\documentclass[final,onefignum,onetabnum]{siamonline250211}

\RequirePackage{amsfonts,amssymb}
\RequirePackage[authoryear]{natbib}
\RequirePackage{graphicx}
\graphicspath{{code_cp/}}

\hypersetup{
  colorlinks=true,
  citecolor=blue,
  urlcolor=blue,
  hypertexnames=false,
  pdftitle={Generalization Error Estimation for Primal--Dual Algorithms in Non-Smooth Regression},
  pdfauthor={Kai Tan and Pierre C. Bellec}
}
\usepackage{derivative}
\usepackage{cuted} 
\usepackage{enumitem}
\usepackage{float}
\usepackage{multirow}
\usepackage{array}

\makeatletter
\@addtoreset{theorem}{section}

\makeatother
\newsiamremark{assumption}{Assumption}
\theoremstyle{plain}
\theoremheaderfont{\normalfont\itshape}
\theorembodyfont{\normalfont}
\theoremseparator{.}
\theoremsymbol{}

\newtheorem{remark}{Remark}[section]
\usepackage{bm}

\newcommand{\boldzero}{\mathbf{0}}
\usepackage{mathtools}
\DeclarePairedDelimiter\norm{\lVert}{\rVert}
\DeclarePairedDelimiter\fnorm{\lVert}{\rVert_{\rm{F}}}
\DeclarePairedDelimiter\opnorm{\lVert}{\rVert_{\rm{op}}}

\def\op{\mathrm{op}}
\def\F{\mathrm{F}}

\def\cbA{\check{\mathbold{A}}}
\def\cbK{\check{\mathbold{K}}}

\def\cbF{\check{\mathbold{F}}}

\def\cbTheta{\check{\mathbold{\Theta}}}

\def\bzero{\mathbf{0}}

\def\bUpsilon{\mathbold{\Upsilon}}

\def\defas{\stackrel{\text{\tiny def}}{=}}

\def\var{{\rm var}}

\newcommand{\limp}{{\smash{\xrightarrow{~\mathrm{p}~}}}}

\def\argmin{\mathop{\rm arg\, min}}
\DeclareMathOperator{\prox}{prox}
\DeclareMathOperator{\soft}{soft}

\DeclareMathOperator{\tr}{Tr}
\DeclareMathOperator{\trace}{Tr}

\let\vec\relax
\DeclareMathOperator{\vec}{\mathbf{vec}}
\DeclareMathOperator{\diag}{\mathbf{diag}}

\usepackage{xspace} 
\def\iid{i.i.d.\@\xspace}

\def\I{{\mathbb{I}}}
\def\R{{\mathbb{R}}}
\def\E{{\mathbb{E}}}

\def\P{{\mathbb{P}}}

\def\mathbold{\boldsymbol} 
\def\cal{\mathcal} 

\def\bvartheta{\mathbold{\vartheta}}

\def\ba{\mathbold{a}}

\def\bA{\mathbold{A}}

\def\bb{\mathbold{b}}
\def\hbb{{\widehat{\bb}}}\def\tbb{{\widetilde{\bb}}}
\def\bB{\mathbold{B}}

\def\bfd{\mathbold{d}}

\def\bD{\mathbold{D}}\def\calD{{\cal D}}

\def\bF{\mathbold{F}}

\def\bG{\mathbold{G}}

\def\bh{\mathbold{h}}
\def\tbh{{\widetilde{\bh}}}
\def\bH{\mathbold{H}}
\def\tbH{{\tilde{\bH}}}

\def\bI{\mathbold{I}}

\def\bJ{\mathbold{J}}

\def\bK{\mathbold{K}}\def\hbK{{\widehat{\bK}}}

\def\bL{\mathbold{L}}

\def\bM{\mathbold{M}}
\def\calM{{\cal M}}

\def\bN{\mathbold{N}}

\def\bP{\mathbold{P}}
\def\calP{{\cal P}}

\def\bq{\mathbold{q}}

\def\bQ{\mathbold{Q}}
\def\calQ{{\cal Q}}

\def\br{\mathbold{r}}

\def\bR{\mathbold{R}}
\def\calR{{\cal R}}

\def\bS{\mathbold{S}}

\def\be{\mathbold{e}}
\def\beps{\mathbold{\epsilon}}

\def\bu{\mathbold{u}}
\def\tbu{{\widetilde{\bu}}}

\def\calU{{\cal U}}

\def\bv{\mathbold{v}}
\def\tbv{{\widetilde{\bv}}}

\def\bw{\mathbold{w}}

\def\bW{\mathbold{W}}\def\hbW{{\widehat{\bW}}}

\def\bx{\mathbold{x}}

\def\bX{\mathbold{X}}\def\tbX{{\tilde{\bX}}}

\def\by{\mathbold{y}}

\def\bz{\mathbold{z}}

\def\bZ{\mathbold{Z}}

\def\bDelta{\mathbold{\Delta}}

\def\ep{\varepsilon}
\def\bep{ {\mathbold{\ep} }}

\def\bzeta{\mathbold{\zeta}}

\def\bfeta{\mathbold{\eta}}

\def\tbfeta{{\widetilde{\bfeta}}}

\def\bTheta{\mathbold{\Theta}}

\def\bxi{\mathbold{\xi}}

\def\tbxi{{\widetilde{\bxi}}}

\def\bSigma{\mathbold{\Sigma}}

\def\bphi{\mathbold{\phi}}

\def\tbphi{{\widetilde{\bphi}}}

\def\bPhi{\mathbold{\Phi}}
\def\tbPhi{{\widetilde{\bPhi}}}

\def\bpsi{\mathbold{\psi}}

\def\tbpsi{{\widetilde{\bpsi}}}

\def\bPsi{\mathbold{\Psi}}

\newcommand\tPhi{\widetilde{\Phi}}
\newcommand\tPsi{\widetilde{\Psi}}

\DeclareMathOperator{\clip}{clip}
\DeclareMathOperator{\poly}{poly}

\begin{document}

\title{
	Generalization Error Estimation for Primal--Dual Algorithms in Non-Smooth Regression
}
\author{Kai Tan\thanks{Department of Statistics, Stanford University, Stanford, USA (\texttt{kaitan9@stanford.edu}).}
\and Pierre C. Bellec\thanks{Department of Statistics, Rutgers University, Piscataway, USA (\texttt{pierre.bellec@rutgers.edu}).}}
\headers{Generalization Error Estimation}{Tan and Bellec}
\maketitle

\begin{abstract}
This paper studies trajectory-wise estimation of generalization error for
primal--dual algorithms in non-smooth regression. Motivating examples include
\(\ell_1\)-penalized least absolute deviations regression and square-root Lasso regression, where the data-fitting loss is non-differentiable and existing risk
estimators for gradient-type optimization paths do not apply directly. We
develop a general recursive framework that includes the Chambolle--Pock
algorithm and related primal--dual splitting methods. We estimate risk by
correcting each in-sample fitted value with a weighted combination of
past dual iterates. The ideal weights are Stein derivative contractions and
depend on the design covariance. We construct replacement weights from
observable derivative contractions of the fitted-signal trajectory, yielding a
covariance-free, data-driven correction. For high-dimensional
Gaussian designs and fixed finite iteration horizon, we prove finite-sample
guarantees for both estimators. For square-root ridge, we further establish a
matched-Gaussian universality result beyond Gaussian designs. Numerical
experiments show that the proposed estimators accurately track the
out-of-sample risk along finite optimization paths.
\end{abstract}




\section{Introduction}

This paper considers regression problems with non-smooth data-fitting losses,
possibly together with non-smooth regularization. Examples include least
absolute deviations regression with an \(\ell_1\) penalty
\citep{wang2007robust} and square-root Lasso regression
\citep{belloni2011square,sun2012scaled}. In least absolute deviations
regression, the usual squared-error loss is replaced by the absolute-error loss,
which reduces sensitivity to heavy-tailed noise and outliers. In square-root
Lasso, a square-root loss is combined with an \(\ell_1\)-penalty to yield a
scale-adaptive high-dimensional sparse estimator. These examples are useful because they depart from smooth
least-squares modeling, and this also changes the optimization
algorithm. Standard differentiable-gradient updates are not directly applicable when the data-fitting loss itself is
non-smooth. Primal--dual splitting methods are designed for such composite
non-smooth problems
\citep{condat2013primal,parikh2014proximal},
and the Chambolle--Pock algorithm \citep{chambolle2011first,chambolle2016introduction} is a prominent
example. We study a general recursive class of primal--dual algorithms that
contains Chambolle--Pock as a special case, and our goal is to estimate the
generalization error of each iterate in a finite trajectory generated by this
recursion.

Prediction risk is often of interest along the entire iteration path, not only
at the final algorithmic output, because prediction performance can improve
during the early iterations and then deteriorate. Reliable risk estimates along
the path can therefore support early stopping and other data-driven tuning
decisions without rerunning the algorithm many times
\citep{raskutti2014early}. Moreover, the exact minimizer is often unavailable in
closed form, and in non-convex problems an algorithm may converge only to a
stationary point. In such settings, the statistically relevant object is the
performance of the iterates actually produced by the algorithm.

Classical methods estimate generalization error in many regression settings, but
they are not designed to provide risk estimates along an iterative algorithmic
path. These methods include cross-validation and sample splitting, generalized
cross-validation for ridge-type estimators \citep{wahba1985comparison},
leave-one-out and approximate leave-one-out methods
\citep{rad2018scalable,wang2018approximate,rad2020error,auddy2024approximate},
and Stein's unbiased risk estimation and related
degrees-of-freedom or derivative-based corrections for penalized regression
\citep{stein1981estimation,zou2007,tibshirani2012,vaiter2012degrees,dossal2013degrees,bellec2020out,bellec2021derivatives}.
These methods mainly target minimizers of the optimization problem, and do not
by themselves provide risk estimators for the iterates of coupled primal--dual
recursions.

Recent work has focused on estimating out-of-sample performance along
iterative optimization paths. \citet{luo2023iterative} proposed iterative
approximate cross-validation for approximating leave-one-out risk along
first-order optimization trajectories. Their approximation is built from
Taylor expansions of a smooth objective, and therefore does not apply directly when the
data-fitting loss itself is non-smooth, as in LAD or square-root-loss
regression. More directly related to our goal are
\cite{bellec2024uncertainty}, which treats full-batch GD and proximal GD for
least-squares problems, and \cite{tan2024estimating}, which treats SGD and
proximal SGD for robust regression with heavy-tailed noise. Both papers study
gradient-type algorithms
whose state is a primal iterate, possibly together with its stochastic-gradient
history. Our paper differs from these works not only in the optimization
algorithm but also in the structure of the correction. For gradient-type primal
algorithms, the correction is built from the primal trajectory: the adjusted
fitted value is obtained by subtracting a weighted combination of past primal
iterates, or of their fitted values. In the primal--dual setting studied here,
our proposed correction to the fitted value is a weighted combination of past dual
iterates. The covariance-dependent estimator uses weights collected in a matrix
\(\bW\), while the covariance-free estimator constructs replacement weights
through an empirical triangular system involving two additional matrices
\(\cbK\) and \(\cbA\). We now preview these two estimators.

Suppose the primal--dual algorithm is run for a fixed finite horizon \(T\),
producing primal iterates \(\hbb^t\) and dual iterates \(\bu^t\) from the data
\((\bX,\by)\).
Here \(\hbb^t\) is the regression-vector iterate and \(\bu^t\) is the
associated dual iterate; the formal recursion is defined in
Section~\ref{sec:chambolle-pock-review}. For a test loss \(\ell\), our target is
the conditional prediction risk
\[
\calR_t
=
\E\!\left[
\ell(y_{\rm new},\bx_{\rm new}^\top\hbb^t)\mid(\bX,\by)
\right],
\]
where \((\bx_{\rm new},y_{\rm new})\) is an independent test observation. We now
preview the two risk estimators proposed in this paper. Both are obtained by
correcting the training error:
\[
\begin{aligned}
\widetilde{\calR}_t
&=
\frac{1}{n}\sum_{i=1}^n
\ell\!\left(y_i,\, \bx_i^\top\hbb^t-\sum_{s=1}^t w_{t+1,s}u_i^s\right),
\\
\widehat{\calR}_t
&=
\frac{1}{n}\sum_{i=1}^n
\ell\!\left(y_i,\, \bx_i^\top\hbb^t-\sum_{s=1}^t \hat w_{t+1,s}u_i^s\right).
\end{aligned}
\]
The distinction of these two estimators is in the weights: \(\widetilde{\calR}_t\) uses
covariance-dependent weights $w_{t,s}$ that require the design covariance \(\bSigma\),
whereas \(\widehat{\calR}_t\) uses fully data-driven weights $\hat w_{t,s}$ that do not require
\(\bSigma\).
Here \(w_{t+1,s}\) and \(\hat w_{t+1,s}\) are entries of two
lower-triangular weight matrices \(\bW\) and \(\hbW\), respectively.
The following trajectory matrices are used to express the correction weights.
Let \(\bb^*\) denote the target regression vector and set
\[
\bH=[\hbb^0-\bb^*,\ldots,\hbb^{T-1}-\bb^*],
\qquad
\bF=[\bu^1,\ldots,\bu^T].
\]
Throughout, \(\be_i\), \(\be_j\), and \(\be_t\) denote the standard basis
vectors of \(\R^n\), \(\R^p\), and \(\R^T\), respectively.
The covariance-dependent estimator $\widetilde{\calR}_t$ uses the population
weight matrix \(\bW\in\R^{T\times T}\). The matrix \(\bW\) is characterized by the derivative-contraction
approximation
\begin{equation}
    \label{eq:W-motivation}
\sum_{i=1}^n\sum_{j=1}^p
\be_i\be_j^\top\bSigma\pdv{\bH}{x_{ij}}
\approx
\bF\bW^\top.
\end{equation}
Thus \(\bW\) is a strictly lower-triangular matrix (because the derivatives of $\hbb^t$ with respect to $x_{ij}$ do not depend on iterates $\bu^{s}$ for $s > t$), and its definition depends
on the design covariance \(\bSigma\).
Here ``derivative-contraction'' means summing out selected indices of a derivative tensor.

To remove this covariance dependence, the covariance-free estimator
\(\widehat{\calR}_t\) replaces \(\bW\) by the data-driven matrix
\(\hbW=\cbK^{-1}\cbA\). The matrices \(\cbK\) and \(\cbA\) are motivated by
similar derivative contractions of the trajectory \(\bX\bH\):

\begin{align}
	\label{eq:K-motivation}
\frac{1}{\sqrt n}\sum_{i=1}^n\sum_{j=1}^p
\be_j\be_i^\top
\pdv{\bX\bH}{x_{ij}}
&\approx
-\bH\cbK^\top,\\
\frac{1}{\sqrt n}\sum_{i=1}^n\sum_{j=1}^p
\be_i\be_j^\top\bX^\top
\pdv{\bX\bH}{x_{ij}}
&\approx
-\bF\cbA^\top.
	\label{eq:A-motivation}
\end{align}
The displays above are only motivational: they involve \(\bH\), and hence the
unknown vector \(\bb^*\). The formal definitions of \(\cbK\) and \(\cbA\),
given in Appendix~\ref{sec:formal-correction-definitions} of the Supplementary
Material, use only the
observed trajectory and its Jacobian blocks; they do not require either
\(\bSigma\) or \(\bb^*\). The matrix \(\cbK\) is lower triangular with diagonal
entries \(-\sqrt n\), so the empirical triangular system is invertible.
The Proposition~\ref{prop:approx-W} in the supplement establishes \(\cbK\bW\approx\cbA\), which
motivates the covariance-free replacement \(\hbW=\cbK^{-1}\cbA\).

The preceding construction leads to three main contributions. First, for a
finite trajectory of the primal--dual recursion, we construct a
covariance-dependent risk estimator whose correction weight matrix \(\bW\) is
built from derivative contractions, and we prove finite-sample accuracy under
high-dimensional Gaussian designs for both centered Lipschitz losses and
centered squared error. Second, we construct a covariance-free estimator by
replacing \(\bW\) with \(\hbW=\cbK^{-1}\cbA\); this replacement is computable by
a triangular solve. We prove finite-sample validity of this fully data-driven
estimator. Third, we prove a universality result for square-root ridge, showing
that the Gaussian-design theory remains valid for covariance-matched
non-Gaussian designs. Numerical experiments support the risk estimates along
finite iteration paths and illustrate the universality phenomenon beyond the
setting covered by the theorem.

The rest of the paper is organized as follows. Section~\ref{sec:chambolle-pock-review}
introduces the general iterative framework and shows how it specializes to the
non-smooth regression problems of interest, with Chambolle--Pock as the motivating
example. Section~\ref{sec:main-result} presents the proposed risk estimators and
their theoretical guarantees. Section~\ref{sec:numerical-experiments} reports
numerical experiments that complement the theory. Technical proofs and
additional simulations are deferred to the Supplementary Material,
in particular Sections~\ref{sec:proof-ingredients} through
\ref{sec:additional-numerical-figures}.

\section{Non-smooth Regression and General Iterations}
\label{sec:chambolle-pock-review}

Suppose we observe data $(\bX,\by)$, with $\bX\in\R^{n\times p}$ and
$\by\in\R^n$, satisfying the linear model
\begin{equation}
	\label{eq:linear-model}
	\by=\bX\bb^*+\bep,
\end{equation}
with unknown regression vector $\bb^*\in\R^p$ and noise vector
$\bep\in\R^n$. We consider penalized regression problems of the form
\begin{equation}
\label{eq:penalized-regression}
	\hbb
	\in
	\argmin_{\bb\in\R^p}
	f_{\by}(\bX\bb)+g(\bb),
\end{equation}
where $f_{\by}:\R^n\to\R$ is a possibly non-differentiable data-fitting loss
and $g:\R^p\to\R$ is a regularization function. Since \(f_{\by}\) may be
non-differentiable and is composed with the linear map \(\bX\bb\), we use a
primal--dual splitting recursion whose updates involve proximal maps of
\(f_{\by}^*\) and \(g\). The concrete losses and penalties used below, together
with their proximal maps, are listed in Tables~\ref{tab:loss-function}
and~\ref{tab:penalty-function}.

We first recall the Chambolle--Pock recursion for
\eqref{eq:penalized-regression}. We then rewrite it using normalized residual
vectors, which is the form used in the derivative analysis.
Given step sizes $\tau_n,\sigma_n>0$, an extrapolation parameter
$\theta\in[0,1]$, and initial values $\bu^0\in\R^n$, $\hbb^0\in\R^p$, with
$\bar{\bb}^0=\hbb^0$, the Chambolle--Pock updates are
\begin{equation}
\label{eq:CP-1}
\begin{aligned}
	\bu^{t+1}
	&=
	\prox[\sigma_n f_{\by}^*]\bigl(\bu^t+\sigma_n\bX\bar{\bb}^t\bigr),\\
	\hbb^{t+1}
	&=
	\prox[\tau_n g]\bigl(\hbb^t-\tau_n\bX^\top \bu^{t+1}\bigr),\\
	\bar{\bb}^{t+1}
	&=
	\hbb^{t+1}+\theta\bigl(\hbb^{t+1}-\hbb^t\bigr),
\end{aligned}
\end{equation}
where $f_{\by}^*$ denotes the convex conjugate of $f_{\by}$, defined by
\begin{align*}
	f_{\by}^*(\bu)
	=
	\sup_{\bv\in\R^n}
	\{\bu^\top\bv-f_{\by}(\bv)\},
\end{align*}
and $\prox[h]$ denotes the proximal operator of a function $h$, namely
\begin{align*}
	\prox[h](\bw)
	=
	\argmin_{\bv}
	\left\{
	\frac12\|\bw-\bv\|^2+h(\bv)
	\right\}.
\end{align*}

\subsection{Examples of common choices of $f$ and $g$}

Tables~\ref{tab:loss-function} and~\ref{tab:penalty-function} list the proximal
maps used in the main examples. They use a generic parameter \(\delta>0\), with
\(\delta=\sigma_n\) in the dual update and \(\delta=\tau_n\) in the primal
update. For the data-fitting losses, the formulas are written with
\(\bv:=\bu-\delta\by\). These examples identify the proximal maps used by the
recursion; the theoretical framework below is stated for a general class of
update maps.

\begin{table}[H]
\centering
\caption{Examples of data-fitting loss functions $f_{\by}$ and their proximal operators $\prox[\delta f_{\by}^*]$.}
\label{tab:loss-function}
\begin{tabular}{l | l}
\hline
$f_{\by}(\bz)$ & $\prox[\delta f_{\by}^*](\bu)$ \\
\hline\hline
$\dfrac{1}{\sqrt{n}}\|\by - \bz\|_1$
&
$\displaystyle
\bigl(\prox[\delta f_{\by}^*](\bu)\bigr)_i
=
\Pi_{[-1/\sqrt n,\,1/\sqrt n]}(u_i-\delta y_i)
$
\\[2ex]
\hline 
$\|\by - \bz\|_2$
&
$\displaystyle
\prox[\delta f_{\by}^*](\bu)
=
\Pi_{\|\cdot\|_2\le 1}(\bu-\delta\by)
$
\\
\hline
\end{tabular}
\end{table}

Here \(\Pi_{[a,b]}\) denotes scalar projection onto \([a,b]\), and
\(\Pi_{\|\cdot\|_2\le 1}\) denotes Euclidean projection onto the unit
\(\ell_2\) ball.

\begin{table}[H]
\centering
\caption{Common penalty functions $g$ and their proximal operators $\prox[\delta g]$.}
\label{tab:penalty-function}
\begin{tabular}{l | l}
\hline
$g(\bb)$ & $\prox[\delta g](\bb)$ \\
\hline\hline

$\lambda \|\bb\|_1$
&
$\displaystyle
\bigl(\prox[\delta g](\bb)\bigr)_j
=
\operatorname{sign}(b_j)
\max\{|b_j|-\delta\lambda,\,0\}
$
\\[2ex]

$\dfrac{\lambda}{2}\|\bb\|_2^2$
&
$\displaystyle
\prox[\delta g](\bb)
=
\frac{1}{1+\delta\lambda}\,\bb
$
\\[2ex]

$\lambda \alpha \|\bb\|_1 + \dfrac{\lambda}{2} (1-\alpha)\|\bb\|_2^2$
&
$\displaystyle
\prox[\delta g](\bb)
=
\frac{1}{1+\delta\lambda(1-\alpha)}
\soft_{\delta\lambda\alpha}(\bb)
$
\\[2ex]

\hline
\end{tabular}
\end{table}

Table~\ref{tab:penalty-function} lists the convex penalties needed for the
main examples and for the square-root ridge universality result.

\subsection{General iteration}

To pass from the concrete Chambolle--Pock update to the abstract recursion, we
first rewrite \eqref{eq:CP-1} in a compact residual form. We use step sizes
\(\sigma_n=\sigma/\sqrt n\) and \(\tau_n=\tau/\sqrt n\), where \(\sigma\) and
\(\tau\) are fixed constants. For the losses in
Table~\ref{tab:loss-function}, we write the dual proximal map as a function of
the residual-form argument \(\ba-\sigma_n\by\), and let \(\bpsi\) denote the
primal proximal map:
\[
\prox[\sigma_n f_{\by}^*](\ba)
=
\bphi(\ba-\sigma_n\by),
\qquad
\bpsi(\cdot)=\prox[\tau_n g](\cdot).
\]
For LAD, \(\bphi\) is the coordinatewise projection onto
\([-1/\sqrt n,1/\sqrt n]\); for square-root loss, it is the Euclidean
projection onto the unit \(\ell_2\) ball. Introduce the normalized residual
vector and the normalized vector
\[
\bv^t=\frac{\by-\bX\hbb^t}{\sqrt n},
\qquad
\bfeta^t=\frac{\bX^\top\bu^t}{\sqrt n}.
\]
Substituting these definitions into \eqref{eq:CP-1} gives the compact recursion
\begin{equation}
\label{eq:CP-update}
\begin{aligned}
	\bu^t
	&=
	\begin{cases}
		\bphi(\bu^{t-1}-\sigma\bv^{t-1}), & t=1,\\
		\bphi\Bigl(\bu^{t-1}-\sigma\bigl((1+\theta)\bv^{t-1}-\theta\bv^{t-2}\bigr)\Bigr), & t\ge 2,
	\end{cases}\\
	\hbb^t
	&=
	\bpsi(\hbb^{t-1}-\tau\bfeta^t).
\end{aligned}
\end{equation}
The extrapolated primal iterate satisfies
\(\bar{\bb}^{t-1}=(1+\theta)\hbb^{t-1}-\theta\hbb^{t-2}\), so
\[
\frac{\by-\bX\bar{\bb}^{t-1}}{\sqrt n}
=
(1+\theta)\bv^{t-1}-\theta\bv^{t-2},
\]
which explains the residual combination in the dual update for \(t\ge2\).
This residual form is the one used in the derivative analysis. It separates the
two kinds of inputs used by the algorithm: residual vectors \(\bv^s\) enter
the dual update, while dual-gradient vectors \(\bfeta^s\) enter the primal
update. In Chambolle--Pock, only the most recent such vectors are active. The
general form below keeps the same separation but allows each update to depend on
the whole available history. Consider
algorithms whose iterates can be written as
\begin{equation}
\label{eq:general-iteration}
\boxed{
\begin{aligned}
\bu^t
&=
\bphi^t(\bu^0,\cdots,\bu^{t-1},
\bv^0,\cdots,
\bv^{t-1}),\\
\hbb^t
&=
\bpsi^t(\hbb^0,\cdots,\hbb^{t-1},
\bfeta^0,\cdots,
\bfeta^t).
\end{aligned}
}
\end{equation}
Here \(\bphi^t\) and \(\bpsi^t\) are history-dependent update maps. They may
depend on \(n\) through the dimensions and step sizes. We assume
\(\bphi^t:\R^{2tn}\to\R^n\) is bounded and Lipschitz continuous, and
\(\bpsi^t:\R^{(2t+1)p}\to\R^p\) is Lipschitz continuous, but not necessarily
bounded. The dimensions correspond to the \(t\) previous \(\bu\)'s and \(t\)
previous \(\bv\)'s in the dual update, and to the \(t\) previous \(\hbb\)'s and
\(t+1\) vectors \(\bfeta\) in the primal update. The Chambolle--Pock recursion
\eqref{eq:CP-update} is obtained by
choosing these maps so that the unused history arguments are ignored.

\begin{remark}
The recursion \eqref{eq:general-iteration} also contains smooth-loss proximal-gradient methods. For example,
if \(f_{\by}(\bz)=\|\by-\bz\|^2/2\), then
\begin{equation}
    \label{prox_gradient_example}
\bu^t=\nabla f_{\by}(\bX\hbb^{t-1}),\qquad
\hbb^t=\prox[\alpha g]\bigl(\hbb^{t-1}-\alpha\bX^\top\bu^t\bigr).
\end{equation}
The form \eqref{eq:general-iteration} also contains linearized ADMM after eliminating the residual
variable and scaled multiplier; see Section~\ref{sec:linearized-admm} of the
Supplementary Material. The
analysis below focuses on the non-smooth primal--dual case.
\end{remark}

Since the update maps are Lipschitz, Rademacher's theorem gives
almost-everywhere Jacobians. At differentiability points, define the Jacobian
blocks
\begin{equation}
\label{eq:DJ-general}
\begin{aligned}
\Phi^u_{t,s}
&=
\pdv{\bphi^t(\cdot)}{\bu^s}\Bigm|_{\substack{\bu^0,\ldots,\bu^{t-1}\\\bv^0,\ldots,\bv^{t-1}}},
\qquad
\Phi^v_{t,s}
=
\pdv{\bphi^t(\cdot)}{\bv^s}\Bigm|_{\substack{\bu^0,\ldots,\bu^{t-1}\\\bv^0,\ldots,\bv^{t-1}}},\\
\Psi^b_{t,s}
&=
\pdv{\bpsi^t(\cdot)}{\hbb^s}\Bigm|_{\substack{\hbb^0,\ldots,\hbb^{t-1}\\\bfeta^0,\ldots,\bfeta^t}},
\qquad
\Psi^\eta_{t,s}
=
\pdv{\bpsi^t(\cdot)}{\bfeta^s}\Bigm|_{\substack{\hbb^0,\ldots,\hbb^{t-1}\\\bfeta^0,\ldots,\bfeta^t}}.
\end{aligned}
\end{equation}
Throughout the derivative analysis, we use the same convention as
\cite[Section~2.2]{bellec2024uncertainty}: the Jacobian blocks are chosen, if
necessary, as modified almost-everywhere derivatives so that the Lipschitz chain
rule holds. When the relevant update map is differentiable at the realized
argument, this convention agrees with the usual Jacobian. For coordinatewise
separable Lipschitz maps, the usual chain rule holds almost everywhere by
\cite[Theorem~2.1.11]{ziemer2012weakly}.

\begin{remark}
For Chambolle--Pock, each update depends only on the recent arguments displayed
in \eqref{eq:CP-update}. Hence many Jacobian blocks in \eqref{eq:DJ-general}
are zero:
\begin{equation}
\begin{aligned}
\Phi^u_{t,s} &= \boldzero_{n\times n},
&& \text{for } 0\le s<t-1,\\
\Phi^v_{t,s} &= \boldzero_{n\times n},
&& \text{for } 0\le s<t-2,\\
\Psi^b_{t,s} &= \boldzero_{p\times p},
&& \text{for } 0\le s<t-1,\\
\Psi^\eta_{t,s} &= \boldzero_{p\times p},
&& \text{for } 0\le s<t.
\end{aligned}
\end{equation}
\end{remark}


\section{Risk Estimation for General Iterations}
\label{sec:main-result}



This section gives the main risk-estimation results for the general recursion
\eqref{eq:general-iteration}. Throughout the section, the iteration horizon
\(T\) is fixed, and constants may depend on \(T\). We first state the
assumptions, define the conditional prediction risk, and explain the high-level
construction of the triangular correction matrices. We then introduce the
covariance-dependent estimator, whose correction uses the population design
covariance \(\bSigma\), followed by the covariance-free estimator that replaces
the population correction by an empirical triangular system. The section ends
with a square-root ridge result showing universality for the true risk and for
both proposed risk estimators.



\begin{assumption}\label{assu:X}
The data follow the linear model
\[
y_i=\bx_i^\top\bb^*+\ep_i,\qquad i=1,\dots,n,
\]
where the design matrix \(\bX\in\R^{n\times p}\) has \iid\ rows drawn from
\(N(\boldzero,\bSigma)\). The noise variables \(\ep_i\) are \iid, independent
of \(\bX\), and satisfy \(\E[\ep_i]=0\). The covariance matrix
\(\bSigma\in\R^{p\times p}\) is positive definite and satisfies
\[
\kappa^{-1}
\le
\lambda_{\min}(\bSigma)
\le
\lambda_{\max}(\bSigma)
\le
\kappa
\]
for some constant \(\kappa>1\). Here \(\lambda_{\min}(\bSigma)\) and
\(\lambda_{\max}(\bSigma)\) denote the smallest and largest eigenvalues of
\(\bSigma\), respectively. In addition, the signal satisfies
\(\norm{\bb^*}\le \zeta\).
\end{assumption}

\begin{assumption}\label{assu:regime}
The sample size $n$ and the ambient dimension $p$ satisfy $p/n\le \gamma$ for some constant $\gamma\in(0,\infty)$.
\end{assumption}

\begin{assumption}\label{assu:algorithm}
The algorithm \eqref{eq:general-iteration} is initialized at
\(\bu^0=\boldzero_n\) and \(\hbb^0=\boldzero_p\). For each
\(t\in\{1,\dots,T\}\), the functions \(\bphi^t\) and \(\bpsi^t\) in
\eqref{eq:general-iteration} are \(\zeta\)-Lipschitz continuous. In addition,
\[
\sup_{\ba\in\R^{2tn}}\norm{\bphi^t(\ba)}\le \zeta,
\qquad
\norm{\bpsi^t(\boldzero)}\le\zeta,
\qquad t=1,\ldots,T,
\]
where the zero vector in the second bound has the appropriate dimension.
\end{assumption}

For fixed iteration horizons, Assumption~\ref{assu:algorithm} is satisfied when
the history-dependent maps \(\bphi^t\) and \(\bpsi^t\), obtained by composing
finitely many one-step updates, have Lipschitz constants depending at most on
the fixed horizon \(T\). This condition follows, for example, if the one-step
maps in a concrete algorithm are uniformly Lipschitz and the intermediate
iterates remain uniformly bounded. In particular, the assumption is satisfied by the
Chambolle--Pock iteration \eqref{eq:CP-update} for the penalized regression
problem \eqref{eq:penalized-regression} with the data-fitting losses listed in
Table~\ref{tab:loss-function} and the convex penalty functions listed in
Table~\ref{tab:penalty-function}.

\subsection{Generalization error}

Let \((\bx_{\rm new},y_{\rm new})\) be an independent copy of a generic training
observation \((\bx_i,y_i)\). For each iterate $\hbb^t$ generated by the
algorithm \eqref{eq:general-iteration}, define its generalization error under a
test loss $\ell(\cdot,\cdot)$ by
\begin{equation}
\label{eq:rt}
\calR_t:=\calR(\hbb^t)
=
\E\bigl[\ell(y_{\rm new},\bx_{\rm new}^\top \hbb^t)\mid (\bX,\by)\bigr].
\end{equation}
We make the following assumptions on the test loss and the noise distribution.
\begin{assumption}\label{assu:test-function}
The test loss is centered at the zero prediction, \(\ell(y,0)=0\). We consider
either the Lipschitz regime
\[
|\ell(y,a)-\ell(y,b)|\le |a-b|,
\qquad y,a,b\in\R,
\qquad \E|\ep_i|\le\zeta,
\]
or the centered squared-error regime
\[
\ell(y,a)=(y-a)^2-y^2,
\qquad \E[\ep_i^2]\le\zeta.
\]
\end{assumption}
This centering is only a convention. Replacing an uncentered loss \(L(y,a)\) by
\(\ell(y,a)=L(y,a)-L(y,0)\) shifts all true and empirical risks by the same
iterate-independent baseline, so risk comparisons and early-stopping decisions
are unchanged. The corresponding uncentered risks can be recovered by adding
this baseline back. The 1-Lipschitz class includes centered absolute, Huber,
pseudo-Huber, \(\epsilon\)-insensitive, and quantile losses after normalization
by their Lipschitz constants.

Before describing the two estimators, we introduce the triangular correction
matrices that appear in their definitions.

\subsection{High-level construction of the triangular matrices}
\label{sec:high-level-formal}

The two estimators below use triangular correction matrices that summarize
derivative contractions of the algorithmic trajectory. The matrix \(\bW\) is the
population correction appearing in the Stein approximation of the leave-one-out
fitted values, and it depends on the design covariance \(\bSigma\). To remove
this covariance dependence, we construct two observable matrices \(\cbK\) and
\(\cbA\) from the observed trajectory and its Jacobian blocks. The formal
definitions are given in Appendix~\ref{sec:formal-correction-definitions} of the
Supplementary Material. This subsection gives the high-level construction. The
matrices \(\bW\), \(\cbK\), and \(\cbA\) are
motivated by the derivative approximations
\eqref{eq:W-motivation}--\eqref{eq:A-motivation}.

At iteration \(t\), the recursion \eqref{eq:general-iteration} proceeds as
follows:
\begin{enumerate}
    \item The dual update map \(\bphi^t\) maps
        \((\bu^0,\ldots,\bu^{t-1},\bv^0,\ldots,\bv^{t-1})\) to
        \(\bu^t\).
    \item Form the dual-gradient vector
        \(\bfeta^t=n^{-1/2}\bX^\top\bu^t\).
    \item The primal update map \(\bpsi^t\) maps
        \((\hbb^0,\ldots,\hbb^{t-1},\bfeta^0,\ldots,\bfeta^t)\) to
        \(\hbb^t\).
    \item Form the residual vector
        \[
        \bv^t=n^{-1/2}(\by-\bX\hbb^t)
        =n^{-1/2}\{\bep-\bX(\hbb^t-\bb^*)\}.
        \]
\end{enumerate}
When we unroll this recursion and apply the chain rule to the derivatives on the
left-hand sides of \eqref{eq:W-motivation}--\eqref{eq:A-motivation}, the two
explicit appearances of \(\bX\) identify two types of trajectory factors:
\begin{itemize}
\item The factor \(n^{-1/2}\bX^\top\bu^t\) in step 2 produces terms
carrying \(\bu^t\), and hence terms associated with the dual-trajectory
matrix \(\bF\) in \eqref{eq:W-motivation} and \eqref{eq:A-motivation}.
\item The factor \(-n^{-1/2}\bX(\hbb^t-\bb^*)\) in step 4 produces terms
carrying \(\hbb^t-\bb^*\), and hence terms associated with the primal-error
matrix \(\bH\) in \eqref{eq:K-motivation}.
\end{itemize}
After recursively substituting the derivatives of earlier iterates, the
chain-rule expansion is grouped by the trajectory factor multiplying each
coefficient. Terms carrying a factor \(\bu^s\) are assigned to the
\(\bF\)-coefficients, and terms carrying a factor \(\hbb^s-\bb^*\) are assigned
to the \(\bH\)-coefficients. The matrix $\bW$ is obtained by collecting all
\(\bF\)-proportional
terms in the chain-rule expansion of the left-hand side of
\eqref{eq:W-motivation}. Similarly, $\cbA$ collects all
\(\bF\)-proportional terms in the expansion of \eqref{eq:A-motivation}, while
\(\cbK\) collects all \(\bH\)-proportional terms in the expansion of
\eqref{eq:K-motivation}. The proof of our main results establishes that the
remaining terms are negligible, so that the approximations
\eqref{eq:W-motivation}--\eqref{eq:A-motivation} hold.

Another way to view this construction is to write
the recursion
\eqref{eq:general-iteration},
with
$$
\bv^t = \frac{\beps-\tbX(\hbb^t - \bb^*)}{\sqrt n},
\qquad
\bfeta^t = \frac{\bar\bX^\top\bu^t}{\sqrt n}
$$
for two matrices $\tbX,\bar\bX$ whose values are always set to $\bX$,
which lets us keep track of the two different appearances of $\bX$ in the chain rule. Holding $\bep$ fixed, the derivative operators in
the coordinate system $\bX$ and the coordinate system
$(\tbX,\bar\bX)$ are related by
$$
\frac{\partial}{\partial x_{ij}}
=
\frac{\partial}{\partial \tilde x_{ij}}
+
\frac{\partial}{\partial \bar x_{ij}}
$$
and the two terms on the right-hand side correspond to the two
bullet points above. Then the description of $\bW$, $\cbK$, and $\cbA$ above is equivalent to
\begin{align*}
\sum_{i=1}^n\sum_{j=1}^p
\be_i\be_j^\top\bSigma\pdv{\bH}{\bar x_{ij}}
&=
\bF\bW^\top,
\\
\frac{1}{\sqrt n}\sum_{i=1}^n\sum_{j=1}^p
\be_j\be_i^\top
\pdv{\tbX\bH}{\tilde x_{ij}}
&=
-\bH\cbK^\top,\\
\frac{1}{\sqrt n}\sum_{i=1}^n\sum_{j=1}^p
\be_i\be_j^\top\bar\bX^\top
\pdv{\tbX\bH}{\bar x_{ij}}
&=
-\bF\cbA^\top
\end{align*}
which are of a similar form
to \eqref{eq:W-motivation}-\eqref{eq:K-motivation}
but with a single derivative operator 
($\partial/\partial \tilde x_{ij}$
or 
$\partial/\partial \bar x_{ij}$)
on the left-hand side
and an equality.


\subsection{Covariance-dependent risk estimator}

We first motivate the estimator from leave-one-out prediction and the Stein
correction, then state its risk-estimation guarantee.

We begin with the leave-one-out (LOO) estimator, which evaluates the \(t\)th
iterate after deleting each observation one at a time:
\begin{align*}
	\frac{1}{n}\sum_{i=1}^n
	\ell\bigl(y_i,\bx_i^\top \hbb^{t,-i}\bigr),
\end{align*}
where \(\hbb^{t,-i}\) denotes the \(t\)th iterate computed after removing the
\(i\)th observation \((\bx_i,y_i)\). Direct computation requires rerunning the
algorithm \(n\) times for each iteration \(t\). Because \(\hbb^{t,-i}\) is
computed without the \(i\)th observation, it is independent of
\((\bx_i,y_i)\), making this quantity the natural leave-one-out proxy for
prediction risk.
The covariance-dependent estimator approximates the LOO fitted values without
recomputing \(\hbb^{t,-i}\). Conditioning on the data with the \(i\)th
observation removed, the vector \(\hbb^{t,-i}\) is fixed. Applying Gaussian
Stein's formula to the map \(\bx_i\mapsto\hbb^t-\hbb^{t,-i}\) gives
\[
\E\Bigl[\bx_i^\top(\hbb^t-\hbb^{t,-i})- \sum_{j=1}^p \be_j^\top \bSigma \pdv{\hbb^t}{x_{ij}}\Bigr]
= 0.
\]
The contraction defining \(\bW\) in \eqref{eq:W-motivation} gives the
approximation
\[
\sum_{j=1}^p \be_j^\top\bSigma\pdv{\hbb^t}{x_{ij}}
\approx
\be_i^\top\bF\bW^\top\be_{t+1},
\qquad
\bF=[\bu^1,\ldots,\bu^T],
\]
where the formal definition of \(\bW\) is given in
Appendix~\ref{sec:formal-correction-definitions}. The remaining
terms are controlled by the risk bound below.
This motivates the fitted-value approximation
\begin{align*}
\bx_i^\top \hbb^{t,-i}
&\approx
\bx_i^\top \hbb^t-\sum_{j=1}^p \be_j^\top \bSigma \pdv{\hbb^t}{x_{ij}}\\
&\approx
\bx_i^\top \hbb^t-\be_i^\top \bF\bW^\top \be_{t+1}.
\end{align*}
Motivated by substituting this approximation for
\(\bx_i^\top\hbb^{t,-i}\) in the LOO estimator, we define the
covariance-dependent estimator as follows.

\begin{definition}
For each $t\in\{0,\ldots,T-1\}$, define the covariance-dependent estimator
$\widetilde{\calR}_t$ of $\calR_t$ in \eqref{eq:rt} by
\begin{align}
\label{eq:rt-tilde}
\widetilde{\calR}_t
=
\frac{1}{n}\sum_{i=1}^n
\ell\Bigl(y_i,\, \bx_i^\top \hbb^t-\sum_{s=1}^t w_{t+1,s}u_i^s\Bigr),
\end{align}
where $w_{t+1,s}$ is the $(t+1,s)$th entry of the weight matrix
$\bW\in\R^{T\times T}$ defined formally in
Appendix~\ref{sec:formal-correction-definitions}. The subtracted term
in the fitted value is the Stein correction.
\end{definition}

The following theorem shows that this estimator is accurate for every fixed
iterate, with error of order \(n^{-1/2}\).

\begin{theorem}
\label{thm:rt-tilde}
Under Assumptions~\ref{assu:X}, \ref{assu:regime}, \ref{assu:algorithm}, and \ref{assu:test-function}, for every fixed \(t\in\{0,\ldots,T-1\}\),
\begin{equation}
\E\bigl[|\widetilde{\calR}_t-\calR_t|\bigr]
\le
\frac{C(T,\gamma,\zeta,\kappa)}{\sqrt n}.
\end{equation}
\end{theorem}
Theorem~\ref{thm:rt-tilde} can be viewed as a primal--dual generalization of
the covariance-dependent risk-estimation result for SGD in
\citet[Theorem~3.6]{tan2024estimating}. The correction is
again derived from Gaussian Stein's formula, but the weights are now generated
by derivative contractions of a coupled primal--dual recursion. 

The result also
covers Lipschitz test losses, in addition to centered squared error; this
requires a different comparison argument because the squared-error expansion used
in the earlier setting is no longer available. The proof is given in
Appendix~\ref{sec:proof-rt-tilde} of the Supplementary Material.

\subsection{Covariance-free risk estimator}

The covariance-dependent estimator uses the matrix \(\bW\), which depends on
\(\bSigma\). The covariance-free estimator replaces it by the observable matrix
\[
\hbW:=\cbK^{-1}\cbA.
\]
The lower-triangular structure of \(\cbK\), with diagonal entries \(-\sqrt n\),
makes this triangular solve well-defined.
Here \(\cbK\) and \(\cbA\) are the observable correction matrices previewed in
\eqref{eq:K-motivation}--\eqref{eq:A-motivation} and described in
Section~\ref{sec:high-level-formal}; their row-by-row computation is given in
Appendix~\ref{sec:computation-tricks}. Replacing \(\bW\) by \(\hbW\) in the
covariance-dependent estimator \eqref{eq:rt-tilde} gives the following
covariance-free estimator.

\begin{definition}
For each \(t\in\{0,\ldots,T-1\}\), define the covariance-free estimator
\(\widehat{\calR}_t\) of \(\calR_t\) by
\begin{equation}
\label{eq:rt-hat}
\widehat{\calR}_t
=
\frac{1}{n}\sum_{i=1}^n
\ell\Bigl(y_i,\bx_i^\top\hbb^t-\sum_{s=1}^t \hat w_{t+1,s}u_i^s\Bigr),
\end{equation}
where $\hat w_{t+1,s}$ denotes the $(t+1,s)$ entry of the matrix $\hbW$.
\end{definition}

This estimator is data-driven: it is computed from the training data, the
observed trajectory, and the corresponding Jacobian blocks, without knowledge of
the population covariance \(\bSigma\). We first compare
\(\widehat{\calR}_t\) to \(\widetilde{\calR}_t\). Combined with
Theorem~\ref{thm:rt-tilde}, this yields consistency for the true risk
\(\calR_t\).
The result is stated on the high-probability event controlling the
operator norm of the design:
\begin{equation}
\label{eq:event-Omega}
\Omega = \left\{\bX\in \R^{n\times p}:
\frac{\opnorm{\bX}}{\sqrt{n}} \le \sqrt{\kappa}(2+\sqrt{\gamma})
\right\}.
\end{equation}
Under Assumptions~\ref{assu:X} and \ref{assu:regime}, standard random matrix theory gives \(\P(\Omega^c)\le e^{-n/2}\).

\begin{theorem}
\label{thm:rt-hat}
Under Assumptions~\ref{assu:X}, \ref{assu:regime}, \ref{assu:algorithm}, and
\ref{assu:test-function}, for every fixed
\(t\in\{0,\ldots,T-1\}\),
\begin{equation}
\E\bigl[\I_\Omega\,|\widehat{\calR}_t-\widetilde{\calR}_t|\bigr]
\le
C(T,\gamma,\zeta,\kappa)\,n^{-1/4}.
\end{equation}
Consequently, combining this bound with Theorem~\ref{thm:rt-tilde}, for every
\(\epsilon>0\),
\begin{equation}
\P\bigl(|\widehat{\calR}_t-\calR_t|>\epsilon\bigr)
\le
e^{-n/2}
+
\frac{C(T,\gamma,\zeta,\kappa)}{\epsilon\,n^{1/4}}.
\end{equation}
\end{theorem}

The slower \(n^{-1/4}\) rate arises from replacing the population correction
matrix \(\bW\) by its data-driven counterpart \(\hbW\). This replacement
introduces an additional error in passing from the covariance-dependent
estimator \(\widetilde{\calR}_t\) to the covariance-free estimator
\(\widehat{\calR}_t\). Our current analysis yields only an
\(n^{-1/4}\) bound for this error, and we do not expect this rate to be
optimal.

Theorem~\ref{thm:rt-hat} extends the data-driven pathwise risk-estimation
results of \citet[Theorem~2.1]{bellec2024uncertainty} and
\citet[Theorem~3.7]{tan2024estimating} to the present primal--dual setting. In
the square-loss full-batch gradient descent and proximal-gradient settings of
\citet{bellec2024uncertainty}, the covariance-free correction is determined by
one empirical memory matrix. In the present notation,
this is the identity
\(
\cbK=-\sqrt n(\bI_T-\cbA/n)
\)
(valid only in the setting of \citet{bellec2024uncertainty}),
so \(\hbW=\cbK^{-1}\cbA\) is a function of \(\cbA\) alone.

This algebraic simplification does not extend beyond squared loss or to SGD, where random sampling breaks the full-batch identity. It also does not extend to
the Chambolle--Pock recursion studied here: the derivative contractions defining
\(\cbK\) and \(\cbA\) remain distinct, so the covariance-free correction must
use the triangular system \(\cbK\hbW=\cbA\). In contrast to the corresponding
matrix in \citet{tan2024estimating}, which is invertible only on a
high-probability event, \(\cbK\) from the current paper is always invertible because it
is lower triangular with diagonal entries \(-\sqrt n\).

The proof of Theorem~\ref{thm:rt-hat} is given in
Appendix~\ref{sec:proof-rt-hat} of the Supplementary Material. It also
justifies using \(\widehat{\calR}_t\) to choose an early stopping time. Since
only finitely many iterates are compared, uniform control of the estimation
error implies that the selected iterate has risk close to the best risk along
the computed path.

\begin{corollary}[Risk estimation for early stopping]
\label{cor:early-stopping}
Under the assumptions of Theorem~\ref{thm:rt-hat}, let
\[
\widehat t\in\argmin_{0\le t<T}\widehat{\calR}_t,
\qquad
t^\star\in\argmin_{0\le t<T}\calR_t,
\]
with ties broken arbitrarily. Then, for every \(\epsilon>0\),
\[
\P\left(
\calR_{\widehat t}-\calR_{t^\star}>\epsilon
\right)
\le
T e^{-n/2}
+
\frac{2T\,C(T,\gamma,\zeta,\kappa)}{\epsilon\,n^{1/4}}.
\]
\end{corollary}

Thus, for fixed \(T\), minimizing \(\widehat{\calR}_t\) is asymptotically
oracle optimal among the first \(T\) iterates.
The proof of Corollary~\ref{cor:early-stopping} is given in Appendix~\ref{sec:proof-early-stopping}
of the Supplementary Material.

\subsection{Beyond Gaussian designs: universality for square-root ridge}

We next give a non-Gaussian extension of the risk-estimation results in a
setting where the primal--dual trajectory can be analyzed spectrally:
square-root ridge at a fixed iteration horizon, with squared test loss. The
restriction to ridge is useful because the ridge proximal map is linear. As a
result, the Chambolle--Pock iterates and the derivative contractions defining
\(\bW,\cbK,\cbA\) admit finite-dimensional spectral representations generated by
\(\bX^\top\bX\) and \(\bSigma\). This structure allows a direct comparison
between a non-Gaussian design and its covariance-matched Gaussian counterpart.
The uniform-design simulations in Section~\ref{sec:numerical-experiments}
suggest that the same universality phenomenon may hold more broadly, but the
theorem below is restricted to this square-root ridge setting.

\begin{assumption}[Square-root ridge universality setting]
\label{assu:sqrt-ridge-universality}
In this subsection, \(p=p_n\), \(\bSigma_n\), and \(\bb^*=\bb_n^*\) may depend
on \(n\). The following conditions hold.
\begin{enumerate}[label=(\roman*)]
\item Square-root ridge is run with penalty
\(g(\bb)=\lambda\|\bb\|_2^2/2\), where \(\lambda>0\) is fixed, and prediction
risk is evaluated using squared test loss. The Chambolle--Pock constants
\((\sigma,\tau,\theta)\) are fixed, and the step sizes are
\(\sigma_n=\sigma/\sqrt n\) and \(\tau_n=\tau/\sqrt n\). The recursion is
initialized at \(\bu^0=\boldzero_n\) and \(\hbb^0=\boldzero_p\).

\item The non-Gaussian design and its matched Gaussian design are
\[
\bX=\bZ\bSigma_n^{1/2},
\qquad
\bX_{\rm G}=\bG\bSigma_n^{1/2},
\]
where \(\bG\) has iid standard Gaussian entries and \(\bZ\) has iid mean-zero,
variance-one entries with uniformly bounded sub-Gaussian norm. The matrices
\(\bZ\) and \(\bG\) are independent.

\item The asymptotic regime satisfies \(p/n\to\gamma_0\in(0,\infty)\). The
covariance matrices \(\bSigma_n\) have eigenvalues bounded above and below by
positive constants, and \(\norm{\bb^*}_2\) is bounded. For each fixed
\(k\ge0\), the empirical spectral moments \(p^{-1}\trace(\bSigma_n^k)\) and the
signal-alignment moments \(\bb^{*\top}\bSigma_n^k\bb^*\) have limits.

\item The noise variables are iid, independent of \(\bZ\) and \(\bG\), with
mean zero, variance \(\sigma_\ep^2\), and finite fourth moment.
\end{enumerate}
\end{assumption}

Let \(\by=\bX\bb^*+\bep\) and
\(\by_{\rm G}=\bX_{\rm G}\bb^*+\bep\), where the noise realization in the
two training samples is the same. Run the same Chambolle--Pock recursion on the
non-Gaussian data \((\bX,\by)\) and on the matched Gaussian data
\((\bX_{\rm G},\by_{\rm G})\). In both models, prediction risk is evaluated
using a fresh Gaussian test point with covariance \(\bSigma_n\) and an
independent noise draw with the same law. Thus the two risks differ only through
the training design, not through the test distribution. The shared-noise
coupling is used only to compare the two trajectories; each marginal model has
the noise distribution specified in
Assumption~\ref{assu:sqrt-ridge-universality}. We use superscript \({\rm Z}\)
for the non-Gaussian design and superscript \({\rm G}\) for the matched
Gaussian design. Thus
\[
\calR_t^{\rm Z},\ \widetilde{\calR}_t^{\rm Z},\ \widehat{\calR}_t^{\rm Z}
\quad\text{and}\quad
\calR_t^{\rm G},\ \widetilde{\calR}_t^{\rm G},\ \widehat{\calR}_t^{\rm G}
\]
denote the corresponding risks and estimators in the two models.

\begin{assumption}[No-boundary condition]
\label{assu:sqrt-ridge-no-boundary}
For the fixed \(t\) under consideration, let \(\mu_s\) denote the deterministic
limit of the squared norm of the vector entering the square-root projection
\(\bz\mapsto \bz/\max\{1,\|\bz\|_2\}\) at iteration \(s\), as constructed in
Appendix~\ref{sec:universality}. We assume
\[
\mu_s\neq 1,\qquad s=1,\ldots,t.
\]
The value \(\mu_s=1\) corresponds to the nondifferentiability boundary of the
projection. Thus this condition keeps both the non-Gaussian and matched
Gaussian trajectories away from that boundary with probability tending to one.
\end{assumption}

Informally, this is a nondegeneracy condition: the limiting trajectory does not
land on the kink of the square-root projection map.
Assumption~\ref{assu:sqrt-ridge-no-boundary} is used only for the estimator
universality statements below. The true-risk comparison uses only the
continuity of the square-root projection, while the correction estimators use
its Jacobian, which is discontinuous at the projection boundary.

\begin{theorem}[Universality for square-root ridge]
\label{thm:main-universality-sqrt-ridge}
Under Assumption~\ref{assu:sqrt-ridge-universality}, for every fixed
\(t\in\{0,\ldots,T-1\}\),
\[
\calR_t^{\rm Z}-\calR_t^{\rm G}\limp 0.
\]
If Assumption~\ref{assu:sqrt-ridge-no-boundary} also holds,
then
\[
\widetilde{\calR}_t^{\rm Z}-\widetilde{\calR}_t^{\rm G}\limp 0,
\qquad
\widehat{\calR}_t^{\rm Z}-\widehat{\calR}_t^{\rm G}\limp 0.
\]
Combining these universality comparisons with the Gaussian-design bounds in
Theorems~\ref{thm:rt-tilde} and~\ref{thm:rt-hat} gives
the following consistency statements for the non-Gaussian design:
\[
\widetilde{\calR}_t^{\rm Z}-\calR_t^{\rm Z}\limp 0,
\qquad
\widehat{\calR}_t^{\rm Z}-\calR_t^{\rm Z}\limp 0.
\]
\end{theorem}

The proof is given in Appendix~\ref{sec:universality} of the
Supplementary Material. It uses a Krylov
representation of the square-root ridge trajectory and of the correction
weights. This representation exploits the linear spectral structure of the ridge
proximal map, and the squared test loss allows the risk to be expressed through
the same spectral quantities. These are the reasons the present universality
proof is limited to ridge penalties and squared test loss.

\section{Numerical experiments}
\label{sec:numerical-experiments}

This section studies the finite-sample behavior of the proposed risk
estimators along the computed Chambolle--Pock path. We first compare the true
conditional risk with the covariance-dependent estimator
\(\widetilde{\calR}_t\) and the covariance-free estimator
\(\widehat{\calR}_t\) for nonsmooth regression problems with several penalties.
We then examine design universality by comparing Gaussian designs with
covariance-matched uniform designs. The main universality comparison is for
square-root ridge, matching Theorem~\ref{thm:main-universality-sqrt-ridge};
the remaining Gaussian-versus-uniform comparisons are empirical robustness
checks beyond the theorem. We next describe the common simulation design, the
algorithms used to generate the trajectories, and the way the risks are
evaluated.

\subsection{Simulation design}

We generate data from the linear model
\(\by=\bX\bb^*+\bep\) with sample size \(n=2000\), ambient dimensions
\(p=1000\) and \(p=2400\), and iteration horizon \(T=300\). The covariance
matrix has Toeplitz form
\((\bSigma)_{jk}=0.5^{|j-k|}\). The main text reports the \(p=1000\) case; the
\(p=2400\) case is deferred to
Section~\ref{sec:additional-numerical-figures} of the Supplementary Material; the
same section also contains additional simulations for the remaining method
combinations.
In the Gaussian-design panels,
the rows of \(\bX\) are sampled independently from \(N(\bzero,\bSigma)\). For the
universality plots, we compare this design with the matched non-Gaussian design
\(\bX=\bZ\bSigma^{1/2}\), where the entries of \(\bZ\) are sampled
independently from the uniform distribution on \([-\sqrt 3,\sqrt 3]\). This
choice gives mean zero and variance one before the covariance transformation,
so the two designs have the same population covariance. The true coefficient
vector \(\bb^*\) is sparse: only the first \(\lfloor p/10\rfloor\) entries are
nonzero, and it is scaled so that \(\|\bSigma^{1/2}\bb^*\|^2=2\). Table
\ref{tab:regression-methods} lists the noise distribution and test loss used for
each data-fitting loss. Each plotted curve is the Monte Carlo mean over \(100\)
independent simulation runs; shaded bands, when present, show \(\pm 2\) standard
errors.

We consider eight procedures obtained by combining
two data-fitting losses with four penalties. The loss is either the normalized
LAD loss \(n^{-1/2}\|\by-\bX\bb\|_1\) or the square-root loss
\(\|\by-\bX\bb\|_2\). The penalty is ridge, Lasso, MCP
\citep{zhang10-mc+}, or SCAD \citep{fan2001variable}. The resulting methods and
evaluation regimes are listed in Table~\ref{tab:regression-methods}.
MCP and SCAD are included only
as empirical checks using their single-valued proximal maps in
Appendix~\ref{sec:additional-proximal-formulas} of the Supplementary Material;
no convergence guarantee is
claimed for these non-convex objectives.

\begin{table}[ht]
\centering
\small
\renewcommand{\arraystretch}{1.2}
\setlength{\tabcolsep}{4pt}
\begin{tabular}{lll|cccc}
\hline
\multirow{2}{*}{\textbf{Data-fitting loss}}
& \multirow{2}{*}{\textbf{Noise}}
& \multirow{2}{*}{\textbf{Test loss}}
& \multicolumn{4}{c}{\textbf{Penalty}} \\
\cline{4-7}
 & & & Ridge & Lasso & MCP & SCAD \\
\hline
LAD
& Student-\(t_2\)
& Absolute value
& \(\checkmark\) & \(\checkmark\) & \(\checkmark\) & \(\checkmark\) \\
Square-root loss
& $N(0,1)$
& Squared error
& \(\checkmark\) & \(\checkmark\) & \(\checkmark\) & \(\checkmark\) \\
\hline
\end{tabular}
\caption{Simulation regimes and penalty combinations considered in the experiments.}
\label{tab:regression-methods}
\end{table}

For all methods in Table~\ref{tab:regression-methods},
we run the residual-form Chambolle--Pock recursion \eqref{eq:CP-update} with
constants \((\theta,\sigma,\tau)=(0.5,0.1,0.1)\), equivalently step sizes
\(\sigma_n=\sigma/\sqrt n\) and \(\tau_n=\tau/\sqrt n\). We initialize
\(\bu^0=\boldzero_n\) and \(\hbb^0=\boldzero_p\), and run \(T=300\)
iterations.

The dual map \(\bphi\) is determined by the data-fitting loss and is listed in
Table~\ref{tab:loss-function}; the primal map \(\bpsi\) is determined by the
penalty. The ridge and Lasso proximal maps are listed in
Table~\ref{tab:penalty-function}, and the MCP/SCAD proximal maps are given in
Section~\ref{sec:additional-proximal-formulas} of the Supplementary Material. We set
\(\lambda=0.5\), \(\gamma_{\mathrm{mcp}}=3\) for MCP, and \(a=3.7\)
for SCAD. Computing the risk estimators requires \(\bW\) and the matrices
\(\cbK\) and \(\cbA\) used to obtain \(\hbW\). We approximate the trace terms
in their entrywise definitions \eqref{eq:W-row-recursion}--\eqref{eq:A-row-recursion}
by Hutchinson's method, using \(m=5\) independent Rademacher probe vectors.
Propagating the probe products through the linearized Chambolle--Pock recursion
requires \(O(mnpT^2)\) operations and \(O(m(n+p)T+T^2)\) memory, followed by an
\(O(T^3)\) lower-triangular solve for \(\hbW\). The true-risk curves below do
not use this stochastic approximation.

The figures use uncentered risks
on the standard test-loss scale; this differs from
Assumption~\ref{assu:test-function} only by an iterate-independent baseline. For
LAD-based procedures, the benchmark is
\[
\calR_t^{\rm abs}
=
\E\bigl[
|y_{\rm new}-\bx_{\rm new}^\top\hbb^t|
\mid(\bX,\by)
\bigr].
\]
For square-root-loss procedures, the benchmark is
\[
\calR_t^{\rm sq}
=
\E\bigl[
(y_{\rm new}-\bx_{\rm new}^\top\hbb^t)^2
\mid(\bX,\by)
\bigr]
=
(\hbb^t-\bb^*)^\top\bSigma(\hbb^t-\bb^*)+1,
\]
and the corrected estimators use the same test loss as their benchmark. For the
absolute-value benchmark, the Gaussian and uniform cases are evaluated by the
corresponding deterministic one-dimensional integrals; the squared-risk
benchmark is computed from the displayed formula.

\subsection{Results}

Figure~\ref{fig:gaussian-risk-comparison} shows the \(n>p\) risk-estimation
comparison under Gaussian design. Columns correspond to penalties, and rows
correspond to data-fitting loss. The LAD rows use
absolute-value risk and the square-root rows use squared risk, so the vertical
scales should be read row by row. The purpose of these plots is to assess risk
estimation along the computed finite path, not only at a terminal or converged
iterate; some trajectories have nearly stabilized by the end of the run, while
others have not. Across all panels, \(\widetilde{\calR}_t\) and
\(\widehat{\calR}_t\) track the true test-risk curve and the location of its minimum along the computed path. The Supplementary Material reports the corresponding
\(n<p\) comparison in Section~\ref{sec:additional-numerical-figures}.

Figure~\ref{fig:universality-square-root-ridge} gives the main empirical
universality check. It focuses on square-root ridge, the setting covered by
Theorem~\ref{thm:main-universality-sqrt-ridge}, and directly compares the
Gaussian design with the covariance-matched uniform design. The three panels
show the true risk \(\calR_t\), the covariance-dependent estimator
\(\widetilde{\calR}_t\), and the covariance-free estimator
\(\widehat{\calR}_t\). In all three panels, the Gaussian and uniform curves
nearly overlap, illustrating the universality comparison for both the true risk
and both risk estimators.
Section~\ref{sec:additional-numerical-figures} includes additional Gaussian-versus-uniform
comparisons for the \(n<p\) square-root ridge setting and other loss--penalty
combinations. These empirical results suggest that the universality phenomenon
may extend more broadly.
\begin{figure}[!tbp]
	\centering
	\includegraphics[width=0.98\textwidth]{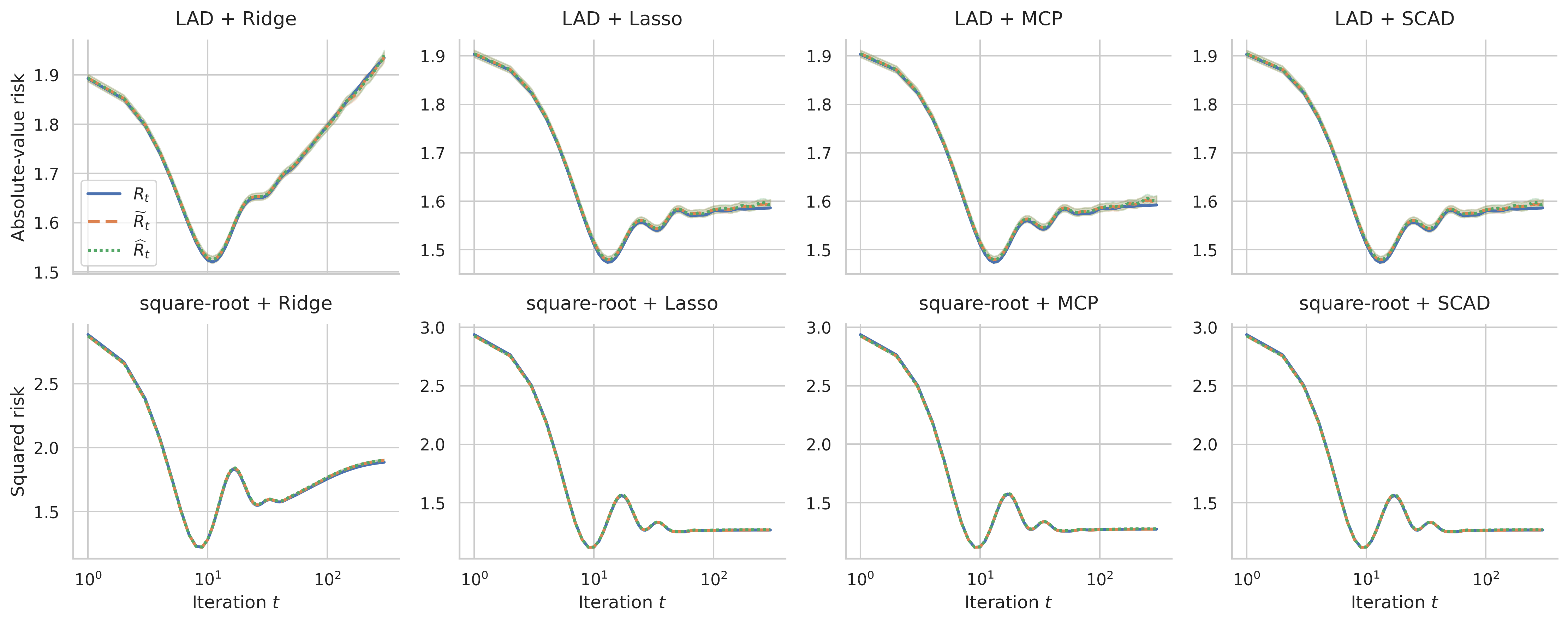}
		\caption[Matched test-loss risk estimates for n larger than p]{Risk estimates over the iteration path with \((n,p,T)=(2000,1000,300)\). Columns correspond to penalties, and rows correspond to data-fitting loss. The curves show the true test risk \(\calR_t\), the covariance-dependent estimator \(\widetilde{\calR}_t\), and the covariance-free estimator \(\widehat{\calR}_t\).}
	\label{fig:gaussian-risk-comparison}
\end{figure}

\begin{figure}[H]
	\centering
	\includegraphics[width=0.96\textwidth]{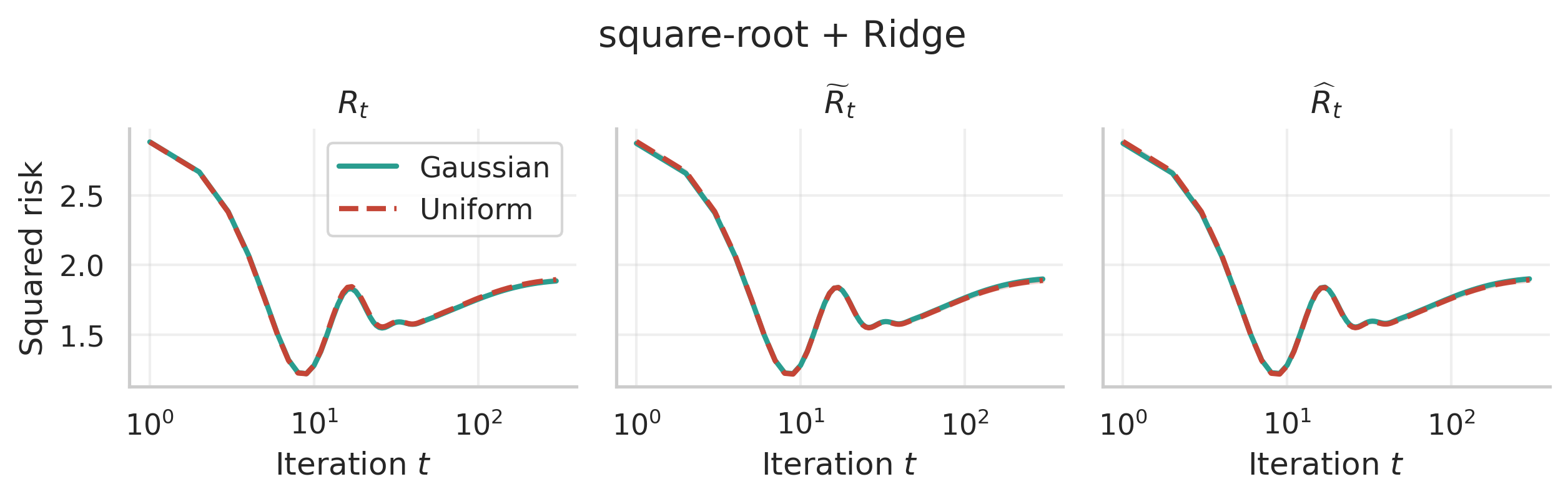}
	\caption[Universality for square-root ridge]{Gaussian-versus-uniform comparison for square-root ridge with \((n,p,T)=(2000,1000,300)\), using squared-error risk. The left, middle, and right panels compare \(\calR_t\), \(\widetilde{\calR}_t\), and \(\widehat{\calR}_t\), respectively; overlap of the paired curves illustrates the universality comparison.}
	\label{fig:universality-square-root-ridge}
\end{figure} 

\section{Discussion}

This paper develops two risk estimators for nonsmooth regression problems solved
by primal--dual algorithms, with the Chambolle--Pock recursion
\eqref{eq:CP-1} as the main example. The first estimator is
covariance-dependent and uses the population covariance in its correction
weights. The second estimator is covariance-free and replaces these weights by
observable derivative contractions computed from the algorithmic trajectory. For
Gaussian designs and a fixed iteration horizon, both estimators are consistent
for the conditional prediction risk. For square-root ridge, we further
prove a universality result showing that the risk and both correction estimators
have the same limits under Gaussian and covariance-matched non-Gaussian designs.

The analysis also points to several directions for further work. The main risk
bounds are finite-time results, with the iteration horizon fixed as the sample
size grows, and are proved under Gaussian designs. The universality theorem is
more specialized: it is established for square-root ridge, where the ridge
proximal map yields a spectral representation of both the iterates and the
correction terms. Extending the theory to growing iteration horizons, to
nonlinear proximal maps such as square-root Lasso, and to broader non-Gaussian
designs would make the results more directly applicable to long optimization
paths and more complex penalties. 
These
directions would further clarify how the proposed risk estimators can be used as a general tool for monitoring and tuning modern iterative regression
algorithms.

\bibliographystyle{plainnat}
\bibliography{biblio}    

\newpage
\setcounter{page}{1}
\begin{center}
{\LARGE\sc Supplementary Material}
\bigskip
\end{center}
\appendix
\noindent\textbf{Appendix roadmap.}
The appendix has two proof streams, followed by computational and numerical
supplements. The first proof stream establishes the Gaussian-design risk bounds
and the covariance-free replacement \(\hbW=\cbK^{-1}\cbA\). The second proves
the square-root ridge universality result. The remaining sections give the
formal correction definitions, their efficient computation, and additional
algorithmic and numerical material. Table~\ref{tab:appendix-roadmap}
summarizes the role of each appendix section and the results it supports.

\begin{table}[!htbp]
\centering
\small
\setlength{\tabcolsep}{4pt}
\renewcommand{\arraystretch}{1.12}
\begin{tabular}{
>{\raggedright\arraybackslash}p{0.18\textwidth}
>{\raggedright\arraybackslash}p{0.48\textwidth}
>{\raggedright\arraybackslash}p{0.25\textwidth}}
\hline
\textbf{Section} & \textbf{Task} & \textbf{Used for}\\
\hline
Section~\ref{sec:proof-ingredients}
& Decomposes the estimation error into \(R_1,\ldots,R_5\) and proves the
leave-one-out, Stein-correction, and residual bounds needed for these terms.
& Theorems~\ref{thm:rt-tilde} and~\ref{thm:rt-hat}.\\
Section~\ref{sec:proof-main}
& Assembles the residual bounds, applies the isotropic reduction, and proves the
main Gaussian-design guarantees.
& Theorems~\ref{thm:rt-tilde}--\ref{thm:rt-hat} and
Corollary~\ref{cor:early-stopping}.\\
Section~\ref{sec:useful-lemmas}
& Derives the pathwise derivative identities, Stein contractions, and algebraic
relations that produce the correction matrices.
& The residual bounds in Section~\ref{sec:proof-ingredients}.\\
Section~\ref{sec:reduction-isotropic}
& Reduces the general covariance model to the isotropic Gaussian analysis while
tracking how \(\bW\), \(\cbK\), and \(\cbA\) transform.
& Theorems~\ref{thm:rt-tilde} and~\ref{thm:rt-hat} under general covariance.\\
Section~\ref{sec:norm-moment-bounds}
& Collects operator-norm, inverse-\(\cbK\), trajectory, and derivative moment
bounds.
& All residual controls, especially the covariance-free \(R_5\) bound.\\
Section~\ref{sec:auxiliary-proofs}
& Proves the general derivative formula used by the Stein and contraction
arguments.
& Section~\ref{sec:useful-lemmas}.\\
Section~\ref{sec:universality}
& Proves the finite-dimensional Krylov representation and polynomial
universality argument for square-root ridge.
& Theorem~\ref{thm:main-universality-sqrt-ridge}.\\
Section~\ref{sec:computation-tricks}
& Gives the formal definitions and row-by-row computation of
\(\bW\), \(\cbK\), and \(\cbA\), including Hutchinson trace approximations.
& Implementation of \(\widetilde{\calR}_t\) and
\(\widehat{\calR}_t\).\\
Section~\ref{sec:algorithmic-supplements}
& Records additional proximal formulas, the linearized ADMM reduction, and
simulation figures omitted from the main text.
& Examples and numerical experiments.\\
\hline
\end{tabular}
\caption{Appendix roadmap.}
\label{tab:appendix-roadmap}
\end{table}

\section{Proof strategy and residual bounds}\label{sec:proof-ingredients}
We first reduce the two risk-estimation claims to five residual terms. The
bounds for these residuals yield the main theorems directly in the next
section. Their technical inputs--the derivative identities, covariance
reduction, and moment estimates--are established later in the appendix.
The analysis is split according to the two cases in
Assumption~\ref{assu:test-function}: globally 1-Lipschitz losses use only the
first moment of the noise, while the centered squared loss uses the second
moment.

Throughout this appendix, for a positive integer \(m\), write
\([m]=\{1,\ldots,m\}\), and write \(\I_E\) for the indicator of an event \(E\).
Let
\(\calD = \{(\bx_i,y_i):i\in[n]\}\) denote the full data set,
\(\calD_{-i} = \{(\bx_k,y_k):k\in[n]\setminus\{i\}\}\) the data set excluding the
\(i\)th sample, and
\(\calD_{-(i,l)} = \{(\bx_k,y_k):k\in[n]\setminus\{i,l\}\}\) the data set
excluding the \(i\)th and \(l\)th samples. We denote the iterates computed from
\(\calD\), \(\calD_{-i}\), and \(\calD_{-(i,l)}\) by \(\hbb^t\),
\(\hbb^{t,-i}\), and \(\hbb^{t,-(i,l)}\), respectively. Throughout this
leave-one-out analysis, the deleted and double-deleted iterates are defined by
rerunning the same recursion on the reduced datasets while keeping the ambient
normalization \(1/\sqrt n\), the step sizes, and all tuning parameters fixed at
their full-data values.

\subsection{Error decomposition}

The generalization risk for the iterate \(\hbb^t\) is
\[
\calR_t \defas \E[\ell(y_{\rm new}, \bx_{\rm new}^\top \hbb^{t}) \mid (\bX,\by)],
\]
where \((\bX,\by)\) is the training data and the expectation is taken over a fresh sample \((\bx_{\rm new},y_{\rm new})\) with the same distribution as the training observations.

The following decomposition is written in isotropic coordinates,
\(\bSigma=\bI_p\). For a general covariance matrix, the Stein divergence in the
original coordinates contains the covariance weight \(\bSigma\); the
change-of-variables argument in Section~\ref{sec:reduction-isotropic} converts
it to the unweighted form displayed below and transfers the final bounds back to
the original model.
We will use the following approximations:
\begin{align*}
	\calR_t 
	& \approx \frac 1n \sum_{i=1}^n \E[\ell(y_i, \bx_i^\top \hbb^{t,-i})| \calD_{-i}]\\
	& \approx \frac 1n \sum_{i=1}^n \ell(y_i, \bx_i^\top \hbb^{t,-i})\\
	&\approx \frac 1n \sum_{i=1}^n \ell(y_i, \bx_i^\top \hbb^{t} - \sum_j \pdv{\hbb^t_j}{x_{ij}}) \\
	& \approx 
	\underbrace{\frac 1n \sum_{i=1}^n \ell\big(y_i, \bx_i^\top \hbb^{t} - 
	\be_i^\top \bF \bW^\top\be_{t+1}\big)}_{\widetilde{\calR}_t}\\
	& \approx 
	\underbrace{\frac 1n \sum_{i=1}^n \ell\big(y_i, \bx_i^\top \hbb^{t} - 
	\be_i^\top \bF \hbW^\top\be_{t+1}\big)}_{\widehat{\calR}_t}.
\end{align*}

We denote the five approximation errors as \(R_1,\dots,R_5\), namely
\begin{align}
	&R_1 = \E[\ell(y_{\rm new}, \bx_{\rm new}^\top \hbb^{t})| \calD] - \frac 1n \sum_{i=1}^n \E[\ell(y_i, \bx_i^\top \hbb^{t,-i})| \calD_{-i}],\label{R1}\\
	&R_2 = \frac 1n \sum_{i=1}^n \E[\ell(y_i, \bx_i^\top \hbb^{t,-i})| \calD_{-i}] - \frac 1n \sum_{i=1}^n [\ell(y_i, \bx_i^\top \hbb^{t,-i})],\label{R2}\\
	&R_3 = \frac 1n \sum_{i=1}^n \ell(y_i, \bx_i^\top \hbb^{t,-i}) - \frac 1n \sum_{i=1}^n \ell\Bigl(y_i, \bx_i^\top \hbb^{t} - \sum_j \pdv{\hbb^t_j}{x_{ij}}\Bigr),\label{R3}\\
	&R_4 = \frac 1n \sum_{i=1}^n \ell\Bigl(y_i, \bx_i^\top \hbb^{t} - \sum_j \pdv{\hbb^t_j}{x_{ij}}\Bigr) - 
	\frac 1n \sum_{i=1}^n \ell\Big(y_i, \bx_i^\top \hbb^{t} - 
	\be_i^\top \bF \bW^\top\be_{t+1}\Big),\label{R4}\\
	&R_5 = \frac 1n \sum_{i=1}^n \ell\big(y_i, \bx_i^\top \hbb^{t} -
	\be_i^\top \bF \bW^\top\be_{t+1}\big) -
	\frac 1n \sum_{i=1}^n \ell\big(y_i, \bx_i^\top \hbb^{t} -
	\be_i^\top \bF \hbW^\top\be_{t+1}\big).\label{R5}
\end{align}
Note that the error terms \(R_1\) and \(R_2\) correspond to the terms \(V_2\) and \(V_1\), respectively, in Appendix C of \cite{rad2020error}.
By the triangle inequality,
\begin{equation}\label{eq:error-bound}
	\E[|\widetilde{\calR}_t - \calR_t|]
	\le \sum_{i=1}^4 \E[|R_i|] 
	\quad \text{and} \quad
	\E[|\widehat{\calR}_t - \calR_t|]
	\le \sum_{i=1}^5 \E[|R_i|]. 
\end{equation}
The proof below is organized around the five residuals. The preliminary lemmas
collect the stability, prediction-moment, and comparison estimates used in both
regimes. Lemmas~\ref{lem:R1}--\ref{lem:R5} then control \(R_1,\ldots,R_5\) in
order; each lemma is stated under Assumption~\ref{assu:test-function}, either in
case \((i)\) or in case \((ii)\), and its proof records which noise moment is
used.

\subsection{Preliminary lemmas}
\label{sec:preliminary-lemmas}

\begin{lemma}[Second-moment leave-one-out stability]
\label{lem:loo-perturbation}
Under Assumptions~\ref{assu:X}--\ref{assu:algorithm} and the second-moment
condition \(\E[\ep_i^2]\le\zeta\), let $(\bu^t,\hbb^t)$ be the iterates
generated by \eqref{eq:general-iteration} on the full dataset, and let
$(\bu^{t,-i},\hbb^{t,-i})$ be the iterates obtained by running the same algorithm
on the dataset with sample $i$ removed. Then for any $t\in\{1,\ldots,T\}$, 
\[
\E\Big[\sum_{i=1}^n \|\hbb^t-\hbb^{t,-i}\|_2^2\Big]
\;\le\;
C(T,\gamma,\zeta,\kappa),
\]
where $C(T,\gamma,\zeta,\kappa)$ does not depend on $n$ or $p$.
\end{lemma}

\begin{proof}[Proof of Lemma~\ref{lem:loo-perturbation}]
\label{proof:loo-stability}
Fix $i\in[n]$. Let $(\bu^t,\hbb^t)$ be the iterates from the full dataset, and
let $(\bu^{t,-i},\hbb^{t,-i})$ be the iterates obtained by rerunning the same
algorithm on the dataset with sample $i$ removed. Since $\bu^{t,-i}\in\R^{n-1}$,
we embed it into $\R^n$ by inserting a zero at coordinate $i$; denote the embedded
vector by $\tilde\bu^{t,-i}\in\R^n$. Define the leave-one-out differences
\[
\Delta_u^{t,i}:=\bu^t-\tilde\bu^{t,-i}\in\R^n,
\qquad
\Delta_b^{t,i}:=\hbb^t-\hbb^{t,-i}\in\R^p.
\]
Also define
\[
\bv^s := \frac{\by-\bX\hbb^s}{\sqrt n},
\qquad
\bfeta^s := \frac{\bX^\top\bu^s}{\sqrt n},
\]
and define $\tilde\bv^{s,-i}$ and $\tilde\bfeta^{s,-i}$ analogously using
$\tilde\bu^{s,-i}$ and $\hbb^{s,-i}$ (with the same embedding convention).
For convenience, also write
\[
v_i^{s,-i}:=\frac{y_i-\bx_i^\top \hbb^{s,-i}}{\sqrt n}.
\]

\medskip

\noindent\textbf{Lipschitz propagation.}
By Assumption~\ref{assu:algorithm}, $\bphi^t$ and $\bpsi^t$ are $\zeta$-Lipschitz
(with respect to the Euclidean norm on the concatenated inputs). Hence for $t\ge1$,
\begin{align}
\|\Delta_u^{t,i}\|_2^2
&\le
\zeta^2\Big(\sum_{s=0}^{t-1}\|\Delta_u^{s,i}\|_2^2
+
\sum_{s=0}^{t-1}\|\Delta_v^{s,i}\|_2^2\Big),
\label{eq:loo-u-Lip}\\
\|\Delta_b^{t,i}\|_2^2
&\le
\zeta^2\Big(\sum_{s=0}^{t-1}\|\Delta_b^{s,i}\|_2^2
+
\sum_{s=0}^{t}\|\Delta_\eta^{s,i}\|_2^2\Big),
\label{eq:loo-b-Lip}
\end{align}
where $\Delta_v^{s,i}:=\bv^s-\tilde\bv^{s,-i}$ and
$\Delta_\eta^{s,i}:=\bfeta^s-\tilde\bfeta^{s,-i}$.

Moreover,
\begin{equation}
\Delta_\eta^{s,i}
=
\frac{1}{\sqrt n}\bX^\top \Delta_u^{s,i},
\qquad
\|\Delta_\eta^{s,i}\|_2^2
\le
\frac{\|\bX\|_{\op}^2}{n}\,\|\Delta_u^{s,i}\|_2^2.
\label{eq:loo-eta}
\end{equation}
For the residual vectors, one checks that
\begin{equation}
\Delta_v^{s,i}
=
-\frac{1}{\sqrt n}\bX\,\Delta_b^{s,i}
+\be_i\,v_i^{s,-i},
\label{eq:loo-v-decomp}
\end{equation}
and therefore
\begin{equation}
\|\Delta_v^{s,i}\|_2^2
\le
\frac{2\|\bX\|_{\op}^2}{n}\,\|\Delta_b^{s,i}\|_2^2
+
2\,(v_i^{s,-i})^2.
\label{eq:loo-v}
\end{equation}

\medskip

\noindent\textbf{Bounding the forcing term.}
Set
\[
\lambda_X:=\frac{\|\bX\|_{\op}^2}{n},
\qquad
\lambda_{-i}:=\frac{\|\bX_{-i}\|_{\op}^2}{n},
\]
where \(\bX_{-i}\) is the design matrix with row \(i\) removed. Since \(\bX_{-i}\)
is a submatrix of \(\bX\), we have \(\lambda_{-i}\le \lambda_X\).

By the same deterministic Gr\"onwall argument as in the proof of
Lemma~\ref{lem:moment-H-F}, applied to the deleted recursion, for every \(s\le T\),
\[
\|\hbb^{s,-i}\|_2
\le
C(T,\zeta)\bigl(1+\lambda_{-i}^{1/2}\bigr).
\]
Moreover,
\[
\bX^\top\bX=\bX_{-i}^\top\bX_{-i}+\bx_i\bx_i^\top,
\]
so
\[
\lambda_X
\le
\lambda_{-i}+\frac{\|\bx_i\|_2^2}{n},
\qquad
1+\lambda_X
\le
\bigl(1+\lambda_{-i}\bigr)\Bigl(1+\frac{\|\bx_i\|_2^2}{n}\Bigr).
\]
Fix an integer \(m\ge 0\). Since \(\hbb^{s,-i}\) depends only on \(\calD_{-i}\),
conditioning on \(\calD_{-i}\) and using the independence of \((\bx_i,\ep_i)\)
from \(\calD_{-i}\), we obtain
\begin{align*}
\E\Big[(1+\lambda_X)^m (v_i^{s,-i})^2 \,\Big|\, \calD_{-i}\Big]
&\le
\frac{C(m)}{n}\bigl(1+\lambda_{-i}\bigr)^m
\E\Bigg[
\Bigl(1+\frac{\|\bx_i\|_2^2}{n}\Bigr)^m
\Bigl(y_i^2+(\bx_i^\top \hbb^{s,-i})^2\Bigr)
\Bigm| \calD_{-i}
\Bigg].
\end{align*}
Because \(\bx_i\sim N(\boldzero,\bSigma)\), \(p/n\le\gamma\), \(\opnorm{\bSigma}\le\kappa\),
\(\|\bb^*\|_2\le\zeta\), and \(\E[\ep_i^2]\le \zeta\), standard Gaussian moment bounds imply
\[
\E\Bigg[
\Bigl(1+\frac{\|\bx_i\|_2^2}{n}\Bigr)^m y_i^2
\Bigg]
\le
C(m,\gamma,\zeta,\kappa).
\]
Likewise, conditional on \(\calD_{-i}\),
\[
\E\Bigg[
\Bigl(1+\frac{\|\bx_i\|_2^2}{n}\Bigr)^m
(\bx_i^\top \hbb^{s,-i})^2
\Bigm| \calD_{-i}
\Bigg]
\le
C(m,\gamma,\kappa)\,\|\hbb^{s,-i}\|_2^2.
\]
Therefore,
\[
\E\Big[(1+\lambda_X)^m (v_i^{s,-i})^2 \,\Big|\, \calD_{-i}\Big]
\le
\frac{C(m,\gamma,\zeta,\kappa)}{n}\,
\bigl(1+\lambda_{-i}\bigr)^m
\Bigl(1+\|\hbb^{s,-i}\|_2^2\Bigr).
\]
Using the deterministic bound on \(\hbb^{s,-i}\), we conclude that
\begin{equation}
\sup_{1\le i\le n}\sup_{0\le s\le T}
\E\Big[(1+\lambda_X)^m (v_i^{s,-i})^2\Big]
\le
\frac{C(T,m,\gamma,\zeta,\kappa)}{n}.
\label{eq:r-bound-weighted}
\end{equation}

\medskip

\noindent\textbf{Closing the recursion.}
Set
\[
A_t:=\|\Delta_u^{t,i}\|_2^2,
\qquad
B_t:=\|\Delta_b^{t,i}\|_2^2,
\qquad
E_t:=A_t+B_t.
\]
Combining \eqref{eq:loo-u-Lip}, \eqref{eq:loo-eta}, and \eqref{eq:loo-v}, we obtain
for \(t\le T\),
\begin{align*}
A_t
&\le
\zeta^2\sum_{s=0}^{t-1}A_s
+
2\zeta^2\lambda_X\sum_{s=0}^{t-1}B_s
+
2\zeta^2\sum_{s=0}^{t-1}(v_i^{s,-i})^2,\\
B_t
&\le
\zeta^2\sum_{s=0}^{t-1}B_s
+
\zeta^2\lambda_X\sum_{s=0}^{t}A_s.
\end{align*}
Hence
\[
E_t
\le
C(T,\zeta)(1+\lambda_X)^2\sum_{s=0}^{t-1}E_s
+
C(T,\zeta)(1+\lambda_X)\sum_{s=0}^{t-1}(v_i^{s,-i})^2.
\]
Since \(E_0=0\) and \(T\) is fixed, a discrete Gr\"onwall argument yields
\[
E_t
\le
C(T,\zeta)(1+\lambda_X)^{2T}
\sum_{s=0}^{t-1}(v_i^{s,-i})^2.
\]
Taking expectations and applying \eqref{eq:r-bound-weighted} with \(m=2T\), we obtain
\[
\E\|\hbb^t-\hbb^{t,-i}\|_2^2
\le
\E[E_t]
\le
\frac{C(T,\gamma,\zeta,\kappa)}{n}.
\]
By exchangeability of the samples,
\[
\E\Big[\sum_{i=1}^n\|\hbb^t-\hbb^{t,-i}\|_2^2\Big]
=
n\,\E\big[\|\hbb^t-\hbb^{t,-1}\|_2^2\big]
\le
C(T,\gamma,\zeta,\kappa),
\]
which completes the proof.
\end{proof}

\begin{lemma}[Weighted second-moment leave-one-out stability]
\label{lem:loo-perturbation-weighted}
Fix an integer \(m\ge 0\), and define
\[
\lambda_X:=\frac{\|\bX\|_{\op}^2}{n},
\qquad
\lambda_{-i}:=\frac{\|\bX_{-i}\|_{\op}^2}{n}.
\]
Under Assumptions~\ref{assu:X}--\ref{assu:algorithm} and the second-moment
condition \(\E[\ep_i^2]\le\zeta\), for any
\(t\in\{1,\ldots,T\}\),
\[
\sup_{1\le i\le n}
\E\Big[(1+\lambda_X)^m \|\hbb^t-\hbb^{t,-i}\|_2^2\Big]
\le
\frac{C(T,m,\gamma,\zeta,\kappa)}{n}.
\]
Consequently,
\[
\sup_{1\le i\ne l\le n}
\E\Big[(1+\lambda_{-i})^m \|\hbb^{t,-i}-\hbb^{t,-(i,l)}\|_2^2\Big]
\le
\frac{C(T,m,\gamma,\zeta,\kappa)}{n}.
\]
\end{lemma}

\begin{proof}[Proof of Lemma~\ref{lem:loo-perturbation-weighted}]
\label{proof:loo-stability-weighted}
The proof of Lemma~\ref{lem:loo-perturbation} already yields the pathwise bound
\[
\|\hbb^t-\hbb^{t,-i}\|_2^2
\le
C(T,\zeta)(1+\lambda_X)^{2T}\sum_{s=0}^{t-1}(v_i^{s,-i})^2.
\]
Multiplying both sides by \((1+\lambda_X)^m\), taking expectations, and applying
\eqref{eq:r-bound-weighted} with \(m+2T\) in place of \(m\), we obtain
\[
\E\Big[(1+\lambda_X)^m \|\hbb^t-\hbb^{t,-i}\|_2^2\Big]
\le
\frac{C(T,m,\gamma,\zeta,\kappa)}{n},
\]
uniformly in \(i\). This proves the first display.

For the second display, fix \(i\). The deleted dataset \(\calD_{-i}\) consists
of \(n-1\) \iid\ observations from the same linear model, and
\(p/(n-1)\le 2\gamma\). Therefore, the first part of the lemma applies to the
deleted recursion with constants depending only on
\((T,m,\gamma,\kappa,\zeta)\). In that recursion, \(\lambda_{-i}\) plays the
role of \(\lambda_X\), and deleting the sample indexed by \(l\ne i\) produces
\(\hbb^{t,-(i,l)}\). Hence
\[
\sup_{l\ne i}
\E\Big[(1+\lambda_{-i})^m \|\hbb^{t,-i}-\hbb^{t,-(i,l)}\|_2^2\Big]
\le
\frac{C(T,m,\gamma,\zeta,\kappa)}{n-1}
\le
\frac{C(T,m,\gamma,\zeta,\kappa)}{n},
\]
and taking the supremum over \(i\) completes the proof.
\end{proof}

\begin{lemma}
\label{lem:deleted-moments}
Under Assumptions~\ref{assu:X}--\ref{assu:algorithm}, for every finite integer \(r\ge 1\),
\[
\sup_{0\le s\le T}
\left(
\E\|\hbb^s\|_2^{2r}
\;+\;
\sup_{1\le i\le n}\E\|\hbb^{s,-i}\|_2^{2r}
\;+\;
\sup_{1\le i\ne l\le n}\E\|\hbb^{s,-(i,l)}\|_2^{2r}
\right)
\le
C(T,r,\gamma,\zeta,\kappa),
\]
where the constant does not depend on \(n\) or \(p\).
\end{lemma}

\begin{proof}[Proof of Lemma~\ref{lem:deleted-moments}]
\label{proof:lem:deleted-moments}
The full-data bound for \(\hbb^s\) is already contained in the proof of
Lemma~\ref{lem:moment-H-F}. It remains to treat the deleted and double-deleted
trajectories.

Fix a dataset obtained from \(\calD\) by deleting either one sample or two
samples, and write its design matrix as \(\bX'\). Let
\(\{\hbb^{s;\bX'}\}_{s=0}^T\) denote the corresponding primal trajectory generated
by the same recursion and the same initialization. Repeating the deterministic
Gr\"onwall argument from the proof of Lemma~\ref{lem:moment-H-F}, with \(\bX\)
replaced by \(\bX'\), yields the pathwise bound
\[
\max_{0\le s\le T}\|\hbb^{s;\bX'}\|_2
\le
C(T,\zeta)\left(1+\frac{\|\bX'\|_{\op}}{\sqrt n}\right).
\]
Since \(\bX'\) is a row-submatrix of \(\bX\), we have
\(\|\bX'\|_{\op}\le \|\bX\|_{\op}\). Therefore, for every \(r\ge 1\),
\[
\max_{0\le s\le T}\|\hbb^{s;\bX'}\|_2^{2r}
\le
C(T,r,\zeta)\left(1+\frac{\|\bX\|_{\op}}{\sqrt n}\right)^{2r}.
\]
Taking expectations and using the Gaussian operator-norm moment bound under
Assumptions~\ref{assu:X} and \ref{assu:regime}, we obtain
\[
\sup_{0\le s\le T}\E\|\hbb^{s;\bX'}\|_2^{2r}
\le
C(T,r,\gamma,\zeta,\kappa).
\]
Applying this to \(\bX'=\bX_{-i}\) and \(\bX'=\bX_{-(i,l)}\), and combining with
the already established full-data bound, proves the claim.
\end{proof}

\begin{lemma}[Prediction Moment Bounds]
\label{lem:prediction-moments}
Under Assumptions~\ref{assu:X}--\ref{assu:algorithm} and the second-moment
condition \(\E[\ep_i^2]\le\zeta\), for each fixed
\(t\in\{0,\ldots,T-1\}\),
\begin{align*}
&\E[y_{\rm new}^2]
+ \E[(\bx_{\rm new}^\top \hbb^t)^2]
+ \sup_{1\le i\le n}\E[(\bx_{\rm new}^\top \hbb^{t,-i})^2]\\
&\qquad
+ \sup_{1\le i\le n}\E[(\bx_i^\top \hbb^{t,-i})^2]
+ \sup_{1\le i\ne l\le n}\E[(\bx_i^\top \hbb^{t,-(i,l)})^2]
\le
C(T,\gamma,\zeta,\kappa).
\end{align*}

Assume further that \(\bSigma=\bI_p\). Define
\[
a_i^{\rm loo}:=\bx_i^\top \hbb^{t,-i},
\qquad
a_i^{\rm stein}:=\bx_i^\top \hbb^t-\sum_{j=1}^p \pdv{\hbb_j^t}{x_{ij}},
\qquad
a_i^W:=\bx_i^\top \hbb^t-\be_i^\top\bF\bW^\top \be_{t+1}.
\]
Then
\[
\sum_{i=1}^n
\E\Big[(a_i^{\rm loo})^2+(a_i^{\rm stein})^2+(a_i^W)^2\Big]
\le
C(T,\gamma,\zeta,\kappa)\,n,
\]
\end{lemma}

\begin{proof}[Proof of Lemma~\ref{lem:prediction-moments}]
\label{proof:lem:prediction-moments}
Since \(y_{\rm new}=\bx_{\rm new}^\top\bb^*+\ep_{\rm new}\), Assumption~\ref{assu:X}
and the displayed second-moment condition imply
\[
\E[y_{\rm new}^2]\le C(\zeta,\kappa).
\]
For the prediction terms, if \(\bx\sim N(\boldzero,\bSigma)\) is independent of a
random vector \(v\), then
\[
\E[(\bx^\top v)^2]=\E[v^\top\bSigma v]\le \kappa\,\E\|v\|_2^2,
\]
and, similarly,
\[
\E[(\bx^\top v)^4]\le C(\kappa)\,\E\|v\|_2^4.
\]
Applying these bounds with \(v\in\{\hbb^t,\hbb^{t,-i},\hbb^{t,-(i,l)}\}\) and
using Lemma~\ref{lem:deleted-moments} with \(r=1\) and \(r=2\) proves the first
two claims.

For the corrected predictors, let \(\bUpsilon_2\in\R^{T\times n}\) denote the
residual matrix defined by
\[
\sum_{j=1}^p \pdv{\be_j^\top\bH}{x_{ij}}
=
\be_i^\top\bF\bW^\top-\be_i^\top\bUpsilon_2^\top,
\qquad i\in[n].
\]
The identity and the moment bound for this residual are proved in
Lemmas~\ref{lem:sum-derivative} and \ref{lem:moment-Upsilon}. Define
\[
d_i^{(3)}
:=
\bx_i^\top(\hbb^{t,-i}-\hbb^t)
+
\sum_{j=1}^p \pdv{\hbb_j^t}{x_{ij}},
\qquad
d_i^{(4)}
:=
\be_i^\top \bUpsilon_2^\top \be_{t+1}.
\]
Then
\[
a_i^{\rm stein}=a_i^{\rm loo}-d_i^{(3)},
\qquad
a_i^W=a_i^{\rm stein}-d_i^{(4)}.
\]
Since \(\hbb^{t,-i}\) does not depend on \((\bx_i,y_i)\), the second-order Stein
identity \cite{bellec2021second} applied to the map
\(\bx_i\mapsto \hbb^t-\hbb^{t,-i}\) gives
\[
\sum_{i=1}^n \E[(d_i^{(3)})^2]
\le
\sum_{i=1}^n \E\|\hbb^{t,-i}-\hbb^t\|_2^2
+
\sum_{i=1}^n\sum_{j=1}^p \E\Big\|\pdv{\hbb^t}{x_{ij}}\Big\|_2^2
\le
C(T,\gamma,\zeta,\kappa),
\]
where we used Lemma~\ref{lem:loo-perturbation} and Lemma~\ref{lem:moment-H-F}
in the last step. Moreover,
\[
\sum_{i=1}^n \E[(d_i^{(4)})^2]
=
\E\|\bUpsilon_2^\top \be_{t+1}\|_2^2
\le
\E\|\bUpsilon_2\|_{\op}^2
\le
\frac{C(T,\gamma,\zeta,\kappa)}{n}
\le
C(T,\gamma,\zeta,\kappa),
\]
by Lemma~\ref{lem:moment-Upsilon}. Therefore,
\begin{align*}
\sum_{i=1}^n \E[(a_i^{\rm stein})^2]
&\le
2\sum_{i=1}^n \E[(a_i^{\rm loo})^2]
+
2\sum_{i=1}^n \E[(d_i^{(3)})^2]
\le
C(T,\gamma,\zeta,\kappa)\,n,\\
\sum_{i=1}^n \E[(a_i^W)^2]
&\le
2\sum_{i=1}^n \E[(a_i^{\rm stein})^2]
+
2\sum_{i=1}^n \E[(d_i^{(4)})^2]
\le
C(T,\gamma,\zeta,\kappa)\,n.
\end{align*}
\end{proof}

The next two deterministic comparison lemmas convert predictor-level
approximation errors into risk errors under the two test-loss regimes.

\begin{lemma}[Weighted Comparison]
\label{lem:weighted-comparison}
Let \(A_i,B_i\in\R\), define
\[
d_i:=A_i-B_i,
\]
and assume Assumption~\ref{assu:test-function}. If
Assumption~\ref{assu:test-function}(i) holds, set \(G_i:=1\). If
Assumption~\ref{assu:test-function}(ii) holds, set
\[
G_i:=1+|y_i|+|A_i|+|B_i|.
\]
Then
\[
\E\Bigg|\frac1n\sum_{i=1}^n
\Bigl(\ell(y_i,A_i)-\ell(y_i,B_i)\Bigr)\Bigg|
\le
\frac{C}{n}
\E\Bigg[\Big(\sum_{i=1}^n G_i^2\Big)^{1/2}
\Big(\sum_{i=1}^n d_i^2\Big)^{1/2}\Bigg].
\]
In particular, if \(\E[\sum_{i=1}^n G_i^2]\le C_0 n\), then
\[
\E\Bigg|\frac1n\sum_{i=1}^n
\Bigl(\ell(y_i,A_i)-\ell(y_i,B_i)\Bigr)\Bigg|
\le
\frac{C C_0^{1/2}}{\sqrt n}
\Bigg(\sum_{i=1}^n \E[d_i^2]\Bigg)^{1/2}.
\]
Moreover, for an event \(E\), if
\[
\E\Big[\I_E \sum_{i=1}^n G_i^2\Big]
\le
C_0 n
+
C_1 \E\Big[\I_E \sum_{i=1}^n d_i^2\Big],
\]
then
\begin{align*}
\E\Bigg[\I_E\Bigg|\frac1n\sum_{i=1}^n
\Bigl(\ell(y_i,A_i)-\ell(y_i,B_i)\Bigr)\Bigg|\Bigg]
\le\;&
\frac{C(C_0)}{\sqrt n}
\E\Big[\I_E \sum_{i=1}^n d_i^2\Big]^{1/2}\\
&+
\frac{C(C_1)}{n}
\E\Big[\I_E \sum_{i=1}^n d_i^2\Big].
\end{align*}
\end{lemma}

\begin{proof}[Proof of Lemma~\ref{lem:weighted-comparison}]
\label{proof:lem:weighted-comparison}
Under Assumption~\ref{assu:test-function}(i), global 1-Lipschitzness gives
\[
\bigl|\ell(y_i,A_i)-\ell(y_i,B_i)\bigr|\le |d_i|=G_i|d_i|.
\]
Under Assumption~\ref{assu:test-function}(ii), the centered squared loss
satisfies the same bound up to an absolute constant with
\(G_i=1+|y_i|+|A_i|+|B_i|\). Hence, in either case,
\[
\Bigg|\frac1n\sum_{i=1}^n
\Bigl(\ell(y_i,A_i)-\ell(y_i,B_i)\Bigr)\Bigg|
\le
\frac{C}{n}\sum_{i=1}^n G_i |d_i|.
\]
Applying Cauchy--Schwarz in the index \(i\) gives the first claim. If
\(\E[\sum_i G_i^2]\le C_0 n\), another Cauchy--Schwarz inequality in probability
yields
\[
\E\Bigg[\Big(\sum_{i=1}^n G_i^2\Big)^{1/2}
\Big(\sum_{i=1}^n d_i^2\Big)^{1/2}\Bigg]
\le
\E\Big[\sum_{i=1}^n G_i^2\Big]^{1/2}
\E\Big[\sum_{i=1}^n d_i^2\Big]^{1/2},
\]
which proves the second claim.

For the event version, the same argument gives
\begin{align*}
\E\Bigg[\I_E\Bigg|\frac1n\sum_{i=1}^n
\Bigl(\ell(y_i,A_i)-\ell(y_i,B_i)\Bigr)\Bigg|\Bigg]
&\le
\frac{C}{n}
\E\Bigg[
\Bigl(\I_E\sum_{i=1}^n G_i^2\Bigr)^{1/2}
\Bigl(\I_E\sum_{i=1}^n d_i^2\Bigr)^{1/2}
\Bigg]\\
&\le
\frac{C}{n}
\E\Big[\I_E\sum_{i=1}^n G_i^2\Big]^{1/2}
\E\Big[\I_E\sum_{i=1}^n d_i^2\Big]^{1/2}.
\end{align*}
Using the assumed bound on \(\E[\I_E\sum_i G_i^2]\) and
\(\sqrt{a+b}\le \sqrt a+\sqrt b\) proves the result.
\end{proof}

\begin{lemma}[Ready-to-use two-case comparison]
\label{lem:ready-two-case-comparison}
Let \(A_i,B_i\in\R\), \(d_i=A_i-B_i\), and
\[
Q(A,B):=
\frac1n\sum_{i=1}^n
\Bigl(\ell(y_i,A_i)-\ell(y_i,B_i)\Bigr).
\]
Under Assumption~\ref{assu:test-function}(i), if
\(\sum_{i=1}^n\E|d_i|\le a_n\), then
\[
\E|Q(A,B)|\le \frac{a_n}{n}.
\]
Under Assumption~\ref{assu:test-function}(ii), if
\[
\E\sum_{i=1}^n\bigl(1+|y_i|+|A_i|+|B_i|\bigr)^2\le M n,
\qquad
\sum_{i=1}^n\E[d_i^2]\le b_n,
\]
then
\[
\E|Q(A,B)|
\le
C(M)\sqrt{\frac{b_n}{n}}.
\]
Moreover, for an event \(E\) and a constant \(M\ge1\), suppose
\[
b_{n,E}:=\E\left[\I_E\sum_{i=1}^n d_i^2\right].
\]
In case \((ii)\), assume additionally that
\[
\E\left[\I_E
\sum_{i=1}^n\bigl(1+|y_i|+|A_i|+|B_i|\bigr)^2
\right]
\le
M n+M b_{n,E}.
\]
Then, in either case of Assumption~\ref{assu:test-function},
\[
\E[\I_E|Q(A,B)|]
\le
\frac{C(M)}{\sqrt n}b_{n,E}^{1/2}
+
\frac{C(M)}{n}b_{n,E}.
\]
\end{lemma}

\begin{proof}[Proof of Lemma~\ref{lem:ready-two-case-comparison}]
The first display follows directly from global 1-Lipschitzness. The second and
event displays are the non-event and event conclusions of
Lemma~\ref{lem:weighted-comparison}; in case \((i)\), the event-weight condition
is automatic with \(G_i\equiv1\).
\end{proof}

When a stronger noise moment is available, we write, for \(q\in[1,2]\),
\begin{equation}
\label{eq:lq-moment-condition}
\E|\ep_i|^q\le\zeta.
\end{equation}
The endpoint \(q=1\) is the first-moment regime in
Assumption~\ref{assu:test-function}(i), while \(q=2\) matches the
second-moment regime in Assumption~\ref{assu:test-function}(ii).

\begin{lemma}[First-moment and \(L^q\) leave-one-out stability]
\label{lem:loo-first-moment}
Fix \(m\ge0\). Under Assumption~\ref{assu:X},
Assumptions~\ref{assu:regime}--\ref{assu:algorithm}, and the first-moment
condition \(\E|\ep_i|\le\zeta\), for every fixed \(t\in\{1,\ldots,T\}\),
\[
\sup_{1\le i\le n}
\E\Big[(1+\lambda_X)^m\|\hbb^t-\hbb^{t,-i}\|_2\Big]
\le
\frac{C(T,m,\gamma,\zeta,\kappa)}{\sqrt n}.
\]
Moreover,
\[
\sup_{1\le i\le n}
\E\Big[(1+\lambda_X)^m\|\hbb^t-\hbb^{t,-i}\|_2^2\Big]
\le
\frac{C(T,m,\gamma,\zeta,\kappa)}{\sqrt n}.
\]
The same conclusions hold for double-deleted iterates: for every \(i\ne l\),
with \(\lambda_{-i}=\|\bX_{-i}\|_{\op}^2/n\),
\[
\E\Big[(1+\lambda_{-i})^m
\|\hbb^{t,-i}-\hbb^{t,-(i,l)}\|_2\Big]
\le
\frac{C(T,m,\gamma,\zeta,\kappa)}{\sqrt n},
\]
and
\[
\E\Big[(1+\lambda_{-i})^m
\|\hbb^{t,-i}-\hbb^{t,-(i,l)}\|_2^2\Big]
\le
\frac{C(T,m,\gamma,\zeta,\kappa)}{\sqrt n}.
\]
More generally, if \eqref{eq:lq-moment-condition} holds for some
\(q\in[1,2]\), then
\[
\sup_{1\le i\le n}
\E\Big[(1+\lambda_X)^m\|\hbb^t-\hbb^{t,-i}\|_2^q\Big]
\le
\frac{C(T,m,q,\gamma,\zeta,\kappa)}{n^{q/2}},
\]
and, for every \(i\ne l\),
\[
\E\Big[(1+\lambda_{-i})^m
\|\hbb^{t,-i}-\hbb^{t,-(i,l)}\|_2^q\Big]
\le
\frac{C(T,m,q,\gamma,\zeta,\kappa)}{n^{q/2}}.
\]
\end{lemma}

\begin{proof}[Proof of Lemma~\ref{lem:loo-first-moment}]
We prove the single-deletion bounds; the double-deletion bounds follow by
applying the same argument to the deleted dataset \(\calD_{-i}\).
Let
\[
a_t=\|\Delta_u^{t,i}\|_2,\qquad
b_t=\|\Delta_b^{t,i}\|_2,\qquad
e_t=a_t+b_t.
\]
The Lipschitz propagation inequalities used in the proof of
Lemma~\ref{lem:loo-perturbation} also give the norm bounds
\[
a_t
\le
\zeta\sum_{s=0}^{t-1}a_s
+
\zeta\sum_{s=0}^{t-1}\|\Delta_v^{s,i}\|_2,
\qquad
b_t
\le
\zeta\sum_{s=0}^{t-1}b_s
+
\zeta\sum_{s=0}^{t}\|\Delta_\eta^{s,i}\|_2.
\]
Using
\[
\|\Delta_\eta^{s,i}\|_2
\le
\lambda_X^{1/2}a_s,
\qquad
\|\Delta_v^{s,i}\|_2
\le
\lambda_X^{1/2}b_s+|v_i^{s,-i}|,
\]
a discrete Gronwall argument over the fixed horizon \(T\) yields the pathwise
bound
\begin{equation}
\label{eq:first-moment-loo-pathwise}
\|\hbb^t-\hbb^{t,-i}\|_2
\le e_t
\le
C(T,\zeta)(1+\lambda_X)^{T/2}
\sum_{s=0}^{t-1}|v_i^{s,-i}|.
\end{equation}

It remains to control the first moment of \(v_i^{s,-i}\). Since
\[
v_i^{s,-i}
=
\frac{\bx_i^\top(\bb^*-\hbb^{s,-i})+\ep_i}{\sqrt n},
\]
\(\hbb^{s,-i}\) is independent of \((\bx_i,\ep_i)\). Also
\[
\lambda_X\le \lambda_{-i}+\frac{\|\bx_i\|_2^2}{n}.
\]
Hence, for any \(r\ge0\), conditioning on \(\calD_{-i}\) and using Gaussian
moment bounds gives
\begin{align}
\E\Big[(1+\lambda_X)^r|v_i^{s,-i}|\,\Bigm|\,\calD_{-i}\Big]
&\le
\frac{C(r,\gamma,\kappa)}{\sqrt n}(1+\lambda_{-i})^r
\Big( \|\bb^*\|_2+\|\hbb^{s,-i}\|_2+\E|\ep_i| \Big).
\label{eq:first-moment-v-bound}
\end{align}
The deterministic iterate bound from the proof of Lemma~\ref{lem:deleted-moments}
does not use any moment of the noise and gives
\[
\|\hbb^{s,-i}\|_2
\le
C(T,\zeta)(1+\lambda_{-i}^{1/2}).
\]
Taking expectations in \eqref{eq:first-moment-v-bound}, using the Gaussian
operator-norm moment bounds, and then applying
\eqref{eq:first-moment-loo-pathwise} gives the first displayed claim.

For the squared claim, use the crude pathwise bound
\[
\|\hbb^t-\hbb^{t,-i}\|_2
\le
\|\hbb^t\|_2+\|\hbb^{t,-i}\|_2
\le
C(T,\zeta)(1+\lambda_X)^{1/2}.
\]
Multiplying this bound by the first-moment bound just proved, with a larger
weight exponent, yields
\[
\E\Big[(1+\lambda_X)^m\|\hbb^t-\hbb^{t,-i}\|_2^2\Big]
\le
\frac{C(T,m,\gamma,\zeta,\kappa)}{\sqrt n}.
\]
This proves the single-deletion statements. The proof for
\(\hbb^{t,-i}-\hbb^{t,-(i,l)}\) is identical after conditioning on the fixed
deleted sample \(i\), because \(\calD_{-i}\) is again an iid sample from the same
model up to the harmless change from \(n\) to \(n-1\).

It remains to prove the \(L^q\) strengthening. Under
\eqref{eq:lq-moment-condition}, the same conditional argument gives, for every
\(r\ge0\),
\begin{align}
\E\Big[(1+\lambda_X)^r|v_i^{s,-i}|^q\,\Bigm|\,\calD_{-i}\Big]
&\le
\frac{C(r,q,\gamma,\kappa)}{n^{q/2}}(1+\lambda_{-i})^r
\Big( \|\bb^*\|_2^q+\|\hbb^{s,-i}\|_2^q+\E|\ep_i|^q \Big).
\label{eq:lq-v-bound}
\end{align}
Using the deterministic deleted-iterate bound and Gaussian operator-norm moment
bounds, we obtain
\[
\sup_{i,s}
\E\Big[(1+\lambda_X)^r|v_i^{s,-i}|^q\Big]
\le
\frac{C(T,r,q,\gamma,\zeta,\kappa)}{n^{q/2}}.
\]
Raising \eqref{eq:first-moment-loo-pathwise} to the power \(q\), using
\((\sum_{s<t} a_s)^q\le T^{q-1}\sum_{s<t}a_s^q\), and applying the last display
with \(r=m+qT/2\) proves the single-deletion \(L^q\) bound. The double-deletion
\(L^q\) bound follows by the same deleted-dataset argument as above.
\end{proof}

To control the empirical fluctuation term \(R_2\), we use the following
leave-two-out decoupling device.

\begin{lemma}[Leave-two-out covariance decoupling]
\label{lem:loo-covariance-decoupling}
Fix \(i\ne l\). Suppose \(U^i,U^l,U^{il},U^{li}\) satisfy
\[
\E[U^i\mid\calD_{-i}]=0,
\qquad
\E[U^l\mid\calD_{-l}]=0,
\qquad
\E[U^{il}\mid\calD_{-(i,l)}]=
\E[U^{li}\mid\calD_{-(i,l)}]=0.
\]
Assume also that \(U^{il}\) is \(\calD_{-l}\)-measurable, \(U^{li}\) is
\(\calD_{-i}\)-measurable, and \(U^{il}\) and \(U^{li}\) are conditionally
independent given \(\calD_{-(i,l)}\). Then
\[
\E[U^lU^{il}]=0,
\qquad
\E[U^{il}U^{li}]=0,
\qquad
\E[U^iU^l]
=
\E[(U^i-U^{il})(U^l-U^{li})].
\]
\end{lemma}

\begin{proof}[Proof of Lemma~\ref{lem:loo-covariance-decoupling}]
Since \(U^{il}\) is \(\calD_{-l}\)-measurable and
\(\E[U^l\mid\calD_{-l}]=0\),
\[
\E[U^lU^{il}]
=
\E\{U^{il}\E[U^l\mid\calD_{-l}]\}
=0.
\]
Similarly, \(\E[U^iU^{li}]=0\). Conditional on \(\calD_{-(i,l)}\), the variables
\(U^{il}\) and \(U^{li}\) are independent and centered, so
\[
\E[U^{il}U^{li}]
=
\E\{\E[U^{il}U^{li}\mid\calD_{-(i,l)}]\}
=0.
\]
Expanding
\(\E[(U^i-U^{il})(U^l-U^{li})]\) and using the three zero covariance identities
gives the claimed equality.
\end{proof}

\begin{lemma}[Centered-sum variance bound]
\label{lem:centered-sum-variance}
Let \(U^1,\ldots,U^n\) be real random variables and set
\[
S_n:=-\frac1n\sum_{i=1}^n U^i.
\]
If
\[
\max_i\E[(U^i)^2]\le M,
\qquad
\max_{i\ne l}|\E[U^iU^l]|\le \frac{M}{n},
\]
then
\[
\E|S_n|\le \frac{C(M)}{\sqrt n}.
\]
\end{lemma}

\begin{proof}[Proof of Lemma~\ref{lem:centered-sum-variance}]
By Cauchy--Schwarz,
\[
\E|S_n|\le \E[S_n^2]^{1/2}.
\]
The assumed diagonal and off-diagonal bounds give
\[
\E[S_n^2]
\le
\frac1{n^2}\sum_{i=1}^n\E[(U^i)^2]
+
\frac1{n^2}\sum_{i\ne l}|\E[U^iU^l]|
\le
\frac{C(M)}{n},
\]
which proves the claim.
\end{proof}

\begin{lemma}[Comparison inputs for \(R_3\) and \(R_4\)]
\label{lem:R34-comparison-inputs}
Under Assumptions~\ref{assu:X}--\ref{assu:algorithm}, with
\(\bSigma=\bI_p\), define
\[
A_i^{(3)}:=\bx_i^\top\hbb^{t,-i},
\qquad
B_i^{(3)}:=\bx_i^\top\hbb^t-\sum_{j=1}^p\pdv{\hbb_j^t}{x_{ij}},
\qquad
d_i^{(3)}:=A_i^{(3)}-B_i^{(3)},
\]
and
\[
A_i^{(4)}:=B_i^{(3)},
\qquad
B_i^{(4)}:=\bx_i^\top\hbb^t-\be_i^\top\bF\bW^\top\be_{t+1},
\qquad
d_i^{(4)}:=A_i^{(4)}-B_i^{(4)}.
\]
Under Assumption~\ref{assu:test-function}(i),
\[
\sum_{i=1}^n\E|d_i^{(3)}|
\le
C(T,\gamma,\zeta,\kappa)\sqrt n,
\qquad
\sum_{i=1}^n\E|d_i^{(4)}|
\le
C(T,\gamma,\zeta,\kappa).
\]
Under Assumption~\ref{assu:test-function}(ii),
\[
\sum_{i=1}^n\E[(d_i^{(3)})^2]
\le
C(T,\gamma,\zeta,\kappa),
\qquad
\sum_{i=1}^n\E[(d_i^{(4)})^2]
\le
\frac{C(T,\gamma,\zeta,\kappa)}{n},
\]
and, for \(r=3,4\),
\[
\E\sum_{i=1}^n
\bigl(1+|y_i|+|A_i^{(r)}|+|B_i^{(r)}|\bigr)^2
\le
C(T,\gamma,\zeta,\kappa)n.
\]
\end{lemma}

\begin{proof}[Proof of Lemma~\ref{lem:R34-comparison-inputs}]
For \(d_i^{(3)}\), the \(L^1\) form of the second-order Stein comparison
\cite{bellec2021second}, applied to the map
\(\bx_i\mapsto\hbb^t-\hbb^{t,-i}\), gives
\[
\sum_{i=1}^n\E|d_i^{(3)}|
\le
C(T,\gamma,\zeta,\kappa)\sqrt n.
\]
This uses Lemma~\ref{lem:loo-first-moment} for the leave-one-out perturbation and
Lemma~\ref{lem:moment-H-F} for the derivative term.

For \(d_i^{(4)}\), Lemma~\ref{lem:sum-derivative} gives
\[
d_i^{(4)}
=
\be_i^\top\bUpsilon_2^\top\be_{t+1}.
\]
Therefore
\[
\sum_{i=1}^n\E|d_i^{(4)}|
\le
\sqrt n\,\E\|\bUpsilon_2^\top\be_{t+1}\|_2
\le
C(T,\gamma,\zeta,\kappa),
\]
where the last step follows from Lemma~\ref{lem:moment-Upsilon}.

Under Assumption~\ref{assu:test-function}(ii), the second-order Stein identity
gives
\[
\sum_{i=1}^n\E[(d_i^{(3)})^2]
\le
\sum_{i=1}^n\E\|\hbb^{t,-i}-\hbb^t\|_2^2
+
\sum_{i=1}^n\sum_{j=1}^p
\E\Big\|\pdv{\hbb^t}{x_{ij}}\Big\|_2^2
\le
C(T,\gamma,\zeta,\kappa),
\]
by Lemmas~\ref{lem:loo-perturbation} and~\ref{lem:moment-H-F}. Also,
\[
\sum_{i=1}^n\E[(d_i^{(4)})^2]
=
\E\|\bUpsilon_2^\top\be_{t+1}\|_2^2
\le
\E\|\bUpsilon_2\|_{\op}^2
\le
\frac{C(T,\gamma,\zeta,\kappa)}{n}.
\]

It remains to verify the squared-loss weights. Since
\[
\bigl(1+|y_i|+|A_i^{(r)}|+|B_i^{(r)}|\bigr)^2
\le
C\bigl(1+|y_i|^2+|A_i^{(r)}|^2+|B_i^{(r)}|^2\bigr),
\]
the second-moment condition in Assumption~\ref{assu:test-function}(ii), together
with Lemma~\ref{lem:prediction-moments}, yields the claimed bound for
\(r=3,4\).
\end{proof}

\begin{lemma}[Comparison inputs for \(R_5\)]
\label{lem:R5-comparison-inputs}
Under Assumptions~\ref{assu:X}--\ref{assu:algorithm}, with
\(\bSigma=\bI_p\), let \(\Omega\) be the event in \eqref{eq:event-Omega} and
let \(\hbW=\cbK^{-1}\cbA\). Define
\[
A_i^{(5)}
:=
\bx_i^\top\hbb^t-\be_i^\top\bF\bW^\top\be_{t+1},
\qquad
B_i^{(5)}
:=
\bx_i^\top\hbb^t-\be_i^\top\bF\hbW^\top\be_{t+1},
\]
and \(d_i^{(5)}:=A_i^{(5)}-B_i^{(5)}\). Then
\[
\E\left[\I_\Omega\sum_{i=1}^n(d_i^{(5)})^2\right]
\le
C(T,\gamma,\zeta,\kappa)\sqrt n.
\]
If Assumption~\ref{assu:test-function}(ii) holds, then
\[
\E\left[
\I_\Omega
\sum_{i=1}^n
\bigl(1+|y_i|+|A_i^{(5)}|+|B_i^{(5)}|\bigr)^2
\right]
\le
C(T,\gamma,\zeta,\kappa)n
+
C\,\E\left[\I_\Omega\sum_{i=1}^n(d_i^{(5)})^2\right].
\]
\end{lemma}

\begin{proof}[Proof of Lemma~\ref{lem:R5-comparison-inputs}]
Set
\[
D_5:=\sum_{i=1}^n(d_i^{(5)})^2
=
\|\bF(\bW-\hbW)^\top\be_{t+1}\|_2^2.
\]
Since \(D_5\le \|\bF(\bW-\hbW)^\top\|_{\rm F}^2\), and since
\[
\|\bA\|_{\F}^2
\le
\sqrt T\,\|\bA^\top\bA\|_{\F}
\quad \text{for } \bA\in\mathbb R^{n\times T},
\]
we have
\[
D_5
\le
\sqrt T\,
\|(\bW-\hbW)\bF^\top\bF(\bW-\hbW)^\top\|_{\F}.
\]
Using \(\bW-\hbW=\cbK^{-1}(\cbK\bW-\cbA)\), we obtain
\begin{align*}
\E[\I_\Omega D_5]
&\le
C(T)\E\!\left[
\I_\Omega
\|\cbK^{-1}\|_{\op}^2
\|\cbK\bW-\cbA\|_{\op}
\|\bF^\top\bF(\cbK\bW-\cbA)^\top\|_{\F}
\right].
\end{align*}
On \(\Omega\), Lemma~\ref{lem:Kinv-opnorm} gives
\(\|\cbK^{-1}\|_{\op}\le C(T,\gamma,\zeta,\kappa)n^{-1/2}\), and
Lemma~\ref{lem:op-bound-WAK} gives
\(\|\cbK\bW-\cbA\|_{\op}\le C(T,\gamma,\zeta,\kappa)n\). Proposition~\ref{prop:approx-W}
therefore implies
\[
\E[\I_\Omega D_5]
\le
C(T,\gamma,\zeta,\kappa)\sqrt n.
\]

For the weight bound, note that \(B_i^{(5)}=A_i^{(5)}-d_i^{(5)}\), so
\[
\sum_{i=1}^n
\bigl(1+|y_i|+|A_i^{(5)}|+|B_i^{(5)}|\bigr)^2
\le
C\sum_{i=1}^n
\bigl(1+|y_i|^2+|A_i^{(5)}|^2+(d_i^{(5)})^2\bigr).
\]
The second-moment condition in Assumption~\ref{assu:test-function}(ii) and
Lemma~\ref{lem:prediction-moments} control the first three terms by
\(C(T,\gamma,\zeta,\kappa)n\), and the remaining term is \(D_5\).
\end{proof}

\subsection{Control of $R_1$}

\begin{lemma}[Control of \(R_1\)]
\label{lem:R1}
Under Assumptions~\ref{assu:X}, \ref{assu:regime}, and
\ref{assu:algorithm}, and under Assumption~\ref{assu:test-function}, either in
case \((i)\) or in case \((ii)\), we have
\[
\E[|R_1|]
\le
\frac{C(T,\gamma,\zeta,\kappa)}{\sqrt n}.
\]
\end{lemma}

\begin{proof}[Proof of Lemma~\ref{lem:R1}]
\label{proof:lem:R1}
Set
\[
\Delta_i:=\hbb^t-\hbb^{t,-i}.
\]

First suppose Assumption~\ref{assu:test-function}(i) holds. The loss is
centered and globally 1-Lipschitz, so
\[
|\ell(y,a)-\ell(y,b)|\le |a-b|.
\]
Since \(\bx_{\rm new}\) is independent of the training data,
\[
\E|R_1|
\le
\frac{1}{n}\sum_{i=1}^n
\E|\bx_{\rm new}^{\top}\Delta_i|
\le
\frac{C(\kappa)}{n}\sum_{i=1}^n\E\|\Delta_i\|_2.
\]
Lemma~\ref{lem:loo-first-moment} gives
\(\sup_i\E\|\Delta_i\|_2\le C(T,\gamma,\zeta,\kappa)n^{-1/2}\), and hence
\[
\E|R_1|\le \frac{C(T,\gamma,\zeta,\kappa)}{\sqrt n}.
\]

Now suppose Assumption~\ref{assu:test-function}(ii) holds. Set
\[
G_i:=1+|y_{\rm new}|+|\bx_{\rm new}^\top \hbb^t|+|\bx_{\rm new}^\top \hbb^{t,-i}|.
\]
By the centered squared-loss increment bound from
Assumption~\ref{assu:test-function}(ii),
\begin{align*}
\E[|R_1|]
&\le
\frac{C}{n}
\sum_{i=1}^n
\E\!\left[
G_i\,|\bx_{\rm new}^\top\Delta_i|
\right].
\end{align*}

\[
G_i^2
\le
4\Bigl(
1+|y_{\rm new}|^2+|\bx_{\rm new}^\top \hbb^t|^2+|\bx_{\rm new}^\top \hbb^{t,-i}|^2
\Bigr).
\]
Lemma~\ref{lem:prediction-moments} gives
\[
\E[G_i^2]\le C(T,\gamma,\zeta,\kappa).
\]
Hence, by Cauchy--Schwarz and the Gaussian second-moment bound,
\[
\E\!\left[
G_i\,|\bx_{\rm new}^\top\Delta_i|
\right]
\le
\E[G_i^2]^{1/2}\E[(\bx_{\rm new}^\top\Delta_i)^2]^{1/2}
\le
C(T,\gamma,\zeta,\kappa)\,\bigl(\E\|\Delta_i\|_2^2\bigr)^{1/2}.
\]
Summing over \(i\) and using \(\sum_{i=1}^n a_i\le \sqrt n (\sum_{i=1}^n a_i^2)^{1/2}\), we obtain
\[
\E[|R_1|]
\le
\frac{C(T,\gamma,\zeta,\kappa)}{\sqrt n}
\Big(\sum_{i=1}^n \E\|\Delta_i\|_2^2\Big)^{1/2}.
\]
The second-moment leave-one-out stability bound in
Lemma~\ref{lem:loo-perturbation} gives
\[
\sum_{i=1}^n \E\|\Delta_i\|_2^2
\le C(T,\gamma,\zeta,\kappa),
\]
so this is \(O(n^{-1/2})\).
This completes the proof.
\end{proof}

\subsection{Control of $R_2$}

\begin{lemma}[Control of \(R_2\)]
\label{lem:R2}
Under Assumptions~\ref{assu:X}, \ref{assu:regime}, and
\ref{assu:algorithm}, and under Assumption~\ref{assu:test-function}, either in
case \((i)\) or in case \((ii)\), we have
\[
\E[|R_2|]
\le 
\frac{C(T,\gamma,\zeta,\kappa)}{\sqrt n}.
\]
\end{lemma}

\begin{proof}[Proof of Lemma~\ref{lem:R2}]
\label{proof:lem:R2}
Write
\[
R_2=-\frac1n\sum_{i=1}^n U^i,
\qquad
U^i:=\ell(y_i,\bx_i^\top\hbb^{t,-i})
-
\E[\ell(y_i,\bx_i^\top\hbb^{t,-i})\mid\calD_{-i}].
\]
For \(i\ne l\), also define
\[
U^{il}
:=
\ell(y_i,\bx_i^\top\hbb^{t,-(i,l)})
-
\E[\ell(y_i,\bx_i^\top\hbb^{t,-(i,l)})\mid\calD_{-(i,l)}],
\]
and define \(U^{li}\) analogously.
The deletion construction gives the centering and measurability requirements of
Lemma~\ref{lem:loo-covariance-decoupling}, as well as the required conditional
independence. Hence
\[
\E[U^iU^l]=\E[(U^i-U^{il})(U^l-U^{li})]
\qquad (i\ne l).
\]
By Lemma~\ref{lem:centered-sum-variance}, it remains to prove, in each case,
\[
\E[(U^i)^2]\le C(T,\gamma,\zeta,\kappa),
\qquad
|\E[U^iU^l]|\le \frac{C(T,\gamma,\zeta,\kappa)}{n}
\quad (i\ne l).
\]

First suppose Assumption~\ref{assu:test-function}(i) holds. Then
\(|\ell(y,a)|\le |a|\) and
\(|\ell(y,a)-\ell(y,b)|\le |a-b|\). The diagonal terms satisfy
\[
\E[(U^i)^2]
\le
\E[\ell(y_i,\bx_i^\top\hbb^{t,-i})^2]
\le
C\E[(\bx_i^\top\hbb^{t,-i})^2]
\le
C(T,\gamma,\zeta,\kappa),
\]
where we conditioned on \(\calD_{-i}\) and used
Lemma~\ref{lem:deleted-moments}. For \(i\ne l\), conditioning on
\(\calD_{-(i,l)}\), applying the Lipschitz bound to the two centered
differences, and then using the conditional double-deletion version of
Lemma~\ref{lem:loo-first-moment} gives
\[
\E\!\left[
\big|(U^i-U^{il})(U^l-U^{li})\big|
\Bigm|\calD_{-(i,l)}
\right]
\le
\frac{C(T,\gamma,\zeta,\kappa)}{n}.
\]
The decoupling identity therefore yields
\[
|\E[U^iU^l]|
\le
\frac{C(T,\gamma,\zeta,\kappa)}{n}.
\]

Now suppose Assumption~\ref{assu:test-function}(ii) holds. Set
\(a_i:=\bx_i^\top \hbb^{t,-i}\). Since
\[
\E[(U^i)^2]
=
\E[\var(\ell(y_i,a_i)\mid\calD_{-i})]
\le
\E[\ell(y_i,a_i)^2],
\]
and the centered squared loss satisfies
\[
|\ell(y_i,a_i)|
\le
C(1+|y_i|+|a_i|)|a_i|,
\qquad
\ell(y_i,a_i)^2
\le
C(a_i^2+y_i^2a_i^2+a_i^4),
\]
it is enough to bound the three displayed moments. Conditioning on
\(\hbb^{t,-i}\), using Gaussian second- and fourth-moment identities, and then
using Lemma~\ref{lem:deleted-moments}, we obtain
\[
\E[a_i^2]+\E[a_i^4]
\le
C(T,\gamma,\zeta,\kappa).
\]
Also, from \(y_i=\bx_i^\top\bb^*+\ep_i\),
\[
\E[y_i^2a_i^2\mid \hbb^{t,-i}]
\le
C(\zeta,\kappa)\|\hbb^{t,-i}\|_2^2,
\]
and hence \(\E[y_i^2a_i^2]\le C(T,\gamma,\zeta,\kappa)\). This proves the
diagonal bound in case \((ii)\).

For the off-diagonal bound, the decoupling identity and Cauchy--Schwarz reduce
the proof to
\[
\E[(U^i-U^{il})^2]\le \frac{C(T,\gamma,\zeta,\kappa)}{n}
\qquad (i\ne l).
\]
Let
\[
\Delta_{il}:=\hbb^{t,-i}-\hbb^{t,-(i,l)},
\qquad
G_{il}:=1+|y_i|+|\bx_i^\top \hbb^{t,-i}|+
|\bx_i^\top \hbb^{t,-(i,l)}|.
\]
Jensen's inequality and the squared-loss increment bound give
\[
\frac12\E[(U^i-U^{il})^2]
\le
C\,\E\big[G_{il}^2(\bx_i^\top\Delta_{il})^2\big].
\]
Condition on \(\calD_{-i}\). Then \(\hbb^{t,-i}\),
\(\hbb^{t,-(i,l)}\), and \(\Delta_{il}\) are fixed, while
\((\bx_i,\ep_i)\) is independent of \(\calD_{-i}\). Using
\[
G_{il}^2\le
C\Bigl(1+y_i^2+(\bx_i^\top\hbb^{t,-i})^2+
(\bx_i^\top\hbb^{t,-(i,l)})^2\Bigr)
\]
and the Gaussian bound
\[
\E[(\bx_i^\top a)^2(\bx_i^\top b)^2]
\le
3\kappa^2\|a\|_2^2\|b\|_2^2,
\]
we get
\[
\E\big[G_{il}^2(\bx_i^\top\Delta_{il})^2\mid\calD_{-i}\big]
\le
C(T,\zeta,\kappa)(1+\lambda_{-i})\|\Delta_{il}\|_2^2,
\]
where
\(\lambda_{-i}=\|\bX_{-i}\|_{\op}^2/n\). Here we used the deterministic
deleted-iterate bound
\[
\|\hbb^{t,-i}\|_2^2+\|\hbb^{t,-(i,l)}\|_2^2
\le
C(T,\zeta)(1+\lambda_{-i}),
\]
which follows from the same Gr\"onwall argument as in
Lemma~\ref{lem:deleted-moments}. Taking expectations and applying
Lemma~\ref{lem:loo-perturbation-weighted} to the deleted dataset \(\calD_{-i}\)
with \(m=1\) gives
\[
\E\big[G_{il}^2(\bx_i^\top\Delta_{il})^2\big]
\le
\frac{C(T,\gamma,\zeta,\kappa)}{n}.
\]
Thus \(\E[(U^i-U^{il})^2]\le C(T,\gamma,\zeta,\kappa)/n\).
Cauchy--Schwarz gives the off-diagonal bound in case \((ii)\).

Combining the diagonal and off-diagonal bounds with
Lemma~\ref{lem:centered-sum-variance} proves the claim.
\end{proof}

\subsection{Control of $R_3$}

\begin{lemma}[Control of \(R_3\)]
\label{lem:R3}
Under Assumptions~\ref{assu:X}, \ref{assu:regime}, and
\ref{assu:algorithm}, under Assumption~\ref{assu:test-function}, either in case
\((i)\) or in case \((ii)\), and with \(\bSigma=\bI_p\), we have
\[
\E[|R_3|]
\le
\frac{C(T,\gamma,\zeta,\kappa)}{\sqrt n}.
\]
\end{lemma}

\begin{proof}[Proof of Lemma~\ref{lem:R3}]
\label{proof:lem:R3}
Let $\hbb^{t,-i}$ denote the $t$-th iterate computed with the $i$-th sample removed.
Define
\[
A_i:=\bx_i^\top \hbb^{t,-i},
\qquad
B_i:=\bx_i^\top \hbb^t-\sum_j \be_j^\top  \pdv{\hbb^t}{x_{ij}},
\qquad
d_i
:=
A_i-B_i.
\]
Lemma~\ref{lem:R34-comparison-inputs} gives the assumptions of
Lemma~\ref{lem:ready-two-case-comparison}: in case \((i)\),
\(\sum_i\E|d_i|\le C(T,\gamma,\zeta,\kappa)\sqrt n\), while in case \((ii)\),
the squared-loss weight is bounded by \(C(T,\gamma,\zeta,\kappa)n\) and
\(\sum_i\E[d_i^2]\le C(T,\gamma,\zeta,\kappa)\). Applying
Lemma~\ref{lem:ready-two-case-comparison} to these predictors proves the claim.
\end{proof}

\subsection{Control of $R_4$}

\begin{lemma}[Control of \(R_4\)]
\label{lem:R4}
Under Assumptions~\ref{assu:X}, \ref{assu:regime}, and
\ref{assu:algorithm}, under Assumption~\ref{assu:test-function}, either in case
\((i)\) or in case \((ii)\), and with \(\bSigma=\bI_p\), we have
\[
\E[|R_4|]
\le
\frac{C(T,\gamma,\zeta,\kappa)}{\sqrt n}.
\]
\end{lemma}

\begin{proof}[Proof of Lemma~\ref{lem:R4}]
\label{proof:lem:R4}
Let
\[
A_i:=\bx_i^\top\hbb^t-\sum_{j=1}^p\pdv{\hbb_j^t}{x_{ij}},
\qquad
B_i:=\bx_i^\top\hbb^t-\be_i^\top\bF\bW^\top\be_{t+1},
\qquad
d_i:=A_i-B_i.
\]
These are the \(r=4\) comparison variables in
Lemma~\ref{lem:R34-comparison-inputs}.
Lemma~\ref{lem:R34-comparison-inputs} gives
\(\sum_i\E|d_i|\le C(T,\gamma,\zeta,\kappa)\) in case \((i)\), and gives the
squared-loss weight bound together with
\(\sum_i\E[d_i^2]\le C(T,\gamma,\zeta,\kappa)/n\) in case \((ii)\). Therefore
Lemma~\ref{lem:ready-two-case-comparison} gives \(\E|R_4|\le
C(T,\gamma,\zeta,\kappa)/n\), which is stronger than the stated bound.
\end{proof}

\subsection{Control of $R_5$}

\begin{lemma}[Control of \(R_5\)]
\label{lem:R5}
Under Assumptions~\ref{assu:X}--\ref{assu:algorithm}, under
Assumption~\ref{assu:test-function}, either in case \((i)\) or in case
\((ii)\), and with \(\bSigma=\bI_p\), let \(\Omega\) be the event defined in
\eqref{eq:event-Omega}.
Using this $\hbW:=\cbK^{-1}\cbA$, define the estimator
\[
\widehat{\calR}_t :=
\frac 1n \sum_{i=1}^n
\ell\!\left(y_i, \bx_i^\top \hbb^{t}
- \be_i^\top \bF \hbW^\top \be_{t+1}\right).
\]
Then
\[
\E[\I_{\Omega}|R_5|]
\le
\frac{C(T,\gamma,\zeta,\kappa)}{n^{1/4}}.
\]
\end{lemma}

\begin{proof}[Proof of Lemma~\ref{lem:R5}]
\label{proof:lem:R5}
By definition,
\[
R_5
=
\frac1n\sum_{i=1}^n
\ell\!\left(y_i,\bx_i^\top\hbb^t-\be_i^\top\bF\bW^\top\be_{t+1}\right)
-
\frac1n\sum_{i=1}^n
\ell\!\left(y_i,\bx_i^\top\hbb^t-\be_i^\top\bF\hbW^\top\be_{t+1}\right).
\]

Define
\[
\begin{aligned}
A_i&:=\bx_i^\top\hbb^t-\be_i^\top\bF\bW^\top\be_{t+1},\\
B_i&:=\bx_i^\top\hbb^t-\be_i^\top\bF\hbW^\top\be_{t+1},\\
d_i&:=A_i-B_i=\be_i^\top\bF(\bW-\hbW)^\top\be_{t+1}.
\end{aligned}
\]
Set
\[
D_5:=\sum_{i=1}^n d_i^2.
\]
Lemma~\ref{lem:R5-comparison-inputs} gives
\(\E[\I_\Omega D_5]\le C(T,\gamma,\zeta,\kappa)\sqrt n\), and in case \((ii)\)
also gives the event-weight hypothesis in
Lemma~\ref{lem:ready-two-case-comparison}. In case \((i)\), that event-weight
hypothesis is automatic because \(G_i\equiv1\). Applying the event part of
Lemma~\ref{lem:ready-two-case-comparison} gives
\[
\E[\I_\Omega |R_5|]
\le
\frac{C(T,\gamma,\zeta,\kappa)}{n^{1/4}}.
\]
\end{proof}

\section{Proofs of the main theorems}\label{sec:proof-main}
This section combines the ingredients from Section~\ref{sec:proof-ingredients} with the isotropic reduction from Section~\ref{sec:reduction-isotropic} to prove Theorems~\ref{thm:rt-tilde} and \ref{thm:rt-hat}, and then proves Corollary~\ref{cor:early-stopping}.

\subsection{Proof of Theorem~\ref{thm:rt-tilde}}
\label{sec:proof-rt-tilde}
Using the change-of-variable argument in Section~\ref{sec:reduction-isotropic}, we can assume without loss of generality that $\bSigma = \bI_p$. For general $\bSigma$, the same results hold with constants depending additionally on $\kappa$ from Assumption~\ref{assu:X}.

\begin{proof}[Proof of Theorem~\ref{thm:rt-tilde}]
Recall from \eqref{eq:error-bound} that
\[
\E[|\widetilde{\calR}_t-\calR_t|]
\le
\sum_{i=1}^4 \E|R_i|.
\]
Lemmas~\ref{lem:R1}, \ref{lem:R2}, \ref{lem:R3}, and~\ref{lem:R4} give
\(\E|R_r|\le C(T,\gamma,\zeta,\kappa)n^{-1/2}\), \(r=1,\ldots,4\), in either
case of Assumption~\ref{assu:test-function}. In case \((i)\), these bounds use
only the first moment of the noise; in case \((ii)\), they use the second moment
required for the centered squared loss. Hence
\[
\E[|\widetilde{\calR}_t-\calR_t|]
\le
\sum_{i=1}^4 \E|R_i|
\le
\frac{C(T,\gamma,\zeta,\kappa)}{\sqrt n}.
\]

Finally, Markov's inequality implies
\[
\P\big(|\widetilde{\calR}_t-\calR_t|>\epsilon\big)
\le
\frac{\E|\widetilde{\calR}_t-\calR_t|}{\epsilon}
\le
\frac{C(T,\gamma,\zeta,\kappa)}{\epsilon\sqrt n}.
\]
\end{proof}

\subsection{Proof of Theorem~\ref{thm:rt-hat}}
\label{sec:proof-rt-hat}

\begin{proof}[Proof of Theorem~\ref{thm:rt-hat}]
As in the proof of Theorem~\ref{thm:rt-tilde}, we first apply the change-of-variables reduction in Section~\ref{sec:reduction-isotropic}. Under this reduction, \(\cbK\), \(\cbA\), and hence \(\hbW=\cbK^{-1}\cbA\) are unchanged, while the transformed population weight matrix is the original covariance-dependent matrix \(\bW\); only the constants acquire their dependence on \(\kappa\). The event \(\Omega\) from \eqref{eq:event-Omega} satisfies \(\P(\Omega)\ge 1-e^{-n/2}\). Since \(R_5=\widetilde{\calR}_t-\widehat{\calR}_t\), the first displayed bound follows from Lemma~\ref{lem:R5}, applied in the corresponding case of Assumption~\ref{assu:test-function}.

Using
\[
|\widehat{\calR}_t-\calR_t|
\le
|\widetilde{\calR}_t-\calR_t|
+
|\widehat{\calR}_t-\widetilde{\calR}_t|
=
|\widetilde{\calR}_t-\calR_t|+|R_5|,
\]
we obtain from Theorem~\ref{thm:rt-tilde}
together with Lemma~\ref{lem:R5} that
\begin{align*}
	\E[\I_\Omega |\widehat{\calR}_t-\calR_t|]
	&\le \E[|\widetilde{\calR}_t-\calR_t|]+\E[\I_\Omega |R_5|]\\
	&\le \frac{C(T,\gamma,\zeta,\kappa)}{\sqrt n} + \frac{C(T,\gamma,\zeta,\kappa)}{n^{1/4}}\\
	&\le \frac{C(T,\gamma,\zeta,\kappa)}{n^{1/4}}.
\end{align*}

For the probability bound,
\[
\P(|\widehat{\calR}_t-\calR_t|>\epsilon)
\le
\P(\Omega^c)
+
\P(\I_\Omega |\widehat{\calR}_t-\calR_t|>\epsilon).
\]
Since \(\P(\Omega^c)\le e^{-n/2}\) and
\[
\P(\I_\Omega |\widehat{\calR}_t-\calR_t|>\epsilon)
\le
\frac{\E[\I_\Omega |\widehat{\calR}_t-\calR_t|]}{\epsilon},
\]
the result follows.
\end{proof}

\subsection{Proof of Corollary~\ref{cor:early-stopping}}
\label{sec:proof-early-stopping}

\begin{proof}[Proof of Corollary~\ref{cor:early-stopping}]
Let
\[
\Delta_T=\max_{0\le t<T}|\widehat{\calR}_t-\calR_t|.
\]
By the definitions of \(\widehat t\) and \(t^\star\),
\[
\calR_{\widehat t}
\le
\widehat{\calR}_{\widehat t}+\Delta_T
\le
\widehat{\calR}_{t^\star}+\Delta_T
\le
\calR_{t^\star}+2\Delta_T.
\]
Therefore,
\[
\{\calR_{\widehat t}-\calR_{t^\star}>\epsilon\}
\subseteq
\{\Delta_T>\epsilon/2\}.
\]
The stated bound follows by applying Theorem~\ref{thm:rt-hat} to each
\(t=0,\ldots,T-1\) and taking a union bound.
\end{proof}

\section{Derivative identities and Stein tools}
\label{sec:useful-lemmas}
This section fixes the notation used in the appendix and records the main derivative identities and algebraic relations behind the correction matrices.
Short proofs are included inline when they are self-contained; longer auxiliary arguments are deferred to Section~\ref{sec:auxiliary-proofs}.

For $a,b \in \mathbb{R}$, we define 
$a \vee b \coloneqq \max(a,b)$ and $a \wedge b \coloneqq \min(a,b)$. 
We use $a \lesssim b$ to denote $a \leq Cb$ for some absolute constant $C > 0$ 
that does not depend on problem parameters. 
We use $C(\cdot)$ to denote a positive constant whose arguments list only the parameters on which it depends. By convention, structural parameters such as $T,m,r,k$ are listed first, followed by distributional parameters such as $\gamma,\zeta,\kappa$.

For each $t$, we define $\bh^t = \hbb^t - \bb^*$, and 
we define the following matrices that will be used in the proof:
\[
\bH = [\bh^0, \bh^1, \dots, \bh^{T-1}] \in \R^{p \times T}
\quad \text{and} \quad
\bF = [\bu^1, \bu^2, \dots, \bu^T] \in \R^{n \times T}.
\]
We also define the signal-error matrix
\[
\cbF
\defas
\frac{\bX\bH}{\sqrt n}
=
\left[
\frac{\bX(\hbb^0-\bb^*)}{\sqrt n},
\ldots,
\frac{\bX(\hbb^{T-1}-\bb^*)}{\sqrt n}
\right]\in\R^{n\times T}.
\]

We focus on the case $\bSigma = \bI_p$. 
The general case with general invertible covariance $\bSigma$ can be reduced to this setting via a change-of-variables argument described in Section~\ref{sec:reduction-isotropic}. Such a reduction is standard in the literature on generalization error estimation; see, for example, \cite{bellec2024uncertainty, tan2024estimating}. The only effect of this transformation is that the constants appearing in our main theorems additionally depend on $\kappa$ from Assumption~\ref{assu:X}.

The event \(\Omega\) defined in \eqref{eq:event-Omega} is used throughout the
finite-sample proofs. Under Assumptions~\ref{assu:X} and \ref{assu:regime}, we
have from \cite[Theorem II.13]{DavidsonS01} that 
\begin{align*}
	\P(\Omega) \ge 1 - e^{-n/2}. 
\end{align*}
Indeed, after writing \(\bG=\bX\bSigma^{-1/2}\), \cite[Theorem II.13]{DavidsonS01} gives
\[
\P\left(\|\bG\|_{\op}>\sqrt n+\sqrt p+s\right)\le e^{-s^2/2}.
\]
Since \(p/n\le\gamma\), the threshold
\((2+\sqrt\gamma)\sqrt n\) is at least \(\sqrt n+\sqrt p+\sqrt n\), so taking
\(s=\sqrt n\) gives the displayed probability.

\begin{lemma}\label{lem:moment-X}
		Under Assumptions~\ref{assu:X}--\ref{assu:regime}, we have for any finite integer $k$,
		\begin{align*}
			\E[\opnorm{\bX}^{2k}] \le n^{k}C(\gamma,\kappa,k).
		\end{align*}
	\end{lemma}
	\begin{proof}[Proof of Lemma~\ref{lem:moment-X}]
	\label{proof:lem:moment-X}
	Let \(\bG=\bX\bSigma^{-1/2}\).  By Assumption~\ref{assu:X},
	\(\opnorm{\bX}\le \sqrt{\kappa}\opnorm{\bG}\).  Moreover,
	\cite[Theorem II.13]{DavidsonS01} gives
	\[
	\P\left(\opnorm{\bG}>\sqrt n+\sqrt p+s\right)\le e^{-s^2/2},
	\qquad s\ge 0 .
	\]
	Set \(a=\sqrt n+\sqrt p\).  For \(r=2k\), the tail integral formula gives
	\[
	\E\opnorm{\bG}^{r}
	\le
	a^r+r\int_0^\infty (a+s)^{r-1}e^{-s^2/2}\,ds
	\le C(\gamma,k)n^{r/2},
	\]
	where the last inequality uses \(p/n\le\gamma\).  Therefore
	\[
	\E[\opnorm{\bX}^{2k}]
	\le \kappa^k\E[\opnorm{\bG}^{2k}]
	\le C(\gamma,\kappa,k)n^k .
	\]
	This completes the proof.
\end{proof}

\subsection{Derivative formulae}
\label{sec:derivative-formulae}

For orientation, fix an entry \(x_{ij}\), and let a dot denote differentiation
with respect to this entry; in particular,
\(\dot\bb^t=\partial\hbb^t/\partial x_{ij}\). Then the Jacobian blocks in
\eqref{eq:DJ-general} give the chain-rule recursions
\[
\dot\bu^t
=
\sum_{s=0}^{t-1}\Phi^u_{t,s}\dot\bu^s
+
\sum_{s=0}^{t-1}\Phi^v_{t,s}\dot\bv^s,
\qquad
\dot\bb^t
=
\sum_{s=0}^{t-1}\Psi^b_{t,s}\dot\bb^s
+
\sum_{s=0}^{t}\Psi^\eta_{t,s}\dot\bfeta^s.
\]
The lemmas below solve these recursions in the block-matrix form used in the
proofs.

Using the Jacobian blocks in \eqref{eq:DJ-general}, define the block matrices
\(\bPhi^u,\bPhi^v\in\R^{Tn\times Tn}\) and
\(\bPsi^b,\bPsi^\eta\in\R^{Tp\times Tp}\) by
\[
(\bPhi^u)_{t,s}=\Phi^u_{t,s}\I\{1\le s<t\},\qquad
(\bPhi^v)_{t,s}=\Phi^v_{t,s-1}\I\{1\le s\le t\},
\]
\[
(\bPsi^b)_{t,s}=\Psi^b_{t,s}\I\{1\le s<t\},\qquad
(\bPsi^\eta)_{t,s}=\Psi^\eta_{t,s}\I\{1\le s\le t\},
\]
for \(1\le t,s\le T\).
Let \(\bL=\sum_{t=2}^T \be_t\be_{t-1}^\top\) be the \(T\times T\) lag
matrix, and define
\[
\bB_1 := \bI_T \otimes
\begin{bmatrix}
\bI_n\\
\bzero_{p\times n}
\end{bmatrix},
\qquad
\bB_2 := \bI_T \otimes
\begin{bmatrix}
\bzero_{n\times p}\\
\bI_p
\end{bmatrix}.
\]
The block matrix used in the derivative formulas is
\begin{equation}
\label{eq:def-M}
\calM
=
\bI_{T(n+p)}
+
\begin{bmatrix}
\bB_1 & \bB_2
\end{bmatrix}
\begin{bmatrix}
-\bPhi^u & \frac{1}{\sqrt n}\,\bPhi^v(\bL\otimes \bX)\\[2mm]
-\frac{1}{\sqrt n}\,\bPsi^\eta (\bI_T\otimes \bX^\top) & -\bPsi^b
\end{bmatrix}
\begin{bmatrix}
\bB_1^\top \\[6pt]
\bB_2^\top
\end{bmatrix}.
\end{equation}
We also set
\begin{equation}
\label{eq:def-JH-JF}
\bJ_{\bPhi^v}
=
\frac{1}{\sqrt n}\bB_1\bPhi^v,
\qquad
\bJ_{\bPsi^\eta}
=
\frac{1}{\sqrt n}\bB_2\bPsi^\eta.
\end{equation}
Finally, define
\begin{equation}
\label{eq:MHMF}
\begin{aligned}
\calM_H
&=
-\bI_{nT}
+
(\bL\otimes \bX)
\bigl(\bI_T\otimes[\boldzero_{p\times n},\bI_p]\bigr)
\calM^{-1}\bJ_{\bPhi^v},\\
\calM_F
&=
(\bL\otimes \bX)
\bigl(\bI_T\otimes[\boldzero_{p\times n},\bI_p]\bigr)
\calM^{-1}\bJ_{\bPsi^\eta}.
\end{aligned}
\end{equation}

\begin{lemma}
\label{lem:dot-b-general}
For ${\bu^t}$ and ${\hbb^t}$ defined in \eqref{eq:general-iteration}, 
if $\bu^0$ and $\hbb^0$ are independent of $\bX$, then for any $t\ge 1$ and $i\in [n], j\in [p]$, we have 
\begin{align*}
\pdv{\bu^t}{x_{ij}}
&= (\be_t^\top \otimes \bI_n)
\bigl(\bI_T \otimes [\bI_n, \boldzero_{n\times p}]\bigr) \calM^{-1} 
\bigl[-\bJ_{\bPhi^v} \bigl((\bH^\top \be_j) \otimes \be_i\bigr) + \bJ_{\bPsi^\eta} \bigl((\bF^\top \be_i) \otimes \be_j\bigr)\bigr],\\
\pdv{\hbb^t}{x_{ij}}
&= (\be_t^\top \otimes \bI_p)
\bigl(\bI_T \otimes [\boldzero_{p\times n}, \bI_p]\bigr) \calM^{-1} 
\bigl[-\bJ_{\bPhi^v} \bigl((\bH^\top \be_j) \otimes \be_i\bigr) + \bJ_{\bPsi^\eta} \bigl((\bF^\top \be_i) \otimes \be_j\bigr)\bigr].
\end{align*}
where \(\calM\), \(\bJ_{\bPhi^v}\), and \(\bJ_{\bPsi^\eta}\) are defined in
\eqref{eq:def-M} and \eqref{eq:def-JH-JF}.
\end{lemma}

Let $\check\bu^t$ denote the \(t\)th column of \(\cbF\), then
\begin{equation}
\label{eq:check-u}
	\check\bu^{t}
	=
	\frac{\bX(\hbb^{t-1}-\bb^*)}{\sqrt n},
	\qquad t=1,\ldots,T,
\end{equation}
and we have the following derivative formula.

\begin{lemma}
\label{lem:dot-check-u}

For the sequence $\{\check\bu^t\}_{t=1}^T$ defined in \eqref{eq:check-u}, we have for any $t\ge 1$ and $i\in [n], j\in [p]$,
\begin{align*}
	\pdv{\check\bu^t}{x_{ij}}
	&= \frac{1}{\sqrt{n}} 
	(\be_t^\top \otimes \bI_n)
	\Bigl(
	-\calM_H \bigl((\bH^\top \be_j)\otimes \be_i\bigr)
	+ \calM_F \bigl((\bF^\top \be_i)\otimes \be_j\bigr)
	\Bigr),
\end{align*}
where \(\calM_H\) and \(\calM_F\) are defined in \eqref{eq:MHMF}.
\end{lemma}

\begin{proof}[Proof of Lemma~\ref{lem:dot-check-u}]
\label{proof:lem:dot-check-u}
We now derive the derivative of $\check\bu^t$ with respect to $x_{ij}$, which plays a key role in defining $\cbK$ and $\cbA$. 
Recall that
\begin{align*}
\check\bu^t
=
\frac{\bX\bh^{t-1}}{\sqrt n},
\qquad
\bh^{t-1}=\hbb^{t-1}-\bb^*.
\end{align*}
Therefore
\begin{align*}
	\pdv{\check\bu^t}{x_{ij}}
	&=
	\frac{1}{\sqrt{n}}
	\bigl(\dot\bX\bh^{t-1} + \bX\dot\bb^{t-1}\bigr),
\end{align*}
which is the same derivative as that of \(-\bv^{t-1}\), since \(-\bv^{t-1}=\bX\bh^{t-1}/\sqrt n-\bep/\sqrt n\) and the noise vector is held fixed in the pathwise derivative.

Since $\bH = [\bh^0, \bh^1, \ldots, \bh^{T-1}]$, we have
\[
\bh^{t-1} = \bH \be_t, \qquad t=1,\ldots,T.
\]
Moreover,
\begin{align*}
\dot\bX \bh^{t-1}
= \be_i \be_j^\top \bh^{t-1}
= \be_i \be_j^\top \bH \be_t
= (\be_t^\top \otimes \bI_n)(\bH^\top \be_j \otimes \be_i).
\end{align*}

For compactness, set
\[
\bq_{ij}
:=
\calM^{-1}
\Bigl[
-\bJ_{\bPhi^v}\bigl((\bH^\top\be_j)\otimes\be_i\bigr)
+\bJ_{\bPsi^\eta}\bigl((\bF^\top\be_i)\otimes\be_j\bigr)
\Bigr].
\]
Then Lemma~\ref{lem:dot-b-general} gives
\begin{align*}
\pdv{\hbb^t}{x_{ij}}
&= (\be_t^\top \otimes \bI_p)
\bigl(\bI_T \otimes [\boldzero_{p\times n}, \bI_p]\bigr)\bq_{ij}.
\end{align*}
Therefore, for $t=1$, we have $\bX \pdv{\hbb^{t-1}}{x_{ij}} = \boldzero$, while for $t\ge 2$,
\begin{align*}
\bX \pdv{\hbb^{t-1}}{x_{ij}}
&= (\be_{t-1}^\top \otimes \bX)
\bigl(\bI_T \otimes [\boldzero_{p\times n}, \bI_p]\bigr)\bq_{ij}\\
&= (\be_{t-1}^\top \otimes \bI_n)
(\bI_T \otimes \bX)
\bigl(\bI_T \otimes [\boldzero_{p\times n}, \bI_p]\bigr)\bq_{ij}.
\end{align*}

Introducing the lag matrix $\bL = \sum_{t=2}^T \be_t \be_{t-1}^\top$, we can write, for all $t\ge 1$,
\begin{align*}
\bX \pdv{\hbb^{t-1}}{x_{ij}}
&= (\be_t^\top \otimes \bI_n)
(\bL \otimes \bX)
\bigl(\bI_T \otimes [\boldzero_{p\times n}, \bI_p]\bigr)\bq_{ij}.
\end{align*}

Combining the above expressions, we obtain
\begin{align*}
	\pdv{\check\bu^t}{x_{ij}}
	&= \frac{1}{\sqrt{n}} 
	(\be_t^\top \otimes \bI_n)
	\Bigl(
	-\calM_H \bigl((\bH^\top \be_j)\otimes \be_i\bigr)
	+ \calM_F \bigl((\bF^\top \be_i)\otimes \be_j\bigr)
	\Bigr),
\end{align*}
with \(\calM_H\) and \(\calM_F\) as in \eqref{eq:MHMF}.
\end{proof}

We also have the following identity that summarizes the tensor contraction of the derivatives of $\bH$ and $\bF$ with respect to $x_{ij}$, which give rise to the definitions of
$\bW$, \(\hbK\), \(\cbK\), and \(\cbA\). Here \(\bW\) is the weight matrix used in \eqref{eq:rt-tilde}; \(\hbK\) is an auxiliary contraction associated with the derivative of \(\bF\) and is not used to define \(\hbW\); and we replace \(\bW\) by \(\hbW := \cbK^{-1}\cbA\) in the definition of \(\widehat{\calR}_t\). The matrices
\(\bUpsilon_1,\bUpsilon_2,\bUpsilon_3,\check\bUpsilon_1,\check\bUpsilon_5\)
below are residual matrices: they collect the terms left over after the leading
contractions are expressed through \(\bW,\hbK,\cbK,\cbA\), and their explicit
algebraic forms are given in the proof of the lemma.

\begin{lemma}
\label{lem:sum-derivative}
For $\bSigma = \bI_p$, we have the following identities:
\begin{align}
	\label{eq:sum-dF}
	&\sum_{i=1}^n \pdv{\be_i^\top \bF}{x_{ij}} 
	= -\be_j^\top \bH \hbK^\top +\be_j^\top \bUpsilon_1^\top,\\
	\label{eq:sum-dH}
	&\sum_{j=1}^p \pdv{\be_j^\top \bH}{x_{ij}}
	= \be_i^\top \bF \bW^\top - \be_i^\top \bUpsilon_2^\top,\\
	\label{eq:sum-dFH}
	&\sum_{i=1}^n\sum_{j=1}^p \pdv{\bF^\top \be_i\be_j^\top \bH}{x_{ij}}
	= - \hbK \bH^\top \bH + \bF^\top \bF \bW^\top 
	- \bUpsilon_3\\
	&\sum_{i=1}^n \pdv{\be_i^\top \cbF}{x_{ij}}
	= -\be_j^\top \bH \cbK^\top + \be_j^\top \check\bUpsilon_1^\top, \label{eq:checkF-derivative-1}\\
	&\sum_{i=1}^n\sum_{j=1}^p \pdv{\bF^\top \be_i\be_j^\top \bX^\top \cbF}{x_{ij}}
	= - \hbK \bH^\top \bX^\top \cbF + p \bF^\top \cbF - \bF^\top \bF \cbA^\top - \check\bUpsilon_5, \label{eq:checkF-derivative-2}
\end{align}
where 
\begin{align*}
\bW 
&= \sum_{j=1}^p 
	(\bI_T\otimes \be_j^\top) 
	(\bL\otimes \bI_p)
	\bigl(\bI_T \otimes [\boldzero_{p\times n}, \bI_p]\bigr) \calM^{-1} \bJ_{\bPsi^\eta} 
	(\bI_T\otimes \be_j),\\
\hbK
&= \sum_{i=1}^n 
	(\bI_T\otimes \be_i^\top) 
	\bigl(\bI_T \otimes [\bI_n, \boldzero_{n\times p}]\bigr) \calM^{-1} \bJ_{\bPhi^v} 
	(\bI_T\otimes \be_i),\\
\cbK
&= \frac{1}{\sqrt{n}} \sum_{i=1}^n (\bI_T\otimes \be_i^\top)\calM_H (\bI_T \otimes \be_i),\\
\cbA
&=-
\frac{1}{\sqrt n}\sum_{j=1}^p
(\bI_T\otimes \be_j^\top \bX^\top)
\calM_F(\bI_T\otimes \be_j).
\end{align*}
\end{lemma}

The triangular structure follows from the time ordering of the derivative
system. Since \(\calM\), \(\bJ_{\bPhi^v}\), and \(\bJ_{\bPsi^\eta}\) are block
lower triangular, the products
\(\calM^{-1}\bJ_{\bPhi^v}\) and \(\calM^{-1}\bJ_{\bPsi^\eta}\) are block lower
triangular. The lag matrix \(\bL\) shifts every block row forward by one time
step, so the definitions of \(\bW\) and \(\cbA\) have zero diagonal entries.
The definitions of \(\hbK\) and \(\cbK\) preserve lower triangularity. On the
diagonal, only the direct block
\(\frac1{\sqrt n}\Phi^v_{t,t-1}\) contributes to \(\hbK\), while only the
\(-\bI_{nT}\) term in \(\calM_H\) contributes to \(\cbK\). Hence
\[
(\hbK)_{t,t}
=
\frac{1}{\sqrt n}\trace(\Phi^v_{t,t-1}), 
\qquad
(\cbK)_{t,t}
=-\sqrt n,
\qquad t=1,\dots,T.
\]
The advantage of using the signal-error stack \(\cbF=\bX\bH/\sqrt n\) is that
the deterministic diagonal of \(\cbK\) guarantees invertibility, which allows us
to define \(\hbW=\cbK^{-1}\cbA\) in a stable manner. In contrast, \(\hbK\) may
not be invertible. For example, if \(\bphi\) is the entrywise clip function
\(\phi(x)=\clip(x,-1/\sqrt n,1/\sqrt n)\), then its derivative is zero whenever
the input is outside the interval \([-1/\sqrt n,1/\sqrt n]\). In this case,
\(\hbK\) can have zero diagonal entries.

Our next proposition provides the weighted relation that makes the
covariance-free estimator possible. The implicit derivative approximations for
\(\bW\), \(\cbK\), and \(\cbA\) stated in the main text are the right
intuition, but they are not sufficient by themselves: the proof also needs the
mixed product-rule cancellation produced when the two derivative contractions
are multiplied. We encode the three formal residuals needed for this argument
as \(\bTheta_1\), \(\cbTheta_2\), and \(\cbTheta_4\). Here
\(\bTheta_1\) is the population \(\bW\) Stein residual, \(\cbTheta_2\) is the
\(\cbA\) Stein residual after the \(\bF^\top\)-weighting, and
\(\cbTheta_4\) is the mixed product-rule residual coupling the \(\cbK\)
contraction with the population correction. The exact bridge to
Proposition~\ref{prop:approx-W} is the algebraic identity in the proof below.

We define the following $T\times T$ matrices: 
\begin{align*}
	\bTheta_1 
	&\defas \bF^\top \bX \bH + \hbK \bH^\top \bH - \bF^\top \bF \bW^\top,\\
	\cbTheta_2 &\defas n^{-1} [\bF^\top \bX \bX^\top \cbF + \hbK \bH^\top \bX^\top \cbF - p\bF^\top \cbF
	+ \bF^\top \bF \cbA^\top],\\
	\cbTheta_4
	&\defas n^{-1} [p \bF^\top \cbF  - (\hbK \bH^\top + \bF^\top \bX)(\cbK \bH^\top + \cbF^\top \bX)^\top].
\end{align*}

\begin{lemma}
	\label{lem:moment-theta}
Under Assumptions~\ref{assu:X},\ref{assu:regime},\ref{assu:algorithm}, and $\bSigma = \bI_p$, we have 
\begin{align*}
	\E[\fnorm{\bTheta_1}^2] 
	&\le C(T,\gamma,\zeta),\\
	\E[\fnorm{\cbTheta_2}^2] 
	&\le n^{-1} C(T,\gamma,\zeta),\\
	\E[\fnorm{\cbTheta_4}] 
	&\le n^{-1/2} C(T,\gamma,\zeta).
\end{align*}
\end{lemma}

\begin{proof}[Proof of Lemma~\ref{lem:moment-theta}]
\label{proof:lem:moment-theta}
	We first notice that the following identities hold:
\begin{align*}
	&\bF^\top \bX \bH 
	- \sum_{i=1}^n\sum_{j=1}^p \pdv{\bF^\top \be_i\be_j^\top \bH}{x_{ij}}
	= \bTheta_1 + \bUpsilon_3,\\
	&\bF^\top \bX \bX^\top \cbF - \sum_{i=1}^n\sum_{j=1}^p \pdv{\bF^\top \be_i\be_j^\top \bX^\top \cbF}{x_{ij}}
	= n \cbTheta_2 + \check\bUpsilon_5,\\
	&\Bigl[p \bF^\top \cbF - \sum_{j=1}^p \Bigl(\sum_{i=1}^n \pdv{\bF^\top \be_i}{x_{ij}} - \bF^\top \bX \be_j\Bigr)
		\Bigl(\sum_{i=1}^n \pdv{\cbF^\top \be_i}{x_{ij}} - \cbF^\top \bX \be_j\Bigr)^\top
	\Bigr]\\
	&= n \cbTheta_4 + 
	\underbrace{\bUpsilon_1 \Bigl(\cbK \bH^\top + \cbF^\top \bX \Bigr)^\top
	+ \Bigl(\hbK \bH^\top + \bF^\top \bX \Bigr) \check\bUpsilon_1^\top
	- \bUpsilon_1 \check\bUpsilon_1^\top}_{\check\bUpsilon_4}.
\end{align*}
It follows from Lemma E.10 of \cite{tan2022noise} that 
\begin{align*}
	\E[\fnorm{\bTheta_1}^2] 
	&\lesssim \E\Bigl[\fnorm{\bF^\top \bX \bH 
	- \sum_{i=1}^n\sum_{j=1}^p \pdv{\bF^\top \be_i\be_j^\top \bH}{x_{ij}}}^2 \Bigr] + \E[\opnorm{\bUpsilon_3}^2],\\
	&\le \E [\fnorm{\bF}^2 \fnorm{\bH}^2]+ \E \sum_{ij}\Big[
	\fnorm{\bH}^2\fnorm{ \pdv{\bF}{x_{ij}} }^2
	+ \fnorm{\bF}^2\fnorm{ \pdv{\bH}{x_{ij}} }^2 \Big] + \E[\fnorm{\bUpsilon_3}^2],\\
	&\le C(T,\gamma,\zeta), 
\end{align*}
where the last inequality follows from Lemma~\ref{lem:moment-H-F} and Lemma~\ref{lem:moment-Upsilon}.

Similarly, we have 
\begin{align*}
	n^2 \E[\fnorm{\cbTheta_2}^2] 
	&\lesssim \E\Bigl[\fnorm{\bF^\top \bX \bX^\top \cbF - \sum_{i=1}^n\sum_{j=1}^p \pdv{\bF^\top \be_i\be_j^\top \bX^\top \cbF}{x_{ij}}}^2\Bigr] + \E[\opnorm{\check\bUpsilon_5}^2],\\
	&\le \E [\fnorm{\cbF}^4 \opnorm{\bX}^2]+ \E \Big[(1+\opnorm{\bX}^2)
	\fnorm{\bF}^2\sum_{ij}\fnorm{ \pdv{\cbF}{x_{ij}} }^2
	\Big] + \E[\opnorm{\check\bUpsilon_5}^2]\\
	&\le n C(T,\gamma,\zeta).
\end{align*}

By definition of \(\check\bUpsilon_4\) and submultiplicativity,
\begin{align*}
\|\check\bUpsilon_4\|_{\op}
\le\;&
\|\bUpsilon_1\|_{\op}\Bigl(\|\cbK\|_{\op}\|\bH\|_{\rm F}
+ \|\cbF\|_{\rm F}\|\bX\|_{\op}\Bigr)\\
&+
\Bigl(\|\hbK\|_{\op}\|\bH\|_{\rm F}
+ \|\bF\|_{\rm F}\|\bX\|_{\op}\Bigr)\|\check\bUpsilon_1\|_{\op}
+ \|\bUpsilon_1\|_{\op}\|\check\bUpsilon_1\|_{\op}.
\end{align*}
Moreover, the definition of \(\hbK\) and the triangle inequality give
\[
\|\hbK\|_{\op}
\le
n\Bigl\|
\bigl(\bI_T \otimes [\bI_n,\boldzero_{n\times p}]\bigr)
\calM^{-1}\bJ_{\bPhi^v}
\Bigr\|_{\op}.
\]
Hence, by Lemmas~\ref{lem:opnorm-Minv-JH-JF}, \ref{lem:moment-H-F},
\ref{lem:moment-Upsilon}, and \ref{lem:op-bound-WAK},
\[
\E[\|\check\bUpsilon_4\|_{\op}^2]
\le
C(T,\gamma,\zeta).
\]

Using Lemma E.12 of \cite{tan2022noise}, we have
\begin{align*}
	n \E[\fnorm{\cbTheta_4}]
	&\lesssim (1 + 2\sqrt{p}) 
	\bigl( 
		\E [\fnorm{\bF}^4]^{1/2}
		+ \E [\norm{\cbF}^4_{\partial}]^{1/2}
		\bigr) 
+ \E[\opnorm{\check\bUpsilon_{4}}^2]\\
&\le \sqrt{n} C(T,\gamma,\zeta).
\end{align*}
\end{proof}

\begin{lemma}
\label{lem:moment-Upsilon}
Under Assumptions~\ref{assu:X},\ref{assu:regime},\ref{assu:algorithm}, and $\bSigma = \bI_p$, the matrices $\bUpsilon_1, \bUpsilon_2, \bUpsilon_3, \check\bUpsilon_1, \check\bUpsilon_5$ in Lemma~\ref{lem:sum-derivative} satisfy 
\begin{align*}
	\max_{k\in\{1,2,3\}}\E[\opnorm{\bUpsilon_k}^2] \le \frac{C(T,\gamma,\zeta)}{n}, \quad 
\E[\opnorm{\check\bUpsilon_1}^2] 
	\le \frac{C(T,\gamma,\zeta)}{n}, \quad
	\E[\opnorm{\check\bUpsilon_5}^2] 
	\le C(T,\gamma,\zeta).
\end{align*}
\end{lemma}

\begin{proof}[Proof of Lemma~\ref{lem:moment-Upsilon}]
\label{proof:moment-Upsilon}
For $\bUpsilon_1$, using the inequality \eqref{eq:opnorm-Upsilon1}, we obtain
\begin{align*}
	\E[\opnorm{\bUpsilon_1}^2]
	&\le \E[(T \opnorm{\calM^{-1} \bJ_{\bPsi^\eta}} \opnorm{\bF})^2]\\
	&\le T^2 \E[\opnorm{\calM^{-1} \bJ_{\bPsi^\eta}}^4]^{1/2} \E[\opnorm{\bF}^4]^{1/2}\\
	&\le \frac{C(T,\gamma,\zeta)}{n},
\end{align*}
where the last inequality follows from Lemma~\ref{lem:opnorm-Minv-JH-JF} and Lemma~\ref{lem:moment-H-F}.

For $\bUpsilon_2$, using the inequality \eqref{eq:opnorm-Upsilon2}, we have
\begin{align*}
	\E[\opnorm{\bUpsilon_2}^2]
	&\le \E[(T \opnorm{\calM^{-1} \bJ_{\bPhi^v}} \opnorm{\bH})^2]\\
	&\le T^2 \E[\opnorm{\calM^{-1} \bJ_{\bPhi^v}}^4]^{1/2} \E[\opnorm{\bH}^4]^{1/2}\\
	&\le \frac{C(T,\gamma,\zeta)}{n}.
\end{align*}

For $\bUpsilon_3$, using the inequality \eqref{eq:opnorm-Upsilon3}, we have
\begin{align*}
	\E[\opnorm{\bUpsilon_3}^2]
	&\le \E[(\opnorm{\bUpsilon_1}\fnorm{\bH} + \fnorm{\bF} \opnorm{\bUpsilon_2})^2]\\
	&\le 2\E[\opnorm{\bUpsilon_1}^2 \fnorm{\bH}^2] + 2\E[\fnorm{\bF}^2 \opnorm{\bUpsilon_2}^2]\\
	&\le 2\E[\opnorm{\bUpsilon_1}^4]^{1/2} \E[\fnorm{\bH}^4]^{1/2} + 2\E[\fnorm{\bF}^4]^{1/2} \E[\opnorm{\bUpsilon_2}^4]^{1/2}\\
	&\le \frac{C(T,\gamma,\zeta)}{n}.
\end{align*}

For $\check\bUpsilon_1$, using the inequality \eqref{eq:opnorm-checkUpsilon1}, we have 
\begin{align*}
	\E[\opnorm{\check\bUpsilon_1}^2]
	&\le \frac{T^2}{n} \E[( \opnorm{\calM_F} \fnorm{\bF})^2]\\
	&\le \frac{C(T,\gamma,\zeta)}{n}.
\end{align*}

For $\check\bUpsilon_5$, recall from the proof of Lemma~\ref{lem:sum-derivative} that
\[
\check\bUpsilon_5
=
\bUpsilon_1\bX^\top \cbF + (\bUpsilon_5^*)^\top,
\]
with
\[
\opnorm{\bUpsilon_5^*}
\le
C(T)\frac{\opnorm{\bX}}{\sqrt n}
\opnorm{\calM_H}\fnorm{\bH}\fnorm{\bF}.
\]
Therefore,
\begin{align*}
	\E[\opnorm{\check\bUpsilon_5}^2]
	&\le
	2\,\E[\opnorm{\bUpsilon_1}^2\opnorm{\bX}^2\fnorm{\cbF}^2]
	+
	2\,\E[\opnorm{\bUpsilon_5^*}^2]\\
	&\le
	C(T,\gamma,\zeta),
\end{align*}
where we used Lemmas~\ref{lem:moment-H-F} and \ref{lem:opnorm-Minv-JH-JF},
and the operator-norm bounds on \(\calM_H\) from
Lemma~\ref{lem:opnorm-Minv-JH-JF}.

This completes the proof of Lemma~\ref{lem:moment-Upsilon}.
\end{proof}

\begin{proposition}
\label{prop:approx-W}
Under Assumptions~\ref{assu:X}--\ref{assu:algorithm}, we have
\begin{align*} 
	\E[\fnorm{\bF^\top \bF (\cbA^\top - \bW^\top \cbK^\top)}] \le \sqrt{n} C(T,\gamma,\zeta).
\end{align*}
\end{proposition}

\begin{proof}[Proof of Proposition~\ref{prop:approx-W}]
\label{proof:prop:approx-W}
By the definitions of \(\bTheta_1\), \(\cbTheta_2\), and \(\cbTheta_4\),
\begin{align*}
&n^{-1}\bTheta_1\cbK^\top+\cbTheta_2+\cbTheta_4\\
&=
\frac1n\Big[
\bigl(\bF^\top\bX\bH+\hbK\bH^\top\bH
-\bF^\top\bF\bW^\top\bigr)\cbK^\top\\
&\qquad{}+\bF^\top\bX\bX^\top\cbF
+\hbK\bH^\top\bX^\top\cbF
-p\bF^\top\cbF
+\bF^\top\bF\cbA^\top\\
&\qquad{}+p\bF^\top\cbF
-(\hbK\bH^\top+\bF^\top\bX)(\cbK\bH^\top+\cbF^\top\bX)^\top
\Big].
\end{align*}
Expanding the final product gives
\begin{align*}
(\hbK\bH^\top+\bF^\top\bX)(\cbK\bH^\top+\cbF^\top\bX)^\top
&=
\hbK\bH^\top\bH\cbK^\top
+\hbK\bH^\top\bX^\top\cbF\\
&\quad{}+\bF^\top\bX\bH\cbK^\top
+\bF^\top\bX\bX^\top\cbF.
\end{align*}
All cross-terms cancel, leaving the exact identity
\begin{align*}
	n^{-1}\bTheta_1 \cbK^\top +\cbTheta_2+\cbTheta_4 
	&= n^{-1}\bF^\top \bF \big(\cbA^\top-\bW^\top \cbK^\top\big).
\end{align*}
Therefore,
\begin{align*}
	\E[\fnorm{\bF^\top \bF (\cbA^\top - \bW^\top \cbK^\top)}]
	&\le \E[\fnorm{\bTheta_1 \cbK^\top}] + n \bigl(\E[\fnorm{\cbTheta_2}] + 
	\E[\fnorm{\cbTheta_4}]\bigr)\\
	&\le \E[\fnorm{\bTheta_1}^2]^{1/2} \E[\opnorm{\cbK}^2]^{1/2} + n \bigl(\E[\fnorm{\cbTheta_2}^2]^{1/2} +
	\E[\fnorm{\cbTheta_4}]\bigr). 
\end{align*}
The desired upper bounds then follow by the moment bounds of $\opnorm{\cbK}$ in Lemma~\ref
{lem:op-bound-WAK}, and moment bounds of $\fnorm{\bTheta_1}$, $\fnorm{\cbTheta_2}$, $\fnorm{\cbTheta_4}$ in Lemma~\ref{lem:moment-theta}.
This completes the proof.
\end{proof}

We now give the deferred proof of Lemma~\ref{lem:sum-derivative}.

\begin{proof}[Proof of Lemma~\ref{lem:sum-derivative}]
\label{pf:sum-derivative}
We first prove \eqref{eq:sum-dF}. We will use the following cancellation identities:
\begin{equation}
\label{eq:sum-to-I}
\sum_{i=1}^n \be_i \be_i^\top = \bI_n,
\qquad
\sum_{j=1}^p \be_j \be_j^\top = \bI_p,
\qquad
\sum_{t=1}^T \be_t \be_t^\top = \bI_T.
\end{equation}
For compactness, write
\begin{align*}
\calP_{\Phi}
&:= \bigl(\bI_T \otimes [\bI_n,\boldzero_{n\times p}]\bigr)
\calM^{-1}\bJ_{\bPhi^v},
&
\calP_{\Psi}
&:= \bigl(\bI_T \otimes [\bI_n,\boldzero_{n\times p}]\bigr)
\calM^{-1}\bJ_{\bPsi^\eta},\\
\calQ_{\Phi}
&:= \bigl(\bI_T \otimes [\boldzero_{p\times n},\bI_p]\bigr)
\calM^{-1}\bJ_{\bPhi^v},
&
\calQ_{\Psi}
&:= \bigl(\bI_T \otimes [\boldzero_{p\times n},\bI_p]\bigr)
\calM^{-1}\bJ_{\bPsi^\eta}.
\end{align*}
By definition,
\[
\bF = [\bu^1, \bu^2, \ldots, \bu^T] = \sum_{t=1}^T \bu^t \be_t^\top.
\]
Therefore,
\begin{align*}
\sum_{i=1}^n \pdv{\be_i^\top \bF}{x_{ij}}
&= \sum_{i=1}^n \sum_{t=1}^T \be_i^\top \pdv{\bu^t}{x_{ij}} \be_t^\top\\
&= \sum_{i=1}^n \sum_{t=1}^T \be_i^\top 
(\be_t^\top \otimes \bI_n)
\Bigl[-\calP_{\Phi}\bigl((\bH^\top \be_j) \otimes \be_i\bigr)
+\calP_{\Psi}\bigl((\bF^\top \be_i) \otimes \be_j\bigr)\Bigr]\be_t^\top\\
&= -\sum_{i=1}^n \sum_{t=1}^T \be_i^\top 
(\be_t^\top \otimes \bI_n)
\calP_{\Phi}\bigl((\bH^\top \be_j) \otimes \be_i\bigr)
\be_t^\top\\
&\qquad
+ \sum_{i=1}^n \sum_{t=1}^T \be_i^\top 
(\be_t^\top \otimes \bI_n)
\calP_{\Psi}\bigl((\bF^\top \be_i) \otimes \be_j\bigr)
\be_t^\top.
\end{align*}
We now simplify the final two terms in the last display.

For the first term,
\begin{align*}
&\sum_{i=1}^n \sum_{t=1}^T \be_i^\top 
(\be_t^\top \otimes \bI_n)
\bigl(\bI_T \otimes [\bI_n, \boldzero_{n\times p}]\bigr) \calM^{-1} \bJ_{\bPhi^v} \bigl((\bH^\top \be_j) \otimes \be_i\bigr)
\be_t^\top\\
&=
\Bigl[\sum_{i=1}^n \sum_{t=1}^T  
\be_t(\be_t^\top \otimes \be_i^\top)
\calP_{\Phi}\bigl((\bH^\top \be_j) \otimes \be_i\bigr)\Bigr]
^\top\\
&=
\Bigl[\underbrace{\sum_{i=1}^n (\bI_T \otimes \be_i^\top)
\calP_{\Phi}(\bI_T \otimes \be_i)}_{\hbK} \bH^\top \be_j\Bigr]
^\top\\
&= \be_j^\top \bH \hbK^\top.
\end{align*}

For the second term,
\begin{align*}
&\sum_{i=1}^n \sum_{t=1}^T \be_i^\top 
(\be_t^\top \otimes \bI_n)
\bigl(\bI_T \otimes [\bI_n, \boldzero_{n\times p}]\bigr) \calM^{-1} \bJ_{\bPsi^\eta} \bigl((\bF^\top \be_i) \otimes \be_j\bigr)
\be_t^\top\\
&=
\Bigl[\sum_{i=1}^n \sum_{t=1}^T 
\be_t(\be_t^\top \otimes \be_i^\top)
\calP_{\Psi}\bigl((\bF^\top \be_i) \otimes \be_j\bigr)
\Bigr]^\top\\
&=
\Bigl[\sum_{i=1}^n 
(\bI_T \otimes \be_i^\top)
\calP_{\Psi}\bigl((\bF^\top \be_i) \otimes \be_j\bigr)
\Bigr]^\top\\
&=
\Bigl[ \underbrace{\sum_{i=1}^n 
(\bI_T \otimes \be_i^\top)
\calP_{\Psi}\bigl((\bF^\top \be_i) \otimes \bI_p\bigr)}_{\bUpsilon_1}\be_j
\Bigr]^\top\\
&= \be_j^\top \bUpsilon_1^\top.
\end{align*}

We rewrite $\bUpsilon_1$ as
\begin{align*}
\bUpsilon_1 
&= 
\sum_{i=1}^n \sum_{t=1}^T
(\bI_T \otimes \be_i^\top)
\calP_{\Psi}\bigl((\be_t\be_t^\top \bF^\top \be_i) \otimes \bI_p\bigr)\\
&= 
\sum_{i=1}^n \sum_{t=1}^T
(\bI_T \otimes (\be_t^\top \bF^\top \be_i \be_i^\top))
\calP_{\Psi}\bigl(\be_t \otimes \bI_p\bigr)\\
&= 
\sum_{t=1}^T
(\bI_T \otimes (\be_t^\top \bF^\top))
\calP_{\Psi}\bigl(\be_t \otimes \bI_p\bigr).
\end{align*}
It follows that
\begin{equation}\label{eq:opnorm-Upsilon1}
\opnorm{\bUpsilon_1} 
\le T \opnorm{\calM^{-1} \bJ_{\bPsi^\eta}} \fnorm{\bF}.
\end{equation}
Combining the above identities proves \eqref{eq:sum-dF}.

Next, we prove \eqref{eq:sum-dH}. By definition,
\[
\bH = [\bh^0, \bh^1, \ldots, \bh^{T-1}] = \sum_{t=1}^{T} \bh^{t-1} \be_t^\top.
\]
Since $\dot\bh^0=\dot\bb^0=\boldzero$ because $\hbb^0$ is independent of the data $(\bX, \by)$, we have
\begin{align*}
\sum_{j=1}^p \pdv{\be_j^\top \bH}{x_{ij}}
&= \sum_{j=1}^p \sum_{t=1}^T \be_j^\top \pdv{\bh^{t-1}}{x_{ij}} \be_t^\top\\
&= \sum_{j=1}^p \sum_{t=1}^T \be_j^\top \pdv{\hbb^{t-1}}{x_{ij}} \be_t^\top\\
&= \sum_{j=1}^p \sum_{t=2}^T \be_j^\top  \pdv{\hbb^{t-1}}{x_{ij}} \be_t^\top \\
&= \sum_{j=1}^p \sum_{t=2}^T \be_j^\top 
(\be_{t-1}^\top \otimes \bI_p)
\Bigl[-\calQ_{\Phi}\bigl((\bH^\top \be_j) \otimes \be_i\bigr)
+\calQ_{\Psi}\bigl((\bF^\top \be_i) \otimes \be_j\bigr)\Bigr]
\be_t ^\top\\
&= \be_i^\top \bF \bW^\top - \be_i^\top \bUpsilon_2^\top.
\end{align*}
Here the last line uses
\begin{align*}
&\sum_{j=1}^p \sum_{t=2}^T \be_j^\top  
(\be_{t-1}^\top \otimes \bI_p)
\calQ_{\Psi}\bigl((\bF^\top \be_i) \otimes \be_j\bigr)
\be_t ^\top \\
&= \Bigl[\sum_{j=1}^p \sum_{t=2}^T \be_t  
(\be_{t-1}^\top \otimes \be_j^\top )
\calQ_{\Psi}\bigl((\bF^\top \be_i) \otimes \be_j\bigr)
\Bigr]^\top \\
&= \Bigl[\underbrace{\sum_{j=1}^p   
(\bL \otimes \be_j^\top )
\calQ_{\Psi}(\bI_T \otimes \be_j)}_{\bW} \bF^\top \be_i
\Bigr]^\top\\
&= \be_i^\top \bF \bW^\top,
\end{align*}
and
\begin{align*}
&\sum_{j=1}^p \sum_{t=2}^T \be_j^\top  (\be_{t-1}^\top \otimes \bI_p)
\calQ_{\Phi}\bigl((\bH^\top \be_j) \otimes \be_i\bigr)\be_t^\top\\
&= \Bigl[\sum_{j=1}^p \sum_{t=2}^T  \be_t (\be_{t-1}^\top \otimes \be_j^\top )
\calQ_{\Phi}\bigl((\bH^\top \be_j) \otimes \be_i\bigr)
\Bigr]^\top\\
&= \Bigl[\underbrace{\sum_{j=1}^p  (\bL \otimes \be_j^\top )
\calQ_{\Phi}\bigl((\bH^\top \be_j) \otimes \bI_n \bigr)}_{\bUpsilon_2} \be_i
\Bigr]^\top
&&\mbox{since $\bL = \sum_{t=2}^T \be_t \be_{t-1}^\top$}\\
&= \be_i^\top \bUpsilon_2^\top.
\end{align*}

We rewrite $\bUpsilon_2$ as
\begin{align*}
\bUpsilon_2
&= \sum_{j=1}^p \sum_{t=1}^T (\bL \otimes \be_j^\top )
\calQ_{\Phi}\bigl((\be_t\be_t^\top \bH^\top \be_j) \otimes \bI_n \bigr)
\\
&= \sum_{j=1}^p \sum_{t=1}^T (\bL \otimes (\be_t^\top \bH^\top \be_j\be_j^\top ))
\calQ_{\Phi}\bigl(\be_t \otimes \bI_n \bigr)
\\
&= \sum_{t=1}^T (\bL \otimes (\be_t^\top \bH^\top ))
\calQ_{\Phi}\bigl(\be_t \otimes \bI_n \bigr).
\end{align*}
It follows that
\begin{equation}
\label{eq:opnorm-Upsilon2}
\opnorm{\bUpsilon_2} 
\le T\opnorm{\calM^{-1} \bJ_{\bPhi^v}} \fnorm{\bH}.
\end{equation}

Next, we prove \eqref{eq:sum-dFH}. Using the product rule for differentiation and \eqref{eq:sum-dF}, \eqref{eq:sum-dH}, we have
\begin{align*}
&\sum_{i=1}^n\sum_{j=1}^p \pdv{\bF^\top \be_i\be_j^\top \bH}{x_{ij}}\\
&= \sum_{i=1}^n\sum_{j=1}^p \pdv{\bF^\top \be_i}{x_{ij}}\be_j^\top \bH + \bF^\top \be_i \pdv{\be_j^\top \bH}{x_{ij}}\\
&= - (\hbK \bH^\top - \bUpsilon_1) \bH  + \bF^\top (\bF \bW^\top - \bUpsilon_2^\top) \\
&= -\hbK \bH^\top \bH + \bF^\top \bF \bW^\top - 
\underbrace{(-\bUpsilon_1 \bH + \bF^\top \bUpsilon_2^\top)}_{\bUpsilon_3}.
\end{align*}
We have
\begin{equation}\label{eq:opnorm-Upsilon3}
\opnorm{\bUpsilon_3}
\le \opnorm{\bUpsilon_1}\fnorm{\bH} + \fnorm{\bF} \opnorm{\bUpsilon_2}.
\end{equation}
We can further bound $\opnorm{\bUpsilon_1}$ and $\opnorm{\bUpsilon_2}$ using
\eqref{eq:opnorm-Upsilon1} and \eqref{eq:opnorm-Upsilon2}.

We now prove \eqref{eq:checkF-derivative-1}. Since
\[
\cbF = [\check\bu^1,\dots,\check\bu^T],
\]
we write
\begin{align*}
\sum_{i=1}^n \pdv{\be_i^\top \cbF}{x_{ij}}
= \sum_{i=1}^n \sum_{t=1}^T \pdv{\be_i^\top \check\bu^t}{x_{ij}} \be_t^\top.
\end{align*}
Substituting the derivative formula
\[
\pdv{\check\bu^t}{x_{ij}}
=
\frac{1}{\sqrt{n}} 
(\be_t^\top \otimes \bI_n)
\Bigl(
-\calM_H \bigl((\bH^\top \be_j) \otimes \be_i\bigr)
+
\calM_F \bigl((\bF^\top \be_i) \otimes \be_j\bigr)
\Bigr),
\]
we obtain
\begin{align*}
\sum_{i=1}^n \pdv{\be_i^\top \cbF}{x_{ij}}
&= \frac{1}{\sqrt{n}} \sum_{i=1}^n\sum_{t=1}^T 
\be_i^\top 
(\be_t^\top \otimes \bI_n)
\Bigl(
-\calM_H \bigl((\bH^\top \be_j) \otimes \be_i\bigr)
+ \calM_F \bigl((\bF^\top \be_i) \otimes \be_j\bigr)
\Bigr)\be_t^\top.
\end{align*}
Rewriting in Kronecker form and summing over \(t\),
\begin{align*}
\sum_{i=1}^n \pdv{\be_i^\top \cbF}{x_{ij}}
&= \frac{1}{\sqrt{n}} \sum_{i=1}^n
\Bigl[
(\bI_T \otimes \be_i^\top)
\Bigl(
-\calM_H \bigl((\bH^\top \be_j) \otimes \be_i\bigr)
+ \calM_F \bigl((\bF^\top \be_i) \otimes \be_j\bigr)
\Bigr)
\Bigr]^\top,
\end{align*}
Separating the two terms yields
\begin{align*}
\sum_{i=1}^n \pdv{\be_i^\top \cbF}{x_{ij}}
&= - \Bigl[
\underbrace{\frac{1}{\sqrt{n}} \sum_{i=1}^n
(\bI_T \otimes \be_i^\top)\calM_H (\bI_T \otimes \be_i)}_{\cbK}
\, \bH^\top \be_j
\Bigr]^\top\\
&\quad
+ \Bigl[
\underbrace{\frac{1}{\sqrt{n}} \sum_{i=1}^n
(\bI_T \otimes \be_i^\top)\calM_F
\bigl((\bF^\top \be_i)\otimes \bI_p\bigr)}_{\check\bUpsilon_1}\be_j
\Bigr]^\top\\
&= - \be_j^\top \bH \cbK^\top + \be_j^\top \check\bUpsilon_1^\top.
\end{align*}
Rewrite $\check \bUpsilon_1$ as sum over \(t\) instead of \(i\):
\begin{align*}
\check\bUpsilon_1
&= \frac{1}{\sqrt{n}} \sum_{i=1}^n
(\bI_T \otimes \be_i^\top)\calM_F
\bigl((\bF^\top \be_i)\otimes \bI_p\bigr)\\
&= \frac{1}{\sqrt{n}} \sum_{i=1}^n \sum_{t=1}^T
(\bI_T \otimes \be_i^\top)\calM_F
\bigl((\be_t\be_t^\top\bF^\top \be_i)\otimes \bI_p\bigr)\\
&= \frac{1}{\sqrt{n}} \sum_{t=1}^T
(\bI_T \otimes (\be_t^\top \bF^\top))\calM_F
\bigl(\be_t \otimes \bI_p\bigr).
\end{align*}
It follows that
\begin{equation}\label{eq:opnorm-checkUpsilon1}
\opnorm{\check\bUpsilon_1}
\le \frac{T}{\sqrt{n}} \opnorm{\calM_F} \fnorm{\bF}.
\end{equation}

\vspace{0.5em}

We now prove \eqref{eq:checkF-derivative-2}. By the product rule,
\begin{align*}
&\sum_{i=1}^n\sum_{j=1}^p \pdv{\bF^\top \be_i\be_j^\top \bX^\top \cbF}{x_{ij}}\\
=& \sum_{i=1}^n\sum_{j=1}^p 
\Bigl[
\pdv{\bF^\top \be_i}{x_{ij}}\be_j^\top \bX^\top \cbF
+ \bF^\top \be_i\be_j^\top \be_j \be_i^\top \cbF
+ \bF^\top \be_i \be_j^\top \bX^\top \pdv{\cbF}{x_{ij}}
\Bigr]\\
=& - (\hbK \bH^\top + \bUpsilon_1)\bX^\top \cbF
+ p\,\bF^\top \cbF
+ \sum_{i=1}^n\sum_{j=1}^p \bF^\top \be_i \be_j^\top \bX^\top \pdv{\cbF}{x_{ij}}.
\end{align*}

It remains to simplify the last term. Writing
\[
\cbF = [\check\bu^1,\dots,\check\bu^T],
\]
we have
\begin{align*}
\sum_{i,j} \bF^\top \be_i \be_j^\top \bX^\top \pdv{\cbF}{x_{ij}}
= \sum_{i,j,t}
\bF^\top \be_i \be_j^\top \bX^\top
\pdv{\check\bu^t}{x_{ij}} \be_t^\top.
\end{align*}
Substituting the derivative expression and regrouping terms, we obtain
\begin{align*}
\sum_{i,j} \bF^\top \be_i \be_j^\top \bX^\top \pdv{\cbF}{x_{ij}}
=
-(\bUpsilon_5^*)^\top
-
(\cbA\,\bF^\top\bF)^\top,
\end{align*}
where
\begin{align*}
\bUpsilon_5^*
&:=
\frac{1}{\sqrt n}\sum_{j=1}^p
(\bI_T\otimes \be_j^\top \bX^\top)
\calM_H\bigl((\bH^\top\be_j)\otimes \bI_n\bigr)\bF,\\
\cbA
&:=- 
\frac{1}{\sqrt n}\sum_{j=1}^p
(\bI_T\otimes \be_j^\top \bX^\top)
\calM_F(\bI_T\otimes \be_j).
\end{align*}

Combining terms yields
\begin{align*}
&\sum_{i,j} \pdv{\bF^\top \be_i\be_j^\top \bX^\top \cbF}{x_{ij}}\\
=& - \hbK \bH^\top \bX^\top \cbF
+ p\,\bF^\top \cbF
- \bF^\top \bF\,\cbA^\top
-\check\bUpsilon_5,
\end{align*}
where
\[
\check\bUpsilon_5
:= \bUpsilon_1\bX^\top \cbF + (\bUpsilon_5^*)^\top.
\]
Hence, 
\begin{align*}
	\opnorm{\check\bUpsilon_5}
	&\le \opnorm{\bUpsilon_1}\opnorm{\bX} \fnorm{\cbF} + \opnorm{\bUpsilon_5^*}\\
	&\le T \opnorm{\calM^{-1} \bJ_{\bPsi^\eta}} \fnorm{\bF} \opnorm{\bX} \fnorm{\cbF} + \opnorm{\bUpsilon_5^*}.
\end{align*}
This proves \eqref{eq:checkF-derivative-2}.

Finally, rewriting \(\bUpsilon_5^*\) by expanding \(\bH^\top \be_j\) gives
\begin{align*}
\bUpsilon_5^*
&=\frac{1}{\sqrt n}\sum_{j=1}^p
(\bI_T\otimes \be_j^\top \bX^\top)\calM_H\bigl((\bH^\top\be_j)\otimes \bI_n\bigr)\bF\\
&=\frac{1}{\sqrt n}\sum_{j=1}^p\sum_{t=1}^T
(\bI_T\otimes \be_j^\top \bX^\top)\calM_H\bigl((\be_t \be_t^\top \bH^\top \be_j)\otimes \bI_n\bigr)\bF\\
&=\frac{1}{\sqrt n}\sum_{j=1}^p\sum_{t=1}^T
(\bI_T\otimes \be_t^\top \bH^\top \be_j\be_j^\top \bX^\top)\calM_H\bigl(\be_t\otimes \bI_n\bigr)\bF\\
&=\frac{1}{\sqrt n}\sum_{t=1}^T
(\bI_T\otimes \be_t^\top \bH^\top \bX^\top)\calM_H\bigl(\be_t\otimes \bI_n\bigr)\bF.
\end{align*}
Therefore, 
\begin{align*}
\opnorm{\bUpsilon_5^*}
&\le C(T) \frac{\opnorm{\bX}}{\sqrt n}  
\opnorm{\calM_H} \fnorm{\bH} \fnorm{\bF}.
\end{align*}
\end{proof}

\section{Reduction to isotropic design and transformed derivatives}
\label{sec:reduction-isotropic}

The previous section developed the derivative identities and correction matrices
in the isotropic setting. This section explains how the same analysis extends
to general covariance. This reduction is invoked at the start of the theorem
proofs in Section~\ref{sec:proof-main}. The key point is that, after a linear
change of variables, the transformed design has i.i.d.\ standard Gaussian rows
and the transformed algorithm preserves the original trajectory up to the
natural rescaling by \(\bSigma^{1/2}\). This allows the finite-sample analysis
to be carried out in the isotropic setting without loss of generality.

Consider the original linear model
\[
\by = \bX \bb^* + \bep,
\]
where the rows of $\bX$ have covariance matrix $\bSigma$. Define
\[
\bG := \bX \bSigma^{-1/2}, 
\qquad 
\tbb^* := \bSigma^{1/2} \bb^*.
\]
Then
\[
\by = \bG \tbb^* + \bep,
\]
and the rows of $\bG$ are i.i.d.\ $N(\boldsymbol{0}, \bI_p)$. Thus, the transformed design is isotropic.

We next define a transformed algorithm $(\tbu^t,\tbb^t)$ corresponding to the original iteration $(\bu^t,\hbb^t)$ in \eqref{eq:general-iteration}. Let
\[
\tbv^t = \frac{\by - \bG \tbb^t}{\sqrt n}, 
\qquad 
\tbfeta^t = \frac{\bG^\top \tbu^t}{\sqrt n}.
\]
We initialize the transformed iteration by
\[
\tbu^0 = \bu^0, 
\qquad 
\tbb^0 = \bSigma^{1/2} \hbb^0.
\]
For $t\ge 0$, the transformed recursion is
\begin{equation}
\label{eq:transformed}
\begin{aligned}
\tbu^t 
&= \tbphi^t(\tbu^0,\ldots,\tbu^{t-1},\tbv^0,\ldots,\tbv^{t-1}), \\
\tbb^t 
&= \tbpsi^t(\tbb^0,\ldots,\tbb^{t-1},\tbfeta^0,\ldots,\tbfeta^t).
\end{aligned}
\end{equation}

We choose $\tbphi^t$ and $\tbpsi^t$ so that the transformed recursion is exactly the original recursion expressed in the new coordinates. Specifically, we set $\tbphi^t=\bphi^t$ and define
\begin{equation}
\label{eq:tbpsi}
\tbpsi^t(\tbb^0,\ldots,\tbb^{t-1},\tbfeta^0,\ldots,\tbfeta^t)
:=
\bSigma^{1/2}
\bpsi^t(
\bSigma^{-1/2}\tbb^0,\ldots,\bSigma^{-1/2}\tbb^{t-1},
\bSigma^{1/2}\tbfeta^0,\ldots,\bSigma^{1/2}\tbfeta^t).
\end{equation}

With these definitions, the transformed iterates satisfy
\[
\tbu^t = \bu^t,
\qquad
\tbb^t = \bSigma^{1/2} \hbb^t
\]
for all $t$. Indeed, for any $s$,
\[
\tbv^s
= \frac{\by - \bG \tbb^s}{\sqrt n}
= \frac{\by - \bX \hbb^s}{\sqrt n}
= \bv^s,
\]
and
\[
\tbfeta^s
= \frac{\bG^\top \tbu^s}{\sqrt n}
= \frac{\bSigma^{-1/2}\bX^\top \bu^s}{\sqrt n}
= \bSigma^{-1/2}\bfeta^s.
\]
Therefore, the transformed residual sequence coincides with the original one, while the transformed dual summary $\tbfeta^s$ is simply a rescaled version of $\bfeta^s$. Substituting these identities into \eqref{eq:transformed} and \eqref{eq:tbpsi} gives
\[
\tbu^t = \bu^t,
\qquad
\tbb^t = \bSigma^{1/2} \hbb^t,
\]
which proves the equivalence between the transformed algorithm and the original one. In the remainder of this section, we keep the notation $\bu^t$ for the dual iterates of the transformed algorithm, since $\tbu^t=\bu^t$, and we write $\tbb^t$ for the transformed primal iterates.

\subsection{Derivatives for transformed algorithm}
We next record the derivative identities for the transformed algorithm \eqref{eq:transformed}. These are the analogues of Lemma~\ref{lem:dot-b-general} for the isotropic representation. We begin by defining the Jacobian blocks of the transformed update maps, in parallel with \eqref{eq:DJ-general}:
\begin{equation}
\label{eq:DJ-general-isotropic}
\begin{aligned}
\tPhi^{u}_{t,s} 
&= \pdv{\tbphi^t(\cdot)}{\tbu^s}
= \pdv{\bphi^t(\cdot)}{\bu^s} = \Phi^u_{t,s},\\
\tPhi^{v}_{t,s} 
&= \pdv{\tbphi^t(\cdot)}{\tbv^s}
= \pdv{\bphi^t(\cdot)}{\bv^s} = \Phi^v_{t,s},
\\
\tPsi^{b}_{t,s} 
&= \pdv{\tbpsi^t(\cdot)}{\tbb^s}
= \bSigma^{1/2} \pdv{\bpsi^t(\cdot)}{\bb^s} \bSigma^{-1/2} = \bSigma^{1/2} \Psi^b_{t,s} \bSigma^{-1/2},\\
\tPsi^{\eta}_{t,s} 
&= \pdv{\tbpsi^t(\cdot)}{\tbfeta^s}
= \bSigma^{1/2} \pdv{\bpsi^t(\cdot)}{\bfeta^s} \bSigma^{1/2} = \bSigma^{1/2} \Psi^\eta_{t,s} \bSigma^{1/2},
\end{aligned}
\end{equation}
where we used $\tbphi^t(\cdot)=\bphi^t(\cdot)$, the definition of $\tbpsi^t(\cdot)$ in \eqref{eq:tbpsi}, and the chain rule.

Let \(\tbh^t=\tbb^t-\tbb^*\), where \(\tbb^*=\bSigma^{1/2}\bb^*\), and set
\[
\tbH = [\tbh^0,\tbh^1,\ldots,\tbh^{T-1}] = \bSigma^{1/2} \bH.
\]

\begin{lemma}
\label{lem:dot-b-transformed}
For the transformed iteration \eqref{eq:transformed}, we have the following derivative formulae with respect to $g_{ij}$:
\begin{align*}
\pdv{\bu^t}{g_{ij}}
&= (\be_t^\top \otimes \bI_n)
\bigl(\bI_T \otimes [\bI_n, \boldzero_{n\times p}]\bigr) \tilde\calM^{-1} 
\bigl[-\bJ_{\tbPhi^v}\bigl((\tbH^\top \be_j) \otimes \be_i\bigr) + \bJ_{\tPsi^\eta} \bigl((\bF^\top \be_i) \otimes \be_j\bigr)\bigr],\\
\pdv{\tbb^t}{g_{ij}}
&= (\be_t^\top \otimes \bI_p)
\bigl(\bI_T \otimes [\boldzero_{p\times n}, \bI_p]\bigr) \tilde\calM^{-1} 
\bigl[-\bJ_{\tbPhi^v}\bigl((\tbH^\top \be_j) \otimes \be_i\bigr) + \bJ_{\tPsi^\eta} \bigl((\bF^\top \be_i) \otimes \be_j\bigr)\bigr],
\end{align*}
where $\tilde\calM, \bJ_{\tbPhi^v}, \bJ_{\tPsi^\eta}$ satisfy the following identities:
\begin{equation}
\label{eq:Minv-J-transformed}
\begin{aligned}
\bigl(\bI_T \otimes [\bI_n, \boldzero_{n\times p}]\bigr) \tilde\calM^{-1} \bJ_{\tbPhi^v}
&= \bigl(\bI_T \otimes [\bI_n, \boldzero_{n\times p}]\bigr) \calM^{-1} \bJ_{\bPhi^v},\\
\bigl(\bI_T \otimes [\bI_n, \boldzero_{n\times p}]\bigr) \tilde\calM^{-1} \bJ_{\tPsi^\eta} 
&= \bigl(\bI_T \otimes [\bI_n, \boldzero_{n\times p}]\bigr) \calM^{-1} \bJ_{\bPsi^\eta} (\bI_T \otimes \bSigma^{1/2}),\\
\bigl(\bI_T \otimes [\boldzero_{p\times n}, \bI_p]\bigr) \tilde\calM^{-1} \bJ_{\tbPhi^v}
&= (\bI_T \otimes \bSigma^{1/2})\bigl(\bI_T \otimes [\boldzero_{p\times n}, \bI_p]\bigr) \calM^{-1} \bJ_{\bPhi^v},\\
\bigl(\bI_T \otimes [\boldzero_{p\times n}, \bI_p]\bigr) \tilde\calM^{-1} \bJ_{\tPsi^\eta}
&= (\bI_T \otimes \bSigma^{1/2})\bigl(\bI_T \otimes [\boldzero_{p\times n}, \bI_p]\bigr) \calM^{-1} \bJ_{\bPsi^\eta} (\bI_T \otimes \bSigma^{1/2}).
\end{aligned}
\end{equation}
and \(\calM\), \(\bJ_{\bPhi^v}\), and \(\bJ_{\bPsi^\eta}\) are defined in
\eqref{eq:def-M} and \eqref{eq:def-JH-JF}.
\end{lemma}
\begin{proof}[Proof of Lemma~\ref{lem:dot-b-transformed}]
\label{proof:lem:dot-b-transformed}

Since $\bG = \bX \bSigma^{-1/2}$ and $\bX = \bG \bSigma^{1/2}$, the entries of $\bG$ and $\bX$ satisfy
\begin{align*}
g_{ij}
&= \be_i^\top \bG \be_j
= \be_i^\top \bX \bSigma^{-1/2} \be_j
= \sum_{k=1}^p x_{ik}\, (\bSigma^{-1/2})_{kj},\\
x_{ij}
&= \be_i^\top \bX \be_j
= \be_i^\top \bG \bSigma^{1/2} \be_j
= \sum_{k=1}^p g_{ik}\, (\bSigma^{1/2})_{kj}.
\end{align*}
Therefore, each entry $x_{ik}$ is a linear function of $\{g_{i1},\ldots,g_{ip}\}$. Applying the chain rule and the formula of $\pdv{\bu^t}{x_{ik}}$ in Lemma~\ref{lem:dot-b-general} gives
{\small
\setlength{\jot}{1pt}
\begin{align*}
	\pdv{\bu^t}{g_{ij}}
	&= \sum_{k=1}^p \pdv{\bu^t}{x_{ik}} \cdot \pdv{x_{ik}}{g_{ij}}\\
	&= \sum_{k=1}^p \pdv{\bu^t}{x_{ik}} \cdot (\bSigma^{1/2})_{kj},\\
	&= \sum_{k=1}^p 
	(\be_t^\top \otimes \bI_n)
\bigl(\bI_T \otimes [\bI_n, \boldzero_{n\times p}]\bigr) \calM^{-1}
\bigl[-\bJ_{\bPhi^v} \bigl((\bH^\top \be_k) \otimes \be_i\bigr) + \bJ_{\bPsi^\eta} \bigl((\bF^\top \be_i) \otimes \be_k\bigr)\bigr]
\cdot (\bSigma^{1/2})_{kj}\\
&= (\be_t^\top \otimes \bI_n)
\bigl(\bI_T \otimes [\bI_n, \boldzero_{n\times p}]\bigr) \calM^{-1}
\bigl[-\bJ_{\bPhi^v} \bigl((\bH^\top \bSigma^{1/2}\be_j) \otimes \be_i\bigr) + \bJ_{\bPsi^\eta} \bigl((\bF^\top \be_i) \otimes \bSigma^{1/2}\be_j\bigr)\bigr]\\
&= (\be_t^\top \otimes \bI_n)
\bigl(\bI_T \otimes [\bI_n, \boldzero_{n\times p}]\bigr) \calM^{-1}\\
&\quad {}\times
\left[
\begin{aligned}
&-\bJ_{\bPhi^v} \bigl((\tbH^\top \be_j) \otimes \be_i\bigr)\\[-2pt]
&\quad {}+ \bJ_{\bPsi^\eta} (\bI_T \otimes \bSigma^{1/2})
\\[-2pt]
&\qquad {}\times \bigl((\bF^\top \be_i) \otimes \be_j\bigr)
\end{aligned}
\right].
\end{align*}
}

Similarly, we have by $\tbb^t = \bSigma^{1/2} \hbb^t$ that 
{\small
\setlength{\jot}{1pt}
\begin{align*}
	\pdv{\tbb^t}{g_{ij}}
	&= \sum_{k=1}^p \pdv{\tbb^t}{x_{ik}} \cdot (\bSigma^{1/2})_{kj}\\
	&= \bSigma^{1/2}\sum_{k=1}^p  \pdv{\hbb^t}{x_{ik}} \cdot (\bSigma^{1/2})_{kj}\\
&= \bSigma^{1/2} \sum_{k=1}^p 
(\be_t^\top \otimes \bI_p)
\!\bigl(\bI_T \otimes [\boldzero_{p\times n}, \bI_p]\bigr)\!\calM^{-1}
\bigl[-\bJ_{\bPhi^v} \bigl((\bH^\top \be_k) \otimes \be_i\bigr) + \bJ_{\bPsi^\eta} \bigl((\bF^\top \be_i) \otimes \be_k\bigr)\bigr]
\cdot (\bSigma^{1/2})_{kj}\\
&= \bSigma^{1/2}  
(\be_t^\top \otimes \bI_p)
\!\bigl(\bI_T \otimes [\boldzero_{p\times n}, \bI_p]\bigr)\!\calM^{-1}
\bigl[-\bJ_{\bPhi^v} \bigl((\bH^\top \bSigma^{1/2} \be_j) \otimes \be_i\bigr) + \bJ_{\bPsi^\eta} \bigl((\bF^\top \be_i) \otimes (\bSigma^{1/2} \be_j)\bigr)\bigr]\\
&=  
(\be_t^\top \otimes \bI_p)
(\bI_T \otimes \bSigma^{1/2})\!\!
\bigl(\bI_T \otimes [\boldzero_{p\times n}, \bI_p]\bigr)\!\calM^{-1}\\
&\quad {}\times
\left[
\begin{aligned}
&-\bJ_{\bPhi^v} \bigl((\tbH^\top \be_j) \otimes \be_i\bigr)\\[-2pt]
&\quad {}+ \bJ_{\bPsi^\eta} (\bI_T \otimes \bSigma^{1/2})
\\[-2pt]
&\qquad {}\times \bigl((\bF^\top \be_i) \otimes \be_j \bigr)
\end{aligned}
\right].
\end{align*}
}
This completes the proof of the lemma by matching the terms in the last lines of the above two displays with the expressions in the statement of the lemma.
\end{proof}

Let $(\bW_*, \cbA_*, \cbK_*)$ denote the analogue matrices of
$(\bW, \cbA, \cbK)$ for the transformed iteration \eqref{eq:transformed}. They
are defined in the same way as in Lemma~\ref{lem:sum-derivative}, but with
\(\bJ_{\tbPhi^v}\), \(\bJ_{\tPsi^\eta}\), and \(\tilde\calM\) in place of
\(\bJ_{\bPhi^v}\), \(\bJ_{\bPsi^\eta}\), and \(\calM\). Using the identities in
\eqref{eq:Minv-J-transformed}, we derive the corresponding expressions for
$(\bW_*, \cbA_*, \cbK_*)$:
{\small
\setlength{\jot}{1pt}
\begin{equation}
\label{eq:K-W-A-transformed}
	\begin{aligned}
		\cbK_* 
		&= 
		\frac{1}{\sqrt n}\sum_{i=1}^n
		(\bI_T\otimes \be_i^\top)\widetilde\calM_H(\bI_T\otimes \be_i)\\
		&= 
		\frac{1}{\sqrt n}\sum_{i=1}^n
		(\bI_T\otimes \be_i^\top)\calM_H(\bI_T\otimes \be_i)\\
		&= \cbK,\\
		\cbA_*
		&=
		-\frac{1}{\sqrt n}\sum_{j=1}^p
		(\bI_T\otimes \be_j^\top\bG^\top)\widetilde\calM_F(\bI_T\otimes \be_j)\\
		&=
		-\frac{1}{\sqrt n}\sum_{j=1}^p
		(\bI_T\otimes \be_j^\top\bX^\top)\calM_F 
		(\bI_T\otimes \be_j)\\
		&= \cbA,\\
		\bW_*
		&= \sum_{j=1}^p
		(\bI_T\otimes \be_j^\top)
	(\bL\otimes \bI_p)
	\bigl(\bI_T \otimes [\boldzero_{p\times n}, \bI_p]\bigr)\\
	&\quad {}\times \tilde\calM^{-1} \bJ_{\tPsi^\eta}
	(\bI_T\otimes \be_j)\\
		&= \sum_{j=1}^p
		(\bI_T\otimes \be_j^\top)
	(\bL\otimes \bI_p)
	(\bI_T\otimes \bSigma^{1/2})\\
	&\quad {}\times
	\bigl(\bI_T \otimes [\boldzero_{p\times n}, \bI_p]\bigr)
	\calM^{-1} \bJ_{\bPsi^\eta}
	(\bI_T\otimes \bSigma^{1/2})
	(\bI_T\otimes \be_j)\\
		&= \sum_{j=1}^p
		(\bI_T\otimes \be_j^\top)
	(\bL\otimes \bSigma)
	\bigl(\bI_T \otimes [\boldzero_{p\times n}, \bI_p]\bigr)\\
	&\quad {}\times \calM^{-1} \bJ_{\bPsi^\eta}
	(\bI_T\otimes \be_j).
	\end{aligned}
\end{equation}
}
The final equality is the basis contraction. The equalities of $\cbK_*$ and $\cbA_*$ use the relationships
\begin{align*}
	\widetilde\calM_H
&=
-\bI_{nT}
+
(\bL\otimes \bG)
\bigl(\bI_T\otimes[\boldzero_{p\times n},\bI_p]\bigr)
\tilde\calM^{-1}\bJ_{\tbPhi^v},\\
&=
-\bI_{nT}
+
(\bL\otimes \bG)
(\bI_T \otimes \bSigma^{1/2})
\bigl(\bI_T\otimes[\boldzero_{p\times n},\bI_p]\bigr)
\calM^{-1}\bJ_{\bPhi^v},\\
&=
-\bI_{nT}
+
(\bL\otimes \bX)
\bigl(\bI_T\otimes[\boldzero_{p\times n},\bI_p]\bigr)
\calM^{-1}\bJ_{\bPhi^v},\\
&= \calM_H,
\end{align*}
and 
\begin{align*}
	\widetilde\calM_F
	&=
	(\bL\otimes \bG)
	\bigl(\bI_T\otimes[\boldzero_{p\times n},\bI_p]\bigr)
	\tilde\calM^{-1}\bJ_{\tPsi^\eta}\\
	&=
	(\bL\otimes \bG) (\bI_T\otimes \bSigma^{1/2})
	\bigl(\bI_T\otimes[\boldzero_{p\times n},\bI_p]\bigr)
	\calM^{-1}\bJ_{\bPsi^\eta}(\bI_T\otimes \bSigma^{1/2})\\
	&=
	(\bL\otimes \bX) 
	\bigl(\bI_T\otimes[\boldzero_{p\times n},\bI_p]\bigr)
	\calM^{-1}\bJ_{\bPsi^\eta}(\bI_T\otimes \bSigma^{1/2}).
\end{align*}

In summary, for general covariance $\bSigma \ne \bI_p$, the weight matrix $\bW$ is transformed to a new matrix $\bW_*$ that depends on $\bSigma$, whereas the matrices $\cbK$ and $\cbA$ remain unchanged under the change of variables. In particular, $\hbW\defas \cbK^{-1}\cbA$ also remains unchanged under the change of variables. Therefore, once the results are proved under the isotropic design with $\bSigma=\bI_p$, the theorem follows for general covariance matrices $\bSigma$ as well. The only change is that the constants in the bounds also depend on $\kappa$ because of the transformed update formula \eqref{eq:tbpsi}.

\section{Norm and moment bounds}
\label{sec:norm-moment-bounds}
This section gathers the uniform operator-norm and moment estimates used in the approximation bounds for \(\widetilde{\calR}_t\) and \(\widehat{\calR}_t\).
These results are invoked repeatedly in Sections~\ref{sec:auxiliary-proofs} and \ref{sec:proof-ingredients}.

\begin{lemma}
\label{lem:opnorm-Minv-JH-JF}
Fix \(T\ge 1\) and suppose Assumptions~\ref{assu:X}--\ref{assu:algorithm} hold and $\bSigma = \bI_p$. Then
\begin{align*}
\bigl\|\calM^{-1}\bJ_{\bPhi^v}\bigr\|_{\op}
\vee
\bigl\|\calM^{-1}\bJ_{\bPsi^\eta}\bigr\|_{\op}
&\le
\frac{C(T,\zeta)}{\sqrt n}\,
\poly_{T-1}\!\left(\frac{\|\bX\|_{\op}}{\sqrt n}\right).
\end{align*}
Here \(\poly_k(x)\) denotes a polynomial in \(x\) of degree at most \(k\), whose coefficients depend only on \(T\) and \(\zeta\).

Moreover,
\begin{align*}
\|\calM_H\|_{\op}
\vee
\|\calM_F\|_{\op}
&\le
C(T,\zeta)\,
\poly_T\!\left(\frac{\|\bX\|_{\op}}{\sqrt n}\right).
\end{align*}
\end{lemma}

Before proving the lemma, we record the moment consequence that will be used
below. For every finite integer \(k\),
\begin{equation}
	\E[ \big\|\calM^{-1} \bJ_{\bPhi^v}\big\|_{\op}^{2k}] 
	\vee \E[ \big\|\calM^{-1} \bJ_{\bPsi^\eta}\big\|_{\op}^{2k}]
	\le \frac{1}{n^k} C(T,\gamma,\zeta,k).
\end{equation}

\begin{lemma}
\label{lem:op-bound-WAK}
Suppose Assumptions~\ref{assu:X}--\ref{assu:algorithm} hold and $\bSigma = \bI_p$. Then the following hold
\begin{align*}
\|\bW\|_{\op}
&\le
C(T,\gamma,\zeta)\,\sqrt n\,
\poly_{T-1}\!\left(\frac{\|\bX\|_{\op}}{\sqrt n}\right),\\
\|\cbK\|_{\op}
&\le
C(T,\gamma,\zeta)\,\sqrt n\,
\poly_T\!\left(\frac{\|\bX\|_{\op}}{\sqrt n}\right),\\
\|\cbA\|_{\op}
&\le
C(T,\gamma,\zeta)\,n\,
\poly_{T+1}\!\left(\frac{\|\bX\|_{\op}}{\sqrt n}\right).
\end{align*}
Here \(\poly_k(x)\) denotes a polynomial in \(x\) of degree at most \(k\), whose coefficients depend only on \(T\) and \(\zeta\).

Consequently, we have 
\begin{align*}
	\E[\opnorm{\cbK}^2] \le n C(T,\gamma,\zeta),~
\E[\opnorm{\bW}^2] \le n C(T,\gamma,\zeta),~
\text{ and } ~
\E[\opnorm{\cbA}^2] \le n^2 C(T,\gamma,\zeta).
\end{align*}

In addition, on the event \(\Omega\) defined in \eqref{eq:event-Omega}, we have 
\begin{align*}
\opnorm{\cbK}
\le 
\sqrt{n} C(T,\gamma,\zeta),
~
\opnorm{\bW}
\le 
\sqrt{n} C(T,\gamma,\zeta),
~ \text{ and } ~
\opnorm{\cbA}
\le 
n C(T,\gamma,\zeta).
\end{align*}
\end{lemma}

\begin{proof}[Proof of Lemma~\ref{lem:op-bound-WAK}]
\label{pf:op-bound-WAK}
By Lemma~\ref{lem:opnorm-Minv-JH-JF},
\begin{equation}
\label{eq:bound-MinvJ-short}
\bigl\|\calM^{-1}\bJ_{\bPhi^v}\bigr\|_{\op}
\vee
\bigl\|\calM^{-1}\bJ_{\bPsi^\eta}\bigr\|_{\op}
\le
\frac{C(T,\zeta)}{\sqrt n}\,
\poly_{T-1}\!\left(\frac{\|\bX\|_{\op}}{\sqrt n}\right),
\end{equation}
and
\begin{equation}
\label{eq:bound-MHMF-short}
\|\calM_H\|_{\op}\vee \|\calM_F\|_{\op}
\le
C(T,\zeta)\,
\poly_T\!\left(\frac{\|\bX\|_{\op}}{\sqrt n}\right).
\end{equation}

We first bound \(\bW\). By definition,
\[
\bW
=
\sum_{j=1}^p
(\bI_T\otimes \be_j^\top)
(\bL\otimes \bI_p)
\bigl(\bI_T\otimes [\boldzero_{p\times n},\bI_p]\bigr)
\calM^{-1}\bJ_{\bPsi^\eta}
(\bI_T\otimes \be_j).
\]
Using the triangle inequality, \(\|\bL\|_{\op}\le 1\), and \eqref{eq:bound-MinvJ-short},
\begin{align*}
\|\bW\|_{\op}
&\le
\sum_{j=1}^p
\left\|
(\bI_T\otimes \be_j^\top)
(\bL\otimes \bI_p)
\bigl(\bI_T\otimes [\boldzero_{p\times n},\bI_p]\bigr)
\calM^{-1}\bJ_{\bPsi^\eta}
(\bI_T\otimes \be_j)
\right\|_{\op}\\
&\le
p\,\bigl\|\calM^{-1}\bJ_{\bPsi^\eta}\bigr\|_{\op}\\
&\le
C(T,\zeta)\,\frac{p}{\sqrt n}\,
\poly_{T-1}\!\left(\frac{\|\bX\|_{\op}}{\sqrt n}\right).
\end{align*}
Using \(p/n\le \gamma\), we obtain
\[
\|\bW\|_{\op}
\le
C(T,\gamma,\zeta)\,\sqrt n\,
\poly_{T-1}\!\left(\frac{\|\bX\|_{\op}}{\sqrt n}\right).
\]

Next, for \(\cbK\),
\[
\cbK
=
\frac{1}{\sqrt n}\sum_{i=1}^n
(\bI_T\otimes \be_i^\top)\calM_H(\bI_T\otimes \be_i).
\]
Hence, by the triangle inequality and \eqref{eq:bound-MHMF-short},
\begin{align*}
\|\cbK\|_{\op}
&\le
\frac{1}{\sqrt n}\sum_{i=1}^n
\left\|(\bI_T\otimes \be_i^\top)\calM_H(\bI_T\otimes \be_i)\right\|_{\op}\\
&\le
\frac{n}{\sqrt n}\,\|\calM_H\|_{\op}\\
&\le
C(T,\zeta)\,\sqrt n\,
\poly_T\!\left(\frac{\|\bX\|_{\op}}{\sqrt n}\right),
\end{align*}
which proves the second bound.

Finally, for \(\cbA\),
\[
\cbA
=
-\frac{1}{\sqrt n}\sum_{j=1}^p
(\bI_T\otimes \be_j^\top \bX^\top)\calM_F(\bI_T\otimes \be_j).
\]
Using the triangle inequality and \eqref{eq:bound-MHMF-short},
\begin{align*}
\|\cbA\|_{\op}
&\le
\frac{1}{\sqrt n}\sum_{j=1}^p
\left\|
(\bI_T\otimes \be_j^\top \bX^\top)\calM_F(\bI_T\otimes \be_j)
\right\|_{\op}\\
&\le
\frac{p}{\sqrt n}\,\|\bX\|_{\op}\,\|\calM_F\|_{\op}\\
&\le
C(T,\zeta)\,\frac{p}{\sqrt n}\,\|\bX\|_{\op}\,
\poly_T\!\left(\frac{\|\bX\|_{\op}}{\sqrt n}\right).
\end{align*}
Using \(p/n\le \gamma\), this yields
\[
\|\cbA\|_{\op}
\le
C(T,\gamma,\zeta)\,n\,
\poly_{T+1}\!\left(\frac{\|\bX\|_{\op}}{\sqrt n}\right).
\]
This completes the proof.
\end{proof}

\begin{proof}[Proof of Lemma~\ref{lem:opnorm-Minv-JH-JF}]
\label{proof:lem:opnorm-Minv-JH-JF}
Recall that
\[
\calM
=
\bI_{T(n+p)}
+
\begin{bmatrix}
\bB_1 & \bB_2
\end{bmatrix}
\begin{bmatrix}
-\bPhi^u & \frac{1}{\sqrt n}\bPhi^v(\bL\otimes \bX)\\[1mm]
-\frac{1}{\sqrt n}\bPsi^\eta(\bI_T\otimes \bX^\top) & -\bPsi^b
\end{bmatrix}
\begin{bmatrix}
\bB_1^\top\\[2mm]
\bB_2^\top
\end{bmatrix}.
\]
Write
\[
\calM=\bI_{T(n+p)}+\bN,
\qquad
\bN:=\calM-\bI_{T(n+p)}.
\]
Since \(\calM\) is block lower triangular with identity diagonal blocks, \(\bN\)
is strictly block lower triangular. Hence \(\bN^{m}=0\) for every \(m\ge T\),
and therefore the finite Neumann series gives
\begin{equation}
\label{eq:Minv-series-poly}
\calM^{-1}=\sum_{m=0}^{T-1}(-\bN)^m,
\qquad
\|\calM^{-1}\|_{\op}\le \sum_{m=0}^{T-1}\|\bN\|_{\op}^m.
\end{equation}

Next, using
\[
\|\bB_1\|_{\op}\vee \|\bB_2\|_{\op}\le 1,
\qquad
\|\bL\otimes \bX\|_{\op}\le \|\bX\|_{\op},
\qquad
\|\bI_T\otimes \bX^\top\|_{\op}=\|\bX\|_{\op},
\]
we obtain
\begin{align*}
\|\bN\|_{\op}
&\le
\|\bPhi^u\|_{\op}
+\|\bPsi^b\|_{\op}
+\frac{\|\bX\|_{\op}}{\sqrt n}\bigl(\|\bPhi^v\|_{\op}+\|\bPsi^\eta\|_{\op}\bigr).
\end{align*}
Under Assumption~\ref{assu:algorithm}, all Jacobian blocks are bounded by \(\zeta\), and each of the block lower-triangular matrices \(\bPhi^u,\bPhi^v,\bPsi^b,\bPsi^\eta\) has at most \(T\) nonzero block columns in each row. Hence
\[
\|\bPhi^u\|_{\op}\vee \|\bPhi^v\|_{\op}\vee \|\bPsi^b\|_{\op}\vee \|\bPsi^\eta\|_{\op}
\le \zeta T,
\]
and so
\[
\|\bN\|_{\op}
\le
2\zeta T\left(1+\frac{\|\bX\|_{\op}}{\sqrt n}\right).
\]
Substituting this into \eqref{eq:Minv-series-poly} yields
\begin{align*}
\|\calM^{-1}\|_{\op}
&\le
\sum_{m=0}^{T-1}
\left[
2\zeta T\left(1+\frac{\|\bX\|_{\op}}{\sqrt n}\right)
\right]^m\\
&\le
C(T,\zeta)\,
\poly_{T-1}\!\left(\frac{\|\bX\|_{\op}}{\sqrt n}\right).
\end{align*}

Now recall that
\[
\bJ_{\bPhi^v}=\frac{1}{\sqrt n}\bB_1\bPhi^v,
\qquad
\bJ_{\bPsi^\eta}=\frac{1}{\sqrt n}\bB_2\bPsi^\eta.
\]
Therefore,
\[
\|\bJ_{\bPhi^v}\|_{\op}\vee \|\bJ_{\bPsi^\eta}\|_{\op}
\le
\frac{1}{\sqrt n}
\bigl(\|\bPhi^v\|_{\op}\vee \|\bPsi^\eta\|_{\op}\bigr)
\le
\frac{\zeta T}{\sqrt n}.
\]
Hence
\begin{align*}
\bigl\|\calM^{-1}\bJ_{\bPhi^v}\bigr\|_{\op}
\vee
\bigl\|\calM^{-1}\bJ_{\bPsi^\eta}\bigr\|_{\op}
&\le
\|\calM^{-1}\|_{\op}
\bigl(\|\bJ_{\bPhi^v}\|_{\op}\vee \|\bJ_{\bPsi^\eta}\|_{\op}\bigr)\\
&\le
\frac{C(T,\zeta)}{\sqrt n}\,
\poly_{T-1}\!\left(\frac{\|\bX\|_{\op}}{\sqrt n}\right).
\end{align*}
This proves the first bound.

Finally, by the definitions
\[
\calM_H
=
-\bI_{nT}
+
(\bL\otimes \bX)
\bigl(\bI_T\otimes [\boldzero_{p\times n},\bI_p]\bigr)\calM^{-1}\bJ_{\bPhi^v},
\]
and
\[
\calM_F
=
(\bL\otimes \bX)
\bigl(\bI_T\otimes [\boldzero_{p\times n},\bI_p]\bigr)\calM^{-1}\bJ_{\bPsi^\eta},
\]
together with
\[
\|\bL\otimes \bX\|_{\op} = \|\bX\|_{\op},
\qquad
\left\|\bI_T\otimes [\boldzero_{p\times n},\bI_p]\right\|_{\op}=1,
\]
we obtain
\begin{align*}
\|\calM_H\|_{\op}
&\le
1+\|\bX\|_{\op}\,\bigl\|\calM^{-1}\bJ_{\bPhi^v}\bigr\|_{\op}\\
&\le
1+
C(T,\zeta)\,
\frac{\|\bX\|_{\op}}{\sqrt n}\,
\poly_{T-1}\!\left(\frac{\|\bX\|_{\op}}{\sqrt n}\right),\\
\|\calM_F\|_{\op}
&\le
\|\bX\|_{\op}\,\bigl\|\calM^{-1}\bJ_{\bPsi^\eta}\bigr\|_{\op}\\
&\le
C(T,\zeta)\,
\frac{\|\bX\|_{\op}}{\sqrt n}\,
\poly_{T-1}\!\left(\frac{\|\bX\|_{\op}}{\sqrt n}\right).
\end{align*}
Since
\[
1+
\frac{\|\bX\|_{\op}}{\sqrt n}\,
\poly_{T-1}\!\left(\frac{\|\bX\|_{\op}}{\sqrt n}\right)
\le
C(T,\zeta)\,
\poly_T\!\left(\frac{\|\bX\|_{\op}}{\sqrt n}\right),
\]
the simplified common bound
\[
\|\calM_H\|_{\op}\vee \|\calM_F\|_{\op}
\le
C(T,\zeta)\,
\poly_T\!\left(\frac{\|\bX\|_{\op}}{\sqrt n}\right)
\]
follows. This completes the proof.
\end{proof}

\begin{lemma}
\label{lem:moment-H-F}
Let 
\[
\bH = [\bh^0, \bh^1, \cdots, \bh^{T-1}], 
\qquad
\bF = [\bu^1, \bu^2, \cdots, \bu^T], 
\qquad
\cbF = [\check \bu^1, \check \bu^2, \cdots, \check \bu^T].
\]
Under Assumptions~\ref{assu:X}--\ref{assu:algorithm}, we have for any finite integer $k \ge 1$,
\begin{align*}
	\E[\fnorm{\bH}^{2k}] \le C(T,\gamma,\zeta,k),  
	\quad
	\E[\fnorm{\bF}^{2k}] \le C(T,\zeta,k),
	\text{ and } \quad
	\E[\fnorm{\cbF}^{2k}] \le C(T,\gamma,\zeta,k).
\end{align*}

Let $\norm{\bA}_{\partial} := [\sum_{i=1}^n \sum_{j=1}^p \|\pdv{\bA}{x_{ij}}\|_{\rm F}^2]^{1/2}$. 
Under Assumptions~\ref{assu:X}--\ref{assu:algorithm}, we have
\begin{align*}
\E[\norm{\bH}_{\partial}^{2k}] 
\vee \E[\norm{\bF}_{\partial}^{2k}] 
\vee \E[\norm{\cbF}_{\partial}^{2k}] 
\le C(T,\gamma,\zeta,k).
\end{align*}
\end{lemma}

\begin{proof}[Proof of Lemma~\ref{lem:moment-H-F}]
\label{proof:moment-H-F}
We first prove the moment bounds for \(\bH,\bF,\cbF\), and then bound the derivative matrices.

By Assumption~\ref{assu:algorithm}, the maps \(\bphi^t\) are uniformly bounded by \(\zeta\). Hence, for every \(t\in[T]\),
\[
\|\bu^t\|_2 \le \zeta.
\]
Therefore
\[
\|\bF\|_{\rm F}^2
=
\sum_{t=1}^T \|\bu^t\|_2^2
\le T\zeta^2.
\]
It follows that, for every finite integer \(k\ge 1\),
\[
\E[\|\bF\|_{\rm F}^{2k}] \le (T\zeta^2)^k.
\]

We next control \(\bH\). Recall that \(\bh^t=\hbb^t-\bb^*\), so it suffices to bound \(\max_{t\le T}\|\hbb^t\|_2\). Since \(\hbb^t\) is defined recursively through \(\bpsi^t\), we use the Lipschitz property of \(\bpsi^t\).

Write the \(b\)-update abstractly as
\[
\hbb^t
=
\bpsi^t(\hbb^0,\dots,\hbb^{t-1},\bfeta^1,\dots,\bfeta^t),
\qquad
\bfeta^s=\frac{\bX^\top \bu^s}{\sqrt n}.
\]
By Assumption~\ref{assu:algorithm}, each \(\bpsi^t\) is \(\zeta\)-Lipschitz. Hence
\[
\|\hbb^t\|_2
\le
\|\bpsi^t(\mathbf 0)\|_2
+
\zeta\sum_{r=0}^{t-1}\|\hbb^r\|_2
+
\zeta\sum_{s=1}^t\|\bfeta^s\|_2.
\]
Using \(\|\bu^s\|_2\le \zeta\), we have
\[
\|\bfeta^s\|_2
=
\left\|\frac{\bX^\top \bu^s}{\sqrt n}\right\|_2
\le
\frac{\|\bX\|_{\op}}{\sqrt n}\,\|\bu^s\|_2
\le
\zeta\,\frac{\|\bX\|_{\op}}{\sqrt n}.
\]
Therefore,
\[
\|\hbb^t\|_2
\le
\|\bpsi^t(\mathbf 0)\|_2
+
\zeta\sum_{r=0}^{t-1}\|\hbb^r\|_2
+
t\zeta^2\frac{\|\bX\|_{\op}}{\sqrt n}.
\]
Since \(T\) is fixed, a discrete Grönwall argument yields
\[
\max_{0\le t\le T-1}\|\hbb^t\|_2
\le
C(T,\zeta)\left(
1+\max_{1\le t\le T}\|\bpsi^t(\mathbf 0)\|_2
+\frac{\|\bX\|_{\op}}{\sqrt n}
\right).
\]
Under Assumption~\ref{assu:algorithm}, \(\|\bpsi^t(\mathbf 0)\|_2\le \zeta\), so
\[
\max_{0\le t\le T-1}\|\hbb^t\|_2
\le
C(T,\zeta)\left(1+\frac{\|\bX\|_{\op}}{\sqrt n}\right).
\]
By Assumption~\ref{assu:X}, \(\|\bb^*\|_2\le \zeta\). Hence
\[
\max_{0\le t\le T-1}\|\bh^t\|_2
\le
C(T,\zeta)\left(1+\frac{\|\bX\|_{\op}}{\sqrt n}\right),
\]
and therefore
\[
\|\bH\|_{\rm F}^{2k}
\le
C(T,\zeta,k)\left(1+\frac{\|\bX\|_{\op}}{\sqrt n}\right)^{2k}.
\]
Taking expectation and using Assumption~\ref{assu:X} gives
\[
\E[\|\bH\|_{\rm F}^{2k}] \le C(T,\gamma,\zeta,k).
\]

It remains to control \(\cbF\). By definition,
\[
\check\bu^t
=
\frac{\bX\bh^{t-1}}{\sqrt n}.
\]
Using the bound on \(\max_{t\le T}\|\bh^t\|_2\) obtained above,
\[
\|\check\bu^t\|_2
\le
C(T,\zeta)
\frac{\|\bX\|_{\op}}{\sqrt n}
\left(1+\frac{\|\bX\|_{\op}}{\sqrt n}\right).
\]
Since \(\|\bX\|_{\op}/\sqrt n\) has bounded moments of every fixed order,
\[
\E[\|\cbF\|_{\rm F}^{2k}]\le C(T,\gamma,\zeta,k).
\]

We now turn to the derivative moment bounds.

For \(\bF\), by Lemma~\ref{lem:dot-b-general},
\begin{align*}
\pdv{\bu^t}{x_{ij}}
&=
(\be_t^\top\otimes \bI_n)
(\bI_T\otimes [\bI_n,\boldzero_{n\times p}])\calM^{-1}
\Bigl[
-\bJ_{\bPhi^v}\bigl((\bH^\top\be_j)\otimes \be_i\bigr)
+
\bJ_{\bPsi^\eta}\bigl((\bF^\top\be_i)\otimes \be_j\bigr)
\Bigr].
\end{align*}
Arguing exactly as before,
\[
\|\bF\|_{\partial}^2
\le
C(T,\gamma,\zeta)\,
\poly_{2T-2}\!\left(\frac{\|\bX\|_{\op}}{\sqrt n}\right)
\bigl(\|\bH\|_{\rm F}^2+\|\bF\|_{\rm F}^2\bigr).
\]
Raising both sides to the \(k\)-th power and using \((a+b)^k\le C(k)(a^k+b^k)\),
\[
\|\bF\|_{\partial}^{2k}
\le
C(T,\gamma,\zeta,k)\,
\poly_{(2T-2)k}\!\left(\frac{\|\bX\|_{\op}}{\sqrt n}\right)
\bigl(\|\bH\|_{\rm F}^{2k}+\|\bF\|_{\rm F}^{2k}\bigr).
\]
Taking expectation and using the moment bounds proved above yields
\[
\E[\|\bF\|_{\partial}^{2k}] \le C(T,\gamma,\zeta,k).
\]

The same argument applied to
\[
\pdv{\hbb^t}{x_{ij}}
=
(\be_t^\top\otimes \bI_p)
(\bI_T\otimes [\boldzero_{p\times n},\bI_p])\calM^{-1}
\Bigl[
-\bJ_{\bPhi^v}\bigl((\bH^\top\be_j)\otimes \be_i\bigr)
+
\bJ_{\bPsi^\eta}\bigl((\bF^\top\be_i)\otimes \be_j\bigr)
\Bigr]
\]
gives
\[
\E[\|\bH\|_{\partial}^{2k}] \le C(T,\gamma,\zeta,k).
\]

Finally, for \(\cbF\), using
\[
\pdv{\check\bu^t}{x_{ij}}
=
\frac{1}{\sqrt n}
(\be_t^\top\otimes \bI_n)
\Bigl[
-\calM_H\bigl((\bH^\top\be_j)\otimes \be_i\bigr)
+
\calM_F\bigl((\bF^\top\be_i)\otimes \be_j\bigr)
\Bigr],
\]
together with Lemma~\ref{lem:opnorm-Minv-JH-JF}, we obtain
\[
\|\cbF\|_{\partial}^2
\le
C(T,\gamma,\zeta)\,
\poly_{2T}\!\left(\frac{\|\bX\|_{\op}}{\sqrt n}\right)
\bigl(\|\bH\|_{\rm F}^2+\|\bF\|_{\rm F}^2\bigr),
\]
and hence
\[
\E[\|\cbF\|_{\partial}^{2k}] \le C(T,\gamma,\zeta,k).
\]

Combining the above estimates proves the lemma.
\end{proof}

\subsection{Properties of $\cbK$}

Recall the event \(\Omega\) is defined in 
\eqref{eq:event-Omega}. 
Under Assumptions~\ref{assu:X} and \ref{assu:regime}, we have 
\[
\P(\Omega)\ge 1-e^{-n/2}.
\]

\begin{lemma}
\label{lem:Kinv-opnorm}
On the event \(\Omega\), we have 
\[
\|\cbK^{-1}\|_{\op}\le \frac{C(T,\gamma,\zeta)}{\sqrt n}.
\]
\end{lemma}

\begin{proof}[Proof of Lemma~\ref{lem:Kinv-opnorm}]
\label{proof:lem:Kinv-opnorm}
Recall that \(\cbK\) is lower triangular with diagonal entries
\[
\cbK_{t,t}
=
-\sqrt n,
\qquad t=1,\dots,T.
\]
Thus the diagonal is deterministic and satisfies
\[
|\cbK_{t,t}|=\sqrt n,\qquad t=1,\dots,T.
\]

We now write
\[
\frac{\cbK}{\sqrt n}=\bDelta+ \bDelta_l,
\]
where
\[
\bDelta:=\diag(\delta_1,\dots,\delta_T),
\qquad
\delta_t:=\frac{\cbK_{t,t}}{\sqrt n},
\]
and \(\bDelta_l\) is the strictly lower triangular part of \(\cbK/\sqrt n\). The previous bound implies
\[
|\delta_t|=1,
\qquad t=1,\dots,T,
\]
and hence
\[
\|\bDelta^{-1}\|_{\op}
=
1.
\]

On the other hand, by Lemma~\ref{lem:op-bound-WAK}, on \(\Omega\),
\[
\|\cbK\|_{\op}\le C(T,\gamma,\zeta)\sqrt n.
\]
Therefore
\[
\left\|\frac{\cbK}{\sqrt n}\right\|_{\op}\le C(T,\gamma,\zeta).
\]
Since
\[
\frac{\cbK}{\sqrt n}=\bDelta+\bDelta_l,
\]
we get
\[
\|\bDelta_l\|_{\op}
\le
\left\|\frac{\cbK}{\sqrt n}\right\|_{\op}+\|\bDelta\|_{\op}
\le
C(T,\gamma,\zeta),
\]
where we used \(\|\bDelta\|_{\op}=1\).

Because \(\bDelta_l\) is strictly lower triangular, it is nilpotent of order at most \(T\). Hence
\[
(\bDelta+\bDelta_l)^{-1}
=
\sum_{k=0}^{T-1}(-\bDelta^{-1}\bDelta_l)^k\bDelta^{-1}.
\]
Therefore
\begin{align*}
\|\cbK^{-1}\|_{\op}
&=
\frac{1}{\sqrt n}\|(\bDelta+\bDelta_l)^{-1}\|_{\op}\\
&\le
\frac{1}{\sqrt n}
\sum_{k=0}^{T-1}
\|\bDelta^{-1}\bDelta_l\|_{\op}^k
\|\bDelta^{-1}\|_{\op}\\
&\le
\frac{1}{\sqrt n}\,
C(T,\gamma,\zeta)
\sum_{k=0}^{T-1}
C(T,\gamma,\zeta)^k.
\end{align*}
Since \(T\) is fixed, the finite sum is bounded by a constant depending only on
the parameters already fixed in the assumptions and in the event construction.
Absorbing the auxiliary constants into the generic constant, we obtain
\[
\|\cbK^{-1}\|_{\op}\le \frac{C(T,\gamma,\zeta)}{\sqrt n}.
\]
This completes the proof.
\end{proof}

\section{Auxiliary proof for the derivative formula}
\label{sec:auxiliary-proofs}
This section contains the proof of the general derivative formula from
Lemma~\ref{lem:dot-b-general}. The leave-one-out stability estimates are stated
and proved with the other preliminary lemmas in
Section~\ref{sec:preliminary-lemmas}.

\subsection{Proof of the general derivative formula}

\begin{proof}[Proof of Lemma~\ref{lem:dot-b-general}]
\label{proof:lem:dot-b-general}
Recall the general iteration in \eqref{eq:general-iteration}:
\begin{equation}
\label{eq:general-iteration-v-eta}
\boxed{
\begin{aligned}
\bu^{t}&=\bphi^{t}(\bu^0,\ldots,\bu^{t-1},
\bv^0,\ldots,\bv^{t-1}),\\
\hbb^{t}&=\bpsi^{t}(\hbb^0,\ldots,\hbb^{t-1},
\bfeta^0,\ldots,\bfeta^{t}),
\end{aligned}
}
\end{equation}
where
\[
\bv^t = \frac{\by - \bX\hbb^t}{\sqrt{n}}, 
\qquad 
\bfeta^t = \frac{\bX^\top\bu^t}{\sqrt{n}}.
\]

Using the Jacobian notation introduced in \eqref{eq:DJ-general}, define
\[
\Phi^{u}_{t,s} = \frac{\partial \bphi^t}{\partial \bu^s},
\quad
\Phi^{v}_{t,s} = \frac{\partial \bphi^t}{\partial \bv^s},
\quad
\Psi^{b}_{t,s} = \frac{\partial \bpsi^t}{\partial \hbb^s},
\quad
\Psi^{\eta}_{t,s} = \frac{\partial \bpsi^t}{\partial \bfeta^s}.
\]

For notational simplicity, let
\[
\dot\bu^t = \frac{\partial \bu^t}{\partial x_{ij}},
\qquad
\dot\bb^t = \frac{\partial \hbb^t}{\partial x_{ij}},
\qquad
\dot\bv^t = \frac{\partial \bv^t}{\partial x_{ij}},
\quad
\dot\bfeta^t = \frac{\partial \bfeta^t}{\partial x_{ij}},
\qquad
\dot\bX = \frac{\partial \bX}{\partial x_{ij}}.
\]

By definition of \(\bv^t\) and \(\bfeta^t\), and using \(\bh^t = \hbb^t - \bb^*\), we have by the product rule
\begin{align*}
\dot\bv^t
&= \frac{1}{\sqrt{n}}
\pdv{\bep - \bX(\hbb^t-\bb^*)}{x_{ij}}
= -\frac{1}{\sqrt{n}}\bigl(\dot\bX\bh^t + \bX\dot\bb^t\bigr),\\
\dot\bfeta^t
&= \frac{1}{\sqrt{n}}
(\dot\bX^\top\bu^t + \bX^\top\dot\bu^t),
\end{align*}
where \(\dot\bX = \be_i \be_j^\top\).

For \(t\ge1\), applying the chain rule to \eqref{eq:general-iteration-v-eta} gives
\begin{align*}
\dot\bu^t 
&= \sum_{s=0}^{t-1} \Phi^{u}_{t,s}\dot\bu^s
   + \sum_{s=0}^{t-1} \Phi^{v}_{t,s}\dot\bv^s \\
&= \sum_{s=0}^{t-1}\Phi^u_{t,s}\dot\bu^s
   -\sum_{s=0}^{t-1}\Phi^v_{t,s}
   \left(\frac{1}{\sqrt{n}}\dot\bX\bh^s
         +\frac{1}{\sqrt{n}}\bX\dot\bb^s\right),\\
\dot\bb^t 
&= \sum_{s=0}^{t-1} \Psi^{b}_{t,s}\dot\bb^s
   + \sum_{s=0}^{t} \Psi^{\eta}_{t,s}\dot\bfeta^s \\
&= \sum_{s=0}^{t-1}\Psi^b_{t,s}\dot\bb^s
   +\sum_{s=0}^{t}\Psi^\eta_{t,s}
   \left(\frac{1}{\sqrt{n}}\dot\bX^\top\bu^s
         +\frac{1}{\sqrt{n}}\bX^\top\dot\bu^s\right).
\end{align*}

Rearranging by collecting the terms \((\dot\bu^t, \dot\bb^t)\) on the left-hand side gives
\begin{align*}
&\dot\bu^t 
- \sum_{s=0}^{t-1}\Phi^u_{t,s}\dot\bu^s 
+ \frac{1}{\sqrt{n}} \sum_{s=0}^{t-1}\Phi^v_{t,s} \bX\dot\bb^s
= -\frac{1}{\sqrt{n}} \sum_{s=0}^{t-1}\Phi^v_{t,s} \dot\bX\bh^s, \\
&\dot\bb^t
- \sum_{s=0}^{t-1}\Psi^b_{t,s}\dot\bb^s
- \frac{1}{\sqrt{n}}\sum_{s=0}^{t}\Psi^\eta_{t,s} \bX^\top\dot\bu^s
= \frac{1}{\sqrt{n}}\sum_{s=0}^{t}\Psi^\eta_{t,s} \dot\bX^\top\bu^s.
\end{align*}

Since \(\dot\bu^0 = \bzero_n\) and \(\dot\bb^0 = \bzero_p\) (because the initialization \(\bu^0,\hbb^0\) is independent of \(\bX\)), we can remove the terms with \(s=0\) in the left-hand side summations. For the case \(t=1\), we adopt the convention that \(\sum_{s=1}^{0} = 0\). The equations therefore simplify to
\begin{align*}
&\dot\bu^t 
- \sum_{s=1}^{t-1}\Phi^u_{t,s}\dot\bu^s 
+ \frac{1}{\sqrt{n}} \sum_{s=1}^{t-1}\Phi^v_{t,s} \bX\dot\bb^s
= -\frac{1}{\sqrt{n}} \sum_{s=0}^{t-1}\Phi^v_{t,s} \dot\bX\bh^s, \\
&\dot\bb^t
- \sum_{s=1}^{t-1}\Psi^b_{t,s}\dot\bb^s
- \frac{1}{\sqrt{n}}\sum_{s=1}^{t}\Psi^\eta_{t,s} \bX^\top\dot\bu^s
= \frac{1}{\sqrt{n}}\sum_{s=0}^{t}\Psi^\eta_{t,s} \dot\bX^\top\bu^s.
\end{align*}
Assuming $\bu^0 = \bzero_n$, the summation in the right-hand side of the second equation starts from $s=1$ instead of $s=0$. 

Denote the right-hand side of the above equations by \(\bw_u^t\) and \(\bw_b^t\), respectively.

We can now stack the above equations for $\dot\bu^t$ and $\dot\bb^t$ ($t=1,\ldots,T$) into the linear system:
\begin{equation}
\label{eq:block-linear-system}
\calM 
\underbrace{
\begin{bmatrix}
\dot\bu^{1}\\ \dot\bb^{1}\\
\dot\bu^{2}\\ \dot\bb^{2}\\
\vdots\\
\dot\bu^{T}\\ \dot\bb^{T}
\end{bmatrix}
}_{\displaystyle \bvartheta \in \R^{T(n+p)}}
= 
\underbrace{\begin{bmatrix}
\bw_u^{1}\\ \bw_b^{1}\\
\bw_u^{2}\\ \bw_b^{2}\\
\vdots\\
\bw_u^{T}\\ \bw_b^{T}
\end{bmatrix}
}_{\bw \in \R^{T(n+p)}}.
\end{equation}

In the above, $\calM\in \R^{T(n+p)\times T(n+p)}$ is a block lower-triangular matrix:
{\tiny
\begin{equation}
\label{eq:def-M-explicit}
\calM=
\begin{bmatrix}
\begin{bmatrix}
\bI_n & \bzero_{n\times p}\\
-\frac{1}{\sqrt n}\Psi^\eta_{1,1}\bX^\top & \bI_p
\end{bmatrix}
& \bzero & \cdots & \bzero\\

\begin{bmatrix}
-\Phi^u_{2,1} & \frac{1}{\sqrt n}\Phi^v_{2,1}\bX\\
-\frac{1}{\sqrt n}\Psi^\eta_{2,1}\bX^\top & -\Psi^b_{2,1}
\end{bmatrix}
&
\begin{bmatrix}
\bI_n & \bzero_{n\times p}\\
-\frac{1}{\sqrt n}\Psi^\eta_{2,2}\bX^\top & \bI_p
\end{bmatrix}
& \ddots & \vdots\\

\vdots & \ddots & \ddots & \bzero\\

\begin{bmatrix}
-\Phi^u_{T,1} & \frac{1}{\sqrt n}\Phi^v_{T,1}\bX\\
-\frac{1}{\sqrt n}\Psi^\eta_{T,1}\bX^\top & -\Psi^b_{T,1}
\end{bmatrix}
& \cdots &
\begin{bmatrix}
-\Phi^u_{T,T-1} & \frac{1}{\sqrt n}\Phi^v_{T,T-1}\bX\\
-\frac{1}{\sqrt n}\Psi^\eta_{T,T-1}\bX^\top & -\Psi^b_{T,T-1}
\end{bmatrix}
&
\begin{bmatrix}
\bI_n & \bzero_{n\times p}\\
-\frac{1}{\sqrt n}\Psi^\eta_{T,T}\bX^\top & \bI_p
\end{bmatrix}
\end{bmatrix}.
\end{equation}
}
It is important to note that $\calM$ is a block lower-triangular matrix with invertible diagonal blocks. Therefore, $\calM$ is invertible, and the solution to \eqref{eq:block-linear-system} is given by
\[
\bvartheta = \calM^{-1}\bw.
\]
We now derive a more compact representation of $\calM$ and $\bw$ that will be useful for the subsequent analysis.

To simplify the notation, we introduce the matrices
{\small
\begin{equation}
\begin{aligned}
	\bPhi^u 
	&=
	\begin{bmatrix}
	\boldzero_{n\times n} &        &        & \\
	\Phi^u_{2,1} & \boldzero_{n\times n} &        & \\
	\vdots    & \ddots & \ddots & \\
	\Phi^u_{T,1} & \cdots & \Phi^u_{T,T-1} & \boldzero_{n\times n}
	\end{bmatrix} \in \R^{Tn\times Tn},
	\quad 
	\bPhi^v
	=
	\begin{bmatrix}
	\Phi^v_{1,0} &        &        & \\
	\Phi^v_{2,0} & \Phi^v_{2,1} &        & \\
	\vdots    & \ddots & \ddots & \\
	\Phi^v_{T,0} & \cdots & \Phi^v_{T,T-2} & \Phi^v_{T,T-1}
	\end{bmatrix} \in \R^{Tn\times Tn},\\
	\bPsi^b
	&=
	\begin{bmatrix}
	\boldzero_{p\times p} &        &        & \\
	\Psi^b_{2,1} & \boldzero_{p\times p} &        & \\
	\vdots     & \ddots & \ddots & \\
	\Psi^b_{T,1} & \cdots & \Psi^b_{T,T-1} & \boldzero_{p\times p}
	\end{bmatrix} \in \R^{Tp\times Tp},
	\quad
	\bPsi^\eta
	=
	\begin{bmatrix}
	 \Psi^\eta_{1,1} &        &        & \\
	\Psi^\eta_{2,1} & \Psi^\eta_{2,2} &        & \\
	\vdots     & \vdots     & \ddots & \ddots & \\
	\Psi^\eta_{T,1} & \cdots & \Psi^\eta_{T,T-1} & \Psi^\eta_{T,T}
	\end{bmatrix}\in \R^{Tp\times Tp}.
\end{aligned}
\end{equation}}

Now we can write $\calM$ in a more compact form
\begin{equation}
\calM
=
\bI_{T(n+p)}
+
\begin{bmatrix}
\bB_1 & \bB_2
\end{bmatrix}
\begin{bmatrix}
-\bPhi^u & \frac{1}{\sqrt n}\,\bPhi^v(\bL\otimes \bX)\\[2mm]
-\frac{1}{\sqrt n}\,\bPsi^\eta (\bI_T\otimes \bX^\top) & -\bPsi^b
\end{bmatrix}
\begin{bmatrix}
\bB_1^\top \\[6pt]
\bB_2^\top
\end{bmatrix},
\end{equation}
where $\bL = \sum_{t=2}^T \be_t \be_{t-1}^\top$ is the $T\times T$ lag matrix, 
and 
\begin{align*}
	\bB_1 := \bI_T \otimes
\begin{bmatrix}
\bI_n\\
\bzero_{p\times n}
\end{bmatrix}
\in\R^{T(n+p)\times Tn},
\qquad
\bB_2 := \bI_T \otimes
\begin{bmatrix}
\bzero_{n\times p}\\
\bI_p
\end{bmatrix}
\in\R^{T(n+p)\times Tp}.
\end{align*}

The compact representation \eqref{eq:def-M} preserves the block lower-triangular structure of $\calM$. Indeed, $\bPhi^u$ and $\bPsi^b$ are strictly lower triangular, $\bPsi^\eta$ is lower triangular, and $\bPhi^v(\bL\otimes \bX)$ is strictly lower triangular because multiplication by $\bL$ removes the $s=0$ column and shifts the remaining columns one step to the left. Hence \eqref{eq:def-M} has no nonzero blocks above the diagonal in the interleaved ordering.

To write $\bw$ compactly, since
\begin{align*}
\bw
&= 
\begin{bmatrix}
-\frac{1}{\sqrt n}\sum_{s=0}^{0}\Phi^v_{1,s} \dot\bX\bh^s\\
\frac{1}{\sqrt n}\sum_{s=0}^{1}\Psi^\eta_{1,s} \dot\bX^\top\bu^s\\
-\frac{1}{\sqrt n}\sum_{s=0}^{1}\Phi^v_{2,s} \dot\bX\bh^s\\
\frac{1}{\sqrt n}\sum_{s=0}^{2}\Psi^\eta_{2,s} \dot\bX^\top\bu^s\\	
\vdots\\
-\frac{1}{\sqrt n}\sum_{s=0}^{T-1}\Phi^v_{T,s} \dot\bX\bh^s\\
\frac{1}{\sqrt n}\sum_{s=0}^{T}\Psi^\eta_{T,s} \dot\bX^\top\bu^s
\end{bmatrix}\\
&= -\frac{1}{\sqrt n} \begin{bmatrix}
\sum_{s=0}^{0}\Phi^v_{1,s} \dot\bX\bh^s\\
\bzero_p\\
\sum_{s=0}^{1}\Phi^v_{2,s} \dot\bX\bh^s\\
\bzero_p\\	
\vdots\\
\sum_{s=0}^{T-1}\Phi^v_{T,s} \dot\bX\bh^s\\
\bzero_p
\end{bmatrix}
+ 
\frac{1}{\sqrt n}
\begin{bmatrix}
\bzero_n\\
\sum_{s=0}^{1}\Psi^\eta_{1,s} \dot\bX^\top\bu^s\\
\bzero_n\\
\sum_{s=0}^{2}\Psi^\eta_{2,s} \dot\bX^\top\bu^s\\	
\vdots\\
\bzero_n\\
\sum_{s=0}^{T}\Psi^\eta_{T,s} \dot\bX^\top\bu^s
\end{bmatrix}\\
&= -\frac{1}{\sqrt n}\bB_1 \bPhi^v \vec(\dot\bX\bH) 
+ \frac{1}{\sqrt n}\bB_2 \bPsi^\eta \vec(\dot\bX^\top \bF)\\
&:= - \bJ_{\bPhi^v}\vec(\dot\bX\bH) + \bJ_{\bPsi^\eta} \vec(\dot\bX^\top \bF).
\end{align*}
where 
\begin{align}
\bJ_{\bPhi^v}
 = \frac{1}{\sqrt n}\bB_1 \bPhi^v,
\quad
\bJ_{\bPsi^\eta} = \frac{1}{\sqrt n}\bB_2 \bPsi^\eta. 
\end{align}
Equivalently, $\bJ_{\bPhi^v}$ and $\bJ_{\bPsi^\eta}$ can be written as block lower-triangular matrices:
\begin{equation}
\begin{aligned}
\bJ_{\bPhi^v}
&=
\frac{1}{\sqrt n}
\begin{bmatrix}
\begin{bmatrix}\Phi^v_{1,0}\\ \bzero_{p\times n}\end{bmatrix} & & & \\
\begin{bmatrix}\Phi^v_{2,0}\\ \bzero_{p\times n}\end{bmatrix} &
\begin{bmatrix}\Phi^v_{2,1}\\ \bzero_{p\times n}\end{bmatrix} & & \\
\vdots & \ddots & \ddots & \\
\begin{bmatrix}\Phi^v_{T,0}\\ \bzero_{p\times n}\end{bmatrix} &
\cdots &
\begin{bmatrix}\Phi^v_{T,T-2}\\ \bzero_{p\times n}\end{bmatrix} &
\begin{bmatrix}\Phi^v_{T,T-1}\\ \bzero_{p\times n}\end{bmatrix}
\end{bmatrix} \in \R^{T(n+p)\times nT},
\\[6pt]
\bJ_{\bPsi^\eta}
&=
\frac{1}{\sqrt n}
\begin{bmatrix}
\begin{bmatrix}\bzero_{n\times p}\\ \Psi^\eta_{1,1}\end{bmatrix} & & & \\
\begin{bmatrix}\bzero_{n\times p}\\ \Psi^\eta_{2,1}\end{bmatrix} &
\begin{bmatrix}\bzero_{n\times p}\\ \Psi^\eta_{2,2}\end{bmatrix} & & \\
\vdots & \ddots & \ddots & \\
\begin{bmatrix}\bzero_{n\times p}\\ \Psi^\eta_{T,1}\end{bmatrix} &
\cdots &
\begin{bmatrix}\bzero_{n\times p}\\ \Psi^\eta_{T,T-1}\end{bmatrix} &
\begin{bmatrix}\bzero_{n\times p}\\ \Psi^\eta_{T,T}\end{bmatrix}
\end{bmatrix} \in \R^{T(n+p)\times pT}.
\end{aligned}
\end{equation}
Solving the linear system \eqref{eq:block-linear-system} yields
\[
\bvartheta = \calM^{-1} \bw,
\]
and 
\begin{align*}
	\dot\bu^t = (\be_t^\top \otimes [\bI_n, \boldzero_{n\times p}]) \bvartheta,
	\quad
	\dot\bb^t = (\be_t^\top \otimes [\boldzero_{p\times n}, \bI_p]) \bvartheta.
\end{align*}
This completes the proof of Lemma~\ref{lem:dot-b-general}.
\end{proof}

\section{Proof of Theorem~\ref{thm:main-universality-sqrt-ridge}}
\label{sec:universality}

This appendix proves Theorem~\ref{thm:main-universality-sqrt-ridge}. We work
at fixed \(T\) with square-root ridge, deterministic signal, deterministic
covariance, and squared test loss up to the zero-prediction centering
convention. The non-Gaussian covariance model is
\[
\bX=\bZ\bSigma_n^{1/2}
\]
and the matched Gaussian model is
\[
\bX_{\rm G}=\bG\bSigma_n^{1/2},
\]
where the deterministic covariance \(\bSigma_n\) is the same in the two models.
The proof has three ingredients: a finite Krylov representation of the
trajectory, matched-Gaussian universality for a finite list of primitive scalar
statistics, and finite-dimensional closure of the correction matrices. These
ingredients are then used to prove the true-risk, covariance-dependent, and
covariance-free parts of the theorem.

Throughout this section, all statements are for the original covariance
coordinates. We consider the square-root ridge problem
\[
\min_{\bb\in\R^p} \|\by-\bX\bb\|_2 + \frac{\lambda}{2}\|\bb\|_2^2,
\qquad
\lambda>0,
\]
and the Chambolle--Pock iteration with fixed parameters
\((\sigma,\tau,\theta)\in(0,\infty)^2\times[0,1]\) through
\(\sigma_n=\sigma/\sqrt n\) and \(\tau_n=\tau/\sqrt n\). In the square-root case the
dual update map is
\[
\bphi(\bu)=\frac{\bu}{\max\{1,\|\bu\|_2\}},
\]
and in the ridge case the primal map is linear,
\[
\bpsi_n(\bb)=\frac{1}{1+\tau_n\lambda}\,\bb.
\]
We write
\[
\begin{aligned}
c_{\lambda,n}
&:=\frac{1}{1+\tau_n\lambda}
=\frac{1}{1+\tau\lambda/\sqrt n},
\qquad 
\tilde\by
:=\frac{\by}{\sqrt n},\\
\bA_n
&:=\frac{\bX\bX^\top}{n},
\qquad
\bB_n:=\frac{\bX\bSigma_n\bX^\top}{n},
\qquad
\bS_n:=\frac{\bX^\top\bX}{n}.
\end{aligned}
\]

\subsection{Trajectory representation}

We begin by showing that the square-root ridge trajectory stays in a finite
Krylov space generated by \(\bA_n\) and \(\tilde\by\).

For \(t\ge 0\), define scalar-valued functions \(q_t^{(n)},r_t^{(n)},b_t^{(n)}\)
recursively as follows. Set
\[
q_0^{(n)}(x)\equiv 1,
\qquad
b_0^{(n)}(x)\equiv 0.
\]
For \(t=1\), define
\[
a_1^{(n)}(x):=-\sigma q_0^{(n)}(x)\equiv -\sigma,
\qquad
\alpha_1^{(n)}
:=
\frac{1}{\max\left\{1,\left(\tilde\by^\top (a_1^{(n)}(\bA_n))^2\tilde\by\right)^{1/2}\right\}},
\]
and set
\[
r_1^{(n)}:=\alpha_1^{(n)} a_1^{(n)},
\qquad
b_1^{(n)}:=c_{\lambda,n} b_0^{(n)}-c_{\lambda,n}\tau r_1^{(n)},
\qquad
q_1^{(n)}(x):=1-x b_1^{(n)}(x).
\]
For \(t\ge 2\), having defined \(q_{t-1}^{(n)},q_{t-2}^{(n)},r_{t-1}^{(n)}\) and
\(b_{t-1}^{(n)}\), let
\[
a_t^{(n)}(x)
:=
r_{t-1}^{(n)}(x)
-\sigma\bigl((1+\theta)q_{t-1}^{(n)}(x)-\theta q_{t-2}^{(n)}(x)\bigr),
\]
\[
\alpha_t^{(n)}
:=
\frac{1}{\max\left\{1,\left(\tilde\by^\top (a_t^{(n)}(\bA_n))^2\tilde\by\right)^{1/2}\right\}},
\]
and set
\[
r_t^{(n)}:=\alpha_t^{(n)} a_t^{(n)},
\qquad
b_t^{(n)}:=c_{\lambda,n} b_{t-1}^{(n)}-c_{\lambda,n}\tau r_t^{(n)},
\qquad
q_t^{(n)}(x):=1-x b_t^{(n)}(x).
\]

\begin{lemma}
\label{lem:universality-krylov}
For every fixed \(t\ge 0\), the iterates of the square-root ridge
Chambolle--Pock recursion satisfy
\[
\bu^t = r_t^{(n)}(\bA_n)\tilde\by,
\qquad
\bv^t = q_t^{(n)}(\bA_n)\tilde\by,
\qquad
\hbb^t = \frac{1}{\sqrt n}\bX^\top b_t^{(n)}(\bA_n)\tilde\by.
\]
In particular, all three iterates belong to the Krylov space generated by
\(\bA_n\) and \(\tilde\by\).
\end{lemma}

\begin{proof}[Proof of Lemma~\ref{lem:universality-krylov}]
\label{proof:lem:universality-krylov}
We argue by induction on \(t\). Since \(\hbb^0=\bzero\), we have
\[
\bv^0=\frac{\by-\bX\hbb^0}{\sqrt n}=\tilde\by=q_0^{(n)}(\bA_n)\tilde\by,
\qquad
\hbb^0=\frac1{\sqrt n}\bX^\top b_0^{(n)}(\bA_n)\tilde\by.
\]
For \(t=1\),
\[
\bu^1
=
\bphi(-\sigma \bv^0)
=
\frac{-\sigma\tilde\by}{\max\{1,\sigma\|\tilde\by\|_2\}}
=
r_1^{(n)}(\bA_n)\tilde\by.
\]
Since \(\hbb^1=\bpsi_n(-\tau\bfeta^1)= -c_{\lambda,n}\tau \bfeta^1\) and
\(\bfeta^1=\bX^\top\bu^1/\sqrt n\),
\[
\hbb^1
=
-\frac{c_{\lambda,n}\tau}{\sqrt n}\bX^\top r_1^{(n)}(\bA_n)\tilde\by
=
\frac1{\sqrt n}\bX^\top b_1^{(n)}(\bA_n)\tilde\by.
\]
Consequently,
\[
\bv^1
=
\tilde\by-\frac1{\sqrt n}\bX\hbb^1
=
\tilde\by-\frac1n \bX\bX^\top b_1^{(n)}(\bA_n)\tilde\by
=
q_1^{(n)}(\bA_n)\tilde\by.
\]

Now assume the claim holds up to time \(t-1\), with \(t\ge 2\). Then
\[
\bu^{t}
=
\bphi\!\left(\bu^{t-1}-\sigma\bigl((1+\theta)\bv^{t-1}-\theta\bv^{t-2}\bigr)\right).
\]
By the induction hypothesis, the argument of \(\bphi\) equals
\[
\left(
r_{t-1}^{(n)}(\bA_n)
-\sigma\bigl((1+\theta)q_{t-1}^{(n)}(\bA_n)-\theta q_{t-2}^{(n)}(\bA_n)\bigr)
\right)\tilde\by
=
a_t^{(n)}(\bA_n)\tilde\by.
\]
Since \(\bphi(\bz)=\bz/\max\{1,\|\bz\|_2\}\), we obtain
\[
\bu^t = \alpha_t^{(n)} a_t^{(n)}(\bA_n)\tilde\by = r_t^{(n)}(\bA_n)\tilde\by.
\]
Next,
\[
\begin{aligned}
\hbb^t
&=
c_{\lambda,n}\left(\hbb^{t-1}-\tau\frac{\bX^\top\bu^t}{\sqrt n}\right)
=
\frac{c_{\lambda,n}}{\sqrt n}\bX^\top
\bigl(b_{t-1}^{(n)}(\bA_n)-\tau r_t^{(n)}(\bA_n)\bigr)\tilde\by\\
&=
\frac1{\sqrt n}\bX^\top b_t^{(n)}(\bA_n)\tilde\by.
\end{aligned}
\]
Finally,
\[
\bv^t
=
\tilde\by-\frac1{\sqrt n}\bX\hbb^t
=
\tilde\by-\frac1n \bX\bX^\top b_t^{(n)}(\bA_n)\tilde\by
=
q_t^{(n)}(\bA_n)\tilde\by.
\]
This completes the induction.
\end{proof}

\subsection{Admissible primitive scalars}

The proof reduces all fixed-time quantities to finitely many scalar polynomial
functionals. We call each such statistic a \emph{primitive scalar}. An
admissible primitive scalar is one of the finite-degree statistics listed
below, or a finite linear combination of them; a primitive vector is a finite
vector of primitive scalars. The symbols \(\xi_{n,j}\) and
\(\zeta_{n,\ell}\) denote primitive scalars, while \(\vartheta_n\) denotes a
primitive vector. Once a quantity is a continuous function of finitely many
admissible primitive scalars, matched-Gaussian universality follows from
convergence of that finite vector.

\begin{assumption}[Primitive comparison setup]
\label{assu:primitive-comparison}
Let \(\bX=\bZ\bSigma_n^{1/2}\), where the entries of \(\bZ\) are iid with mean
\(0\), variance \(1\), and uniformly bounded sub-Gaussian norm. Let
\(\bX_{\rm G}=\bG\bSigma_n^{1/2}\), where \(\bG\) has iid \(N(0,1)\) entries.
Both designs use the same deterministic covariance matrix \(\bSigma_n\), the
same deterministic signal \(\bb^*=\bb_n^*\), and the same noise vector \(\bep\).
The noise variables are iid, independent of the designs, with mean \(0\),
variance \(\sigma_\ep^2\), and finite fourth moment. Moreover,
\[
p/n\to\gamma_0\in(0,\infty),\qquad
\|\bSigma_n\|_{\op}\vee\|\bSigma_n^{-1}\|_{\op}\le C,\qquad
\|\bb^*\|_2\le C.
\]
For every fixed \(k\ge0\), the two deterministic quantities
\[
\frac1p\trace(\bSigma_n^k),
\qquad
\bb^{*\top}\bSigma_n^k\bb^*
\]
have limits. The second convergence is the usual signal-alignment condition for
\(\bb^*\) relative to the eigenspaces of \(\bSigma_n\).
\end{assumption}

Define
\[
\bA_n=\frac{\bX\bX^\top}{n},\qquad
\bS_n=\frac{\bX^\top\bX}{n},\qquad
\bR_n=\frac{\bX^\top\tilde\by\tilde\by^\top\bX}{n}.
\]

Let \(W(\bS_n,\bSigma_n)\) denote a product of finitely many factors, each equal
to \(\bS_n\) or \(\bSigma_n\). A primitive scalar is \emph{admissible} if it is
a finite linear combination of scalars of the forms
\[
\begin{aligned}
&n^{-1/2},\quad n^{-1},\quad
\frac1n\trace(W(\bS_n,\bSigma_n)),\quad
\bb^{*\top}W(\bS_n,\bSigma_n)\bb^*,\\
&\tilde\by^\top \pi(\bA_n)\tilde\by,\quad
\frac1n\trace\!\left(
\bM_n W_1(\bS_n,\bSigma_n)\bR_n W_2(\bS_n,\bSigma_n)
\right).
\end{aligned}
\]
where \(W,W_1,W_2\) denote such finite products, \(\pi\) is a finite-degree
polynomial, \(\bM_n\in\{\bI_p,\bSigma_n\}\), and
\(\tilde\by=(\bX\bb^*+\bep)/\sqrt n\). For example,
\(n^{-1}\trace(\bS_n\bSigma_n\bS_n)\) is admissible.
The matched-Gaussian version of an admissible primitive scalar is obtained by
replacing \(\bX,\bA_n,\bS_n,\bR_n,\tilde\by\) by the corresponding objects built
from \(\bX_{\rm G}\), while the deterministic scalars \(n^{-1/2}\) and
\(n^{-1}\) are unchanged. 
We include \(n^{-1/2}\) explicitly because the finite-\(n\) ridge proximal
factor is
\(
c_{\lambda,n}=\frac{1}{1+\tau\lambda n^{-1/2}}.
\)

The next lemma states that any fixed finite list of such
statistics has the same asymptotic value in the two matched designs.

\begin{lemma}[Primitive universality for the covariance model]
\label{lem:universality-primitive-moments}
Under Assumption~\ref{assu:primitive-comparison}, for every fixed finite
collection of admissible primitive scalars \(\{\xi_{n,j}\}_{j=1}^m\) and their
matched-Gaussian counterparts \(\{\xi_{n,j}^{\rm G}\}_{j=1}^m\),
\[
\max_{1\le j\le m}|\xi_{n,j}-\xi_{n,j}^{\rm G}|\limp 0.
\]
Moreover, each such finite collection has a deterministic limit depending only
on \(\gamma_0\), \(\sigma_\ep^2\), the limiting spectral moments of
\(\bSigma_n\), and the limiting signal-alignment moments.
\end{lemma}

\begin{proof}[Proof of Lemma~\ref{lem:universality-primitive-moments}]
\label{proof:lem:universality-primitive-moments}
We first prove one representative case. Take
\[
\xi_n=\tilde\by^\top \bA_n^2\tilde\by,
\qquad
\xi_n^{\rm G}=(\tilde\by^{\rm G})^\top(\bA_n^{\rm G})^2\tilde\by^{\rm G}.
\]
Since \(\tilde\by=(\bX\bb^*+\bep)/\sqrt n\),
\[
\xi_n
=
\bb^{*\top}\bS_n^3\bb^*
+\frac{2}{n}\bep^\top\bA_n^2\bX\bb^*
+\frac1n\bep^\top\bA_n^2\bep .
\]
The mixed term is negligible. Indeed, conditioning on \(\bX\), it has mean zero
and variance bounded by
\[
\frac{C}{n^2}\|\bA_n^2\bX\bb^*\|_2^2
\le
\frac{C}{n^2}\|\bA_n\|_{\op}^4\|\bX\bb^*\|_2^2
=O_{\P}(1/n),
\]
because \(\|\bA_n\|_{\op}=O_{\P}(1)\) and
\(\|\bX\bb^*\|_2^2=n\,\bb^{*\top}\bS_n\bb^*=O_{\P}(n)\). For the noise
quadratic form,
\[
\frac1n\bep^\top\bA_n^2\bep
=
\sigma_\ep^2\frac1n\trace(\bA_n^2)+o_{\P}(1)
=
\sigma_\ep^2\frac1n\trace(\bS_n^2)+o_{\P}(1).
\]
Thus
\[
\xi_n
=
\bb^{*\top}\bS_n^3\bb^*
+\sigma_\ep^2\frac1n\trace(\bS_n^2)
+o_{\P}(1).
\]
The same expansion holds for \(\xi_n^{\rm G}\), with \(\bS_n\) replaced by
\(\bS_n^{\rm G}\). The two remaining terms,
\(\bb^{*\top}\bS_n^3\bb^*\) and \(n^{-1}\trace(\bS_n^2)\), are fixed-degree
polynomial functionals of the design. By the standard Lindeberg replacement
argument, or equivalently by the fixed-degree moment expansion, these terms
have the same deterministic limits under \(\bZ\) and under the matched Gaussian
\(\bG\). Hence \(\xi_n-\xi_n^{\rm G}\limp0\) for this example.

The preceding example illustrates a general reduction. After expanding
\(\tilde\by=(\bX\bb^*+\bep)/\sqrt n\), and also expanding \(\bR_n\) whenever
it appears, every admissible primitive scalar is, apart from the deterministic
terms \(n^{-1/2}\) and \(n^{-1}\), a finite linear combination of terms of the
following forms:
\[
\frac1n\trace W(\bS_n,\bSigma_n),
\qquad
\bb^{*\top}W(\bS_n,\bSigma_n)\bb^*,
\qquad
\frac1n\{\bep^\top\bM_n\bep-\sigma_\ep^2\trace(\bM_n)\},
\qquad
\frac1n\bep^\top\bv_n.
\]
Here \(W\) is a fixed-degree product, while \(\bM_n\) and \(\bv_n\) are
fixed-degree functions of the design. The bounded-spectrum and proportional
asymptotic assumptions imply
\[
\|\bM_n\|_{\op}=O_{\P}(1),
\qquad
\|\bv_n\|_2^2=O_{\P}(n).
\]
The deterministic terms are identical in the two matched models and converge to
zero.

The trace and deterministic-signal terms are fixed-degree polynomial
functionals of the entries of \(\bZ\), with deterministic coefficients formed
from \(\bSigma_n\) and \(\bb^*\). The normalizations in the definition of
admissible primitive scalars, together with \(p/n=O(1)\),
\(\|\bSigma_n\|_{\op}=O(1)\), and \(\|\bb^*\|_2=O(1)\), give the usual
Lindeberg derivative bounds for Chatterjee's invariance principle
\cite[Theorem~1.1]{chatterjee2006generalization}. Equivalently, in the direct
moment expansion, only pairings of design entries contribute at order one, while
terms involving third or higher cumulants have fewer free indices and are
\(o(1)\). Therefore these terms have the same deterministic limits under
\(\bZ\) and under the matched Gaussian \(\bG\).

For the centered noise quadratic forms, condition on the design. The
quadratic-form variance bound gives
\[
\E\!\left[
\left(
\frac1n\bep^\top \bM_n\bep
-
\sigma_\ep^2\frac1n\trace(\bM_n)
\right)^2
\middle|\,\bX
\right]
\le
\frac{C}{n^2}\trace(\bM_n\bM_n^\top).
\]
The right-hand side is \(O_{\P}(1/n)\), hence \(o_{\P}(1)\), because
\[
\trace(\bM_n\bM_n^\top)=\|\bM_n\|_F^2
\le n\|\bM_n\|_{\op}^2
=O_{\P}(n).
\]
Thus the noise quadratic forms reduce to trace functionals already covered
above. The mixed terms have conditional mean zero and satisfy
\[
\E\!\left[
\left(\frac1n\bep^\top\bv_n\right)^2
\middle|\,\bX
\right]
=\frac{\sigma_\ep^2}{n^2}\|\bv_n\|_2^2
=O_{\P}(1/n),
\]
so they are \(o_{\P}(1)\) as well. The same expansion holds in the matched
Gaussian model, with the same limits for its design-only terms.

Combining these observations proves matched-Gaussian universality for each
primitive scalar. Applying the preceding expansion to products of two admissible
primitive scalars gives convergence of their variances, so the limits are
deterministic. These limits depend only on \(\gamma_0\), \(\sigma_\ep^2\), the
limiting covariance spectral moments, and the limiting signal-alignment moments.
Since the collection is finite, the coordinatewise statements imply the
displayed maximum convergence.
\end{proof}

We also need a version of the primitive-scalar comparison in which the
coefficients are continuous functions of finitely many primitive scalars.

\begin{lemma}[Random-coefficient primitive extension]
\label{lem:universality-random-polynomial}
Let \(\xi_n(\vartheta_n)\) be a scalar obtained by taking a finite linear
combination of admissible primitive scalars, with coefficients that are
continuous functions of a finite primitive vector \(\vartheta_n\). Let
\(\xi_n^{\rm G}(\vartheta_n^{\rm G})\) be the matched-Gaussian counterpart, with
the same deterministic coefficient map. Then
\[
\xi_n(\vartheta_n)-\xi_n^{\rm G}(\vartheta_n^{\rm G})\limp 0.
\]
\end{lemma}

\begin{proof}[Proof of Lemma~\ref{lem:universality-random-polynomial}]
\label{proof:lem:universality-random-polynomial}
First consider a concrete example with deterministic continuous coefficient
maps \(h\) and \(a\). Write
\[
\zeta_{n,1}=\frac1n\trace(\bS_n\bSigma_n\bS_n),
\qquad
\zeta_{n,2}=\tilde\by^\top\bA_n^2\tilde\by,
\]
so that
\[
\xi_n(\vartheta_n)=h(\vartheta_n)\zeta_{n,1}+a(\vartheta_n)\zeta_{n,2},
\qquad
\xi_n^{\rm G}(\vartheta_n^{\rm G})
=h(\vartheta_n^{\rm G})\zeta_{n,1}^{\rm G}
+a(\vartheta_n^{\rm G})\zeta_{n,2}^{\rm G}.
\]
The coordinates of \(\vartheta_n\), together with \(\zeta_{n,1}\) and
\(\zeta_{n,2}\), form a finite vector of admissible primitive scalars. Hence
Lemma~\ref{lem:universality-primitive-moments} gives a deterministic vector
\(v_*\) such that
\[
\left(\vartheta_n,\zeta_{n,1},\zeta_{n,2}\right)\limp v_*,
\qquad
\left(\vartheta_n^{\rm G},\zeta_{n,1}^{\rm G},\zeta_{n,2}^{\rm G}\right)
\limp v_* .
\]
It follows that
\[
\|\vartheta_n-\vartheta_n^{\rm G}\|_\infty
+|\zeta_{n,1}-\zeta_{n,1}^{\rm G}|
+|\zeta_{n,2}-\zeta_{n,2}^{\rm G}|
\limp 0.
\]
By continuity of \(h\) and \(a\),
\[
|h(\vartheta_n)-h(\vartheta_n^{\rm G})|
+|a(\vartheta_n)-a(\vartheta_n^{\rm G})|
\limp 0.
\]
The same convergence also implies
\[
|\zeta_{n,1}|+|\zeta_{n,2}|+
|h(\vartheta_n^{\rm G})|+|a(\vartheta_n^{\rm G})|
=O_{\P}(1).
\]
Therefore
\begin{align*}
&\left|\xi_n(\vartheta_n)-\xi_n^{\rm G}(\vartheta_n^{\rm G})\right| \\
&\quad\le
|h(\vartheta_n)-h(\vartheta_n^{\rm G})|\,|\zeta_{n,1}|
+|h(\vartheta_n^{\rm G})|\,|\zeta_{n,1}-\zeta_{n,1}^{\rm G}| \\
&\qquad+
|a(\vartheta_n)-a(\vartheta_n^{\rm G})|\,|\zeta_{n,2}|
+|a(\vartheta_n^{\rm G})|\,|\zeta_{n,2}-\zeta_{n,2}^{\rm G}|
=o_{\P}(1).
\end{align*}

For a general finite sum, write
\[
\xi_n(\vartheta_n)=\sum_{\ell=1}^L g_\ell(\vartheta_n)\zeta_{n,\ell},
\qquad
\xi_n^{\rm G}(\vartheta_n^{\rm G})
=\sum_{\ell=1}^L g_\ell(\vartheta_n^{\rm G})\zeta_{n,\ell}^{\rm G},
\]
where \(L\) is fixed, the \(\zeta_{n,\ell}\)'s are admissible primitive
scalars, and the \(g_\ell\)'s are deterministic continuous functions. Applying
Lemma~\ref{lem:universality-primitive-moments} and continuity,
\(\|\vartheta_n-\vartheta_n^{\rm G}\|_\infty\limp0\),
\(\max_\ell|\zeta_{n,\ell}-\zeta_{n,\ell}^{\rm G}|\limp0\), and
\(\max_\ell|g_\ell(\vartheta_n)-g_\ell(\vartheta_n^{\rm G})|\limp0\). The primitive
scalars and Gaussian coefficients are \(O_{\P}(1)\). Therefore
\begin{align*}
\left|\xi_n(\vartheta_n)-\xi_n^{\rm G}(\vartheta_n^{\rm G})\right|
&\le
\sum_{\ell=1}^L
\left|g_\ell(\vartheta_n)-g_\ell(\vartheta_n^{\rm G})\right|
|\zeta_{n,\ell}| \\
&\quad+
\sum_{\ell=1}^L
|g_\ell(\vartheta_n^{\rm G})|
\left|\zeta_{n,\ell}-\zeta_{n,\ell}^{\rm G}\right|
=o_{\P}(1).
\end{align*}
\end{proof}

\subsection{True-risk component}

For every fixed polynomial \(\pi\), Lemma~\ref{lem:universality-primitive-moments}
implies that
\[
\tilde\by^\top \pi(\bA_n)\tilde\by
\]
has the same deterministic limit in the non-Gaussian and matched Gaussian
models. Denote this limiting linear functional by \(\mathfrak m_\Sigma(\pi)\).
Unlike the isotropic case, \(\mathfrak m_\Sigma\) is not a simple
Marchenko--Pastur expression; it depends on the limiting spectral distribution
of \(\bSigma_n\) and on the limiting alignment of \(\bb^*\) with the eigenspaces
of \(\bSigma_n\). Define the deterministic square-root ridge recursion driven
by \(\mathfrak m_\Sigma\) as follows. The finite-sample ridge factor
\(c_{\lambda,n}\) converges to one because \(\tau_n=\tau/\sqrt n\), so the
limiting fixed-\(T\) recursion uses the coefficient \(1\) in place of
\(c_{\lambda,n}\). Set
\[
q_0(x)\equiv 1,
\qquad
b_0(x)\equiv 0.
\]
For \(t=1\), let
\[
a_1(x)\equiv -\sigma,
\qquad
\alpha_1
:=
\frac{1}{\max\left\{1,\left(\mathfrak m_\Sigma(a_1^2)\right)^{1/2}\right\}},
\]
and define
\[
r_1:=\alpha_1 a_1,
\qquad
b_1:=b_0-\tau r_1,
\qquad
q_1(x):=1-x b_1(x).
\]
For \(t\ge 2\), define recursively
\[
a_t(x):=
r_{t-1}(x)-\sigma\bigl((1+\theta)q_{t-1}(x)-\theta q_{t-2}(x)\bigr),
\]
\[
\alpha_t
:=
\frac{1}{\max\left\{1,\left(\mathfrak m_\Sigma(a_t^2)\right)^{1/2}\right\}},
\]
\[
r_t:=\alpha_t a_t,
\qquad
b_t:=b_{t-1}-\tau r_t,
\qquad
q_t(x):=1-x b_t(x).
\]
For every fixed \(t\), the coefficients of the finite-\(n\) polynomials
\(a_s^{(n)},r_s^{(n)},b_s^{(n)},q_s^{(n)}\), \(s\le t\), are continuous
functions of finitely many admissible primitive scalars, including \(n^{-1/2}\).
Consequently, Lemma~\ref{lem:universality-random-polynomial} implies that
the finite-\(n\) coefficient vector differs from the coefficient vector of
\((a_s,r_s,b_s,q_s)\) by \(o_{\P}(1)\), and the same convergence holds under
the matched Gaussian design.

These preparations yield the true-risk component of
Theorem~\ref{thm:main-universality-sqrt-ridge}.

\begin{lemma}[True-risk universality]
\label{lem:universality-true-risk-sqrt-ridge}
Under Assumption~\ref{assu:primitive-comparison}, for the matched
non-Gaussian and Gaussian square-root ridge experiments described above, the
conditional squared-loss risks satisfy, for every fixed
\(t\in\{0,\dots,T-1\}\),
\[
\calR_t^{\rm Z}-\calR_t^{\rm G}\limp 0.
\]
This conclusion does not require the no-boundary condition.
\end{lemma}

\begin{proof}[Proof of Lemma~\ref{lem:universality-true-risk-sqrt-ridge}]
\label{proof:lem:universality-true-risk-sqrt-ridge}
Under squared loss and covariance \(\bSigma_n\), the out-of-sample risk is
\[
\calR_t
=
\sigma_\ep^2+
(\hbb^t-\bb^*)^\top \bSigma_n(\hbb^t-\bb^*),
\]
because \(\E[\bx_{\rm new}]=\bzero\),
\(\E[\bx_{\rm new}\bx_{\rm new}^\top]=\bSigma_n\), and
\((\bx_{\rm new},\ep_{\rm new})\) is independent of \(\hbb^t\).

By Lemma~\ref{lem:universality-krylov},
\[
\hbb^t=\frac1{\sqrt n}\bX^\top b_t^{(n)}(\bA_n)\tilde\by.
\]
Using \(\tilde\by=(\bX\bb^*+\bep)/\sqrt n\), we have
\[
\hbb^t
=
\bS_n b_t^{(n)}(\bS_n)\bb^*
+
\frac1n\bX^\top b_t^{(n)}(\bA_n)\bep.
\]
Thus
\[
\bb^*-\hbb^t
=
q_t^{(n)}(\bS_n)\bb^*
-
\frac1n\bX^\top b_t^{(n)}(\bA_n)\bep,
\]
and the risk decomposes as
\begin{align}
\calR_t
&=
\sigma_\ep^2
+
\bb^{*\top}q_t^{(n)}(\bS_n)\bSigma_n q_t^{(n)}(\bS_n)\bb^*
\nonumber\\
&\quad
+
\frac1n\bep^\top b_t^{(n)}(\bA_n)\bB_n b_t^{(n)}(\bA_n)\bep
-
\frac{2}{n}\bb^{*\top}q_t^{(n)}(\bS_n)\bSigma_n\bX^\top
b_t^{(n)}(\bA_n)\bep.
\label{eq:general-sigma-risk-decomp}
\end{align}

The coefficients of
\(a_s^{(n)},r_s^{(n)},b_s^{(n)},q_s^{(n)}\) up to any fixed time \(t\) are
continuous functions of finitely many quantities
\(\tilde\by^\top \bA_n^\ell\tilde\by\) and of the deterministic scalar
\(n^{-1/2}\), through \(c_{\lambda,n}=1/(1+\tau\lambda n^{-1/2})\). These are
admissible primitive scalars.
Lemma~\ref{lem:universality-random-polynomial} therefore implies that the
corresponding coefficient vectors for \(\bX\) and \(\bX_{\rm G}\) differ by
\(o_{\P}(1)\).

Each term in \eqref{eq:general-sigma-risk-decomp} is a finite linear
combination of admissible primitive scalars with those random coefficients. For the
noise term, conditioning on the design replaces the quadratic form by its
trace,
\[
\frac{\sigma_\ep^2}{n}\trace\!\left(
b_t^{(n)}(\bA_n)\bB_n b_t^{(n)}(\bA_n)
\right),
\]
up to \(o_{\P}(1)\); this trace equals a finite linear combination of
\(n^{-1}\trace(\bSigma_n\bS_n^\ell)\). The mixed signal-noise term has
conditional mean zero and conditional variance \(o_{\P}(1)\). Hence
Lemma~\ref{lem:universality-random-polynomial} gives the same limit for the
non-Gaussian and matched Gaussian versions of every term in
\eqref{eq:general-sigma-risk-decomp}. Therefore
\[
\calR_t^{\rm Z}-\calR_t^{\rm G}\limp0.
\]
\end{proof}

\subsection{Covariance-dependent estimator component}

Recall from the primitive-scalar setup that
\[
\bR_n:=\frac{\bX^\top \tilde\by \tilde\by^\top \bX}{n}\in\R^{p\times p}.
\]
In the general-covariance proof, \(\vartheta_n^{(t)}\) denotes a finite vector
of admissible primitive scalars large enough to contain all trace and
quadratic-form quantities generated by the closure argument below, including
\[
n^{-1/2},\quad n^{-1},\quad
\tilde\by^\top \bA_n^\ell\tilde\by,\quad
n^{-1}\trace W(\bS_n,\bSigma_n),
\]
and
\[
\frac1n\trace\!\left(
\bM_nW_1(\bS_n,\bSigma_n)\bR_nW_2(\bS_n,\bSigma_n)
\right),
\qquad \bM_n\in\{\bI_p,\bSigma_n\},
\]
for the finitely many products that arise up to time \(t\). Its
matched-Gaussian counterpart is denoted \(\vartheta_{n,{\rm G}}^{(t)}\).

The correction estimators use the Jacobian of the square-root projection, so we
also require the deterministic no-boundary condition
\[
\mathfrak m_\Sigma(a_s^2)\neq 1,\qquad s=1,\ldots,t.
\]
It ensures that the vector entering the projection at each fixed step stays on
one side of the nondifferentiability boundary \(\|\bz\|_2=1\) with probability
tending to one. This condition is not needed for true-risk universality.

\begin{lemma}[Covariance-dependent universality]
\label{lem:universality-rt-tilde-sqrt-ridge}
Fix \(t\in\{0,\ldots,T-1\}\). Under the assumptions of
Lemma~\ref{lem:universality-true-risk-sqrt-ridge}, if
\(\mathfrak m_\Sigma(a_s^2)\neq 1\) for \(s=1,\ldots,t\), the
covariance-dependent estimator satisfies
\[
\widetilde{\calR}_t^{\rm Z}-\widetilde{\calR}_t^{\rm G}\limp 0.
\]
\end{lemma}

\begin{proof}[Proof of Lemma~\ref{lem:universality-rt-tilde-sqrt-ridge}]
\label{proof:lem:universality-rt-tilde-sqrt-ridge}
Lemma~\ref{lem:universality-krylov} gives, for each \(s\le t\),
\[
\bu^s=r_s^{(n)}(\bA_n)\tilde\by,
\qquad
\bv^s=q_s^{(n)}(\bA_n)\tilde\by,
\qquad
\hbb^s=\frac1{\sqrt n}\bX^\top b_s^{(n)}(\bA_n)\tilde\by.
\]
Since \(t\) is fixed, there exists a finite degree \(d_t\), depending only on
\(t\), such that the polynomials
\[
a_s^{(n)},\ r_s^{(n)},\ q_s^{(n)},\ b_s^{(n)},\qquad s\le t,
\]
all have degree at most \(d_t\), and their coefficients are continuous
functions of the primitive moments listed in \(\vartheta_n^{(t)}\). The
coordinate \(n^{-1/2}\) in \(\vartheta_n^{(t)}\) accounts for the factor
\(c_{\lambda,n}=1/(1+\tau\lambda n^{-1/2})\) introduced by the ridge proximal
map.

We next describe the square-root Jacobians on the no-boundary event. Let
\[
\bz_s
:=
\bu^{s-1}-\sigma\bigl((1+\theta)\bv^{s-1}-\theta \bv^{s-2}\bigr)
=
a_s^{(n)}(\bA_n)\tilde\by,
\qquad s\ge 1,
\]
with the convention \(\bv^{-1}=\bzero\). For the square-root map
\[
\bphi(\bz)=\frac{\bz}{\max\{1,\|\bz\|_2\}},
\]
its Jacobian at \(\bz_s\), whenever \(\|\bz_s\|_2\neq 1\), is
\[
\bD_s
:=
D\bphi(\bz_s)
=
\begin{cases}
\bI_n, & \|\bz_s\|_2<1,\\[1mm]
\|\bz_s\|_2^{-1}\bI_n-\|\bz_s\|_2^{-3}\bz_s\bz_s^\top, & \|\bz_s\|_2>1.
\end{cases}
\]
Now
\[
\|\bz_s\|_2^2
=
\tilde\by^\top (a_s^{(n)}(\bA_n))^2\tilde\by
\limp
\mathfrak m_\Sigma(a_s^2)
\]
by Lemma~\ref{lem:universality-random-polynomial}, using the
coefficient convergence of \(a_s^{(n)}\) to \(a_s\) and the inclusion of
\(n^{-1/2}\) among the primitives. Since
\(\mathfrak m_\Sigma(a_s^2)\neq 1\) for \(s\le t\), there exists an open neighborhood
\(\calU_t\) of the deterministic limit of \(\vartheta_n^{(t)}\) such that on
\(\{\vartheta_n^{(t)}\in\calU_t\}\), every \(\|\bz_s\|_2\) stays on one fixed
side of \(1\). Hence, on that event,
\[
\bD_s
=
\alpha_s(\vartheta_n^{(t)})\bI_n
+
\gamma_s(\vartheta_n^{(t)})\,
a_s^{(n)}(\bA_n)\tilde\by\tilde\by^\top a_s^{(n)}(\bA_n),
\]
where \(\alpha_s\) and \(\gamma_s\) are continuous on \(\calU_t\).

For square-root ridge, the only nonzero Jacobian blocks are therefore
\[
\begin{aligned}
\Phi^u_{s,s-1}&=\bD_s,\\
\Phi^v_{1,0}&=-\sigma \bD_1,\\
\Phi^v_{s,s-1}&=-\sigma(1+\theta)\bD_s,\\
\Phi^v_{s,s-2}&=\sigma\theta\,\bD_s \qquad (s\ge 2),
\end{aligned}
\]
and for the primal ridge map,
\[
\Psi^b_{s,s-1}=c_{\lambda,n} \bI_p,\qquad
\Psi^\eta_{s,s}=-c_{\lambda,n}\tau \bI_p,
\]
with all remaining \(\Phi^u,\Phi^v,\Psi^b,\Psi^\eta\) equal to zero.

We now show that every block of
\(\bQ=\calM^{-1}\bJ_{\bPsi^\eta}\) is described by finitely many admissible
primitive scalars. The classes below are stable under multiplication by
\(\bX\), multiplication by \(n^{-1}\bX^\top\), and application of \(\bD_s\).

Define the matrix classes
\[
\mathcal H_d
:=
\left\{
\sum_{\ell=0}^d a_\ell \bA_n^\ell \bX
+
\sum_{a,b=0}^d c_{ab}\,\bA_n^a \tilde\by \tilde\by^\top \bX \bS_n^b
\right\}\subset\R^{n\times p},
\]
\[
\mathcal G_d
:=
\left\{
\sum_{\ell=0}^d \bar a_\ell \bS_n^\ell
+
\sum_{a,b=0}^d \bar c_{ab}\,\bS_n^a \bR_n \bS_n^b
\right\}\subset\R^{p\times p},
\]
where all scalar coefficients are deterministic continuous functions of
\(\vartheta_n^{(t)}\) on \(\calU_t\); in particular, the coefficients may use
the coordinate \(n^{-1/2}\) through \(c_{\lambda,n}\).

The definitions above are designed so that the following three closure
relations hold.

If \(M\in\mathcal G_d\), then \(\bX M\in\mathcal H_{d+1}\). Indeed,
\[
\bX \bS_n^\ell=\bA_n^\ell \bX,
\qquad
\bX \bS_n^a \bR_n \bS_n^b
=
\bA_n^a \bX \frac{\bX^\top \tilde\by \tilde\by^\top \bX}{n}\bS_n^b
=
\bA_n^{a+1}\tilde\by \tilde\by^\top \bX \bS_n^b.
\]
If \(N\in\mathcal H_d\), then \(n^{-1}\bX^\top N\in\mathcal G_{d+1}\), since
\[
\frac1n\bX^\top \bA_n^\ell \bX=\bS_n^{\ell+1},
\qquad
\frac1n\bX^\top \bA_n^a \tilde\by \tilde\by^\top \bX \bS_n^b
=
\bS_n^a \bR_n \bS_n^b.
\]
Finally, if \(N\in\mathcal H_d\), then \(\bD_s N\in\mathcal H_{d+d_t}\). This is
immediate for the identity part of \(\bD_s\). For the rank-one part, writing
\(p_s:=a_s^{(n)}\), we have
\[
p_s(\bA_n)\tilde\by \tilde\by^\top p_s(\bA_n)\,\bA_n^\ell \bX
=
p_s(\bA_n)\tilde\by \tilde\by^\top \bX \bS_n^\ell\in\mathcal H_{d+d_t},
\]
and
\begin{align*}
p_s(\bA_n)\tilde\by \tilde\by^\top p_s(\bA_n)
\bA_n^a \tilde\by \tilde\by^\top \bX \bS_n^b
&=
\underbrace{\tilde\by^\top p_s(\bA_n)\bA_n^a \tilde\by}_{\text{primitive moment}}
\cdot p_s(\bA_n)\tilde\by \tilde\by^\top \bX \bS_n^b\\
&\in\mathcal H_{d+d_t}.
\end{align*}

We now prove by induction on \(r\le t\) that, for every \(1\le s\le r\), the
blocks of \(\bQ\) satisfy
\[
\bQ_{r,s}\in \frac1n \mathcal H_{L_t},
\qquad
\widetilde{\bQ}_{r,s}\in \frac1{\sqrt n}\mathcal G_{L_t},
\]
for some finite \(L_t\) depending only on \(t\).

For \(r=1\), the explicit initialization gives
\[
\bQ_1=\boldzero,
\qquad
\widetilde{\bQ}_{1,1}
=
\frac1{\sqrt n}\Psi^\eta_{1,1}
=
-\frac{c_{\lambda,n}\tau}{\sqrt n}\bI_p,
\]
so the claim holds.

For the inductive step, assume the claim holds up to time \(r-1\) and fix a
column index \(s\le r\). Extracting the \(s\)-th block column from
\eqref{eq:Qt-recursion-explicit} gives
\[
\bQ_{r,s}
=
\bD_r \bQ_{r-1,s}
+
\frac{\sigma(1+\theta)}{\sqrt n}\,\bD_r \bX \widetilde{\bQ}_{r-1,s}
-
\frac{\sigma\theta}{\sqrt n}\,\bD_r \bX \widetilde{\bQ}_{r-2,s},
\]
with the convention \(\widetilde{\bQ}_{0,s}=\widetilde{\bQ}_{-1,s}=\boldzero\).
By the induction hypothesis,
\[
\bQ_{r-1,s}\in \frac1n\mathcal H_{L_t},
\qquad
\widetilde{\bQ}_{r-1,s},\widetilde{\bQ}_{r-2,s}\in \frac1{\sqrt n}\mathcal G_{L_t}.
\]
Since \(\bX\mathcal G_{L_t}\subset \mathcal H_{L_t+1}\) and
\(\bD_r \mathcal H_{L_t+1}\subset \mathcal H_{L_t+1+d_t}\), each term on the
right-hand side belongs to \(n^{-1}\mathcal H_{L_t'}\) for a finite degree
\(L_t'\) depending only on \(t\). Enlarging \(L_t\) to dominate this degree
gives
\[
\bQ_{r,s}\in \frac1n\mathcal H_{L_t}.
\]
For the \(\widetilde{\bQ}\)-blocks, if \(s<r\), then the forcing term in the
\(s\)-th block column vanishes and we obtain
\[
\widetilde{\bQ}_{r,s}
=
c_{\lambda,n} \widetilde{\bQ}_{r-1,s}
-
\frac{c_{\lambda,n}\tau}{\sqrt n}\,\bX^\top \bQ_{r,s}.
\]
Here the first term belongs to \(n^{-1/2}\mathcal G_{L_t}\) by induction, while
the second belongs to \(n^{-1/2}\mathcal G_{L_t+1}\) because
\(n^{-1}\bX^\top \mathcal H_{L_t}\subset \mathcal G_{L_t+1}\). With the same
finite-degree convention,
\[
\widetilde{\bQ}_{r,s}\in \frac1{\sqrt n}\mathcal G_{L_t}
\qquad (s<r).
\]
For the diagonal block \(s=r\), the lower-triangular structure gives
\(\widetilde{\bQ}_{r-1,r}=\boldzero\), so
\[
\widetilde{\bQ}_{r,r}
=
-\frac{c_{\lambda,n}\tau}{\sqrt n}\bI_p
\in
\frac1{\sqrt n}\mathcal G_0.
\]
Thus every block in row \(r\) remains in the same finite-dimensional algebra,
which gives the required blockwise induction.

Since \(w_{t+1,s}=\trace(\bSigma_n\widetilde{\bQ}_{t,s})\) by
\eqref{eq:W-entrywise},
the previous representation implies that on \(\{\vartheta_n^{(t)}\in\calU_t\}\),
\[
\frac{w_{t+1,s}}{\sqrt n}
=
\omega_{t+1,s}(\vartheta_n^{(t)}),
\qquad s=1,\ldots,t,
\]
for deterministic continuous functions \(\omega_{t+1,s}\) on \(\calU_t\). The
normalization by \(\sqrt n\) is the trace scale: since
\(\widetilde{\bQ}_{t,s}\in n^{-1/2}\mathcal G_{L_t}\) and \(p\asymp n\),
tracing over \(p\) coordinates makes \(w_{t+1,s}\) of order \(\sqrt n\). The
needed trace reductions are admissible primitive scalars of the form
\[
\frac1n\trace(\bSigma_n W(\bS_n,\bSigma_n))
\quad\text{or}\quad
\frac1n\trace(\bM_n W_1(\bS_n,\bSigma_n)\bR_n
W_2(\bS_n,\bSigma_n)).
\]
Here \(\bM_n\) is either \(\bI_p\) or \(\bSigma_n\), depending on whether the
trace comes from a \(\cbK\)-type or \(\bW\)-type quantity.

It remains to express the corrected residual in the same finite-dimensional
form. Under squared loss,
\[
\widetilde{\calR}_t
=
\frac1n\left\|\by-\bX\hbb^t+\bF\bW^\top \be_{t+1}\right\|_2^2.
\]
By Lemma~\ref{lem:universality-krylov},
\[
\frac1{\sqrt n}\bigl(\by-\bX\hbb^t\bigr)=\bv^t=q_t^{(n)}(\bA_n)\tilde\by.
\]
Moreover,
\[
\frac1{\sqrt n}\bF\bW^\top \be_{t+1}
=
\sum_{s=1}^t \frac{w_{t+1,s}}{\sqrt n}\,\bu^s
=
\sum_{s=1}^t \omega_{t+1,s}(\vartheta_n^{(t)})\,r_s^{(n)}(\bA_n)\tilde\by.
\]
Hence, on \(\{\vartheta_n^{(t)}\in\calU_t\}\),
\[
\frac1{\sqrt n}\bigl(\by-\bX\hbb^t+\bF\bW^\top \be_{t+1}\bigr)
=
\rho_t(\vartheta_n^{(t)};\bA_n)\tilde\by,
\]
where \(\rho_t(\vartheta;\cdot)\) is a polynomial of degree at most \(L_t\)
whose coefficients are continuous functions of \(\vartheta\in\calU_t\). Indeed,
\(\rho_t\) is obtained by combining the polynomial \(q_t^{(n)}\) from
the Krylov representation with the finitely many random scalar correction
coefficients \(\omega_{t+1,s}(\vartheta_n^{(t)})\) obtained above. Therefore
\[
\widetilde{\calR}_t
=
\tilde\by^\top \bigl(\rho_t(\vartheta_n^{(t)};\bA_n)\bigr)^2 \tilde\by.
\]
Expanding the square shows that \(\widetilde{\calR}_t\) is a deterministic
continuous function of the finitely many moments listed in
\(\vartheta_n^{(t)}\). This proves the existence of \(\widetilde\Phi_t\).

Finally, by Lemma~\ref{lem:universality-primitive-moments},
\(\vartheta_n^{(t)}-\vartheta_{n,{\rm G}}^{(t)}\limp0\), and both
vectors converge to the same deterministic limit point. Since
\(\widetilde\Phi_t\) is continuous on a neighborhood of that limit point,
\[
\widetilde\Phi_t(\vartheta_n^{(t)})-\widetilde\Phi_t(\vartheta_{n,{\rm G}}^{(t)})
\limp 0.
\]
Because the identities
\[
\widetilde{\calR}_t^{\rm Z}=\widetilde\Phi_t(\vartheta_n^{(t)}),
\qquad
\widetilde{\calR}_t^{\rm G}
=
\widetilde\Phi_t(\vartheta_{n,{\rm G}}^{(t)})
\]
hold with probability tending to one, we conclude that
\[
\widetilde{\calR}_t^{\rm Z}-\widetilde{\calR}_t^{\rm G}\limp 0.
\]
\end{proof}

\begin{corollary}
\label{cor:universality-rt-tilde-sqrt-ridge}
Under the assumptions of Lemma~\ref{lem:universality-true-risk-sqrt-ridge} and
Lemma~\ref{lem:universality-rt-tilde-sqrt-ridge}, for every fixed
\(t\in\{0,\ldots,T-1\}\),
\[
\widetilde{\calR}_t^{\rm Z}-\calR_t^{\rm Z}\limp 0.
\]
This is the covariance-dependent consistency conclusion in the non-Gaussian
model of Assumption~\ref{assu:primitive-comparison}.
\end{corollary}

\begin{proof}[Proof of Corollary~\ref{cor:universality-rt-tilde-sqrt-ridge}]
\label{proof:cor:universality-rt-tilde-sqrt-ridge}
Let \(\calR_t^{\rm G}\) and \(\widetilde{\calR}_t^{\rm G}\) denote the true risk
and the covariance-dependent estimator under the matched Gaussian design. Then
\[
\bigl|\widetilde{\calR}_t^{\rm Z}-\calR_t^{\rm Z}\bigr|
\le
\bigl|\widetilde{\calR}_t^{\rm Z}-\widetilde{\calR}_t^{\rm G}\bigr|
+
\bigl|\widetilde{\calR}_t^{\rm G}-\calR_t^{\rm G}\bigr|
+
\bigl|\calR_t^{\rm G}-\calR_t^{\rm Z}\bigr|.
\]
The first term is \(o_{\P}(1)\) by
Lemma~\ref{lem:universality-rt-tilde-sqrt-ridge}; the third term is
\(o_{\P}(1)\) by Lemma~\ref{lem:universality-true-risk-sqrt-ridge}; and the
middle term is \(o_{\P}(1)\) by Theorem~\ref{thm:rt-tilde}, applied to the
matched Gaussian design. The claim follows.
\end{proof}

\subsection{Covariance-free estimator component}

In this subsection we use the covariance-free estimator from
Theorem~\ref{thm:rt-hat}, specialized to the signal-error matrix
\(\cbF=\bX\bH/\sqrt n\). Let
\(\cbK\), \(\cbA\), \(\hbW\), and \(\widehat{\calR}_t^{\rm Z}\) denote the
resulting correction matrices and estimator under the non-Gaussian design, and
let \(\widehat{\calR}_t^{\rm G}\) denote the matched-Gaussian counterpart.

\begin{lemma}[Covariance-free universality]
\label{lem:universality-rhat-sqrt-ridge}
Fix \(t\in\{0,\ldots,T-1\}\). Under the assumptions of
Lemma~\ref{lem:universality-rt-tilde-sqrt-ridge}, the covariance-free
estimator satisfies
\[
\widehat{\calR}_t^{\rm Z}-\widehat{\calR}_t^{\rm G}\limp 0.
\]
\end{lemma}

\begin{proof}[Proof of Lemma~\ref{lem:universality-rhat-sqrt-ridge}]
\label{proof:lem:universality-rhat-sqrt-ridge}
We enlarge \(\vartheta_n^{(t)}\), if needed, by finitely many admissible
primitive scalars generated below; Lemma~\ref{lem:universality-primitive-moments}
still applies to this finite vector. Let \(\widehat{\calU}_t\) be an open
neighborhood of its common deterministic limit where the no-boundary formulas
from Lemma~\ref{lem:universality-rt-tilde-sqrt-ridge} hold. We prove the
analogue of the \(Q\)-closure argument for
\(\bP=\calM^{-1}\bJ_{\bPhi^v}\), whose blocks define \(\cbK\), using an
\(n\times n\) residual-space class and a \(p\times n\)
coefficient-residual class.

Define
\[
\mathcal N_d
:=
\left\{
\sum_{\ell=0}^d a_\ell \bA_n^\ell
+
\sum_{a,b=0}^d c_{ab}\,\bA_n^a \tilde\by \tilde\by^\top \bA_n^b
\right\}\subset\R^{n\times n},
\]
\[
\mathcal T_d
:=
\left\{
\sum_{\ell=0}^d \bar a_\ell \bX^\top \bA_n^\ell
+
\sum_{a,b=0}^d \bar c_{ab}\,\bS_n^a \bX^\top \tilde\by \tilde\by^\top \bA_n^b
\right\}\subset\R^{p\times n},
\]
where all coefficients are deterministic continuous functions of
\(\vartheta_n^{(t)}\) on \(\widehat{\calU}_t\), including the dependence on
\(c_{\lambda,n}\) through the deterministic coordinate \(n^{-1/2}\).

The following closure relations are the reason for these definitions:
\[
\bD_r \mathcal N_d \subset \mathcal N_{d+d_t},
\qquad
\bX^\top \mathcal N_d \subset \mathcal T_d,
\qquad
\frac1n \bX \mathcal T_d \subset \mathcal N_{d+1}.
\]
The inclusion \(\bD_r \mathcal N_d \subset \mathcal N_{d+d_t}\) follows from
the representation of \(\bD_r\) obtained in the covariance-dependent component
above. The inclusion \(\bX^\top \mathcal N_d \subset \mathcal T_d\) is built
into the definition of \(\mathcal T_d\). Finally,
\(\frac1n\bX\mathcal T_d \subset \mathcal N_{d+1}\) follows from
\[
\frac1n \bX \bX^\top \bA_n^\ell=\bA_n^{\ell+1},
\qquad
\frac1n \bX \bS_n^a \bX^\top \tilde\by \tilde\by^\top \bA_n^b
=
\bA_n^{a+1}\tilde\by \tilde\by^\top \bA_n^b.
\]

We claim that, after enlarging \(L_t\) if necessary,
\[
\bP_{r,s}\in \frac1{\sqrt n}\mathcal N_{L_t},
\qquad
\widetilde{\bP}_{r,s}\in \frac1n \mathcal T_{L_t},
\qquad 1\le s\le r\le t.
\]
For \(r=1\),
\[
\bP_{1,1}
=
\frac1{\sqrt n}\Phi^v_{1,0}
=
-\frac{\sigma}{\sqrt n}\bD_1
\in
\frac1{\sqrt n}\mathcal N_{L_t},
\]
and
\[
\widetilde{\bP}_{1,1}
=
\frac1{\sqrt n}\Psi^\eta_{1,1}\bX^\top \bP_{1,1}
=
-\frac{c_{\lambda,n}\tau}{\sqrt n}\bX^\top \bP_{1,1}
\in
\frac1n \mathcal T_{L_t}.
\]
Now assume the claim holds up to time \(r-1\). Extracting block columns from
\eqref{eq:Pt-recursion-explicit} and using the specialized Jacobian structure
from Lemma~\ref{lem:universality-rt-tilde-sqrt-ridge}, we obtain
\[
\bP_{r,s}
=
\bD_r \bP_{r-1,s}
+
\frac{\sigma(1+\theta)}{\sqrt n}\,\bD_r \bX \widetilde{\bP}_{r-1,s}
-
\frac{\sigma\theta}{\sqrt n}\,\bD_r \bX \widetilde{\bP}_{r-2,s}
+
\frac1{\sqrt n}\Phi^v_{r,s-1},
\]
where \(\Phi^v_{r,s-1}\) is either zero or a constant multiple of \(\bD_r\). By
the induction hypothesis and the closure relations above, each term on the
right-hand side belongs to \(n^{-1/2}\mathcal N_{L_t'}\) for a finite degree
\(L_t'\) depending only on \(t\). Enlarging \(L_t\) to dominate this degree gives
\[
\bP_{r,s}\in \frac1{\sqrt n}\mathcal N_{L_t}.
\]
Similarly, \eqref{eq:Ptilde-recursion-explicit} gives
\[
\widetilde{\bP}_{r,s}
=
c_{\lambda,n} \widetilde{\bP}_{r-1,s}
-
\frac{c_{\lambda,n}\tau}{\sqrt n}\,\bX^\top \bP_{r,s},
\]
so the induction hypothesis and \(\bX^\top\mathcal N_{L_t}\subset\mathcal T_{L_t}\)
imply
\[
\widetilde{\bP}_{r,s}\in \frac1n \mathcal T_{L_t}.
\]
The induction proves the asserted closure for \(\bP\) and \(\widetilde{\bP}\).

We next identify \(\cbK\), \(\cbA\), and \(\hbW\) as continuous functions of
\(\vartheta_n^{(t)}\). The signal-error row-recursion formulas
\eqref{eq:K-row-recursion-offdiag}--\eqref{eq:A-row-recursion} reduce to
\[
\begin{aligned}
\cbK_{t+1,s}
&=
\frac1{\sqrt n}\trace(\bX\widetilde{\bP}_{t,s}),
&
\cbK_{t+1,t+1}&=-\sqrt n,\\
\cbA_{t+1,s}
&=
-\frac1{\sqrt n}\trace(\bX^\top \bX \widetilde{\bQ}_{t,s}).
\end{aligned}
\]
The relevant scales are \(\cbK/\sqrt n\), \(\cbA/n\), and hence
\(\hbW/\sqrt n\), since \(\hbW=\cbK^{-1}\cbA\).

By the \(P\)-closure just proved,
\[
\widetilde{\bP}_{t,s}\in \frac1n \mathcal T_{L_t},
\]
so
\[
\frac{\cbK_{t+1,s}}{\sqrt n}
=
\frac1n\trace(\bX\widetilde{\bP}_{t,s})
\]
is a deterministic continuous function of \(\vartheta_n^{(t)}\). Indeed, each
term reduces either to \(n^{-1}\trace(\bS_n^\ell)\) or to
\(\tilde\by^\top \bA_n^\ell\tilde\by\), possibly multiplied by the
deterministic scalars \(n^{-1/2}\) or \(n^{-1}\); all of these are included
among the admissible primitive scalars in \(\vartheta_n^{(t)}\).

By the \(Q\)-closure argument in Lemma~\ref{lem:universality-rt-tilde-sqrt-ridge},
\[
\widetilde{\bQ}_{t,s}\in \frac1{\sqrt n}\mathcal G_{L_t}.
\]
Therefore
\[
\frac{\cbA_{t+1,s}}{n}
=
-\frac1n \trace\!\left(\bS_n \cdot \sqrt n\,\widetilde{\bQ}_{t,s}\right)
\]
is also a deterministic continuous function of \(\vartheta_n^{(t)}\). The
terms are normalized traces of finite products in \(\bS_n\) and \(\bR_n\), hence
are admissible primitive scalars.

Thus the scaled matrices
\[
\frac{\cbK}{\sqrt n},
\qquad
\frac{\cbA}{n}
\]
are deterministic continuous functions of \(\vartheta_n^{(t)}\). Since
\(\cbK/\sqrt n\) is lower triangular with diagonal entries all
equal to \(-1\), it is invertible for every realization. Hence
\[
\frac{\hbW}{\sqrt n}
=
\left(\frac{\cbK}{\sqrt n}\right)^{-1}
\left(\frac{\cbA}{n}\right)
\]
is again a deterministic continuous function of \(\vartheta_n^{(t)}\).

It remains to express the covariance-free corrected residual in Krylov form.
Under squared loss,
\[
\widehat{\calR}_t
=
\frac1n\left\|\by-\bX\hbb^t+\bF\hbW^\top \be_{t+1}\right\|_2^2.
\]
By Lemma~\ref{lem:universality-krylov},
\[
\frac1{\sqrt n}\bigl(\by-\bX\hbb^t\bigr)=\bv^t=q_t^{(n)}(\bA_n)\tilde\by,
\]
and
\[
\frac1{\sqrt n}\bF\hbW^\top \be_{t+1}
=
\sum_{s=1}^t \frac{\hat w_{t+1,s}}{\sqrt n}\,\bu^s
=
\sum_{s=1}^t \widehat\omega_{t+1,s}(\vartheta_n^{(t)})\,
r_s^{(n)}(\bA_n)\tilde\by,
\]
where \(\widehat\omega_{t+1,s}\) are deterministic continuous scalar maps
obtained from \(\hbW/\sqrt n\). Therefore, on
\(\{\vartheta_n^{(t)}\in\widehat{\calU}_t\}\),
\[
\frac1{\sqrt n}\bigl(\by-\bX\hbb^t+\bF\hbW^\top \be_{t+1}\bigr)
=
\widehat\rho_t(\vartheta_n^{(t)};\bA_n)\tilde\by,
\]
where \(\widehat\rho_t(\vartheta;\cdot)\) is a polynomial in its
matrix argument whose coefficients depend continuously on
\(\vartheta\in\widehat{\calU}_t\). Consequently,
\[
\widehat{\calR}_t
=
\tilde\by^\top
\bigl(\widehat\rho_t(\vartheta_n^{(t)};\bA_n)\bigr)^2
\tilde\by.
\]
Expanding the square shows that
\(\widehat{\calR}_t\) is a deterministic continuous function of the
primitive vector \(\vartheta_n^{(t)}\). This proves the existence of
\(\widehat\Phi_t\).

By Lemma~\ref{lem:universality-primitive-moments},
\(\vartheta_n^{(t)}-\vartheta_{n,{\rm G}}^{(t)}\limp0\), and both
vectors converge to the same deterministic limit point. By continuity of
\(\widehat\Phi_t\),
\[
\widehat\Phi_t(\vartheta_n^{(t)})
-
\widehat\Phi_t(\vartheta_{n,{\rm G}}^{(t)})
\limp 0.
\]
Since the identities
\[
\widehat{\calR}_t^{\rm Z}
=
\widehat\Phi_t(\vartheta_n^{(t)}),
\qquad
\widehat{\calR}_t^{\rm G}
=
\widehat\Phi_t(\vartheta_{n,{\rm G}}^{(t)})
\]
hold with probability tending to one, we conclude that
\[
\widehat{\calR}_t^{\rm Z}-\widehat{\calR}_t^{\rm G}\limp 0.
\]
\end{proof}

\begin{corollary}
\label{cor:universality-rhat-transfer}
Under the assumptions of
Lemma~\ref{lem:universality-rhat-sqrt-ridge},
\[
\widehat{\calR}_t^{\rm Z}-\calR_t^{\rm Z}\limp 0.
\]
This is the covariance-free consistency conclusion in the non-Gaussian model of
Assumption~\ref{assu:primitive-comparison}.
\end{corollary}

\begin{proof}[Proof of Corollary~\ref{cor:universality-rhat-transfer}]
\label{proof:cor:universality-rhat-transfer}
By the triangle inequality,
\[
\bigl|\widehat{\calR}_t^{\rm Z}-\calR_t^{\rm Z}\bigr|
\le
\bigl|\widehat{\calR}_t^{\rm Z}-\widehat{\calR}_t^{\rm G}\bigr|
+
\bigl|\widehat{\calR}_t^{\rm G}-\calR_t^{\rm G}\bigr|
+
\bigl|\calR_t^{\rm G}-\calR_t^{\rm Z}\bigr|.
\]
The first term is \(o_{\P}(1)\) by
Lemma~\ref{lem:universality-rhat-sqrt-ridge}; the second is \(o_{\P}(1)\)
by Theorem~\ref{thm:rt-hat} applied to the matched Gaussian square-root ridge
model with the same signal-error matrix construction. The assumptions of that
theorem hold here because the square-root dual map is bounded and Lipschitz,
the ridge proximal map is Lipschitz, and Lemma~\ref{lem:moment-H-F} gives the
required moment bounds for \(\cbF\). The third term is \(o_{\P}(1)\) by
Lemma~\ref{lem:universality-true-risk-sqrt-ridge}.
\end{proof}

\subsection{Completion of Theorem~\ref{thm:main-universality-sqrt-ridge}}

Assumption~\ref{assu:primitive-comparison} restates the square-root ridge
universality setting from
Assumption~\ref{assu:sqrt-ridge-universality}. Lemma~\ref{lem:universality-true-risk-sqrt-ridge}
therefore proves the first assertion of
Theorem~\ref{thm:main-universality-sqrt-ridge}, namely
\(\calR_t^{\rm Z}-\calR_t^{\rm G}\limp0\).

If Assumption~\ref{assu:sqrt-ridge-no-boundary} also holds, then its
deterministic boundary limit \(\mu_s\) is the quantity
\(\mathfrak m_\Sigma(a_s^2)\) used above. Hence
Lemma~\ref{lem:universality-rt-tilde-sqrt-ridge} proves
\(\widetilde{\calR}_t^{\rm Z}-\widetilde{\calR}_t^{\rm G}\limp0\), and
Lemma~\ref{lem:universality-rhat-sqrt-ridge} proves
\(\widehat{\calR}_t^{\rm Z}-\widehat{\calR}_t^{\rm G}\limp0\). These are the
two estimator universality assertions in the theorem.

It remains only to justify the two consistency conclusions after the word
``Consequently'' in Theorem~\ref{thm:main-universality-sqrt-ridge}. These are
exactly Corollaries~\ref{cor:universality-rt-tilde-sqrt-ridge}
and~\ref{cor:universality-rhat-transfer}, which combine the three universality
lemmas above with the Gaussian consistency theorems
\ref{thm:rt-tilde} and~\ref{thm:rt-hat}. This completes the proof of
Theorem~\ref{thm:main-universality-sqrt-ridge}.

\begin{remark}
The key algebraic point is that, at fixed \(T\), the iterate path, dual
Jacobians, auxiliary matrices, and correction weights close on a
finite-dimensional algebra generated by \(\bA_n\), \(\bS_n\), \(\bSigma_n\),
and \(\bX^\top \tilde\by \tilde\by^\top \bX/n\). Extending the universality
argument beyond ridge or squared test loss remains open, as does the
correction-based boundary case \(\mathfrak m_\Sigma(a_s^2)=1\).
\end{remark}

\section{Computational details}
\label{sec:computation-tricks}
This section is logically independent of the proof of the main theorems and
serves mainly as an implementation guide.
Directly forming the matrices in \eqref{eq:WKA} is computationally expensive because they are defined through the inverse of the large block matrix \(\calM\). In this section, we show how to compute the \(T\times T\) matrices \(\bW\), \(\cbK\), and \(\cbA\) efficiently. The key idea is to first compute the auxiliary matrices \(\bP=\calM^{-1}\bJ_{\bPhi^v}\) and \(\bQ=\calM^{-1}\bJ_{\bPsi^\eta}\) recursively, and then recover \(\bW\), \(\cbK\), and \(\cbA\) row by row from their block entries. We also explain how Hutchinson's method can be used to approximate the trace terms when \(n\) or \(p\) is large.

\subsection{Formal correction definitions}
\label{sec:formal-correction-definitions}

We first give the full matrix definitions used in
Theorems~\ref{thm:rt-tilde} and~\ref{thm:rt-hat}. The population correction
matrix is \(\bW\), and the covariance-free correction is
\(\hbW=\cbK^{-1}\cbA\).

The derivative bookkeeping behind the proposed estimator uses the
matrices
\[
\begin{aligned}
\cbF
\defas\frac{\bX\bH}{\sqrt n}
&=
\left[
\frac{\bX(\hbb^0-\bb^*)}{\sqrt n},
\ldots,
\frac{\bX(\hbb^{T-1}-\bb^*)}{\sqrt n}
\right] \in \R^{n\times T},\\
\bH
&=[\hbb^0-\bb^*,\ldots,\hbb^{T-1}-\bb^*]
\in \R^{p\times T}
.
\end{aligned}
\]

\begin{definition}
We define $\bW,\cbK,\cbA\in\R^{T\times T}$ as follows. Standard basis vectors
are interpreted in the dimension dictated by the surrounding matrix product:
\(\be_i\in\R^n\), \(\be_j\in\R^p\), and \(\be_t\in\R^T\).
\begin{equation}
\label{eq:WKA}
\begin{aligned}
	\bW
	&=
	\sum_{j=1}^p
	(\bI_T\otimes \be_j^\top)
	(\bL\otimes \bSigma)
	\bigl(\bI_T\otimes [\boldzero_{p\times n},\bI_p]\bigr)
	\calM^{-1}\bJ_{\bPsi^\eta}
	(\bI_T\otimes \be_j),\\
	\cbK
	&=
	\frac{1}{\sqrt n}\sum_{i=1}^n
	(\bI_T\otimes \be_i^\top)\calM_H(\bI_T\otimes \be_i),\\
	\cbA
	&=
	-\frac{1}{\sqrt n}\sum_{j=1}^p
	(\bI_T\otimes \be_j^\top\bX^\top)\calM_F(\bI_T\otimes \be_j),
\end{aligned}
\end{equation}
where
\(
\bL=\sum_{t=2}^T \be_t\be_{t-1}^\top
\)
is the \(T\times T\) shift matrix with ones on the subdiagonal and zeros
elsewhere.
\end{definition}

The remaining quantities in \eqref{eq:WKA} are as follows. We first stack the
Jacobian blocks from \eqref{eq:DJ-general} across time. The matrices
\(\bPhi^u,\bPhi^v\in\R^{Tn\times Tn}\) and
\(\bPsi^b,\bPsi^\eta\in\R^{Tp\times Tp}\) are block lower-triangular matrices.
Their \((t,s)\) blocks, for \(1\le t,s\le T\), are
\[
(\bPhi^u)_{t,s}
=
\begin{cases}
\Phi^u_{t,s}, & s<t,\\
\boldzero_{n\times n}, & s\ge t,
\end{cases}
\qquad
(\bPhi^v)_{t,s}
=
\begin{cases}
\Phi^v_{t,s-1}, & s\le t,\\
\boldzero_{n\times n}, & s>t,
\end{cases}
\]
and
\[
(\bPsi^b)_{t,s}
=
\begin{cases}
\Psi^b_{t,s}, & s<t,\\
\boldzero_{p\times p}, & s\ge t,
\end{cases}
\qquad
(\bPsi^\eta)_{t,s}
=
\begin{cases}
\Psi^\eta_{t,s}, & s\le t,\\
\boldzero_{p\times p}, & s>t.
\end{cases}
\]
Thus the columns of \(\bPhi^v\) correspond to the residual sequence
\((\bv^0,\ldots,\bv^{T-1})\), while the columns of \(\bPsi^\eta\) correspond to
\((\bfeta^1,\ldots,\bfeta^T)\). The initial quantities \(\bu^0,\hbb^0\), and
\(\bfeta^0\) are fixed under our initialization and therefore do not appear as
unknowns in this stacked linear system.

Next define
\begin{align*}
\calM
&=
\bI_{T(n+p)}
+
\begin{bmatrix}
\bB_1 & \bB_2
\end{bmatrix}
\begin{bmatrix}
-\bPhi^u & \frac{1}{\sqrt n}\,\bPhi^v(\bL\otimes \bX)\\[2mm]
-\frac{1}{\sqrt n}\,\bPsi^\eta(\bI_T\otimes \bX^\top) & -\bPsi^b
\end{bmatrix}
\begin{bmatrix}
\bB_1^\top\\[6pt]
\bB_2^\top
\end{bmatrix},
\end{align*}
where
\begin{align*}
\bB_1
&:=
\bI_T\otimes
\begin{bmatrix}
\bI_n\\
\bzero_{p\times n}
\end{bmatrix}
\in\R^{T(n+p)\times Tn},
\qquad
\bB_2
:=
\bI_T\otimes
\begin{bmatrix}
\bzero_{n\times p}\\
\bI_p
\end{bmatrix}
\in\R^{T(n+p)\times Tp}.
\end{align*}
We also define
\begin{align*}
\bJ_{\bPhi^v}
=
\frac{1}{\sqrt n}\bB_1\bPhi^v,
\qquad
\bJ_{\bPsi^\eta}
=
\frac{1}{\sqrt n}\bB_2\bPsi^\eta,
\end{align*}
and
\begin{equation}
\begin{aligned}
\calM_H
&=
-\bI_{nT}
+
(\bL\otimes \bX)
\bigl(\bI_T\otimes[\boldzero_{p\times n},\bI_p]\bigr)
\calM^{-1}\bJ_{\bPhi^v},\\
\calM_F
&=
(\bL\otimes \bX)
\bigl(\bI_T\otimes[\boldzero_{p\times n},\bI_p]\bigr)
\calM^{-1}\bJ_{\bPsi^\eta}.
\end{aligned}
\end{equation}

\subsection{Recursive computation of the auxiliary matrices}
\label{sec:WAK-recursion}
We begin with recursive formulas for the auxiliary matrices \(\bP\) and \(\bQ\), since these matrices are the basic ingredients from which \(\bW\), \(\cbK\), and \(\cbA\) are assembled.
Define
\[
\bP:=\calM^{-1}\bJ_{\bPhi^v} \in \R^{T(n+p)\times Tn},
\qquad
\bQ:=\calM^{-1}\bJ_{\bPsi^\eta} \in \R^{T(n+p)\times Tp}.
\]
Since \(\calM\) is block lower triangular with invertible diagonal blocks,
\(\calM^{-1}\) is also block lower triangular. Moreover, \(\bJ_{\bPhi^v}\) and
\(\bJ_{\bPsi^\eta}\) are block lower triangular. Consequently, both \(\bP\) and
\(\bQ\) are block lower triangular.

Partition these matrices by block rows as
\[
\bP=
\begin{bmatrix}
\bP_1\\ \widetilde{\bP}_1\\
\bP_2\\ \widetilde{\bP}_2\\
\vdots\\
\bP_T\\ \widetilde{\bP}_T
\end{bmatrix},
\qquad
\bQ=
\begin{bmatrix}
\bQ_1\\ \widetilde{\bQ}_1\\
\bQ_2\\ \widetilde{\bQ}_2\\
\vdots\\
\bQ_T\\ \widetilde{\bQ}_T
\end{bmatrix},
\]
where
\[
\bP_t\in\R^{n\times Tn},\quad
\widetilde{\bP}_t\in\R^{p\times Tn},\quad
\bQ_t\in\R^{n\times Tp},\quad
\widetilde{\bQ}_t\in\R^{p\times Tp}.
\]

We first derive a recursion for \(\bP\) from
\[
\calM\bP=\bJ_{\bPhi^v}.
\]
The $t$-th block row of $\calM$ has the form
\[
\begin{bmatrix}
\begin{bmatrix}
-\Phi^u_{t,1} & \frac{1}{\sqrt n}\Phi^v_{t,1}\bX
\end{bmatrix}
& \cdots &
\begin{bmatrix}
-\Phi^u_{t,t-1} & \frac{1}{\sqrt n}\Phi^v_{t,t-1}\bX
\end{bmatrix}
&
\begin{bmatrix}
\bI_n & \boldzero_{n\times p}
\end{bmatrix}
& \boldzero
\\[6pt]
\begin{bmatrix}
-\frac{1}{\sqrt n}\Psi^\eta_{t,1} \bX^\top & -\Psi^b_{t,1}
\end{bmatrix}
& \cdots &
\begin{bmatrix}
-\frac{1}{\sqrt n}\Psi^\eta_{t,t-1} \bX^\top & -\Psi^b_{t,t-1}
\end{bmatrix}
&
\begin{bmatrix}
-\frac{1}{\sqrt n}\Psi^\eta_{t,t} \bX^\top & \bI_p
\end{bmatrix}
& \boldzero
\end{bmatrix}.
\]
The \(t\)-th block row of \(\bJ_{\bPhi^v}\) is
\[
\frac{1}{\sqrt n}
\begin{bmatrix}
\Phi^v_{t,0} & \Phi^v_{t,1} & \cdots & \Phi^v_{t,t-1} & \boldzero & \cdots & \boldzero\\
\boldzero & \boldzero & \cdots & \boldzero & \boldzero & \cdots & \boldzero
\end{bmatrix}.
\]
Hence, for each \(t=1,\ldots,T\),
\begin{align}
\bP_t
-\sum_{s=1}^{t-1}\Phi^u_{t,s}\bP_s
+\frac{1}{\sqrt n}\sum_{s=1}^{t-1}\Phi^v_{t,s}\bX\widetilde{\bP}_s
&=
\frac{1}{\sqrt n}
\begin{bmatrix}
\Phi^v_{t,0} & \Phi^v_{t,1} & \cdots & \Phi^v_{t,t-1} & \boldzero & \cdots & \boldzero
\end{bmatrix},
\label{eq:Pt-recursion}
\\
\widetilde{\bP}_t
-\sum_{s=1}^{t-1}\Psi^b_{t,s}\widetilde{\bP}_s
-\frac{1}{\sqrt n}\sum_{s=1}^{t}\Psi^\eta_{t,s}\bX^\top\bP_s
&=\boldzero.
\label{eq:Ptilde-recursion}
\end{align}
Equivalently,
\begin{align}
\bP_t
&=
\sum_{s=1}^{t-1}\Phi^u_{t,s}\bP_s
-\frac{1}{\sqrt n}\sum_{s=1}^{t-1}\Phi^v_{t,s}\bX\widetilde{\bP}_s
+\frac{1}{\sqrt n}
\begin{bmatrix}
\Phi^v_{t,0} & \Phi^v_{t,1} & \cdots & \Phi^v_{t,t-1} & \boldzero & \cdots & \boldzero
\end{bmatrix},
\label{eq:Pt-recursion-explicit}
\\
\widetilde{\bP}_t
&=
\sum_{s=1}^{t-1}\Psi^b_{t,s}\widetilde{\bP}_s
+\frac{1}{\sqrt n}\sum_{s=1}^{t}\Psi^\eta_{t,s}\bX^\top\bP_s.
\label{eq:Ptilde-recursion-explicit}
\end{align}

For \(t=1\),
\[
\bP_1
=
\frac{1}{\sqrt n}
\begin{bmatrix}
\Phi^v_{1,0} & \boldzero & \cdots & \boldzero
\end{bmatrix},
\qquad
\widetilde{\bP}_1
=
\frac{1}{\sqrt n}\Psi^\eta_{1,1}\bX^\top\bP_1.
\]
Thus \(\bP_t\) and \(\widetilde{\bP}_t\) can be computed recursively for \(t=2,\ldots,T\).

Similarly, from
\[
\calM\bQ=\bJ_{\bPsi^\eta},
\]
and the fact that the \(t\)-th block row of \(\bJ_{\bPsi^\eta}\) is
\[
\frac{1}{\sqrt n}
\begin{bmatrix}
\boldzero & \cdots & \boldzero\\
\Psi^\eta_{t,1} & \Psi^\eta_{t,2} & \cdots & \Psi^\eta_{t,t} & \boldzero & \cdots & \boldzero
\end{bmatrix},
\]
we obtain
\begin{align}
\bQ_t
-\sum_{s=1}^{t-1}\Phi^u_{t,s}\bQ_s
+\frac{1}{\sqrt n}\sum_{s=1}^{t-1}\Phi^v_{t,s}\bX\widetilde{\bQ}_s
&=\boldzero,
\label{eq:Qt-recursion}
\\
\widetilde{\bQ}_t
-\sum_{s=1}^{t-1}\Psi^b_{t,s}\widetilde{\bQ}_s
-\frac{1}{\sqrt n}\sum_{s=1}^{t}\Psi^\eta_{t,s}\bX^\top\bQ_s
&=
\frac{1}{\sqrt n}
\begin{bmatrix}
\Psi^\eta_{t,1} & \Psi^\eta_{t,2} & \cdots & \Psi^\eta_{t,t} & \boldzero & \cdots & \boldzero
\end{bmatrix}.
\label{eq:Qtilde-recursion}
\end{align}
Equivalently,
\begin{align}
\bQ_t
&=
\sum_{s=1}^{t-1}\Phi^u_{t,s}\bQ_s
-\frac{1}{\sqrt n}\sum_{s=1}^{t-1}\Phi^v_{t,s}\bX\widetilde{\bQ}_s,
\label{eq:Qt-recursion-explicit}
\\
\widetilde{\bQ}_t
&=
\sum_{s=1}^{t-1}\Psi^b_{t,s}\widetilde{\bQ}_s
+\frac{1}{\sqrt n}\sum_{s=1}^{t}\Psi^\eta_{t,s}\bX^\top\bQ_s
+\frac{1}{\sqrt n}
\begin{bmatrix}
\Psi^\eta_{t,1} & \Psi^\eta_{t,2} & \cdots & \Psi^\eta_{t,t} & \boldzero & \cdots & \boldzero
\end{bmatrix}.
\label{eq:Qtilde-recursion-explicit}
\end{align}

For \(t=1\),
\[
\bQ_1=\boldzero,
\qquad
\widetilde{\bQ}_1=
\frac{1}{\sqrt n}
\begin{bmatrix}
\Psi^\eta_{1,1} & \boldzero & \cdots & \boldzero
\end{bmatrix}.
\]
Thus \(\bQ_t\) and \(\widetilde{\bQ}_t\) can also be computed recursively for \(t=2,\ldots,T\).

\subsection{Entrywise formulas for the triangular matrices}

We now express \(\bW\), \(\cbK\), and \(\cbA\) in terms of the recursively computed blocks of \(\bP\) and \(\bQ\). These formulas show that each row of \(\bW\), \(\cbK\), and \(\cbA\) only depends on previously computed blocks, so the three matrices can themselves be assembled recursively.

First, partition \(\bP_t\) and \(\widetilde{\bP}_t\) by column blocks as
\[
\bP_t=
\begin{bmatrix}
\bP_{t,1} & \bP_{t,2} & \cdots & \bP_{t,T}
\end{bmatrix},
\qquad
\widetilde{\bP}_t=
\begin{bmatrix}
\widetilde{\bP}_{t,1} & \widetilde{\bP}_{t,2} & \cdots & \widetilde{\bP}_{t,T}
\end{bmatrix},
\]
where
\[
\bP_{t,s}\in\R^{n\times n},
\qquad
\widetilde{\bP}_{t,s}\in\R^{p\times n},
\qquad 1\le s\le T.
\]
Likewise, partition \(\bQ_t\) and \(\widetilde{\bQ}_t\) by column blocks as
\[
\bQ_t=
\begin{bmatrix}
\bQ_{t,1} & \bQ_{t,2} & \cdots & \bQ_{t,T}
\end{bmatrix},
\qquad
\widetilde{\bQ}_t=
\begin{bmatrix}
\widetilde{\bQ}_{t,1} & \widetilde{\bQ}_{t,2} & \cdots & \widetilde{\bQ}_{t,T}
\end{bmatrix},
\]
where
\[
\bQ_{t,s}\in\R^{n\times p},
\qquad
\widetilde{\bQ}_{t,s}\in\R^{p\times p},
\qquad 1\le s\le T.
\]

Since \(\bP\) and \(\bQ\) are block lower triangular, these column blocks satisfy
\[
\bP_{t,s}=\boldzero,\quad \widetilde{\bP}_{t,s}=\boldzero,\quad
\bQ_{t,s}=\boldzero,\quad \widetilde{\bQ}_{t,s}=\boldzero,
\qquad s>t.
\]
In other words, 
\begin{align*}
\bP_t&=
\begin{bmatrix}\bP_{t,1} & \bP_{t,2} & \cdots & \bP_{t,t} & \boldzero & \cdots & \boldzero\end{bmatrix},
\\
\widetilde{\bP}_t&=
\begin{bmatrix}\widetilde{\bP}_{t,1} & \widetilde{\bP}_{t,2} & \cdots & \widetilde{\bP}_{t,t} & \boldzero & \cdots & \boldzero\end{bmatrix},
\\
\bQ_t&=
\begin{bmatrix}\bQ_{t,1} & \bQ_{t,2} & \cdots & \bQ_{t,t} & \boldzero & \cdots & \boldzero\end{bmatrix},
\\
\widetilde{\bQ}_t&=
\begin{bmatrix}\widetilde{\bQ}_{t,1} & \widetilde{\bQ}_{t,2} & \cdots & \widetilde{\bQ}_{t,t} & \boldzero & \cdots & \boldzero\end{bmatrix}.
\end{align*}

\medskip
\noindent
\textbf{Entrywise form for \(\bW\).} 
By the definition of \(\bW\) in \eqref{eq:WKA},
\[
\bW
=
\sum_{j=1}^p
(\bI_T\otimes \be_j^\top)
(\bL\otimes \bSigma)
(\bI_T\otimes[\boldzero_{p\times n},\bI_p])
\bQ
(\bI_T\otimes \be_j).
\]
Using the block structure of \(\bL=\sum_{t=2}^T \be_t\be_{t-1}^\top\), we obtain
\begin{equation}
\label{eq:W-entrywise}
\bW
=
\begin{bmatrix}
0&\\
\trace(\bSigma\widetilde{\bQ}_{1,1})& 0\\
\vdots & \ddots &\ddots\\
\trace(\bSigma\widetilde{\bQ}_{T-1,1}) & \ldots & \trace(\bSigma\widetilde{\bQ}_{T-1,T-1})&0
\end{bmatrix}.
\end{equation}

\medskip
\noindent
\textbf{Entrywise form for \(\cbK\).} 
By definition,
\begin{align*}
\calM_H
&=
-\bI_{nT}
+
(\bL \otimes \bX)
\bigl(\bI_T \otimes [\boldzero_{p\times n}, \bI_p]\bigr)\calM^{-1}\bJ_{\bPhi^v}.
\end{align*}
Recall that 
\[
\calM^{-1}\bJ_{\bPhi^v}
=
\bP
=
\begin{bmatrix}
\bP_1\\
\widetilde{\bP}_1\\
\vdots\\
\bP_T\\
\widetilde{\bP}_T
\end{bmatrix},
\qquad
\widetilde{\bP}_t
= 
\begin{bmatrix}
\widetilde{\bP}_{t,1} & \cdots & \widetilde{\bP}_{t,t} & \boldzero & \cdots & \boldzero
\end{bmatrix}.
\]
where \(\widetilde{\bP}_{t,s}\in\R^{p\times n}\). 
Therefore,
\[
\bigl(\bI_T \otimes [\boldzero_{p\times n}, \bI_p]\bigr)\calM^{-1}\bJ_{\bPhi^v}
=
\begin{bmatrix}
\widetilde{\bP}_1\\
\vdots\\
\widetilde{\bP}_T
\end{bmatrix},
\]
and thus
\[
\calM_H
=
-\bI_{nT}
+
(\bL \otimes \bX)
\begin{bmatrix}
\widetilde{\bP}_1\\
\vdots\\
\widetilde{\bP}_T
\end{bmatrix}.
\]
Using the structure of \(\bL\), we obtain the block form
\[
\calM_H
=
\begin{bmatrix}
-\bI_n & \boldzero & \boldzero & \cdots & \boldzero\\
\bX\widetilde{\bP}_{1,1} & -\bI_n & \boldzero & \cdots & \boldzero\\
\bX\widetilde{\bP}_{2,1} & \bX\widetilde{\bP}_{2,2} & -\bI_n & \ddots & \vdots\\
\vdots & \vdots & \ddots & \ddots & \boldzero\\
\bX\widetilde{\bP}_{T-1,1} & \bX\widetilde{\bP}_{T-1,2} & \cdots & \bX\widetilde{\bP}_{T-1,T-1} & -\bI_n
\end{bmatrix}.
\]
Therefore, 
\begin{align*}
\cbK
&= \frac{1}{\sqrt{n}} \sum_{i=1}^n (\bI_T\otimes \be_i^\top)\calM_H (\bI_T \otimes \be_i),\\
&= \frac{1}{\sqrt{n}}
\begin{bmatrix}
-n & \boldzero & \boldzero & \cdots & \boldzero\\
\tr(\bX\widetilde{\bP}_{1,1}) & -n & \boldzero & \cdots & \boldzero\\
\tr(\bX\widetilde{\bP}_{2,1}) & \tr(\bX\widetilde{\bP}_{2,2}) & -n & \ddots & \vdots\\
\vdots & \vdots & \ddots & \ddots & \boldzero\\
\tr(\bX\widetilde{\bP}_{T-1,1}) & \tr(\bX\widetilde{\bP}_{T-1,2}) & \cdots & \tr(\bX\widetilde{\bP}_{T-1,T-1}) & -n
\end{bmatrix}.
\end{align*}

\medskip
\noindent
\textbf{Entrywise form for \(\cbA\).} 
Similar to $\cbK$, we now derive the entrywise expression for $\cbA$. 
Since 
\begin{align*}
		\calM^{-1}\bJ_{\bPsi^\eta} 
		= \bQ
		= 
		\begin{bmatrix}
	\bQ_1\\
	\widetilde{\bQ}_1\\
	\vdots\\
	\bQ_T\\
	\widetilde{\bQ}_T
	\end{bmatrix}
\end{align*}
where \(\widetilde{\bQ}_t = [\widetilde{\bQ}_{t,1}, \dots, \widetilde{\bQ}_{t,t}, \bzero, \dots, \bzero]\). We have 
\begin{align*}
\calM_F
&=
(\bL \otimes \bX)
\bigl(\bI_T \otimes [\boldzero_{p\times n}, \bI_p]\bigr)\calM^{-1}\bJ_{\bPsi^\eta}\\
&=
(\bL \otimes \bX)
\begin{bmatrix}
\widetilde{\bQ}_1\\
\vdots\\
\widetilde{\bQ}_T
\end{bmatrix}\\
	&=
	\begin{bmatrix}
	\boldzero & \boldzero & \boldzero & \cdots & \boldzero\\
	\bX\widetilde{\bQ}_{1,1} & \boldzero & \boldzero & \cdots & \boldzero\\
	\bX\widetilde{\bQ}_{2,1} & \bX\widetilde{\bQ}_{2,2} & \boldzero & \cdots & \boldzero\\
	\vdots & \vdots & \ddots & \ddots & \vdots\\
	\bX\widetilde{\bQ}_{T-1,1} & \bX\widetilde{\bQ}_{T-1,2} & \cdots & \bX\widetilde{\bQ}_{T-1,T-1} & \boldzero
	\end{bmatrix}.
\end{align*}
It follows that 
\begin{align*}
\cbA
&=-
\frac{1}{\sqrt n}\sum_{j=1}^p
(\bI_T\otimes \be_j^\top \bX^\top)
\calM_F(\bI_T\otimes \be_j)\\
&=-
	\frac{1}{\sqrt n}
	\begin{bmatrix}
	\boldzero & \boldzero & \boldzero & \cdots & \boldzero\\
	\tr(\bX^\top \bX\widetilde{\bQ}_{1,1}) & \boldzero & \boldzero & \cdots & \boldzero\\
	\tr(\bX^\top \bX\widetilde{\bQ}_{2,1}) & \tr(\bX^\top \bX\widetilde{\bQ}_{2,2}) & \boldzero & \cdots & \boldzero\\
	\vdots & \vdots & \ddots & \ddots & \vdots\\
	\tr(\bX^\top \bX\widetilde{\bQ}_{T-1,1}) & \tr(\bX^\top \bX\widetilde{\bQ}_{T-1,2}) & \cdots & \tr(\bX^\top \bX\widetilde{\bQ}_{T-1,T-1}) & \boldzero
	\end{bmatrix}.
\end{align*}

Combining the formulas above, we obtain the following row-by-row computation rule. For each \(t=1,\ldots,T-1\), once the blocks \(\widetilde{\bP}_{t,s}\) and \(\widetilde{\bQ}_{t,s}\) have been computed for \(1\le s\le t\), the \((t+1)\)-st rows of \(\bW\), \(\cbK\), and \(\cbA\) are given by
\begin{align}
\bW_{t+1,s}
&=
\trace(\bSigma\widetilde{\bQ}_{t,s}),
\qquad 1\le s\le t,
\label{eq:W-row-recursion}
\\
\cbK_{t+1,s}
&=
\frac{1}{\sqrt n}\trace(\bX\widetilde{\bP}_{t,s}),
\qquad 1\le s\le t,
\label{eq:K-row-recursion-offdiag}
\\
\cbK_{t+1,t+1}
&=
-\sqrt n,
\label{eq:K-row-recursion-diag}
\\
\cbA_{t+1,s}
&=
-\frac{1}{\sqrt n}\trace(\bX^\top\bX\widetilde{\bQ}_{t,s}),
\qquad 1\le s\le t.
\label{eq:A-row-recursion}
\end{align}
All remaining entries in the \((t+1)\)-st row are zero by lower triangularity. Therefore, after computing \(\bP_t,\widetilde{\bP}_t,\bQ_t,\widetilde{\bQ}_t\), one can immediately update the next rows of \(\bW\), \(\cbK\), and \(\cbA\).

\subsection{Worked example: the case \(T=3\)}

For convenience, we record the specialization of \eqref{eq:W-entrywise} when
\(T=3\). In this case,
\[
\bW=
\begin{bmatrix}
0 & 0 & 0\\
w_{2,1} & 0 & 0\\
w_{3,1} & w_{3,2} & 0
\end{bmatrix},
\]
where
\begin{align}
w_{2,1}
&=
\trace(\bSigma\widetilde{\bQ}_{1,1}),
\label{eq:W21-T3}
\\
w_{3,1}
&=
\trace(\bSigma\widetilde{\bQ}_{2,1}),
\label{eq:W31-T3}
\\
w_{3,2}
&=
\trace(\bSigma\widetilde{\bQ}_{2,2}).
\label{eq:W32-T3}
\end{align}

Using the recursion in \eqref{eq:Qt-recursion-explicit}--\eqref{eq:Qtilde-recursion-explicit}, these blocks can be written explicitly. Since
\[
\widetilde{\bQ}_{1,1}
=
\frac{1}{\sqrt n}\Psi^\eta_{1,1},
\qquad
\bQ_{2,1}
=
-\frac{1}{\sqrt n}\Phi^v_{2,1}\bX\widetilde{\bQ}_{1,1},
\]
we have
\begin{align}
\widetilde{\bQ}_{2,2}
&=
\frac{1}{\sqrt n}\Psi^\eta_{2,2},
\label{eq:Q22-T3}
\\
\widetilde{\bQ}_{2,1}
&=
\Psi^b_{2,1}\widetilde{\bQ}_{1,1}
+
\frac{1}{\sqrt n}\Psi^\eta_{2,1}
+
\frac{1}{\sqrt n}\Psi^\eta_{2,2}\bX^\top \bQ_{2,1}
\notag\\
&=
\Psi^b_{2,1}\widetilde{\bQ}_{1,1}
+
\frac{1}{\sqrt n}\Psi^\eta_{2,1}
-
\frac{1}{n}\Psi^\eta_{2,2}\bX^\top \Phi^v_{2,1}\bX\widetilde{\bQ}_{1,1}.
\label{eq:Q21-T3}
\end{align}
Therefore,
\begin{align}
w_{2,1}
&=
\frac{1}{\sqrt n}\trace(\bSigma\Psi^\eta_{1,1}),
\label{eq:W21-T3-explicit}
\\
w_{3,2}
&=
\frac{1}{\sqrt n}\trace(\bSigma\Psi^\eta_{2,2}),
\label{eq:W32-T3-explicit}
\\
w_{3,1}
&=
\frac{1}{\sqrt n}\trace(\bSigma\Psi^\eta_{2,1})
+
\frac{1}{\sqrt n}\trace(\bSigma\Psi^b_{2,1}\Psi^\eta_{1,1})
\notag\\
&\quad
-
\frac{1}{n^{3/2}}
\trace\!\bigl(
\bSigma\Psi^\eta_{2,2}\bX^\top \Phi^v_{2,1}\bX\Psi^\eta_{1,1}
\bigr).
\label{eq:W31-T3-explicit}
\end{align}
These formulas are simply the \(T=3\) specialization of the general block
recursion and make clear how the third-row entry \(w_{3,1}\) already contains
an interaction between the first and second updates through
\(\Phi^v_{2,1}\), \(\Psi^b_{2,1}\), and \(\Psi^\eta_{2,2}\).

The corresponding empirical matrix
\[
\hbW:=\cbK^{-1}\cbA
\]
has the same strictly lower-triangular form,
\[
\hbW=
\begin{bmatrix}
0 & 0 & 0\\
\hat w_{2,1} & 0 & 0\\
\hat w_{3,1} & \hat w_{3,2} & 0
\end{bmatrix}.
\]
Using the \(T=3\) forms of \(\cbK\) and \(\cbA\) and solving
\(\cbK\hbW=\cbA\) entry by entry, we obtain
\begin{align}
\hat w_{2,1}
&=
\frac{\trace(\bX^\top \bX\widetilde{\bQ}_{1,1})}{n},
\label{eq:chatW21-T3}
\\
\hat w_{3,2}
&=
\frac{\trace(\bX^\top \bX\widetilde{\bQ}_{2,2})}{n},
\label{eq:chatW32-T3}
\\
\hat w_{3,1}
&=
\frac{
\trace(\bX^\top \bX\widetilde{\bQ}_{2,1})
-
\trace(\bX\widetilde{\bP}_{2,2})\,\hat w_{2,1}
}{n}.
\label{eq:chatW31-T3}
\end{align}
Equivalently, after substituting \eqref{eq:Q22-T3} and \eqref{eq:Q21-T3},
\begin{align}
\hat w_{2,1}
&=
\frac{1}{\sqrt n}
\frac{\trace(\bX^\top \bX\Psi^\eta_{1,1})}{n},
\label{eq:chatW21-T3-explicit}
\\
\hat w_{3,2}
&=
\frac{1}{\sqrt n}
\frac{\trace(\bX^\top \bX\Psi^\eta_{2,2})}{n},
\label{eq:chatW32-T3-explicit}
\\
\hat w_{3,1}
&=
\frac{
\trace\!\Bigl(
\bX^\top \bX
\Bigl[
\Psi^b_{2,1}\widetilde{\bQ}_{1,1}
+
\frac{1}{\sqrt n}\Psi^\eta_{2,1}
-
\frac{1}{n}\Psi^\eta_{2,2}\bX^\top \Phi^v_{2,1}\bX\widetilde{\bQ}_{1,1}
\Bigr]
\Bigr)
-
\trace(\bX\widetilde{\bP}_{2,2})\,\hat w_{2,1}
}{n}.
\label{eq:chatW31-T3-explicit}
\end{align}
Thus the only additional feature in \(\hbW\), compared with \(\bW\), is the
lower-triangular back-substitution through \(\cbK^{-1}\), which introduces the
extra coupling term \(\trace(\bX\widetilde{\bP}_{2,2})\,\hat w_{2,1}\) in
\(\hat w_{3,1}\).

\subsection{Hutchinson's trace approximation}
The formulas in \eqref{eq:W-row-recursion}--\eqref{eq:A-row-recursion} reduce the computation of \(\bW\), \(\cbK\), and \(\cbA\) to a collection of trace evaluations. When \(n\) or \(p\) is large, computing these traces exactly may still be expensive. To accelerate the computation, we use Hutchinson's stochastic trace estimator \cite{hutchinson1990stochastic}.

If \(\bM\in\R^{d\times d}\) and \(\bzeta_1,\ldots,\bzeta_m\in\R^d\) are independent Rademacher vectors with entries in \(\{-1/\sqrt m,1/\sqrt m\}\), then
\[
\trace(\bM)\approx \sum_{r=1}^m \bzeta_r^\top \bM \bzeta_r.
\]
Equivalently, if \(\bZ=[\bzeta_1,\ldots,\bzeta_m]\in\R^{d\times m}\), then
\[
\trace(\bM)\approx \trace(\bZ^\top \bM \bZ).
\]

We apply this approximation separately to the \(n\times n\) traces in \(\cbK\) and the \(p\times p\) traces in \(\bW\) and \(\cbA\). Let
\[
\bxi=[\bxi_1,\ldots,\bxi_m]\in\R^{n\times m},
\qquad
\tbxi=[\tbxi_1,\ldots,\tbxi_m]\in\R^{p\times m},
\]
where the entries of \(\bxi\) and \(\tbxi\) are i.i.d.\ from \(\{-1/\sqrt m,1/\sqrt m\}\). Then, for \(1\le s\le t\le T-1\),
\begin{align}
\bW_{t+1,s}
&\approx
\trace\!\left(\tbxi^\top \bSigma \widetilde{\bQ}_{t,s}\tbxi\right),
\label{eq:W-hutch}
\\
\cbK_{t+1,s}
&\approx
\frac{1}{\sqrt n}\trace\!\left(\bxi^\top \bX\widetilde{\bP}_{t,s}\bxi\right),
\label{eq:K-hutch-offdiag}
\\
\cbA_{t+1,s}
&\approx
-\frac{1}{\sqrt n}\trace\!\left(\tbxi^\top \bX^\top \bX\widetilde{\bQ}_{t,s}\tbxi\right).
\label{eq:A-hutch}
\end{align}
The diagonal entries of \(\cbK\) are inexpensive to compute exactly from \eqref{eq:K-row-recursion-diag}, so in practice we keep
\[
\cbK_{t,t}
=
-\sqrt n
\]
without approximation.

To use \eqref{eq:W-hutch}--\eqref{eq:A-hutch} efficiently, we do not form the blocks \(\bP_{t,s}\), \(\widetilde{\bP}_{t,s}\), \(\bQ_{t,s}\), and \(\widetilde{\bQ}_{t,s}\) explicitly. Instead, we propagate only their products with the probe matrices \(\bxi\) and \(\tbxi\).

Define
\begin{align*}
\bP_t^*
&\defas
\begin{bmatrix}
\bP_{t,1}\bxi & \cdots & \bP_{t,t}\bxi
\end{bmatrix}
\in \R^{n\times tm},\\
\widetilde{\bP}_t^*
&\defas
\begin{bmatrix}
\widetilde{\bP}_{t,1}\bxi & \cdots & \widetilde{\bP}_{t,t}\bxi
\end{bmatrix}
\in \R^{p\times tm},
\end{align*}
and, for \(t\ge 2\),
\begin{align*}
\bQ_t^*
&\defas
\begin{bmatrix}
\bQ_{t,1}\tbxi & \cdots & \bQ_{t,t-1}\tbxi
\end{bmatrix}
\in \R^{n\times (t-1)m},\\
\widetilde{\bQ}_t^*
&\defas
\begin{bmatrix}
\widetilde{\bQ}_{t,1}\tbxi & \cdots & \widetilde{\bQ}_{t,t}\tbxi
\end{bmatrix}
\in \R^{p\times tm}.
\end{align*}
For convenience, set \(\bQ_1^*=\boldzero\).

To handle dimension mismatches in the recursion, we introduce the zero-padding operator
\[
\operatorname{pad}_{a\to b}(\bA)
=
\begin{bmatrix}
\bA & \boldzero
\end{bmatrix},
\]
where the zero block has \((b-a)m\) columns.

Define
\[
\bR_t^\xi
=
\begin{bmatrix}
\Phi^v_{t,0}\bxi & \Phi^v_{t,1}\bxi & \cdots & \Phi^v_{t,t-1}\bxi
\end{bmatrix},
\qquad
\widetilde{\bR}_t^{\tbxi}
=
\begin{bmatrix}
\Psi^\eta_{t,1}\tbxi & \Psi^\eta_{t,2}\tbxi & \cdots & \Psi^\eta_{t,t}\tbxi
\end{bmatrix}.
\]

Using the block recursions derived previously, we obtain
\begin{align}
\bP_t^*
&=
\sum_{r=1}^{t-1}\Phi^u_{t,r}\,\operatorname{pad}_{r\to t}(\bP_r^*)
-\frac{1}{\sqrt n}\sum_{r=1}^{t-1}\Phi^v_{t,r}\bX\,\operatorname{pad}_{r\to t}(\widetilde{\bP}_r^*)
+\frac{1}{\sqrt n}\bR_t^\xi,
\label{eq:Ptstar-recursion}
\\
\widetilde{\bP}_t^*
&=
\sum_{r=1}^{t-1}\Psi^b_{t,r}\,\operatorname{pad}_{r\to t}(\widetilde{\bP}_r^*)
+\frac{1}{\sqrt n}\sum_{r=1}^{t}\Psi^\eta_{t,r}\bX^\top\,\operatorname{pad}_{r\to t}(\bP_r^*),
\label{eq:Ptilde-star-recursion}
\end{align}
and, for \(t\ge 2\),
\begin{align}
\bQ_t^*
&=
\sum_{r=1}^{t-1}\Phi^u_{t,r}\,\operatorname{pad}_{r-1\to t-1}(\bQ_r^*)
-\frac{1}{\sqrt n}\sum_{r=1}^{t-1}\Phi^v_{t,r}\bX\,\operatorname{pad}_{r\to t-1}(\widetilde{\bQ}_r^*),
\label{eq:Qtstar-recursion}
\\
\widetilde{\bQ}_t^*
&=
\sum_{r=1}^{t-1}\Psi^b_{t,r}\,\operatorname{pad}_{r\to t}(\widetilde{\bQ}_r^*)
+\frac{1}{\sqrt n}\sum_{r=1}^{t}\Psi^\eta_{t,r}\bX^\top\,\operatorname{pad}_{r-1\to t}(\bQ_r^*)
+\frac{1}{\sqrt n}\widetilde{\bR}_t^{\tbxi}.
\label{eq:Qtilde-star-recursion}
\end{align}

The initial values are
\[
\bP_1^*=\frac{1}{\sqrt n}\Phi^v_{1,0}\bxi,
\qquad
\widetilde{\bP}_1^*=\frac{1}{n}\Psi^\eta_{1,1}\bX^\top\Phi^v_{1,0}\bxi,
\qquad
\bQ_1^*=\boldzero,
\qquad
\widetilde{\bQ}_1^*=\frac{1}{\sqrt n}\Psi^\eta_{1,1}\tbxi.
\]

Thus, the matrices \(\bP_t^*, \widetilde{\bP}_t^*, \bQ_t^*, \widetilde{\bQ}_t^*\) can be computed recursively without explicitly forming \(\bP\) or \(\bQ\). Once these quantities are available, the Hutchinson approximations of \(\bW\), \(\cbK\), and \(\cbA\) are obtained block by block from \eqref{eq:W-hutch}--\eqref{eq:A-hutch}. More precisely, the \(s\)-th \(m\)-column block of \(\widetilde{\bP}_t^*\) equals \(\widetilde{\bP}_{t,s}\bxi\), and the \(s\)-th \(m\)-column block of \(\widetilde{\bQ}_t^*\) equals \(\widetilde{\bQ}_{t,s}\tbxi\). We use the superscript \(H\) to denote these Hutchinson approximations. Therefore, for \(1\le s\le t\le T-1\),
\begin{align*}
\bW_{t+1,s}^{H}
&=
\trace\!\left(\tbxi^\top\bSigma\,(\widetilde{\bQ}_t^*)_{[s]}\right),\\
\cbK_{t+1,s}^{H}
&=
\frac{1}{\sqrt n}\trace\!\left(\bxi^\top\bX\,(\widetilde{\bP}_t^*)_{[s]}\right),\\
\cbA_{t+1,s}^{H}
&=
-\frac{1}{\sqrt n}\trace\!\left(\tbxi^\top\bX^\top\bX\,(\widetilde{\bQ}_t^*)_{[s]}\right),
\end{align*}
where \((\widetilde{\bP}_t^*)_{[s]}\) and \((\widetilde{\bQ}_t^*)_{[s]}\) denote the \(s\)-th \(m\)-column blocks of \(\widetilde{\bP}_t^*\) and \(\widetilde{\bQ}_t^*\), respectively. In summary, the computation proceeds in two stages: first recursively compute the probe products in \eqref{eq:Ptstar-recursion}--\eqref{eq:Qtilde-star-recursion}, and then recover the Hutchinson approximations of the entries of \(\bW\), \(\cbK\), and \(\cbA\) from the resulting block traces.

\section{Algorithmic and numerical supplements}
\label{sec:algorithmic-supplements}

This final appendix collects algorithmic examples and numerical material that
support, but are not needed for, the main theoretical development.

\subsection{Additional proximal formulas}
\label{sec:additional-proximal-formulas}

This subsection records the non-convex penalties used in the numerical experiments.
They are not needed for the main theoretical statements, but their proximal
maps fit the same update notation whenever the displayed denominators are
bounded away from zero.

The minimax concave penalty (MCP) \cite{zhang10-mc+} with parameters
\(\lambda>0\) and \(\gamma_{\mathrm{mcp}}>1\) is
\[
g(\bb)
=
\sum_{j=1}^p
\begin{cases}
\lambda |b_j|-\dfrac{|b_j|^2}{2\gamma_{\mathrm{mcp}}},
& |b_j|\le \gamma_{\mathrm{mcp}}\lambda, \\[8pt]
\dfrac{\gamma_{\mathrm{mcp}}\lambda^2}{2},
& |b_j|>\gamma_{\mathrm{mcp}}\lambda.
\end{cases}
\]
Its proximal map with a generic proximal parameter \(\delta>0\) is coordinatewise:
\[
\bigl(\prox[\delta g](\bb)\bigr)_j
=
\begin{cases}
0, & |b_j|\le \delta\lambda, \\[6pt]
\dfrac{\operatorname{sign}(b_j)\bigl(|b_j|-\delta\lambda\bigr)}
{1-\tfrac{\delta}{\gamma_{\mathrm{mcp}}}},
& \delta\lambda<|b_j|\le \gamma_{\mathrm{mcp}}\lambda, \\[10pt]
b_j, & |b_j|>\gamma_{\mathrm{mcp}}\lambda.
\end{cases}
\]

The smoothly clipped absolute deviation (SCAD) penalty \cite{fan2001variable}
with parameters \(\lambda>0\) and \(a>2\) is
\[
g(\bb)
=
\sum_{j=1}^p
\begin{cases}
\lambda |b_j|,
& |b_j|\le \lambda, \\[8pt]
\dfrac{-|b_j|^2+2a\lambda |b_j|-\lambda^2}{2(a-1)},
& \lambda<|b_j|\le a\lambda, \\[10pt]
\dfrac{(a+1)\lambda^2}{2},
& |b_j|>a\lambda.
\end{cases}
\]
Its proximal map with a generic proximal parameter \(\delta>0\) is coordinatewise:
\[
\bigl(\prox[\delta g](\bb)\bigr)_j
=
\begin{cases}
0, & |b_j|\le \delta\lambda, \\[6pt]
\operatorname{sign}(b_j)(|b_j|-\delta\lambda),
& \delta\lambda<|b_j|\le \lambda(1+\delta), \\[8pt]
\operatorname{sign}(b_j)\dfrac{(a-1)|b_j|-a\delta\lambda}{a-1-\delta},
& \lambda(1+\delta)<|b_j|\le a\lambda, \\[12pt]
b_j, & |b_j|>a\lambda.
\end{cases}
\]

\subsection{Linearized ADMM as a special case}
\label{sec:linearized-admm}

This subsection gives the reduction referenced after \eqref{eq:general-iteration}.
Introduce the scaled residual variable
\[
\br=\frac{\by-\bX\bb}{\sqrt n}
\]
and the residual loss \(H(\br):=f_{\by}(\by-\sqrt n\,\br)\). Then
\eqref{eq:penalized-regression} is equivalent to
\[
\min_{\bb,\br}\ H(\br)+g(\bb)
\qquad
\text{subject to}\qquad
\br+\frac{\bX\bb}{\sqrt n}=\frac{\by}{\sqrt n}.
\]
Starting from the scaled ADMM form \citep{boyd2011distributed} with penalty
parameter \(\rho>0\) and scaled multiplier \(\bfd^t\), linearizing the
\(\bb\)-subproblem with step size \(\alpha>0\), and taking the standard
initialization \(\bfd^0=\boldzero_n\), the residual variable and scaled
multiplier can be eliminated. With
\[
\bv^t=\frac{\by-\bX\hbb^t}{\sqrt n},
\qquad
\bfeta^t=\frac{\bX^\top\bu^t}{\sqrt n},
\]
the resulting recursion is
\[
\begin{aligned}
\bu^1
&=
\rho\bigl(\prox[H/\rho](\bv^0)-\bv^0\bigr),\\
\bu^t
&=
\bu^{t-1}
+\rho\left[
\prox[H/\rho]\left(
2\bv^{t-1}-\bv^{t-2}-\frac{\bu^{t-1}}{\rho}
\right)
+\bv^{t-2}-2\bv^{t-1}
\right],\\
\hbb^t
&=
\prox[\alpha g]\bigl(\hbb^{t-1}-\alpha\bfeta^t\bigr).
\end{aligned}
\]
The second line applies for \(t\ge2\). The right-hand sides of the first two
lines define the maps \(\bphi^1\) and \(\bphi^t\), evaluated at
\((\bu^0,\ldots,\bu^{t-1},\bv^0,\ldots,\bv^{t-1})\), and the last line defines
\(\bpsi^t\), evaluated at
\((\hbb^0,\ldots,\hbb^{t-1},\bfeta^0,\ldots,\bfeta^t)\). Hence this
linearized ADMM recursion is a special case of \eqref{eq:general-iteration}
whenever these maps satisfy the boundedness and Lipschitz assumptions imposed
in Assumption~\ref{assu:algorithm}. A fixed nonzero initial multiplier can be
absorbed into the first map, so the same inclusion still holds.

\subsection{Additional numerical figures}
\label{sec:additional-numerical-figures}

This subsection reports additional numerical results omitted from the main text.
Table~\ref{tab:additional-figure-guide} summarizes the contents of this
appendix figure subsection. These figures use the same matched test-loss
convention as Figure~\ref{fig:gaussian-risk-comparison}: LAD-based procedures
are evaluated by absolute-value risk, while square-root-loss procedures are
evaluated by squared risk.

\begin{table}[H]
\centering
\small
\begingroup
\setlength{\tabcolsep}{3pt}
\begin{tabular}{@{}p{0.22\textwidth}p{0.36\textwidth}p{0.36\textwidth}@{}}
\hline
\textbf{Figure} & \textbf{Setting} & \textbf{Contents} \\
\hline
Figure~\ref{fig:gaussian-risk-comparison-p2400}
& Risk-estimation comparison; \(p=2400\) (\(n<p\))
& Matched test-loss risk curves for all eight loss--penalty
combinations. \\
\hline
Figures~\ref{fig:universality-square-root-ridge-p2400}--\ref{fig:universality-square-root-scad}
& Square-root loss with ridge, Lasso, MCP, and SCAD
& Gaussian-versus-uniform comparisons by penalty; the ridge case is covered by
Theorem~\ref{thm:main-universality-sqrt-ridge}. \\
\hline
Figures~\ref{fig:universality-lad-ridge}--\ref{fig:universality-lad-scad}
& LAD loss with ridge, Lasso, MCP, and SCAD
& Empirical universality comparisons beyond square-root loss. \\
\hline
\end{tabular}
\endgroup
\caption{Guide to the additional numerical figures.}
\label{tab:additional-figure-guide}
\end{table}

\begin{figure}[p]
\centering
\includegraphics[width=0.98\textwidth]{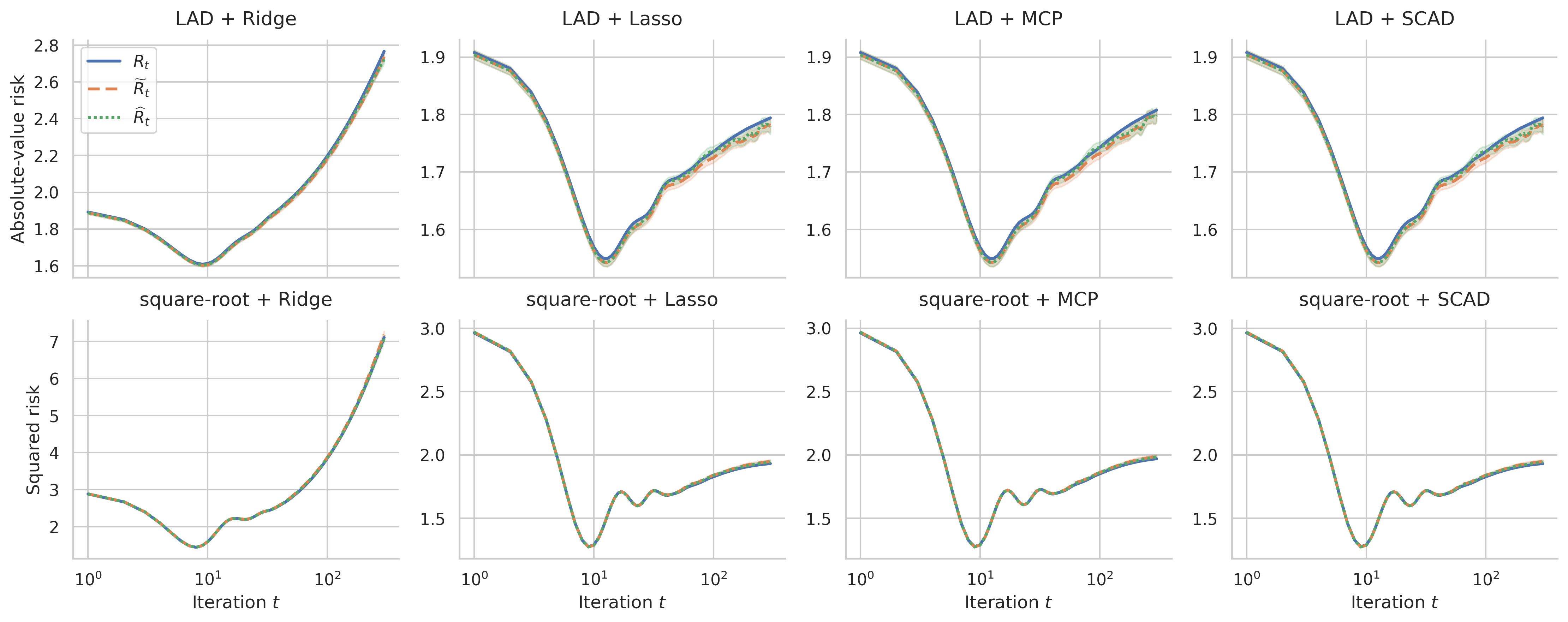}
\caption[Matched test-loss risk estimates for n less than p]{Risk estimates over the iteration path with \((n,p,T)=(2000,2400,300)\), corresponding to the \(n<p\) regime. The curves show \(\calR_t\), \(\widetilde{\calR}_t\), and \(\widehat{\calR}_t\). The LAD rows use absolute-value test loss with Student-\(t_2\) noise; the square-root rows use squared test loss with Gaussian noise. Columns correspond to penalties, and rows correspond to data-fitting loss.}
\label{fig:gaussian-risk-comparison-p2400}
\end{figure}

The remaining figures compare Gaussian and covariance-matched uniform designs
for additional loss--penalty combinations. Only the square-root ridge case is
covered by Theorem~\ref{thm:main-universality-sqrt-ridge}; the other panels are
included as empirical checks of the same phenomenon.

\begin{figure}[p]
\centering
\includegraphics[width=0.96\textwidth]{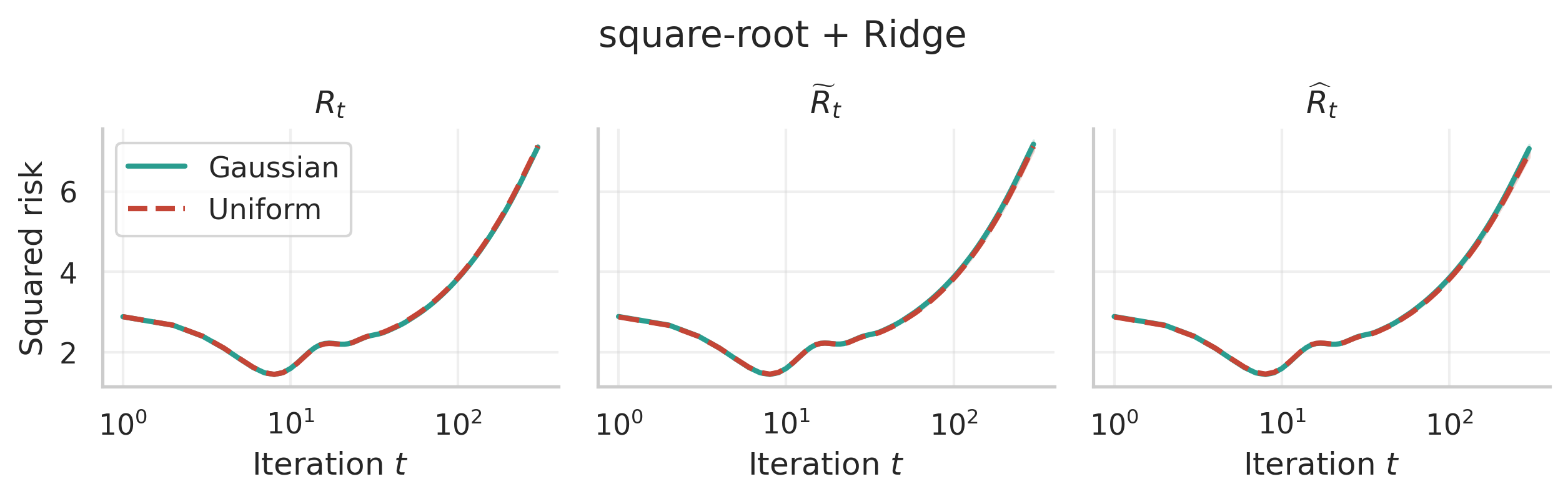}
\caption[Universality for square-root ridge when n is less than p]{Gaussian-versus-uniform comparison for square-root ridge with \((n,p,T)=(2000,2400,300)\), using squared-error risk. The left, middle, and right panels compare \(\calR_t\), \(\widetilde{\calR}_t\), and \(\widehat{\calR}_t\), respectively.}
\label{fig:universality-square-root-ridge-p2400}
\end{figure}

\begin{figure}[p]
\centering
\includegraphics[width=0.90\textwidth]{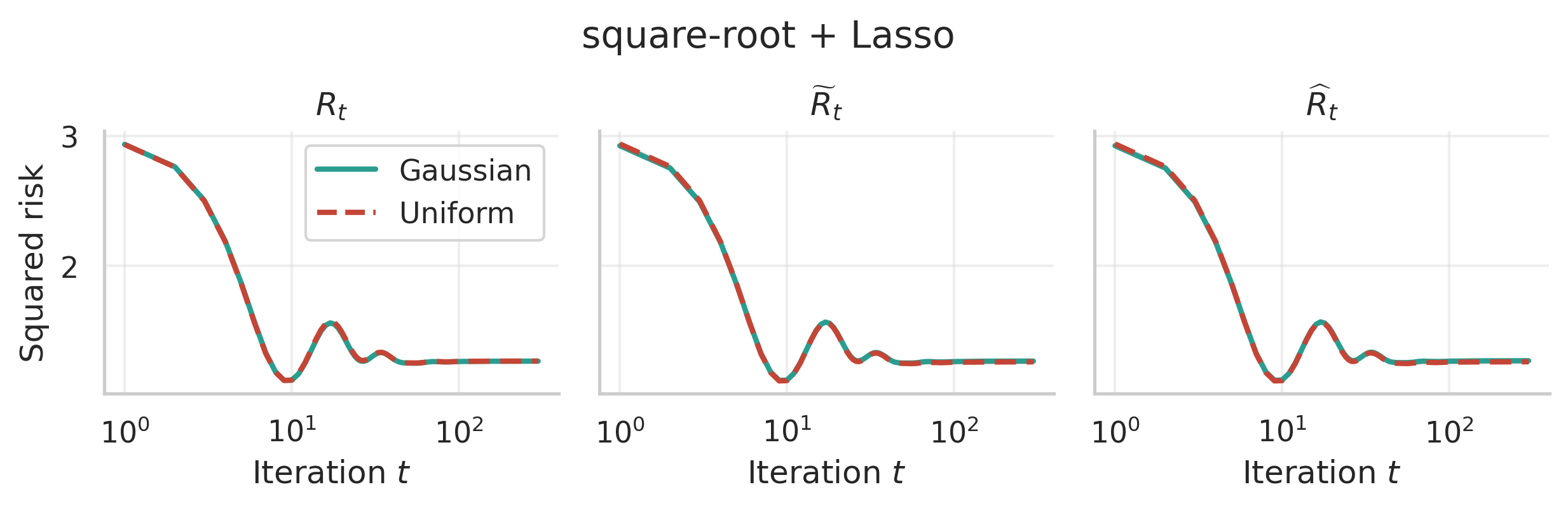}
\vspace{2mm}
\includegraphics[width=0.90\textwidth]{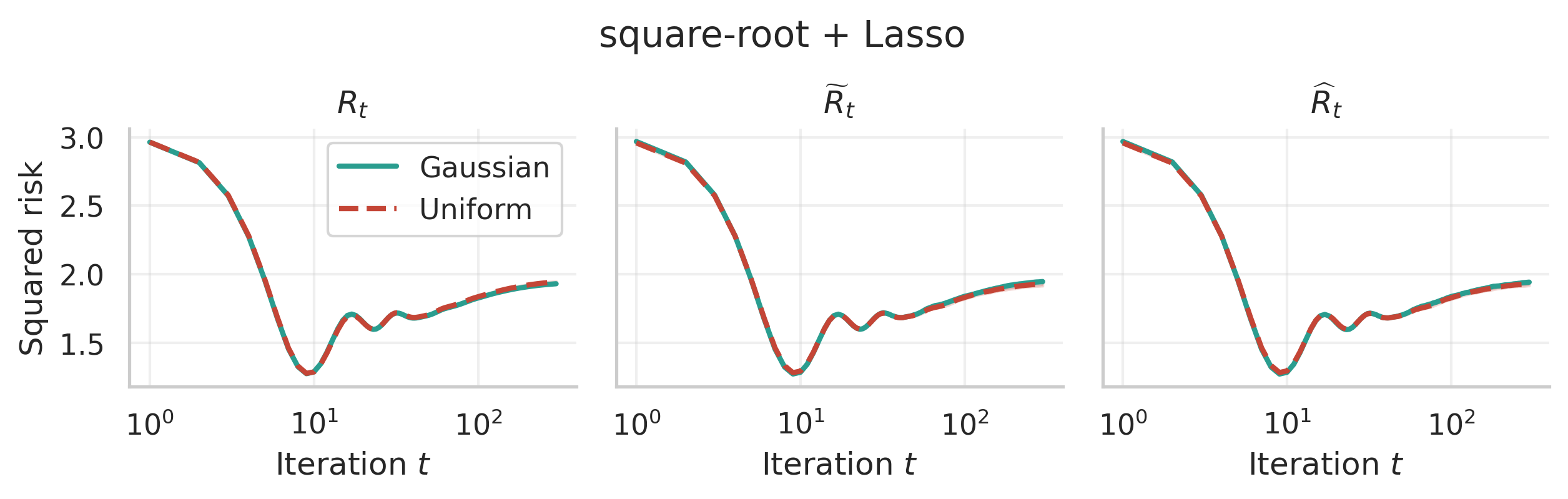}
\caption[Universality for square-root Lasso]{Gaussian-versus-uniform comparison for square-root Lasso with \((n,T)=(2000,300)\), using squared-error risk. The top panel uses \(p=1000\) and the bottom panel uses \(p=2400\). In each panel, the three subpanels compare \(\calR_t\), \(\widetilde{\calR}_t\), and \(\widehat{\calR}_t\).}
\label{fig:universality-square-root-lasso}
\end{figure}

\begin{figure}[p]
\centering
\includegraphics[width=0.90\textwidth]{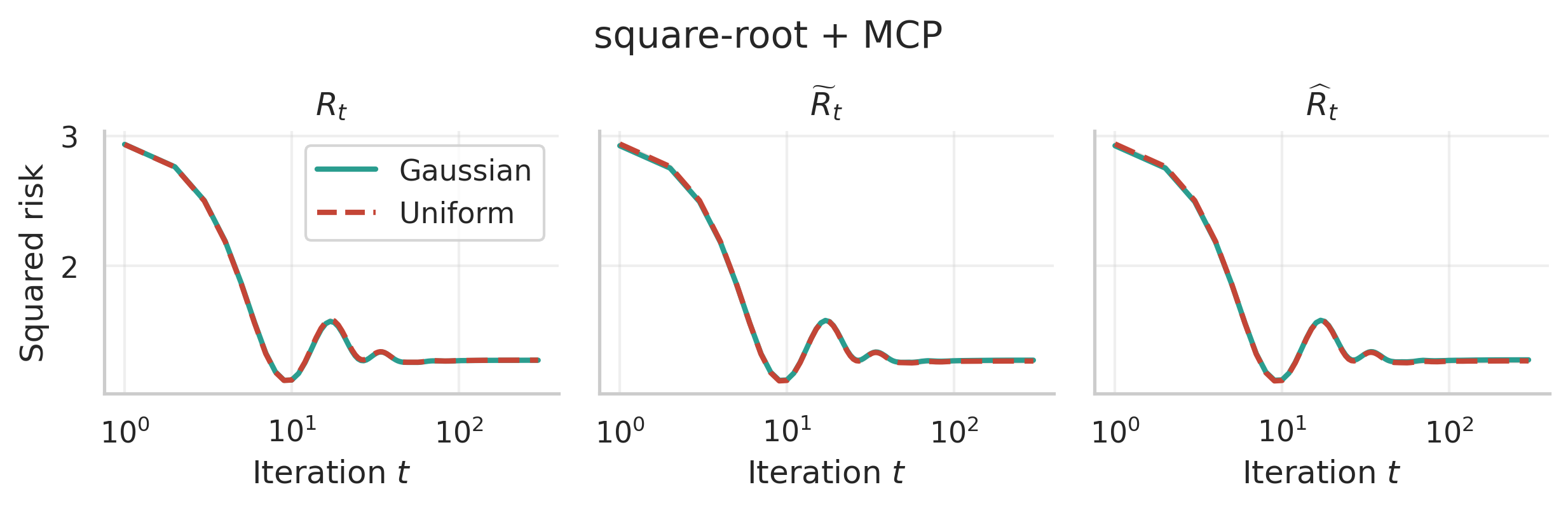}
\vspace{2mm}
\includegraphics[width=0.90\textwidth]{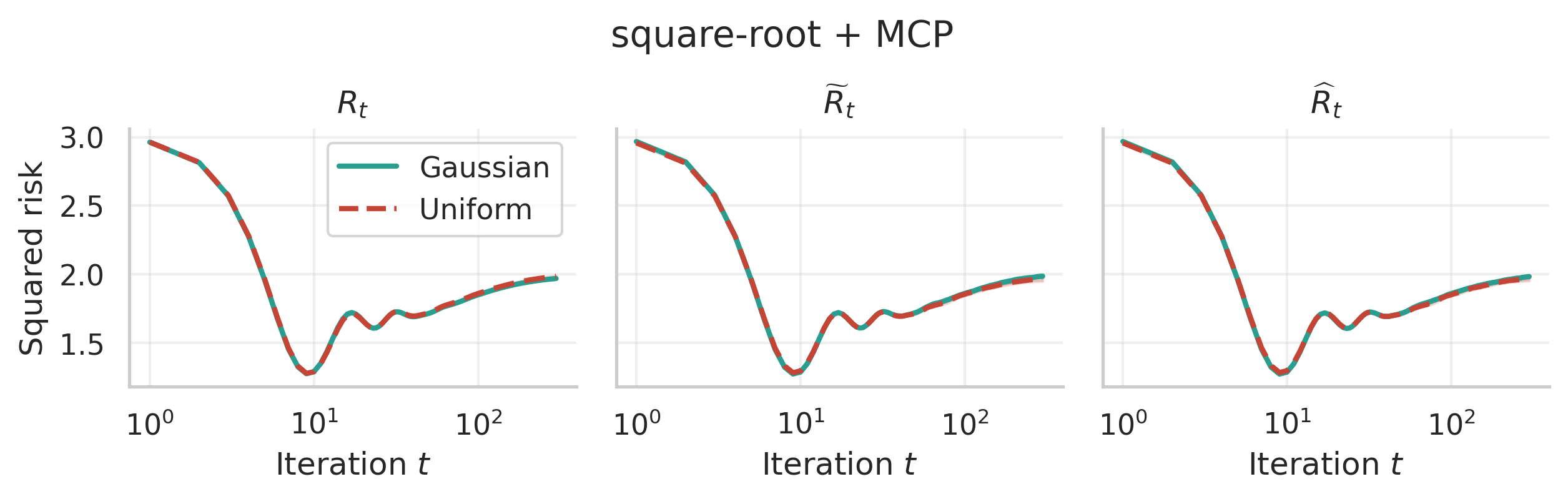}
\caption[Universality for square-root MCP]{Gaussian-versus-uniform comparison for square-root loss with the MCP penalty and \((n,T)=(2000,300)\), using squared-error risk. The top panel uses \(p=1000\) and the bottom panel uses \(p=2400\). In each panel, the three subpanels compare \(\calR_t\), \(\widetilde{\calR}_t\), and \(\widehat{\calR}_t\).}
\label{fig:universality-square-root-mcp}
\end{figure}

\begin{figure}[p]
\centering
\includegraphics[width=0.90\textwidth]{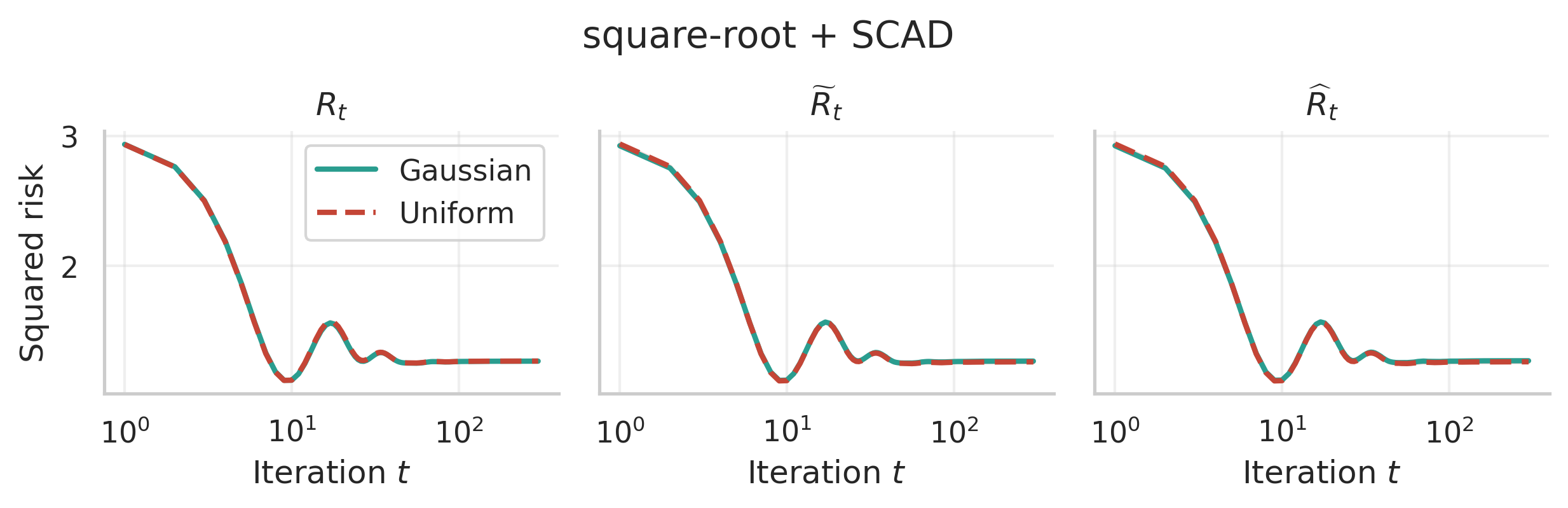}
\vspace{2mm}
\includegraphics[width=0.90\textwidth]{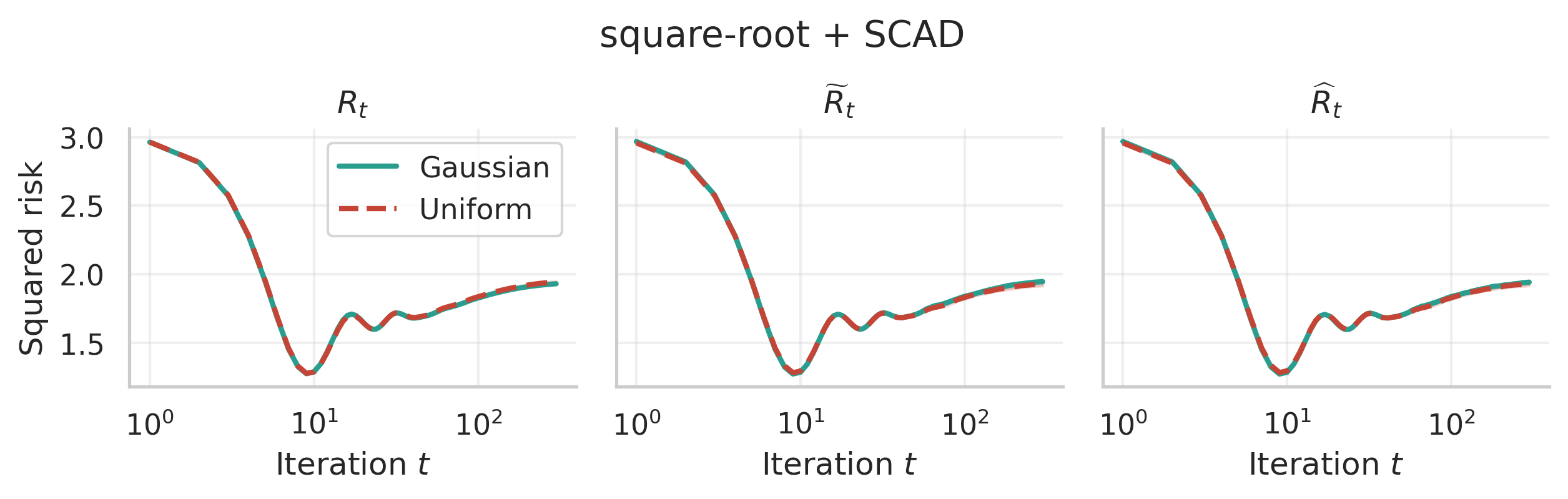}
\caption[Universality for square-root SCAD]{Gaussian-versus-uniform comparison for square-root loss with the SCAD penalty and \((n,T)=(2000,300)\), using squared-error risk. The top panel uses \(p=1000\) and the bottom panel uses \(p=2400\). In each panel, the three subpanels compare \(\calR_t\), \(\widetilde{\calR}_t\), and \(\widehat{\calR}_t\).}
\label{fig:universality-square-root-scad}
\end{figure}

\begin{figure}[p]
\centering
\includegraphics[width=0.90\textwidth]{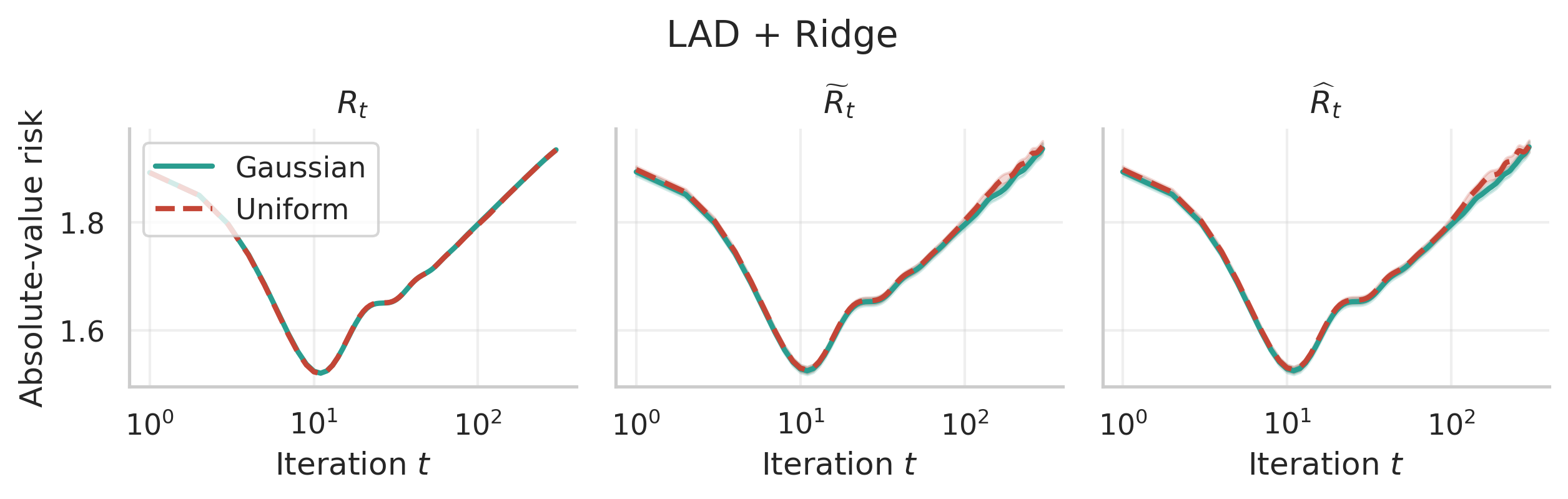}
\vspace{2mm}
\includegraphics[width=0.90\textwidth]{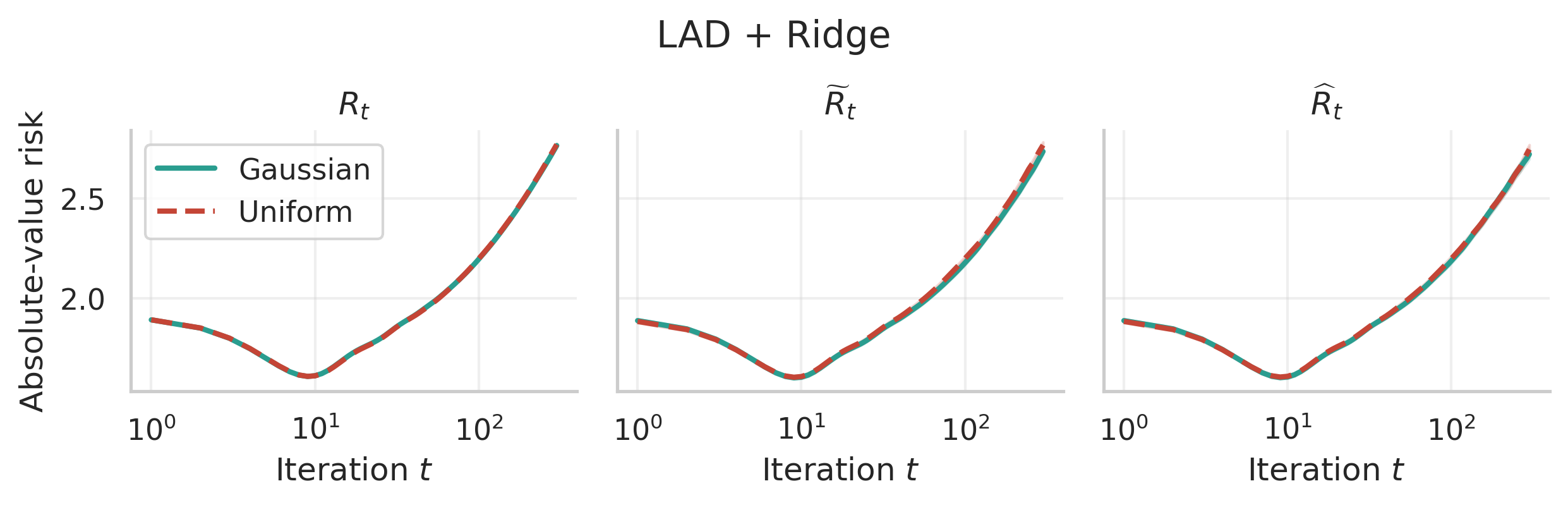}
\caption[Universality for LAD ridge]{Gaussian-versus-uniform comparison for LAD ridge with \((n,T)=(2000,300)\), using absolute-value risk. The top panel uses \(p=1000\) and the bottom panel uses \(p=2400\). In each panel, the three subpanels compare \(\calR_t\), \(\widetilde{\calR}_t\), and \(\widehat{\calR}_t\).}
\label{fig:universality-lad-ridge}
\end{figure}

\begin{figure}[p]
\centering
\includegraphics[width=0.90\textwidth]{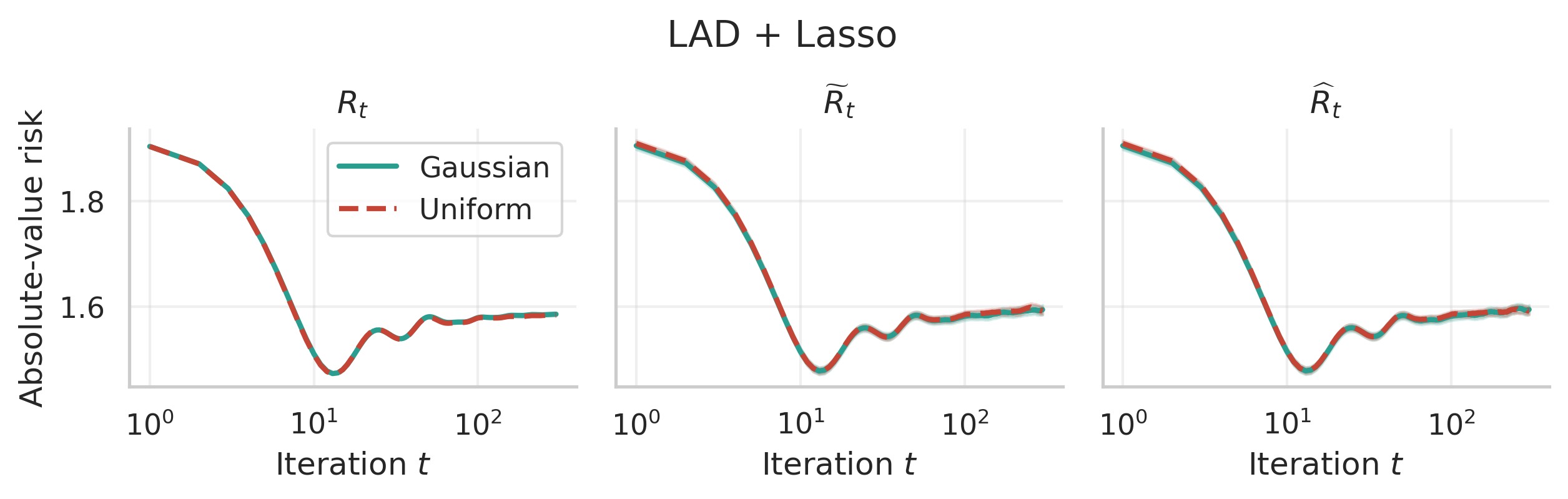}
\vspace{2mm}
\includegraphics[width=0.90\textwidth]{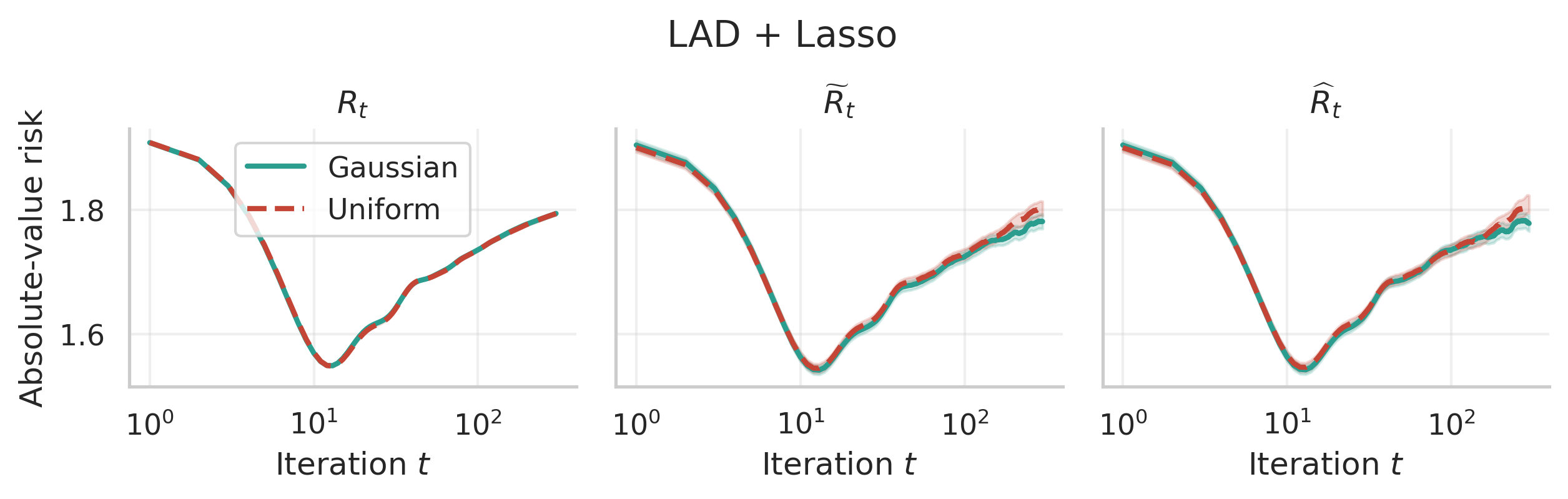}
\caption[Universality for LAD Lasso]{Gaussian-versus-uniform comparison for LAD Lasso with \((n,T)=(2000,300)\), using absolute-value risk. The top panel uses \(p=1000\) and the bottom panel uses \(p=2400\). In each panel, the three subpanels compare \(\calR_t\), \(\widetilde{\calR}_t\), and \(\widehat{\calR}_t\).}
\label{fig:universality-lad-lasso}
\end{figure}

\begin{figure}[p]
\centering
\includegraphics[width=0.90\textwidth]{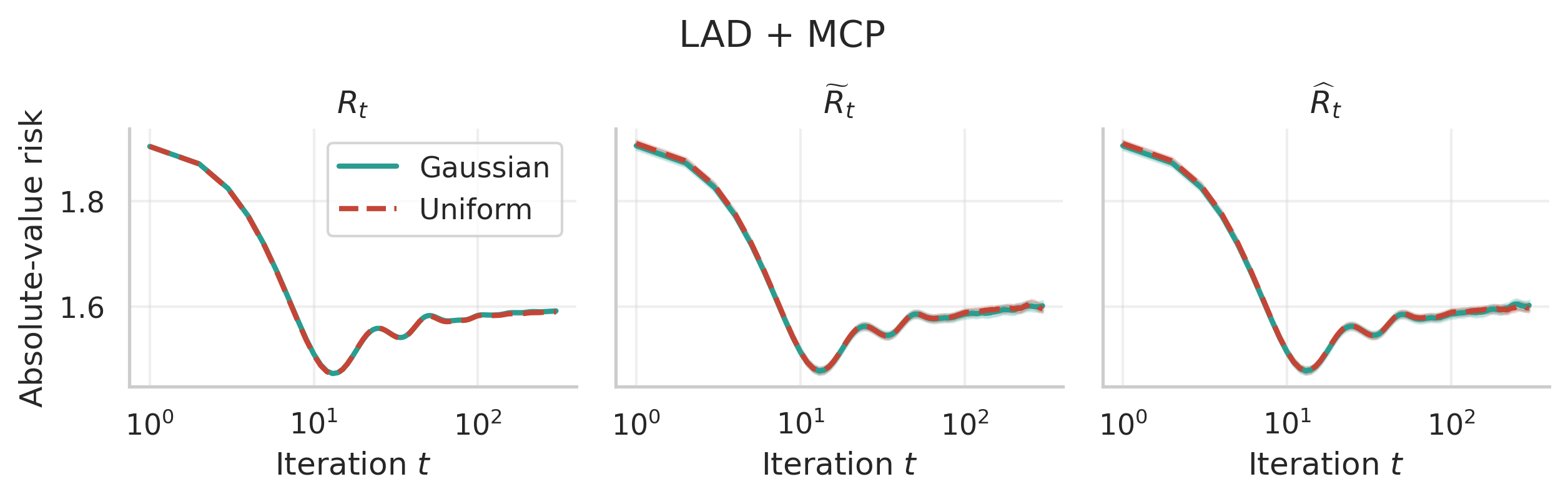}
\vspace{2mm}
\includegraphics[width=0.90\textwidth]{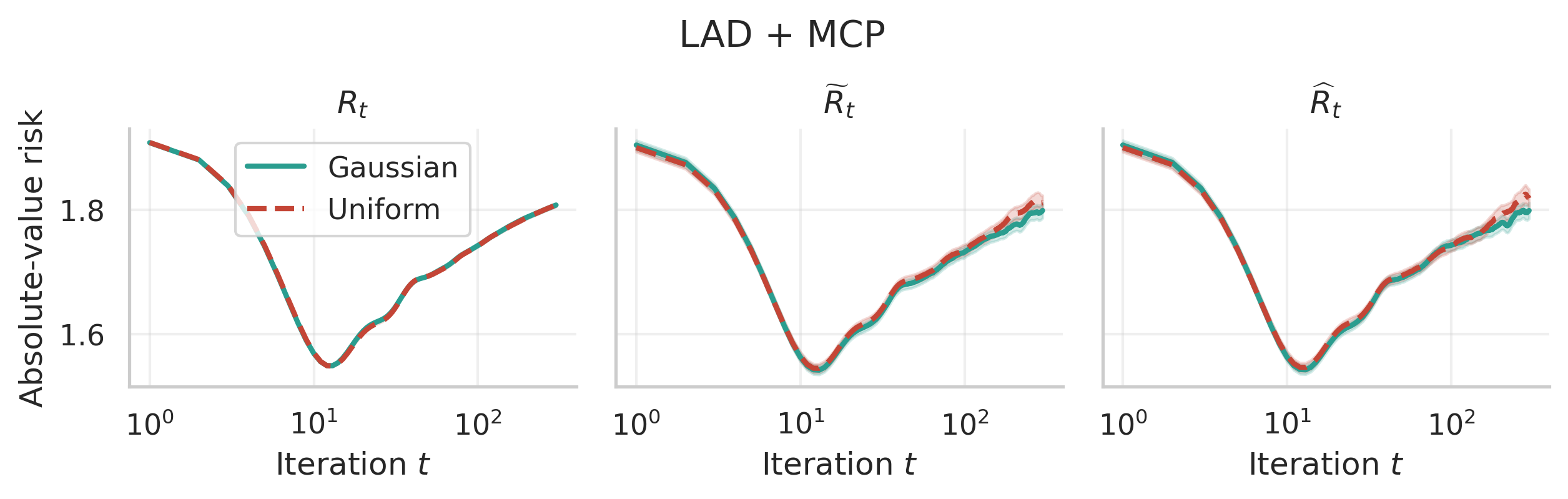}
\caption[Universality for LAD MCP]{Gaussian-versus-uniform comparison for LAD loss with the MCP penalty and \((n,T)=(2000,300)\), using absolute-value risk. The top panel uses \(p=1000\) and the bottom panel uses \(p=2400\). In each panel, the three subpanels compare \(\calR_t\), \(\widetilde{\calR}_t\), and \(\widehat{\calR}_t\).}
\label{fig:universality-lad-mcp}
\end{figure}

\begin{figure}[p]
\centering
\includegraphics[width=0.90\textwidth]{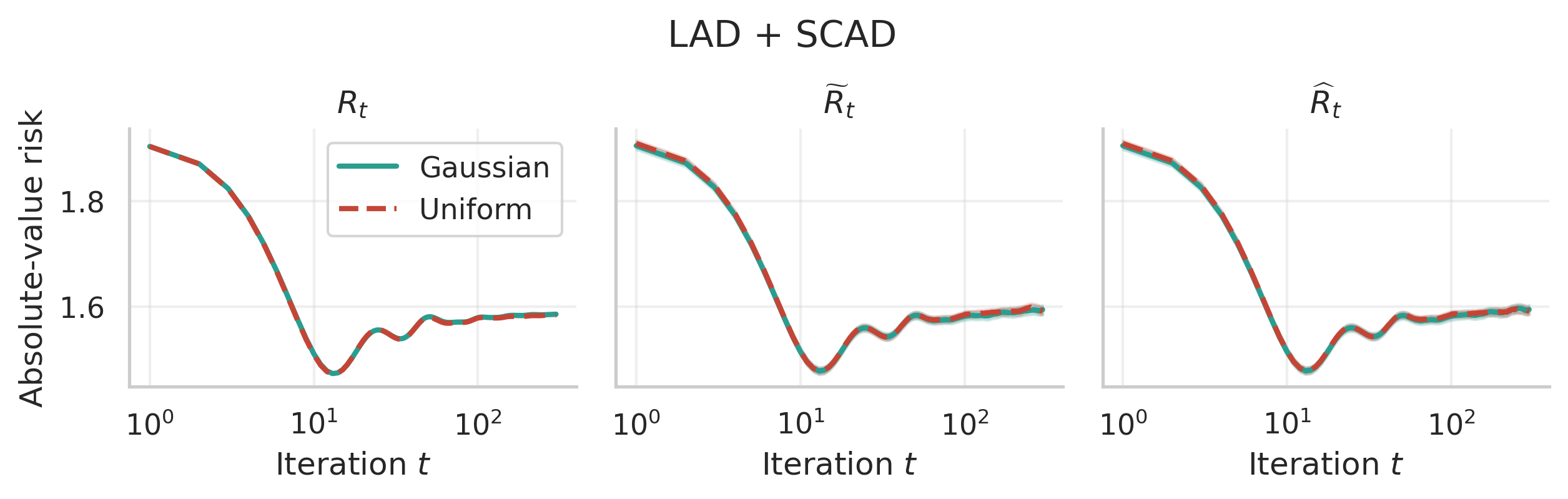}
\vspace{2mm}
\includegraphics[width=0.90\textwidth]{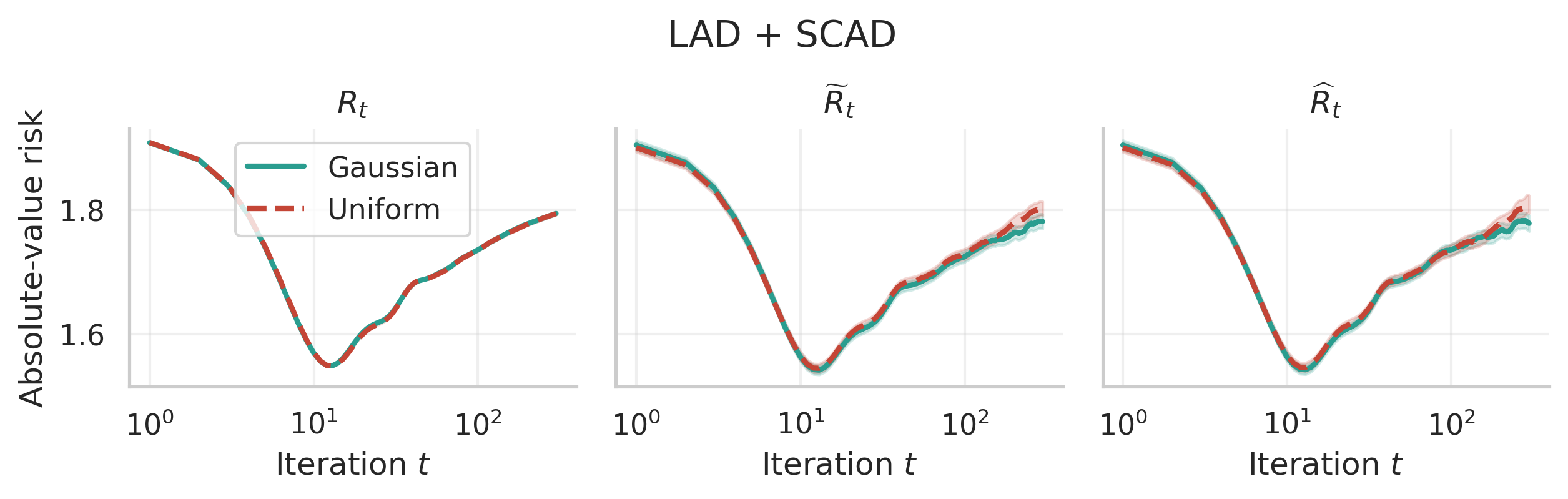}
\caption[Universality for LAD SCAD]{Gaussian-versus-uniform comparison for LAD loss with the SCAD penalty and \((n,T)=(2000,300)\), using absolute-value risk. The top panel uses \(p=1000\) and the bottom panel uses \(p=2400\). In each panel, the three subpanels compare \(\calR_t\), \(\widetilde{\calR}_t\), and \(\widehat{\calR}_t\).}
\label{fig:universality-lad-scad}
\end{figure}

\end{document}